\documentclass{amsart}

\usepackage{macros}
\standardsettings

\draftfalse

\newcommand{\limitset}{\mathcal J}
\newcommand{\rep}{\pi}
\newcommand{\refl}{\mathrm{ref \hspace{0 in} l}}

\begin{document}
\title[Diophantine approximation in hyperbolic metric spaces]{Diophantine approximation and the geometry of limit sets in Gromov hyperbolic metric spaces}

\authorlior\authordavid\authormariusz

\keywords{Diophantine approximation, Schmidt's game, hyperbolic geometry, Gromov hyperbolic metric spaces.}

\begin{Abstract}
In this paper, we provide a complete theory of Diophantine approximation in the limit set of a group acting on a Gromov hyperbolic metric space. This summarizes and completes a long line of results by many authors, from Patterson's classic '76 paper to more recent results of Hersonsky and Paulin ('02, '04, '07). Concrete examples of situations we consider which have not been considered before include geometrically infinite Kleinian groups, geometrically finite Kleinian groups where the approximating point is not a fixed point of any element of the group, and groups acting on infinite-dimensional hyperbolic space. Moreover, in addition to providing much greater generality than any prior work of which we are aware, our results also give new insight into the nature of the connection between Diophantine approximation and the geometry of the limit set within which it takes place. Two results are also contained here which are purely geometric: a generalization of a theorem of Bishop and Jones ('97) to Gromov hyperbolic metric spaces, and a proof that the uniformly radial limit set of a group acting on a proper geodesic Gromov hyperbolic metric space has zero Patterson--Sullivan measure unless the group is quasiconvex-cocompact. The latter is an application of a Diophantine theorem.
\end{Abstract}
\maketitle

\tableofcontents

\draftnewpage\section{Introduction}
Four of the greatest classical theorems of Diophantine approximation are Dirichlet's theorem on the approximability of every point with respect to the function $1/q^2$ (supplemented by Liouville's result regarding the optimality of this function), Jarn\'ik's theorem stating that the set of badly approximable numbers has full dimension in $\R$ (which we will henceforth call the Full Dimension Theorem), Khinchin's theorem on metric Diophantine approximation, and the theorem of Jarn\'ik and Besicovitch regarding the dimension of certain sets of well approximable numbers. As observed S. J. Patterson \cite{Patterson1}, these theorems can be put in the context of Fuchsian groups by noting that the set of rational numbers is simply the orbit of $\infty$ under the Fuchsian group $\SL_2(\Z)$ acting by M\"obius transformations on the upper half-plane $\H^2$. In this context, they admit natural generalizations by considering a different group, a different orbit, or even a different space. Beginning with Patterson, many authors \cite{BDV, DMPV, HillVelani, Patterson1, Stratmann3, Stratmann4, StratmannVelani, Velani1, Velani2, Velani3} have considered generalizations of these theorems to the case of a nonelementary geometrically finite group acting on standard hyperbolic space $X = \H^{d + 1}$ for some $d\in\N$, considering a parabolic fixed point if $G$ has at least one cusp, and a hyperbolic fixed point otherwise. Moreover, many theorems in the literature concerning the asymptotic behavior of geodesics in geometrically finite manifolds \cite{Aravinda, AravindaLeuzinger, Dani3, FernandezMelian, McMullen_absolute_winning, MelianPestana, Stratmann1, Sullivan_disjoint_spheres} can be recast in terms of Diophantine approximation. To date, the most general setup is that of S. D. Hersonsky and F. Paulin, who generalized Dirichlet and Khinchin's theorems to the case of pinched Hadamard manifolds \cite{HP_Dirichlet, HP_Khinchin} and proved an analogue of Khinchin's theorem for uniform trees \cite{HP_Khinchin_trees}; nevertheless, they still assume that the group is geometrically finite, and they only approximate by parabolic orbits.\Footnote{Although they do not fall into our framework, let us mention in passing Hersonsky and Paulin's paper \cite{HP_Jarnik-Besicovitch_internal}, which estimates the dimension of the set of geodesic rays in a pinched Hadamard manifold $X$ which return exponentially close to a given point of $X$ infinitely often, and the two papers \cite{AGP, KleinbockMargulis}, which generalize Khinchin's theorem in a somewhat different direction than us.}


In this paper we provide a far-reaching generalization of these results to the situation of a strongly discrete group $G$ of general type acting by isometries on an arbitrary Gromov hyperbolic metric space $X$ (so that $X = \H^{d + 1}$ is a special case), and we consider the approximation of a point $\eta$ in the radial limit set of $G$ by an arbitrary point $\xi$ in the Gromov boundary of $X$. Thus we are generalizing
\begin{itemize}
\item the type of group being considered - we do not assume that $G$ is geometrically finite,
\item the space being acted on - we do not even assume that $X$ is proper or geodesic, and
\item the point being approximated with - we do not assume that $\xi$ is a parabolic or hyperbolic fixed point of $G$.
\end{itemize}
We note that any one of these generalizations would be new by itself. For example, although the Fuchsian group $\SL_2(\Z)$ is a geometrically finite group acting on $\H^2$, the approximation of points in its limit set by an orbit other than the orbit of $\infty$ has not been considered before; cf. Examples \ref{exampleSL2Zjarniknew} and \ref{exampleSL2Zkhinchinnew}. As another example, we remark that losing the properness assumption allows us to consider infinite-dimensional hyperbolic space as a concrete example of a space for which our theorems have not previously been proven. There are a large variety of groups acting on infinite-dimensional hyperbolic spaces which are not reducible in any way to finite-dimensional examples, see e.g. Appendix \ref{appendixorbitalcounting}.

In addition to generalizing known results, the specific form of our theorems sheds new light on the interplay between the geometry of a group acting on a hyperbolic metric space and the Diophantine properties of its limit set. In particular, our formulation of the Jarn\'ik--Besicovitch theorem (Theorem \ref{theoremjarnikbesicovitch}) provides a clear relation between the density of the limit set around a particular point (measured as the rate at which the Hausdorff content of the radial limit set shrinks as smaller and smaller neighborhoods of the point are considered) and the degree to which other points can be approximated by it. Seeing how this reduces to the known result of R. M. Hill and S. L. Velani \cite{HillVelani} is instructive; cf. Theorem \ref{theoremboundedparabolicjarnik} and its proof. The same is true for our formulation of Khinchin's theorem (Theorem \ref{theoremkhinchin}), although to a lesser extent since we are only able to provide a complete analogue of Khinchin's theorem in special cases (which are nevertheless more general than previous results). Although our formulation of the Full Dimension Theorem (Theorem \ref{theoremfulldimension}) is more or less the same as previous results (we have incorporated the countable intersection property of winning sets into the statement of the theorem), the geometrically interesting part here is the proof. It depends on constructing Ahlfors regular subsets of the uniformly radial limit set with dimension arbitrarily close to the Poincar\'e exponent $\delta$ (Theorem \ref{theorembishopjones}); in this way, we also generalize the seminal theorem of C. J. Bishop and P. W. Jones \cite{BishopJones} which states that the radial and uniformly radial limit sets $\Lr(G)$ and $\Lur(G)$ each have dimension $\delta$. Theorem \ref{theorembishopjones} will appear in a stronger form in \cite[Theorem 1.2.1]{DSU}, whose authors include the second- and third-named authors of this paper.\Footnote{Let us make it clear from the start that although we cite the preprint \cite{DSU} frequently for additional background, the results which we quote from \cite{DSU} are not used in any of our proofs.}

Finally, as an application of our generalization of Khinchin's theorem we show that for any group acting on a proper geodesic hyperbolic metric space, the uniformly radial limit set has zero Patterson--Sullivan measure if and only if the group is not quasiconvex-cocompact (Proposition \ref{propositionmuLurzero}). In the case where $X = \H^{d + 1}$ and $G$ is geometrically finite, this follows from the well-known ergodicity of the geodesic flow \cite[Theorem 1$'$]{Sullivan_entropy}, see also \cite[Th\'eor\`eme 1.7]{Roblin1}.

{\bf Convention 1.} In the introduction, propositions which are proven later in the paper will be numbered according to the section they are proven in. Propositions numbered as 1.\# are either straightforward, proven in the introduction, or quoted from the literature.

{\bf Convention 2.} The symbols $\lesssim$, $\gtrsim$, and $\asymp$ will denote asymptotics; a subscript of $\plus$ indicates that the asymptotic is additive, and a subscript of $\times$ indicates that it is multiplicative. For example, $A\lesssim_{\times,K} B$ means that there exists a constant $C > 0$ (the \emph{implied constant}), depending only on $K$, such that $A\leq C B$. $A\lesssim_{\plus,\times}B$ means that there exist constants $C_1,C_2 > 0$ so that $A\leq C_1 B + C_2$. In general, dependence of the implied constant(s) on universal objects such as the metric space $X$, the group $G$, and the distinguished points $\zero\in X$ and $\xi\in\del X$ will be omitted from the notation.

{\bf Convention 3.} The notation $x_n\tendsto n x$ means $x_n\to x$ as $n\to \infty$. The notation $x_n\tendsto{n,\plus} x$ means
\[
x \asymp_\plus \limsup_{n\to\infty}x_n \asymp_\plus \liminf_{n\to\infty} x_n,
\]
and similarly for $x_n\tendsto{n,\times} x$.

{\bf Convention 4.} Let $G$ be a geometrically finite subgroup of $\Isom(\H^{d + 1})$ for some $d\in\N$. In the literature, authors have often considered ``a parabolic fixed point of $G$ if such a point exists, and a hyperbolic fixed point otherwise''. To avoid repeating this long phrase we shall call such a point a \emph{conventional} point.

{\bf Convention 5.} The symbol $\triangleleft$ will be used to indicate the end of a nested proof.
\\

{\bf Acknowledgements.} The first-named author was supported in part by the Simons Foundation grant \#245708. The third-named author was supported in part by the NSF grant DMS-1001874. The authors thank two referees who gave valuable suggestions and comments on earlier versions of this paper.

\subsection{Preliminaries}

\subsubsection{Visual metrics}
Let $(X,\dist)$ be a hyperbolic metric space (see Definition \ref{definitiongromovhyperbolic}). In the theorems below, we will consider its \emph{Gromov boundary} $\del X$ (see Definition \ref{definitiongromovboundary}), together with a \emph{visual metric} $\Dist = \Dist_{b,\zero}$ on $\del X$ satisfying
\begin{repequation}{distanceasymptotic}
\Dist_{b,\zero}(\xi,\eta) \asymp_\times b^{-\lb\xi|\eta\rb_\zero},
\end{repequation}
where $\zero\in X$ and $b > 1$, and where $\lb\cdot|\cdot\rb$ denotes the Gromov product (see Definition \ref{definitiongromovproduct}). For $b > 1$ sufficiently close to $1$, such a metric exists for all $\zero\in X$ (Proposition \ref{propositionDist}). We remark that when $X$ is the ball model of $\H_\R^{d + 1}$, $b = e$, and $\zero = \0$, then we have $\del X = S^d$ and we may take
\begin{equation}
\label{euclideanball}
\Dist_{e,\0}(\xx,\yy) = \frac{1}{2}\|\yy - \xx\| \all \xx,\yy\in S^d.
\end{equation}
(See Proposition \ref{propositioneuclideanball} below.)

\subsubsection{Strongly discrete groups}
Let $(X,\dist)$ be a metric space, and let $G\leq\Isom(X)$.\Footnote{Here and from now on $\Isom(X)$ denotes the group of isometries of a metric space $(X,\dist)$. The notation $G\leq\Isom(X)$ means that $G$ is a subgroup of $\Isom(X)$.} There are several notions of what it means for $G$ to be discrete, which are inequivalent in general but agree in the case $X = \H^{d + 1}$; see \cite[\65]{DSU} for details. For the purposes of this paper, the most useful notion is \emph{strong discreteness}.
\begin{definition}
A group $G$ is \emph{strongly discrete} if for every $r > 0$, the set
\[
\{g\in G: g(\zero)\in B(\zero,r)\}
\]
is finite, where $\zero\in X$ is a distinguished point.\Footnote{Here and from now on $B(x,r)$ denotes the closed ball centered at $x$ of radius $r$.}
\end{definition}
\begin{remark}
Strongly discrete groups are known in the literature as \emph{metrically proper}. However, we prefer the term ``strongly discrete'', among other reasons, because it emphasizes that the notion is a generalization of discreteness.
\end{remark}

\subsubsection{Groups of general type}
\begin{definition}
An isometry $g$ of a hyperbolic metric space $(X,\dist)$ is \emph{loxodromic} if it has two fixed points $g_+$ and $g_-$ in $\del X$, which are attracting and repelling in the sense that for every neighborhood $U$ of $g_-$, the iterates $g^n$ converge to $g_+$ uniformly on $\del X\setminus U$, and if $g^{-1}$ has the same property.
\end{definition}

\begin{definition}
Let $(X,\dist)$ be a hyperbolic metric space. A group $G\leq\Isom(X)$ is \emph{of general type}\Footnote{This terminology was introduced in \cite{CCMT} to describe the cases of a classification of Gromov \cite[\sectionsymbol 3.1]{Gromov3}.\label{footnoteCCMT}} if it contains two loxodromic isometries $g_1$ and $g_2$, such that the fixed points of $g_1$ are disjoint from the fixed points of $g_2$.
\end{definition}

As in the case of $X = \H^{d + 1}$, we say that a group $G\leq\Isom(X)$ is \emph{nonelementary} if its limit set (see Definition \ref{definitionlimitset}) contains at least three points.\Footnote{This definition is \emph{not} equivalent to the also common definition of a nonelementary group as one which does not preserve any point or geodesic in the compactification/bordification $\bord X$; rather, a group satisfies this second definition if and only if it is of general type.} Since the fixed points of a loxodromic isometry are clearly in the limit set, we have:

\begin{observation}
Every group of general type is nonelementary.
\end{observation}
The converse, however, is not true; there exists groups which are nonelementary and yet are not of general type. Such a group has a unique fixed point on the boundary of $X$, but it is not the only point in the limit set. A good example is the subgroup of $\Isom(\H^2)$ generated by the isometries
\[
g_1(z) := 2z,\hspace{.5 in} g_2(z) := z + 1.
\]
A group which is nonelementary but not of general type is called a \emph{focal} group.\ignorespaces
\Footnote{See footnote \ref{footnoteCCMT}.}
\comdavid{If the pagebreaks move, switch which of the preceding lines is commented out.}
\begin{proposition}[{\cite[Proposition 6.4.1]{DSU}}] \label{propositiongeneraltype}
Every nonelementary strongly discrete group is of general type, or in other words, no strongly discrete group is focal.
\end{proposition}

\begin{remark}
Due to Proposition \ref{propositiongeneraltype}, in our standing assumptions below we could replace the hypothesis that $G$ is of general type with the hypothesis that $G$ is nonelementary. We have elected not to do so since we would prefer this paper to be as self-contained as possible.
\end{remark}

\begin{observation}
\label{observationnoglobalfixedpoint}
If $G$ is a group of general type then $G$ has no global fixed points; i.e. for every $\xi\in\del X$ there exists $h\in G$ such that $h(\xi)\neq\xi$.
\end{observation}

\subsubsection{Standing assumptions}\label{standingassumptions}
In the statements below, we will have the following standing assumptions:
\begin{itemize}
\item[(I)] $(X,\dist)$ is a hyperbolic metric space
\item[(II)] $\zero\in X$ is a distinguished point
\item[(III)] $b > 1$ is a parameter close enough to $1$ to guarantee the existence of a visual metric $\Dist = \Dist_{b,\zero}$
\item[(IV)] $G\leq\Isom(X)$ is a strongly discrete group of general type.
\end{itemize}
If $X = \H^{d + 1} = \{(x_1,\ldots,x_{d + 1}):x_{d + 1} > 0\}$ for some $d\in\N$, then we will moreover assume that $\dist = \frac{\|\dee\xx\|}{x_{d + 1}}$, $\zero = \ee_{d + 1}$, and $b = e$. Similarly, if $X = \B^{d + 1}$, we will assume that $\dist = \frac{2\|\dee\xx\|}{1 - \|\xx\|^2}$, $\zero = \0$, and $b = e$.

\subsection{The Bishop--Jones theorem and its generalization}
Before stating our main results about Diophantine approximation, we present the Bishop--Jones theorem, a theorem which is crucial for understanding the geometry of the limit set of a Kleinian group, and its generalization to hyperbolic metric spaces.

\begin{theorem}[C. J. Bishop and P. W. Jones, \cite{BishopJones}]
\label{theorembishopjonesclassical}
Fix $d\in\N$, and let $G$ be a nonelementary discrete subgroup of $\Isom(\H^{d + 1})$. Let $\delta$ be the Poincar\'e exponent of $G$, and let $\Lr$ and $\Lur$ be the radial and uniformly radial limit sets of $G$, respectively. Then
\begin{equation}
\label{bishopjones}
\HD(\Lr) = \HD(\Lur) = \delta.\HDfootnote
\end{equation}
\end{theorem}

The following theorem is a generalization of the Bishop--Jones theorem:
\begin{reptheorem}{theorembishopjones}
Let $(X,\dist,\zero,b,G)$ be as in \sectionsymbol\ref{standingassumptions}.  Let $\delta$ be the Poincar\'e exponent of $G$ relative to $b$ (see \eqref{poincareexponent}), and let $\Lr$ and $\Lur$ be the radial and uniformly radial limit sets of $G$, respectively (see Definition \ref{definitionlimitset}). Then \eqref{bishopjones} holds; moreover, for every $0 < s < \delta$ there exists an Ahlfors $s$-regular (see Definition \ref{definitionahlforsregular}) set $\limitset_s\subset\Lur$.
\end{reptheorem}

The ``moreover'' clause is new even in the case which Bishop and Jones considered, demonstrating that the limit set $\Lur$ can be approximated by subsets which are particularly well distributed from a geometric point of view. It does not follow from their theorem since it is possible for a set to have large Hausdorff dimension without having any closed Ahlfors regular subsets of positive dimension (much less full dimension); in fact it follows from the work of D. Y. Kleinbock and B. Weiss \cite{KleinbockWeiss1} that the set of well approximable numbers forms such a set.\Footnote{It could be objected that this set is not closed and so should not constitute a counterexample; however, since it has full measure, it has closed subsets of arbitrarily large measure (which in particular still have dimension 1).}

\begin{remark}
In Theorem \ref{theorembishopjones} and in Theorem \ref{theoremfulldimension} below, the case $\delta = \infty$ is possible; see \cite[\613.2]{DSU} for spme examples.
\end{remark}

\begin{remark}
The equation $\HD(\Lr) = \delta$ was proven in \cite{Paulin} if $X$ is a word-hyperbolic group, a CAT(-1) manifold,\Footnote{See e.g. \cite[Part II]{BridsonHaefliger} for the definition of a CAT(-1) space.} or a uniform tree.\Footnote{A locally finite tree is \emph{uniform} if it admits a discrete cocompact action.}
\end{remark}

\begin{remark}
\label{remarkmayedamerrill}
A weaker form of the ``moreover'' clause was cited in \cite{MayedaMerrill} as a private communication from the authors of this paper \cite[Lemma 5.2]{MayedaMerrill}.
\end{remark}



As an application of Theorem \ref{theorembishopjones}, we deduce the following result:
\begin{reptheorem}{theoremextrinsic}
Fix $d\in\N$, let $X = \H^{d + 1}$, and let $G\leq\Isom(X)$ be a discrete group acting irreducibly on $X$. Then
\[
\HD(\BA_d\cap \Lur) = \delta = \HD(\Lr),
\]
where $\BA_d$ denotes the set of badly approximable vectors in $\R^d$.
\end{reptheorem}
For the definitions and proof see Section \ref{sectionextrinsic}.

\subsection{Dirichlet's theorem, generalizations, and optimality}
The most fundamental result in Diophantine approximation is the following:

\begin{theorem}[Dirichlet 1842]
\label{theoremdirichletclassical}
For every irrational $x\in\R$,
\[
\label{dirichletclassical}
\left|x - \frac{p}{q}\right| < \frac{1}{q^2}\textup{ for infinitely many }p/q\in\Q.
\]
\end{theorem}
Dirichlet's theorem can be viewed (up to a multiplicative constant) as the special case of the following theorem which occurs when
\begin{equation}
\label{SL2Z}
X = \H^2, \hspace{.5 in}
G = \SL_2(\Z), \hspace{.5 in}
\xi = \infty:
\end{equation}
\begin{reptheorem}{theoremdirichlet}[Generalization of Dirichlet's Theorem]
Let $(X,\dist,\zero,b,G)$ be as in \sectionsymbol\ref{standingassumptions}, and fix a distinguished point $\xi\in\del X$. Then for every $\sigma > 0$, there exists a number $C = C_\sigma > 0$ such that for every $\eta\in\Lrsigma(G)$,
\begin{repequation}{dirichlet}
\Dist(g(\xi),\eta) \leq Cb^{-\dist(\zero,g(\zero))}\textup{ for infinitely many }g\in G.
\end{repequation}
Here $\Lrsigma(G)$ denotes the $\sigma$-radial limit set of $G$ (see Definition \ref{definitionlimitset}).

Furthermore, there exists a sequence of isometries $(g_n)_1^\infty$ in $G$ such that
\[
g_n(\xi) \tendsto n \eta \text{ and } g_n(\zero) \tendsto n \eta,
\]
and such that \eqref{dirichlet} holds with $g = g_n$ for all $n\in\N$.
\end{reptheorem}

\begin{remark}
In some sense, it is not surprising that a point in the radial limit set should be able to be approximated well, since by definition a radial limit point already has a sequence of something resembling ``good approximations'', although they are in the interior rather than in the boundary. (Indeed, radial limit points are sometimes called \emph{points of approximation}.) The proof of Theorem \ref{theoremdirichlet} is little more than simply taking advantage of this fact. Nevertheless, Theorem \ref{theoremdirichlet} is important because of its place in our general framework; in particular, its optimality (Theorem \ref{theoremfulldimension} below) is not obvious.\comdavid{Is it clear?}
\end{remark}

The fact that Theorem \ref{theoremdirichlet} reduces to Dirichlet's theorem in the special case \eqref{SL2Z} can be deduced from the following well-known observations:

\begin{observation}
\label{observationSL2Z}
The orbit of $\infty$ under $\SL_2(\Z)$ is precisely the extended rationals $\what\Q$. For each $K > 0$ and for each $p/q\in\Q$ in reduced form with $|p/q|\leq K$ we have
\begin{equation}
\label{q2edg}
q^2 \asymp_{\times,K} \min_{\substack{g\in G \\ g(\infty) = p/q}}e^{\dist(\zero,g(\zero))}.
\end{equation}
\end{observation}
\begin{proof}
It is well-known and easy to check that the Ford circles $\big(B_\euc((p/q,1/(2q^2)),1/(2q^2))\big)_{p/q\in\Q}$ together with the set $H := \{(x_1,x_2):x_2 > 1\}$ constitute a disjoint $G = \SL_2(\Z)$-invariant collection of horoballs. (Here a subscript of $\euc$ denotes the Euclidean metric on $\R^2$.) Now note that
\[
\dist_\mathrm H(\del H,G_\infty(\zero)) < \infty,
\]
where $\dist_\mathrm H$ denotes the Hausdorff metric induced by the hyperbolic metric on $\H^2$, and $G_\infty$ is the stabilizer of $\infty$ in $G$. Thus for all $g\in G$
\[
\dist_\mathrm H(\del g(H),\{\w g\in G:\w g(\infty) = g(\infty)\}) \asymp_\plus 0;
\]
letting $p/q = g(\infty)$,
\begin{align*}
\min_{\substack{\w g\in G \\ \w g(\infty) = p/q}}\dist(\zero,\w g(\zero))
&=_\pt \dist(\zero,\{\w g\in G:\w g(\infty) = g(\infty)\})\\
&\asymp_\plus \dist(\zero,\del g(H))\\
&=_\pt \dist(\zero,B_\euc((p/q,1/(2q^2)),1/(2q^2)))\\
&\asymp_\plus \log(q^2) + \log(1 + (p/q)^2)
\asymp_{\plus,K} \log(q^2).
\end{align*}
Exponentiating finishes the proof.
\end{proof}

\begin{observation}
\label{observationSL2Z two}
In the special case \eqref{SL2Z}, the visual metric $\Dist$ is equal to the spherical metric $\Dist_\sph = |\dee x|/(1 + |x|^2)$, and is therefore bi-Lipschitz equivalent to the Euclidean metric $\Dist_\euc = |\dee x|$ on compact subsets of $\R$.
\end{observation}
\begin{proof}
Apply the Cayley transform to \eqref{euclideanball} above.
\end{proof}

\begin{observation}
\label{observationSL2Z three}
The radial limit set of $\SL_2(\Z)$ is precisely the set of irrational numbers; in fact $\R\butnot\Q = \Lrsigma$ for all $\sigma > 0$ sufficiently large.
\end{observation}
\begin{proof}
The proof of (GF1 \implies GF2) in \cite{Bowditch_geometrical_finiteness} shows that for any geometrically finite group $G$, there exists $\sigma > 0$ such that $\Lambda = \Lrsigma \sqcup \Lbp$, where $\Lbp$ is the set of (bounded) parabolic points of $G$. (Here $\sqcup$ denotes a disjoint union.) In particular, when $G = \SL_2(\Z)$, we have $\Lambda = \what\R$ and $\Lbp = \what\Q$, so $\R\butnot\Q = \Lrsigma$.
\end{proof}

Note that in deducing Theorem \ref{theoremdirichletclassical} from Theorem \ref{theoremdirichlet}, the numerator of (\ref{dirichletclassical}) must be replaced by a constant.\Footnote{This constant does not depend on $x$ by Observation \ref{observationSL2Z three}.} This sacrifice is necessary since the constant $1$ depends on information about $\SL_2(\Z)$ which is not available in the general case.

\ \newline{\bf History of Theorem \ref{theoremdirichlet}.}
Theorem \ref{theoremdirichlet} was previously proven
\begin{itemize}
\item in the case where $X = \H^2$, $G$ is finitely generated, and with no assumption on $\xi$ by S. J. Patterson in 1976 \cite[Theorem 3.2]{Patterson1}.
\item in the case where $X = \H^{d + 1}$, $G$ is geometrically finite with exactly one cusp,\Footnote{If $G$ has more than one cusp, then Stratmann and Velani (and later Hersonsky and Paulin) essentially prove that for every $\eta\in\Lr(G)$, there exists $\xi\in P$ such that \eqref{dirichlet} holds, where $P$ is a maximal set of inequivalent parabolic points.} and $\xi$ is a parabolic fixed point by B. O. Stratmann and S. L. Velani in 1995 \cite[Theorem 1]{StratmannVelani}.
\item in the case where $X$ is a pinched Hadamard manifold, $G$ is geometrically finite with exactly one cusp, and $\xi$ is a parabolic fixed point by S. D. Hersonsky and F. Paulin in 2002 \cite[Theorem 1.1]{HP_Dirichlet}.\Footnote{It should be noted that the metric which Hersonsky and Paulin use, the Hamenst\"adt metric, corresponds to the visual metric coming from $\xi$ rather than the visual metric coming from a point in $X$; however, these metrics are locally bi-Lipschitz equivalent on $\del X\butnot\{\xi\}$. Moreover, their ``depth'' function is asymptotic to the expression $\min_{g(\xi) = \eta}\dist(\zero,g(\zero))$, by an argument similar to the proof of Observation \ref{observationSL2Z} above. Because our setup is only equivalent asymptotically to theirs, it makes no sense for us to be interested in their computation of the exact value of the constant in Dirichlet's theorem, which of course we do not claim to generalize.} \comdavid{Is the preceding footnote clear?}
\item in the case where $X$ is a locally finite tree, $G$ is geometrically finite, and $\xi$ is a parabolic fixed point by F. Paulin \cite[Th\'eor\`eme 1.2]{Paulin2}
\end{itemize}

\subsubsection{Optimality}
A natural question is whether the preceding theorems, Dirichlet's theorem and its generalization, can be improved. In each case we do not consider improvement of the constant factor to be significant; the question is whether it is possible for the functions on the right hand side of (\ref{dirichletclassical}) and (\ref{dirichlet}) to be multiplied by a factor tending to zero, and for the theorems to remain true. More formally, we will say that \emph{Theorem \ref{theoremdirichlet} is optimal} for $(X,\dist,\zero,b,G)$ if there does not exist a function $\phi\to 0$ such that for all $\eta\in\Lr(G)$, there exists a constant $C = C_\eta > 0$ such that
\[
\Dist(g(\xi),\eta)\leq C b^{-\dist(\zero,g(\zero))}\phi(b^{\dist(\zero,g(\zero))}) \text{ for infinitely many $g\in G$.}
\]
\begin{definition}
\label{definitionbadlyapproximable}
Let $(X,\dist,\zero,b,G)$ be as in \sectionsymbol\ref{standingassumptions}, and fix a distinguished point $\xi\in\del X$. A point $\eta\in\Lr(G)$ is called \emph{badly approximable with respect to $\xi$} if there exists $\varepsilon > 0$ such that 
\[
\Dist(g(\xi),\eta) \geq \varepsilon b^{-\dist(\zero,g(\zero))}
\]
for all $g\in G$. We denote the set of points that are badly approximable with respect to $\xi$ by $\BA_\xi$, and those that are well approximable by $\WA_\xi = \Lr\butnot \BA_\xi$. In the special case \eqref{SL2Z}, the set $\BA := \BA_\infty$ is the classical set of badly approximable numbers.
\end{definition}
Trivially (using strong discreteness), the existence of badly approximable points in the radial limit set of $G$ implies that Theorem \ref{theoremdirichlet} is optimal.\Footnote{In some contexts, a Dirichlet theorem can be optimal even when badly approximable points do not exist; see \cite{FishmanSimmons5} for such an example.} In fact, we have the following:

\begin{theorem}[Full Dimension Theorem, Jarn\'ik 1928]
$\HD(\BA) = 1$, so Theorem \ref{theoremdirichletclassical} is optimal.
\end{theorem}
\begin{reptheorem}{theoremfulldimension}
[Generalization of the Full Dimension Theorem]
Let $(X,\dist,\zero,b,G)$ be as in \sectionsymbol\ref{standingassumptions}, and let $(\xi_k)_1^\ell$ be a countable (finite or infinite) sequence in $\del X$. Then
\[
\HD\left(\bigcap_{k = 1}^\ell\BA_{\xi_k}\cap\Lur\right) = \delta = \HD(\Lr).
\]
In particular $\BA_\xi\neq\emptyset$ for every $\xi\in\del X$, so Theorem \ref{theoremdirichlet} is optimal.
\end{reptheorem}

Note that the Full Dimension Theorem is simply the special case \eqref{SL2Z} of Theorem \ref{theoremfulldimension}.

We shall prove Theorem \ref{theoremfulldimension} by proving that the sets $\BA_{\xi_k}$ are absolute winning in the sense of C. T. McMullen \cite{McMullen_absolute_winning} on the sets $\limitset_s$ described in Theorem \ref{theorembishopjones}. Our theorem then can be deduced from the facts that the countable intersection of absolute winning sets is absolute winning, and that an absolute winning subset of an Ahlfors $s$-regular set has dimension $s$ ((i), (ii), and (iii) of Proposition \ref{propositionwinningproperties} below).

\begin{remark}
We use the invariance of McMullen's game under quasisymmetric maps to prove Theorem \ref{theoremfulldimension}; cf. Lemma \ref{lemmaquasisymmetric}. For this reason, it is necessary in our proof to use McMullen's game rather than Schmidt's original game which is not invariant under quasisymmetric maps (see \cite{McMullen_absolute_winning}).
\end{remark}

\subsubsection{Geometrically finite groups and approximation by parabolic points}
In the following, we shall assume that $X = \H^{d + 1}$ for some $d\in\N$, and that $G\leq\Isom(X)$ is a nonelementary geometrically finite group. We will say that a point $\eta\in\del X$ is \emph{badly approximable by parabolic points} if $\eta\in\bigcap_{k = 1}^\ell\BA_{\xi_k}$ where $\{\xi_1,\ldots,\xi_\ell\}$ is a complete set of inequivalent parabolic fixed points. The set of points which are badly approximable by parabolic points admits a number of geometric interpretations:

\begin{proposition}
\label{propositionuniformlyradialequivalent}
For a point $\eta\in\del X$, the following are equivalent:
\begin{itemize}
\item[(A)] $\eta$ is badly approximable by parabolic points.
\item[(B)] $\geo\zero\eta$\geofootnote avoids some disjoint $G$-invariant collection of horoballs centered at the parabolic points of $G$.
\item[(C)] The image of $\geo\zero\eta$ in the quotient manifold $X/G$ is bounded.
\item[(D)] $\eta$ is a uniformly radial limit point of $G$.
\end{itemize}
\end{proposition}
The proof of this proposition is not difficult and will be omitted. Certain equivalences have been noted in the literature (e.g. the the equivalence of (B) and (C) was noted in \cite{Stratmann1}).

Note that according to Theorem \ref{theoremfulldimension}, the set of points which are badly approximable by parabolic points has full dimension in the radial limit set.

\ \newline{\bf History of Theorem \ref{theoremfulldimension}.}
Theorem \ref{theoremfulldimension} was previously proven
\begin{itemize}
\item in the case where $X = \H^2$, $G$ is a lattice, and $\{\xi_1,\ldots,\xi_\ell\}$ is a finite set of conventional points by Patterson in 1976 \cite[the Theorem of \sectionsymbol 10]{Patterson1}.
\item in the case where $X = \H^{d + 1}$, $G$ is convex-cocompact, and $\xi$ is a hyperbolic fixed point by B. O. Stratmann in 1994 \cite[Theorem B]{Stratmann3}.
\item in the case where $X = \H^{d + 1}$, $G$ is a lattice, and $\{\xi_1,\ldots,\xi_\ell\}$ is a finite set of parabolic points by B. O. Stratmann in 1994 \cite[Theorem C]{Stratmann3}.
\end{itemize}
Moreover, using Proposition \ref{propositionuniformlyradialequivalent}, one can easily deduce some special cases of Theorem \ref{theoremfulldimension} from known results:
\begin{itemize}
\item The case where $X = \H^{d + 1}$, $G$ is a lattice, and $\{\xi_1,\ldots,\xi_\ell\}$ is a finite set of parabolic points can be deduced from \cite[Corollary 5.2]{Dani3}.
\item The case where $X = \H^{d + 1}$, $G$ is geometrically finite, and $\{\xi_1,\ldots,\xi_\ell\}$ is a finite set of parabolic points can be deduced from either \cite{FernandezMelian}, \cite{Stratmann1}, or \cite{BishopJones}.
\end{itemize}
Also note that some further results on winning properties of the set $\BA_\xi$ can be found in \cite{Aravinda}, \cite{AravindaLeuzinger}, \cite{Dani3}, \cite{McMullen_absolute_winning}, and \cite{MayedaMerrill}.

\begin{remark}
It is likely that the techniques used in \cite{MayedaMerrill} can be used to prove Theorem \ref{theoremfulldimension} in the case where $X$ is a proper geodesic hyperbolic metric space, $G$ is any strongly discrete group of general type, and $\{\xi_1,\ldots\}$ is a set of bounded parabolic points (see Subsection \ref{subsectionparabolic}). 
However, note that such a proof would depend on our Theorem \ref{theorembishopjones}; cf. Remark \ref{remarkmayedamerrill}.\comdavid{Should I say more? or less?}
\end{remark}

\subsection{The Jarn\'ik--Besicovitch theorem and its generalization}
For any irrational $x\in\R$, we can define the \emph{exponent of irrationality}
\[
\omega(x) := \limsup_{\substack{p/q\in\Q \\ q\to\infty}}\frac{-\log|x - p/q|}{\log(q)}\cdot
\]
Then Dirichlet's theorem (Theorem \ref{theoremdirichletclassical}) implies that $\omega(x)\geq 2$ for every irrational $x\in\R$. For each $c > 0$ let
\[
W_c := \{x\in\R\butnot\Q:\omega(x)\geq 2 + c\}.
\]
Furthermore, let
\begin{align*}
\VWA &:= \{x\in\R\butnot\Q:\omega(x) > 2\} = \bigcup_{c > 0}W_c\\
\Liou &:= \{x\in\R\butnot\Q:\omega(x) = \infty\} = \bigcap_{c > 0}W_c
\end{align*}
be the set of \emph{very well approximable} numbers and \emph{Liouville} numbers, respectively.
\begin{theorem}[Jarn\'ik 1929 and Besicovitch 1934]
\label{theoremjarnikbesicovitchclassical}
For every $c > 0$, we have
\[
\HD(W_c) = \frac{2}{2 + c}.
\]
In particular $\HD(\VWA) = 1$ and $\HD(\Liou) = 0$.
\end{theorem}

Let $(X,\dist,\zero,b,G)$ be as in \sectionsymbol\ref{standingassumptions}. For any point $\xi\in\del X$, we can analogously define
\begin{align*}
\omega_\xi(\eta) &:= \limsup_{\substack{g\in G \\ \dist(\zero,g(\zero))\to\infty}}\frac{-\log_b\Dist(g(\xi),\eta))}{\dist(\zero,g(\zero))}\\
W_{c,\xi} &:= \{\eta\in\Lr:\omega_\xi(\eta)\geq 1 + c\}\\
\VWA_\xi &:= \{\eta\in\Lr:\omega_\xi(\eta) > 1\} = \bigcup_{c > 0}W_{c,\xi}\\
\Liou_\xi &:= \{\eta\in\Lr:\omega_\xi(\eta) = \infty\} = \bigcap_{c > 0}W_{c,\xi}.
\end{align*}
Note that in the special case \eqref{SL2Z}, we have $W_{c,\infty} = W_{2c}$, $\VWA_\infty = \VWA$, and $\Liou_\infty = \Liou$.

To formulate our generalization of Theorem \ref{theoremjarnikbesicovitchclassical}, we need some preparation. Let $(X,\dist,\zero,b,G)$ be as in \sectionsymbol\ref{standingassumptions}, and fix a distinguished point $\xi\in\del X$. We recall that for $s \geq 0$ and $S\subset \del X$, the \emph{Hausdorff content} of $S$ in dimension $s$ is defined as
\begin{equation}
\label{Hausdorffcontentdef}
\HH_\infty^s(S) := \inf\left\{\sum_{A\in\CC}\Diam^s(A):\CC\text{ is a countable cover of $S$}\right\}.\Footnote{The subscript of $\infty$ is meant to indicate that the diameter of elements of the cover may be any number $\leq \infty$; more generally, if $\delta\in \OC 0\infty$ then $\HH_\delta^s(S)$ denotes the same infimum taken only over covers whose elements all have diameter $\leq\delta$.}
\end{equation}
Let $\delta$ be the Poincar\'e exponent of $G$, and fix $s \in [0,\delta]$. For each $r,\sigma > 0$ we define
\begin{align} \tag{\ref{Hxisigmadef}}
H_{\xi,\sigma,s}(r) &:= \HH_\infty^s(B(\xi,r)\cap \Lrsigma(G))\\ \tag{\ref{Pxidef}}
P_\xi(s) &:= \lim_{\sigma\to\infty}\liminf_{r\to 0}\frac{\log_b(H_{\xi,\sigma,s}(r))}{\log_b(r)}\cdot
\end{align}

\begin{reptheorem}{theoremjarnikbesicovitch}[Generalization of the Jarn\'ik--Besicovitch Theorem\Footnote{The fact that Theorem \ref{theoremjarnikbesicovitch} is truly a generalization of Theorem \ref{theoremjarnikbesicovitchclassical} is demonstrated in Remark \ref{remarkSL2Zjarnik} below.}]
Let $(X,\dist,\zero,b,G)$ be as in \sectionsymbol\ref{standingassumptions}, and fix a distinguished point $\xi\in\del X$. Let $\delta$ be the Poincar\'e exponent of $G$, and suppose that $\delta < \infty$. Then for each $c > 0$,
\begin{repequation}{jarnikbesicovitch}
\begin{split}
\HD(W_{c,\xi}) &= \sup\left\{s\in(0,\delta): P_\xi(s) \leq \frac{\delta - s}{c}\right\}\\
&= \inf\left\{s\in(0,\delta): P_\xi(s) \geq \frac{\delta - s}{c}\right\}
\leq \frac{\delta}{c + 1}\cdot
\end{split}
\end{repequation}
In particular
\begin{repequation}{HDVWAxi}
\HD(\VWA_\xi) = \sup\{s\in (0,\delta):P_\xi(s) < +\infty\} = \inf\left\{s\in(0,\delta): P_\xi(s) = +\infty\right\}
\end{repequation}
and $\HD(\Liou_\xi) = 0$. In these formulas we let $\sup\emptyset = 0$ and $\inf\emptyset = \delta$.
\end{reptheorem}

\begin{remark}
For the remainder of this subsection, we assume that $\delta < \infty$.
\end{remark}

\begin{corollary}
\label{corollaryHDVWA}
$\HD(\VWA_\xi) = \delta$ if and only if $P_\xi(s) < +\infty$ for all $s\in(0,\delta)$.
\end{corollary}

\begin{repcorollary}{corollaryjarnikbesicovitchcontinuous}
The function $c\mapsto \HD(W_{c,\xi})$ is continuous on $(0,\infty)$.
\end{repcorollary}

It is obvious from (\ref{jarnikbesicovitch}) that $\HD(W_{c,\xi})$ depends on the function $P_\xi$, which in turn depends on the geometry of the limit set of $G$ near the point $\xi$. We will therefore consider special cases for which the function $P_\xi$ can be computed explicitly.

\subsubsection{Special cases}\label{subsubsectionspecialcasesjarnik}
The simplest case for which the function $P_\xi$ can be computed explicitly is the case where $\xi$ is a radial limit point. Indeed, we have the following:

\begin{repexample}{exampleradialjarnik}
If $\xi\in\Lr$ then $P_\xi(s) = s$. Thus in this case (\ref{jarnikbesicovitch}) reduces to
\begin{repequation}{radialjarnik}
\HD(W_{c,\xi}) = \frac{\delta}{c + 1},
\end{repequation}
i.e. the inequality of (\ref{jarnikbesicovitch}) is an equality.
\end{repexample}
This example demonstrates that the expression $\HD(W_{c,\xi})$ is maximized when $\xi$ is a radial limit point. The intuitive reason for this is that there are a lot of ``good approximations'' to $\xi$, so by a sort of duality principle there are many points which are approximated very well by $\xi$.

We remark that if $G$ is of divergence type, then a point in $\Lr$ represents a ``typical'' point; see \cite[Theorem 1.4.1]{DSU}. Thus, generically speaking, the Hausdorff dimension of the set $W_{c,\xi}$ depends only on $c$ and $\delta$, and depends neither on the point $\xi$ nor the group $G$.

\begin{example}
\label{exampleSL2Zjarniknew}
Let $X = \H^2$ and let $G = \SL_2(\Z)$. Example \ref{exampleradialjarnik} shows that for every $\xi\in \Lr = \R\butnot\Q$,
\[
\HD(W_{c,\xi}) = \frac{1}{c + 1}.
\]
This does not follow from any known result.
\end{example}

Next we consider the case where $\xi$ is a bounded parabolic point (see Subsection \ref{subsectionparabolic}):

\begin{reptheorem}{theoremboundedparabolicjarnik}
Let $(X,\dist,\zero,b,G)$ be as in \sectionsymbol\ref{standingassumptions}, and let $\xi$ be a bounded parabolic point of $G$. Let $\delta$ and $\delta_\xi$ be the Poincar\'e exponents of $G$ and $G_\xi$, respectively. Then for all $s\in(0,\delta)$
\begin{repequation}{prevelanihill}
P_\xi(s) = \max(s,2s - 2\delta_\xi)
= \begin{cases}
s & s \geq 2\delta_\xi\\
2s - 2\delta_\xi & s\leq 2\delta_\xi
\end{cases}.
\end{repequation}
Thus for each $c > 0$
\begin{repequation}{velanihillgeneral}
\HD(W_{c,\xi}) = \begin{cases}
\delta/(c + 1) & c + 1 \geq \delta/(2\delta_\xi)\\
\displaystyle\frac{\delta + 2\delta_\xi c}{2c + 1} & c + 1\leq \delta/(2\delta_\xi)
\end{cases}.
\end{repequation}
\end{reptheorem}

\begin{remark}
\label{remarksurprising}
It is surprising that the formula \eqref{prevelanihill} depends only on the Poincar\'e exponent $\delta_\xi$. One would expect it to depend only on the orbital counting function of $G_\xi$ (see \eqref{orbitalcountingfunction}), but this function can be quite wild (see Appendix \ref{appendixorbitalcounting}), even sometimes failing to satisfy a doubling condition.
\end{remark}

\begin{remark}
\label{remarkvelaihill}
If $X = \H^{d + 1}$, then $\delta_\xi = k_\xi/2$, where $k_\xi$ is the rank of $\xi$. Thus \eqref{velanihillgeneral} reduces to
\begin{equation}
\label{velanihill}
\HD(W_{c,\xi}) = \begin{cases}
\delta/(c + 1) & c + 1 \geq \delta/k_\xi\\
\displaystyle\frac{\delta + k_\xi c}{2c + 1} & c + 1\leq \delta/k_\xi
\end{cases}
\end{equation}
in this case.
\end{remark}

\begin{remark}
\label{remarkreductionjarnik}
If $\delta_\xi\geq\delta/2$, then $c + 1 \geq \delta/(2\delta_\xi)$ for all $c > 0$, so \eqref{velanihillgeneral} reduces to \eqref{radialjarnik}. In particular, this happens if
\begin{itemize}
\item[(1)] $X$ is a real, complex, or quaternionic hyperbolic space, and $G$ is a lattice,
\item[(2)] $X$ is a uniform tree which contains no vertex of degree $\leq 2$, and $G$ is a lattice, or if
\item[(3)] $X = \H^2$.
\end{itemize}
\end{remark}
\begin{proof}
(1) is (i) of Proposition \ref{propositionreduction}; (2) is the fifth part of \cite[Proposition 3.1]{Broise-AlamichelPaulin}. To demonstrate (3), note that $\delta_\xi = k_\xi/2 = 1/2$ since 1 is the only possible rank of a parabolic point in a Fuchsian group. On the other hand, it is well known that $\delta\leq 1$ for a Fuchsian group (e.g. it follows from the Bishop--Jones theorem).
\end{proof}

\begin{remark}
\label{remarkSL2Zjarnik}
Theorem \ref{theoremjarnikbesicovitchclassical} follows from Remark \ref{remarkreductionjarnik} applied to the special case \eqref{SL2Z}.
\end{remark}

\begin{remark}
If $G$ is geometrically finite, then Theorem \ref{theoremboundedparabolicjarnik} and Example \ref{exampleradialjarnik} together cover all possible cases for $\xi$.
\end{remark}

Moving beyond the geometrically finite case, let us return to the idea of a duality principle between approximations of a point and the size of the set of points which are well approximable by that point. Another common notion of approximability is the notion of a \emph{horospherical} limit point, i.e. a point $\xi$ for which every horoball centered at $\xi$ (see Definition \ref{definitionhoroball}) is penetrated by the orbit of $\zero$.

\begin{repexample}{examplehorosphericaljarnik}
If $\xi$ is a horospherical limit point of $G$ then $P_\xi(s) \leq 2s$. Thus in this case we have
\[
\frac{\delta}{2c + 1}\leq \HD(W_{c,\xi}) \leq \frac{\delta}{c + 1}.
\]
More generally, let $\lb\cdot|\cdot\rb$ denote the Gromov product (see Definition \ref{definitiongromovproduct}). If we let
\begin{repequation}{betadef}
\beta = \beta(\xi) := \limsup_{\substack{g\in G \\ \dist(\zero,g(\zero))\to\infty}}\frac{\lb g(\zero)|\xi\rb_\zero}{\dist(\zero,g(\zero))} \ ,
\end{repequation}
then $P_\xi(s)\leq \beta^{-1}s$ and so
\[
\frac{\delta}{\beta^{-1} c + 1}\leq \HD(W_{c,\xi}) \leq \frac{\delta}{c + 1}.
\]
In particular, if $\beta > 0$ then $\HD(\VWA_\xi) = \delta$.
\end{repexample}

\begin{remark}
If $X = \B^{d + 1}$, then $\lb \xx|\yy\rb_\0 \asymp_\plus -\log(\|\yy - \xx\|)$ whenever at least one of $\xx,\yy$ is in $S^d$ (e.g. \cite[Lemma 3.5.1]{DSU}). Thus we may rewrite \eqref{betadef} as
\[
\beta(\xi) = \limsup_{\substack{g\in G \\ \|g(\0)\|\to 1}}\frac{\log\|\xi - g(\0)\|}{\log(1 - \|g(\0)\|)}\cdot
\]
From this equation, it can be seen that $\beta(\xi)$ is the supremum of $\alpha\in[0,1]$ for which $\xi\in\LL(\alpha)$, where $\LL(\alpha)$ is defined as in \cite{Lundh}.\comdavid{Is this remark too confusing?}
\end{remark}
In particular, notice that if there is a sequence $g_n(\0)\tendsto n \xi$ such that $1/\|\xi - g_n(\0)\|$ is bounded polynomially in terms of $1/(1 - \|g_n(\0)\|)$, then $\HD(\VWA_\xi) = \delta$. On the other hand,

\begin{repproposition}{propositionVWAHDzero}
There exists a (discrete) nonelementary Fuchsian group $G$ and a point $\xi\in\Lambda$ such that $\HD(\VWA_\xi) = 0$.
\end{repproposition}

\ \newline{\bf History of Theorem \ref{theoremjarnikbesicovitch}.}
Although Theorem \ref{theoremjarnikbesicovitch} has not been stated in the literature before (we are the first to define the function $P_\xi$), the equations \eqref{radialjarnik} and \eqref{velanihill} have both appeared in the literature. \eqref{radialjarnik} was proven\Footnote{Recall that \eqref{radialjarnik} holds whenever $\xi$ is either a radial limit point (e.g. a hyperbolic fixed point) or a (bounded) parabolic point satisfying $\delta_\xi\geq \delta/2$.}

\begin{itemize}
\item in the case where $X = \H^{d + 1}$, $G$ is a lattice, and $\xi$ is a parabolic fixed point by M. V. Meli\'an and D. Pestana in 1993 \cite[Theorems 2 and 3]{MelianPestana}.
\item in the case where $X = \H^{d + 1}$, $G$ is a lattice, and $\xi$ is a conventional point by S. L. Velani in 1994 \cite[Theorem 2]{Velani2} (case $d = 1$ earlier in 1993 \cite[Theorem 2]{Velani1}).
\item in the case where $X = \H^{d + 1}$, $G$ is convex-cocompact, and $\xi$ is a hyperbolic fixed point by M. M. Dodson, M. V. Meli\'an, D. Pestana, and S. L. Velani in 1995 \cite[Theorem 3]{DMPV}.
\item in the case where $X = \H^{d + 1}$, $G$ is geometrically finite and satisfies $\delta\leq k_{\max}$,\Footnote{Since $k_\xi = 2\delta_\xi$, the condition $\delta\leq k_{\max}$ is equivalent to the condition that $\delta_\xi\geq \delta/2$ for every parabolic point $\xi$, i.e. the necessary and sufficient condition for \eqref{velanihillgeneral} to reduce to \eqref{radialjarnik} (Remark \ref{remarkreductionjarnik}). When $\delta > k_{\max}$, Stratmann proves an upper and a lower bound for $\HD(W_{c,\xi})$, neither of which are equal to the correct value \eqref{velanihillgeneral}.} and $\xi$ is a parabolic fixed point by B. O. Stratmann in 1995 \cite[Theorem B]{Stratmann4}.
\end{itemize}
Additionally, the $\leq$ direction of \eqref{radialjarnik} was proven
\begin{itemize}
\item in the case where $X = \H^{d + 1}$, $G$ is geometrically finite, and $\xi$ is a conventional point by S. L. Velani in 1994 \cite[Theorem 1]{Velani2} (case $d = 1$ earlier in 1993 \cite[Theorem 1]{Velani1}).
\end{itemize}
Finally, \eqref{velanihill} has been proven in the literature only once before,
\begin{itemize}
 \item in the case where $X = \H^{d + 1}$, $G$ is geometrically finite, and $\xi$ is a parabolic fixed point by R. M. Hill and S. L. Velani in 1998 \cite[Theorem 1]{HillVelani}.\Footnote{To verify that \eqref{velanihillgeneral} is equal to the expression given in \cite{HillVelani}, note that $k_\xi = 2\delta_\xi$ for a parabolic point $\xi\in\del\H^{d + 1}$.}
\end{itemize}
\begin{remark}
Theorem \ref{theoremboundedparabolicjarnik} is more general than \cite[Theorem 1]{HillVelani} not only because it applies to spaces $X\neq\H^{d + 1}$, but also because it applies to geometrically infinite groups $G$ acting on $\H^{d + 1}$; $\xi$ must be a bounded parabolic point, but a group can have one bounded parabolic point without being geometrically finite.
\end{remark}

\subsection{Khinchin's theorem and its generalization}
We now consider metric Diophantine approximation, or the study of the Diophantine properties of a point chosen at random with respect to a given measure.\Footnote{In this paper, measures are assumed to be finite and Borel.} In this case, we would like to consider a function $\Psi:\N\to(0,\infty)$ and consider the points which are well or badly approximable with respect to $\Psi$. The classical definition and theorem are as follows:
\begin{definition}
An irrational point $x\in\R$ is \emph{$\Psi$-approximable} if there exist infinitely many $p/q\in\Q$ such that
\[
\left|x - \frac{p}{q}\right| \leq \Psi(q).
\]
\end{definition}

\begin{theorem}[Khinchin 1924]
\label{theoremkhinchinclassical}
Let $\Psi:\N\to(0,\infty)$ be any function.
\begin{itemize}
\item[(i)] If the series
\begin{equation}
\label{khinchinclassical}
\sum_{q = 1}^\infty q\Psi(q)
\end{equation}
converges, then almost no point is $\Psi$-approximable.
\item[(ii)] Suppose that the function $q\mapsto q^2\Psi(q)$ is nonincreasing.\Footnote{There are a number of results which allow this restriction to be weakened; see \cite{Bugeaud} for a detailed account.} If the series \eqref{khinchinclassical} diverges, then almost every point is $\Psi$-approximable.
\end{itemize}
In both cases the implicit measure is Lebesgue.
\end{theorem}

Again, we would like to consider a generalization of this theorem to the setting of \sectionsymbol\ref{standingassumptions}. We will replace Lebesgue measure in the above theorem by an ergodic $\delta$-quasiconformal measure (see Definition \ref{definitionquasiconformal}), where $\delta$ is the Poincar\'e exponent of $G$. Sufficient conditions for the existence and uniqueness of such a measure are given in Subsection \ref{subsectionquasiconformal}. Thus, we will assume that we have such a measure, and that we want to determine the Diophantine behavior of almost every point with respect to this measure.


We will now define our generalization of the notion of a $\Psi$-approximable point:

\begin{definition}
Let $(X,\dist,\zero,b,G)$ be as in \sectionsymbol\ref{standingassumptions}, and fix a distinguished point $\xi\in\del X$. Let $\Phi:\Rplus\to(0,\infty)$ be a function such that the function $t\mapsto t\Phi(t)$ is nonincreasing.\Footnote{Although this is not the weakest assumption which we could have used here, it has the advantage of avoiding technicalities about what regularity hypotheses on $\Phi$ are necessary for which statements; this hypothesis suffices for every statement we will consider. We have added the hypothesis to the definition itself because if a weaker hypothesis is used, then the equation \eqref{Phixiwellapproximable} may need to be modified.} We say that a point $\eta\in \Lambda(G)$ is \emph{$\Phi,\xi$-well approximable} if for every $K > 0$ there exists $g\in G$ such that
\begin{equation}
\label{Phixiwellapproximable}
\Dist(g(\xi),\eta) \leq \Phi(K b^{\dist(\zero,g(\zero))}).
\end{equation}
We denote the set of all $\Phi,\xi$-well approximable points by $\WA_{\Phi,\xi}$, and its complement by $\BA_{\Phi,\xi}$. 
\end{definition}

Note that our definition differs from the classical one in that a factor of $K$ has been added.\Footnote{The reason that we have called points satisfying our definition $\Phi,\xi$-\emph{well} approximable rather than $\Phi,\xi$-approximable is to indicate this distinction.} This is in keeping with the philosophy of hyperbolic metric spaces, since most quantities are considered to be defined only ``up to a constant''. Furthermore, our definition is independent of the chosen fixed point $\zero$. The relation between our definition and the classical one is as follows:

\begin{observation}
\label{observationPsiPhiapproximable}
Let $\Phi:\Rplus\to(0,\infty)$ be a function such that the function $t\mapsto t\Phi(t)$ is nonincreasing. In the case \eqref{SL2Z}, a point $x\in\R = \Lambda(G)\butnot\{\infty\}$ is $\Phi,\infty$-well approximable if and only if for every $K > 0$ it is $\Psi$-approximable where $\Psi(q) = \Phi(K q^2)$.
\end{observation}
\noindent Of course, this observation follows from Observations \ref{observationSL2Z} - \ref{observationSL2Z three}.

\begin{reptheorem}{theoremkhinchin}[Generalization of Khinchin's Theorem\Footnote{The fact that Theorem \ref{theoremkhinchin} is truly a generalization of Theorem \ref{theoremkhinchinclassical} is proven in Example \ref{exampleSL2Zkhinchin} below.}]
Let $(X,\dist,\zero,b,G)$ be as in \sectionsymbol\ref{standingassumptions}, and fix a distinguished point $\xi\in\del X$. Let $\delta$ be the Poincar\'e exponent of $G$, and suppose that $\mu\in\MM(\Lambda)$ is an ergodic $\delta$-quasiconformal probability measure satisfying $\mu(\xi) = 0$. Let
\[
\Delta_{\mu,\xi}(r) = \mu(B(\xi,r)).
\]
Let $\Phi:\Rplus\to(0,\infty)$ be a function such that the function $t\mapsto t\Phi(t)$ is nonincreasing. Then:
\begin{itemize}
\item[(i)] If the series
\begin{repequation}{khinchindiverges}
\sum_{g\in G}b^{-\delta\dist(\zero,g(\zero))} \Delta_{\mu,\xi} \left(b^{\dist(\zero,g(\zero))} \Phi(K b^{\dist(\zero,g(\zero))})\right)
\end{repequation}
diverges for all $K > 0$, then $\mu(\WA_{\Phi,\xi}) = 1$.
\item[(iiA)] If the series
\begin{repequation}{khinchinconverges}
\sum_{g\in G}(g'(\xi))^\delta \Delta_{\mu,\xi} \left(\Phi(K b^{\dist(\zero,g(\zero))})/g'(\xi)\right)
\end{repequation}
converges for some $K > 0$, then $\mu(\WA_{\Phi,\xi}) = 0$.
\item[(iiB)] If $\xi$ is a bounded parabolic point of $G$ (see Definition \ref{definitionboundedparabolic}), and if the series \eqref{khinchindiverges} converges for some $K > 0$, then $\mu(\WA_{\Phi,\xi}) = 0$.
\end{itemize}
\end{reptheorem}
\begin{remark}
The fact that $\mu(\WA_{\Phi,\xi})$ is either $0$ or $1$ is an easy consequence of the assumption that $\mu$ is ergodic, since $\WA_{\Phi,\xi}$ is a measurable $G$-invariant set.
\end{remark}
\begin{remark}
(iiB) is not a special case of (iiA) as demonstrated by Proposition \ref{propositionparabolic43} below. In particular, the proof of (iiB) depends crucially on the fact that if $\xi$ is a parabolic point, then the map $g\mapsto g(\xi)$ is highly non-injective, which reduces the number of potential approximations to consider. In general the map $g\mapsto g(\xi)$ is usually injective, so this reduction cannot be used. Thus, it seems unlikely that the converse of (i) of Theorem \ref{theoremkhinchin} should hold in general.
\end{remark}

\subsubsection{Special cases} \label{subsubsectionspecialcaseskhinchin}
When considering special cases of Theorem \ref{theoremkhinchin}, we will consider the following three questions:
\begin{itemize}
\item[1.] Under what hypotheses are the series \eqref{khinchindiverges} and \eqref{khinchinconverges} asymptotic? That is, when do parts (i) and (ii) of Theorem \ref{theoremkhinchin} give a value for $\mu(\WA_{\Phi,\xi})$ for every possible function $\Phi$?
\item[2.] Under what hypotheses is it possible to simplify (asymptotically) the series \eqref{khinchindiverges}?
\item[3.] Under what hypotheses can we prove that $\mu(\BA_\xi) = 0$ and/or $\mu(\VWA_\xi) = 0$?
\end{itemize}
To answer these questions we consider the following asymptotic formulas:
\begin{align} \tag{\ref{Deltamuxiasymp}}
\Delta_{\mu,\xi}(r) &\asymp_\times r^\delta \all r > 0 \\ \tag{\ref{orbitalcountingfunctionasymp}}
f_G(t) := \#\{g\in G:\dist(\zero,g(\zero))\leq t\} &\asymp_\times b^{\delta t} \all t > 0.\footnotemark{}
\end{align}\footnotetext{We remark that if either formula holds with $\delta$ replaced by some $s > 0$, then $s = \delta$.}
These asymptotics will not be satisfied for every group $G$ and every point $\xi$, but we will see below that there are many reasonable hypotheses which ensure that they hold. When the asymptotics are satisfied, we have the following:

\begin{repproposition}{propositionkhinchinpowerlaw}[Reduction of \eqref{khinchindiverges} and \eqref{khinchinconverges} assuming power law asymptotics]
If \eqref{orbitalcountingfunctionasymp} holds then for each $K > 0$
\begin{repequation}{orbitalcountingfunctionasympimplies}
\eqref{khinchindiverges} \asymp_{\plus,\times} \int_{t = 0}^\infty \Delta_{\mu,\xi}(K^{-1} e^t \Phi(e^t))\dee t.
\end{repequation}
If \eqref{Deltamuxiasymp} holds then for each $K > 0$
\begin{repequation}{Deltamuxiasympimplies}
\eqref{khinchindiverges} \asymp_\times \eqref{khinchinconverges} \asymp_\times \sum_{g\in G}\Phi^\delta(K b^{\dist(\zero,g(\zero))}).
\end{repequation}
If both \eqref{orbitalcountingfunctionasymp} and \eqref{Deltamuxiasymp} hold then for each $K > 0$
\begin{repequation}{perfectkhinchin}
\eqref{khinchindiverges} \asymp_\times \eqref{khinchinconverges} \asymp_{\plus,\times} \int_{t = 0}^\infty e^{\delta t} \Phi^\delta(e^t) \dee t,
\end{repequation}
and so by Theorem \ref{theoremkhinchin}
\begin{repequation}{perfectkhinchin2}
\mu(\WA_{\Phi,\xi}) = 0 \;\;\Leftrightarrow\;\; \int_{t = 0}^\infty e^{\delta t} \Phi^\delta(e^t) \dee t < \infty.
\end{repequation}
\end{repproposition}
Note in particular that the last integral depends only on the Poincar\'e exponent $\delta$, and on no other geometric information.

\begin{repcorollary}{corollarykhinchinpowerlaw}[Measures of $\BA_\xi$ and $\VWA_\xi$ assuming power law asymptotics]
Assume $\xi\in\Lambda$. If \eqref{Deltamuxiasymp} holds, then $\mu(\VWA_\xi) = 0$, and $\mu(\BA_\xi) = 0$ if and only if $G$ is of divergence type. If \eqref{orbitalcountingfunctionasymp} holds, then $\mu(\BA_\xi) = 0$. If both hold, then $\mu(\BA_\xi) = \mu(\VWA_\xi) = 0$.
\end{repcorollary}

\begin{remark}
\label{remarknonlimitpoint}
The hypothesis $\xi\in\Lambda$ is necessary in Corollary \ref{corollarykhinchinpowerlaw}. Indeed, if $\xi\in\del X\butnot\Lambda$, then $\BA_\xi = \Lr$ (see Proposition \ref{propositionnonlimitpoint}).
\end{remark}

\subsubsection{Verifying the hypotheses \eqref{Deltamuxiasymp} and \eqref{orbitalcountingfunctionasymp}}
\label{subsubsectionverifying}
Let us now consider (more or less) concrete examples of pairs $(G,\xi)$ for which \eqref{Deltamuxiasymp} and/or \eqref{orbitalcountingfunctionasymp} are satisfied. As for the Jarn\'ik--Besicovitch theorem, the simplest case is when $\xi$ is a uniformly radial point, in which case Sullivan's Shadow Lemma (Lemma \ref{Sullivan's Shadow Lemma}) yields the following:

\begin{repexample}{exampleuniformlyradialkhinchin}
If $\xi$ is a uniformly radial limit point, then \eqref{Deltamuxiasymp} holds.
\end{repexample}

Let us call a group $G$ \emph{perfect} if \eqref{Deltamuxiasymp} and \eqref{orbitalcountingfunctionasymp} (and thus \eqref{perfectkhinchin2}) are satisfied for all $\xi\in\Lambda_G$.

\begin{repexample}{examplequasiconvexcocompact}[Quasiconvex-cocompact group]
Let $X$ be a proper and geodesic hyperbolic metric space, and let $G\leq\Isom(X)$. The group $G$ is called \emph{quasiconvex-cocompact} if the set $G(\zero)$ is \emph{quasiconvex}, i.e. there exists $\rho > 0$ such that for all $x,y\in G(\zero)$,
\begin{repequation}{quasiconvexcocompact}
\geo xy \subset B(G(\zero),\rho).
\end{repequation}
Any quasiconvex-cocompact group is perfect.
\end{repexample}

\begin{example}[Lattice on a real, complex, or quaternionic hyperbolic space]
\label{examplekhinchinlattice1}
Let $X$ be a real, complex, or quaternionic hyperbolic space, and let $G\leq\Isom(X)$ be a lattice. Then $G$ is perfect.
\end{example}
\begin{proof}
\eqref{Deltamuxiasymp} follows from (ii) of Proposition \ref{propositionreduction}; \eqref{orbitalcountingfunctionasymp} is a special case of Example \ref{examplepinched} below.
\end{proof}

\begin{example}[The classical Khinchin theorem]
\label{exampleSL2Zkhinchin}
Since $\SL_2(\Z)$ is a lattice, Example \ref{examplekhinchinlattice1} implies that $\SL_2(\Z)$ is perfect. Moreover, using the substitution $t = \log(q^2)$ and the value $\delta = 1$, we have
\begin{equation}
\label{qPhiq2}
\int_{t = 0}^\infty e^{\delta t} \Phi^\delta(e^t)\dee t
= \int_{t = 0}^\infty e^t \Phi(e^t)\dee t
= 2\int_{q = 1}^\infty q^2 \Phi(q^2)\frac{\dee q}{q}
\asymp_{\plus,\times} \sum_{q = 1}^\infty q\Phi(q^2).
\end{equation}
Thus if we let $\Psi(q) = \Phi(q^2)$, then the series \eqref{khinchindiverges}, \eqref{khinchinconverges}, and \eqref{khinchinclassical} are all $\plus,\times$-asymptotic. Thus, by applying Observation \ref{observationPsiPhiapproximable}, we are able to deduce the difficult direction (ii) of Theorem \ref{theoremkhinchinclassical} from Theorem \ref{theoremkhinchin} .\Footnote{(i) of Theorem \ref{theoremkhinchinclassical} is not deduced from Theorem \ref{theoremkhinchin} for two trivial reasons: first of all, the hypotheses of (i) Theorem \ref{theoremkhinchinclassical} are weaker; second of all, $\Phi,\infty$-well-approximability implies $\Psi$-approximability but not vice-versa.}
\end{example}

\begin{example}
\label{exampleSL2Zkhinchinnew}
Let $X = \H^2$ and let $G = \SL_2(\Z)$. Theorem \ref{theoremkhinchin} shows that for every $\xi\in\del X = \what\R$ and for every $\Phi:\Rplus\to(0,\infty)$ such that $t\mapsto t\Phi(t)$ is nonincreasing, $\WA_{\Phi,\xi}$ is either a Lebesgue null set or of full Lebesgue measure according to whether the series \eqref{khinchinclassical} converges or diverges. This result is new except when $\xi$ is rational.
\end{example}

Our last example of a perfect group is a little more exotic:
\begin{example}[Lattice on a uniform tree]
\label{examplekhinchinlattice2}
Let $X$ be a uniform tree which contains no vertex of degree $\leq 2$. Then if $G\leq\Isom(X)$ is a lattice, then $G$ is perfect.
\end{example}
\begin{proof}
The estimate \eqref{orbitalcountingfunctionasymp} is \cite[Theorem 2.1(2)]{HP_Khinchin_trees}. To demonstrate \eqref{Deltamuxiasymp}, let $H$ be a discrete cocompact subgroup of $\Isom(X)$. By Example \ref{examplequasiconvexcocompact}, the Patterson--Sullivan measure of $H$ satisfies \eqref{Deltamuxiasymp}, so it remains to show that the Patterson--Sullivan measure of $H$ is also $\delta$-conformal with respect to $G$. But by \cite[Theorem 6.3]{HersonskyHubbard}, the Patterson--Sullivan measure of $H$ is a constant multiple of the $\delta$-dimensional Hausdorff measure on $\del X$, and it is well known (e.g. \cite[p. 174]{Sullivan_density_at_infinity}) that this is $\delta$-conformal for any group.
\end{proof}

Note that for each $n\geq 3$, the unique tree such that each vertex has degree $n$ is a uniform tree.

In our next example, \eqref{orbitalcountingfunctionasymp} is satisfied, but \eqref{Deltamuxiasymp} is usually not satisfied. Of course, \eqref{Deltamuxiasymp} is still satisfied if $\xi$ is a uniformly radial point, so in this case we have \eqref{perfectkhinchin2}. Moreover, by Corollary \ref{corollarykhinchinpowerlaw} we have $\mu(\BA_\xi) = 0$ for all $\xi\in\Lambda_G$.

\begin{example}[Geometrically finite group on a pinched Hadamard manifold]
\label{examplepinched}
Let $X$ be a pinched Hadamard manifold, and let $G\leq\Isom(X)$ be geometrically finite (meaning that $\Lambda = \Lr\cup\Lbp$). Suppose that one of the following conditions holds:
\begin{itemize}
\item[(A)] $X$ is a real, complex, or quaternionic hyperbolic space.
\item[(B)] For each $\xi\in\Lbp$, the orbital counting function $f_\xi$ (see \eqref{orbitalcountingfunction}) of $G_\xi$ with respect to $\Dist_\euc$ satisfies a \emph{power law} i.e.
\begin{equation}
\label{powerlaw}
f_\xi(R) \asymp_\times R^{2\delta_\xi},
\end{equation}
where $\delta_\xi$ is the Poincar\'e exponent of the parabolic subgroup $G_\xi$.
\item[(C)] For each $\xi\in\Lbp$, $G_\xi$ is of divergence type.
\item[(D)] For each $\xi\in\Lbp$, $\delta_\xi < \delta$.
\end{itemize}
Then \eqref{orbitalcountingfunctionasymp} is satisfied. In particular, if \textup{(B)} holds, then for each $\xi\in\Lbp$ we have
\begin{equation}
\label{pinched}
\mu(\WA_{\Phi,\xi}) = 0 \;\;\Leftrightarrow\;\; \int_{t = 0}^\infty (e^t\Phi(e^t))^{2(\delta - \delta_\xi)}\dee t < \infty
\end{equation}
for every function $\Phi:\Rplus\to(0,\infty)$ such that $t\mapsto t\Phi(t)$ is nonincreasing.
\end{example}
\begin{proof}
The implication (D) \implies \eqref{orbitalcountingfunctionasymp} was proven by S. D. Hersonsky and F. Paulin \cite[Lemma 3.6]{HP_Khinchin} using a result of T. Roblin \cite[Corollaire 2 of Th\'eor\`eme 4.1.1]{Roblin1}. We claim that (A) \implies (B) \implies (C) \implies (D). Indeed, it is obvious that (B) \implies (C). For (A) \implies (B), see e.g. \cite[Lemma 3.5]{Newberger}. Finally, (C) \implies (D) is \cite[Proposition 2]{DOP}.

To demonstrate \eqref{pinched}, note that since \eqref{orbitalcountingfunctionasymp} holds, by Proposition \ref{propositionkhinchinpowerlaw} we have
\begin{align*}
\eqref{khinchindiverges} &\asymp_{\plus,\times} \int_{t = 0}^\infty \Delta_{\mu,\xi}(e^t \Phi(e^t))\dee t\\
&\asymp_{\times\phantom{,\plus}}  \int_{t = 0}^\infty (e^t\Phi(e^t))^{2(\delta - \delta_\xi)}\dee t. \by{\cite[Th\'eor\`eme 3.2]{Schapira}}
\end{align*}
We remark that in the last step, we used the hypothesis (B) to apply B. Schapira's Global Measure Formula for pinched Hadamard manifolds.

Now since $\xi$ is assumed to be a bounded parabolic point, Theorem \ref{theoremkhinchin} (parts (i) and (iiB)) tells us that $\mu(\WA_{\Phi,\xi}) = 0$ or $1$ according to whether the series \eqref{khinchindiverges} converges or diverges, respectively.
\end{proof}

\begin{remark}
The implication (A) \implies \eqref{orbitalcountingfunctionasymp} was proven in the case of real hyperbolic space by B. O. Stratmann and S. L. Velani \cite[Corollary 4.2]{StratmannVelani}. The idea of their proof, combined with the more general global measure formulas found in \cite{Newberger} and \cite{Schapira}, can be used to prove that (A) \implies \eqref{orbitalcountingfunctionasymp} (in the general case) and that (B) \implies \eqref{orbitalcountingfunctionasymp}, respectively.
\end{remark}

\begin{remark}
\label{remarkreductionkhinchin}
If $\delta_\xi = \delta/2$, then \eqref{pinched} reduces to \eqref{perfectkhinchin2} (cf. Remark \ref{remarkreductionjarnik}). In particular this happens if $X$ is a real, complex, or quaternionic hyperbolic space and $G$ is a lattice.
\end{remark}
\begin{proof}
See (i) of Proposition \ref{propositionreduction}.
\end{proof}

\subsubsection{A remark on $\mu(\VWA_\xi)$ in the above example}
In Example \ref{examplepinched}, since \eqref{orbitalcountingfunctionasymp} is satisfied but not \eqref{Deltamuxiasymp}, Corollary \ref{corollarykhinchinpowerlaw} says that $\mu(\BA_\xi) = 0$ but gives no information on $\mu(\VWA_\xi)$ (unless $\xi$ is a uniformly radial point). 

\begin{proposition}
Let $X = \H^{d + 1}$ for some $d\in\N$, and let $G\leq\Isom(X)$ be geometrically finite. Then $\mu(\VWA_\xi) = 0$.
\end{proposition}
\begin{proof}
For each $c > 0$, by Theorem \ref{theoremjarnikbesicovitch} we have $\HD(W_{c,\xi})\leq\frac{\delta}{c + 1} < \delta$. But by \cite[Proposition 4.10]{StratmannVelani}, any set of full $\mu$-measure has Hausdorff dimension at least $\delta$. Thus $\mu(W_{c,\xi}) = 0$ for all $c > 0$, and so $\mu(\VWA_\xi) = 0$.
\end{proof}

\begin{remark}
It appears to be difficult to prove that $\mu(\VWA_\xi) = 0$ directly from Theorem \ref{theoremkhinchin} (assuming that \eqref{Deltamuxiasymp} is not satisfied and that $\xi\notin\Lbp$), since then the series \eqref{khinchinconverges} needs to be analyzed directly, as it is not clear whether it is asymptotic to \eqref{khinchindiverges}.
\end{remark}

\subsubsection{An example where $\eqref{khinchindiverges}\not\asymp \eqref{khinchinconverges}$}
In Example \ref{examplepinched}, we used (iiB) of Theorem \ref{theoremkhinchin} to prove \eqref{pinched}. A natural question is whether we could have used (iiA) instead, as it is more general. However, we could not, as the following proposition shows:

\begin{repproposition}{propositionparabolic43}
Let $X = \H^{d + 1}$ for some $d\in\N$, and let $G\leq\Isom(X)$ be geometrically finite. If $\xi$ is a bounded parabolic point whose rank is strictly greater than $4\delta/3$, then there exists a function $\Phi$ such that $t\mapsto t\Phi(t)$ is nonincreasing and such that for every $K > 0$, \eqref{khinchindiverges} converges but \eqref{khinchinconverges} diverges.
\end{repproposition}
If $G$ is a nonelementary geometrically finite group and if $\xi$ is a parabolic point of rank $k_\xi$, then $\delta > k_\xi/2$ (e.g. \cite{Beardon1}), and this bound is optimal (e.g. \cite[Theorem 1]{Patterson3}). In particular, it may happen that $\delta < 3k_\xi/4$ or equivalently $k_\xi > 4\delta/3$. This shows that Proposition \ref{propositionparabolic43} is not vacuous.

\subsubsection{The measure of $\Lur$}
The following result is an application of Theorem \ref{theoremkhinchin}:

\begin{repproposition}{propositionmuLurzero}
Let $(X,\dist,\zero,b,G)$ be as in \sectionsymbol\ref{standingassumptions}, with $X$ proper and geodesic. Let $\mu$ be a $\delta_G$-quasiconformal measure on $\Lambda$. Then $\mu(\Lur) = 0$ if and only if $G$ is not quasiconvex-cocompact.
\end{repproposition}

\subsubsection{Relation between Theorems \ref{theoremkhinchin} and \ref{theoremjarnikbesicovitch} via the mass transference principle}
We now remark on the relation between Theorems \ref{theoremkhinchin} and \ref{theoremjarnikbesicovitch} provided by the Mass Transference Principle of V. V. Beresnevich and S. L. Velani \cite{BeresnevichVelani}. Informally, the Mass Transference Principle states that given any setup of Diophantine approximation taking place in an Ahlfors regular metric space, the analogue of Khinchin's theorem implies the analogue of the Jarn\'ik--Besicovitch theorem. In particular, applying \cite[Theorem 3 on p.991]{BeresnevichVelani} to Theorem \ref{theoremkhinchin} (together with \eqref{Deltamuxiasympimplies}) immediately yields the divergence case of the following:

\begin{theorem}
Let $(X,\dist,\zero,b,G)$ be as in \sectionsymbol\ref{standingassumptions}, and fix a distinguished point $\xi\in\del X$. Let $\delta$ be the Poincar\'e exponent of $G$, and suppose that there exists an ergodic $\delta$-quasiconformal probability measure on $\Lambda$ which is Ahlfors $\delta$-regular (see Definition \ref{definitionahlforsregular}). Let $f,\Phi:(0,\infty)\to(0,\infty)$ be functions such that the function $t\mapsto t^{-\delta}f(t)$ is nonincreasing and the function $t\mapsto t^\delta f\circ\Phi(t)$ is monotonic. Then
\[
\HH^f(\WA_{\Phi,\xi}) = \begin{cases}
0 &\text{if } \sum_{g\in G} f\circ\Phi(K b^{\dist(\zero,g(\zero))}) < \infty \text{ for some $K > 0$}\\
\infty &\text{if } \sum_{g\in G} f\circ\Phi(K b^{\dist(\zero,g(\zero))}) = \infty \text{ for all $K > 0$}
\end{cases}.
\]
\end{theorem}
\noindent The convergence case follows immediately from the Hausdorff--Cantelli lemma (see e.g. \cite[Lemma 3.10]{BernikDodson}).

In particular, letting $f(t) = t^s$ and $\Phi_c(t) = t^{-(1 + c)}$, we get
\[
\HH^s(\WA_{\Phi_c,\xi}) = \begin{cases}
0 &\text{if } s > \frac{\delta}{1 + c}\\
\infty &\text{if } s < \frac{\delta}{1 + c}\\
0 &\text{if } s = \frac{\delta}{1 + c} \text{ and $G$ is of convergence type}\\
\infty &\text{if } s = \frac{\delta}{1 + c} \text{ and $G$ is of divergence type}\\
\end{cases}
\]
for all $s \leq \delta$. Since $W_{c,\xi} = \bigcap_{c' < c}\WA_{\Phi_{c'},\xi}\cap\Lr \supset \WA_{\Phi_c,\xi}\cap\Lr$, if $\HH^s(\Lambda\butnot\Lr) = 0$ then we have
\[
\HH^s(W_{c,\xi}) = \begin{cases}
0 &\text{if } s > \frac{\delta}{1 + c}\\
\infty &\text{if } s \leq \frac{\delta}{1 + c}
\end{cases}\cdot
\]
Summarizing, we have the following:

\begin{corollary}
\label{corollaryahlforsjarnik}
Let $(X,\dist,\zero,b,G)$ be as in \sectionsymbol\ref{standingassumptions}, and fix a distinguished point $\xi\in\del X$. Let $\delta$ be the Poincar\'e exponent of $G$, and suppose that there exists an ergodic $\delta$-quasiconformal probability measure on $\Lambda$ which is Ahlfors $\delta$-regular. Then for $c > 0$,
\[
\HD(W_{c,\xi}) = \frac{\delta}{1 + c}
\]
as long as $\HD(\Lambda\butnot\Lr) < \frac{\delta}{1 + c}$.
\end{corollary}

Corollary \ref{corollaryahlforsjarnik} can be deduced independently from Theorem \ref{theoremjarnikbesicovitch}; see below. Moreover, in this case the hypothesis that $\HD(\Lambda\butnot\Lr) < \frac{\delta}{1 + c}$ may be weakened to the hypothesis that $\mu(\Lambda\butnot\Lr) = 0$. However, we remark that the assumption $\HD(\Lambda\butnot\Lr) = 0$ is satisfied in all the examples of \sectionsymbol\ref{subsubsectionverifying} which are Ahlfors regular (in fact, $\Lambda\butnot\Lr$ is countable in these examples).

\begin{proof}[Proof of Corollary \ref{corollaryahlforsjarnik} using Theorem \ref{theoremjarnikbesicovitch}]
Fix $\sigma > 0$ such that $\mu(\Lrsigma) > 0$. Since $\mu$ is ergodic, $\mu(\Lambda\butnot G(\Lrsigma)) = 0$. By Lemma \ref{lemmanearinvariance}, for all $\tau > 0$ sufficiently large we have $G(\Lrsigma)\subset\Lrtau$; for such $\tau$, we have $\mu(\Lambda\butnot\Lrtau) = 0$.

Since $\mu$ is Ahlfors $\delta$-regular, it follows from the mass distribution principle that
\[
\HH_\infty^\delta(B(\xi,r)\cap\Lrtau)\gtrsim_\times \mu(B(\xi,r)\cap\Lrtau) \asymp r^\delta,
\]
and so $P_\xi(\delta) = \delta$. By Proposition \ref{propositionPxifacts}, we have $P_\xi(s) = s$ for all $0 < s < \delta$.
\end{proof}

\ \newline{\bf History of Theorem \ref{theoremkhinchin}.}
Although Theorem \ref{theoremkhinchin} has not been stated in the literature previously, the equivalences \eqref{perfectkhinchin2} and \eqref{pinched} have both appeared in the literature. \eqref{perfectkhinchin2} was proven
\begin{itemize}
\item in the case where $X = \H^2$, $G$ is a lattice, and $\xi$ is a conventional point by S. J. Patterson in 1976 \cite[the Theorem of \sectionsymbol 9]{Patterson1}.
\item in the case where $X = \H^{d + 1}$, $G$ is a lattice, $\xi$ is a parabolic point, and $\Phi(e^t) = e^{-t}(t\vee 1)^{-(1 + \varepsilon)/\delta}$ for some $\varepsilon\geq 0$ (the ``logarithm law'') by D. P. Sullivan in 1982 \cite[Theorem 6]{Sullivan_disjoint_spheres}.
\item in the case where $X = \H^{d + 1}$, $G$ is convex-cocompact, and $\xi$ is a hyperbolic point by B. O. Stratmann in 1994 \cite[Theorem A]{Stratmann3}.
\item in the case where $X$ is a uniform tree which contains no vertex of degree $\leq 2$, $G$ is a lattice, and $\xi$ is a parabolic point by S. D. Hersonsky and F. Paulin in 2007 \cite[Theorem 1.1]{HP_Khinchin_trees}.
\end{itemize}
\eqref{pinched} was proven
\begin{itemize}
\item in the case where $X = \H^{d + 1}$, $G$ is geometrically finite, and $\xi$ is a parabolic point by B. O. Stratmann and S. L. Velani in 1995 \cite[Theorem 4]{StratmannVelani}.
\item in the case where $X$ is a pinched Hadamard manifold, $G$ is geometrically finite, and $\xi$ is a parabolic point satisfying \eqref{powerlaw} by S. D. Hersonsky and F. Paulin in 2004 \cite[Theorem 3]{HP_Khinchin}.
\end{itemize}
Note that in the above lists, we ignore all differences between regularity hypotheses on the function $\Phi$.

\draftnewpage\section{Gromov hyperbolic metric spaces} \label{sectionhyperbolic}
In this section we review the theory of hyperbolic metric spaces and discrete groups of isometries acting on them. There is a vast literature on this subject; see for example \cite{BridsonHaefliger, KapovichBenakli, Oshika, Vaisala}. We would like to single out \cite{BonkSchramm}, \cite{BridsonHaefliger}, and \cite{Vaisala} as exceptions that work in not necessarily proper settings and make reference to infinite-dimensional hyperbolic space.

\begin{definition}
\label{definitiongromovproduct}
Let $(X,\dist)$ be a metric space. For any three points $x,y,z\in X$, the \emph{Gromov product} of $x$ and $y$ with respect to $z$ is defined by
\[
\lb x|y\rb_z := \frac{1}{2}[\dist(x,z) + \dist(y,z) - \dist(x,y)].
\]
\end{definition}
Intuitively, the Gromov product measures the ``defect in the triangle inequality''.

\begin{definition}
\label{definitiongromovhyperbolic}
A metric space $(X,\dist)$ is called \emph{hyperbolic} (or \emph{Gromov hyperbolic}) if for every four points $x,y,z,w\in X$ we
have 
\begin{equation}
\label{gromov}
\lb x|z\rb_w \gtrsim_\plus \min(\lb x|y\rb_w,\lb y|z\rb_w).
\end{equation}
We will refer to this inequality as \emph{Gromov's inequality}.
\end{definition}
Hyperbolic metric spaces should be thought of as ``those which have uniformly negative curvature''. Note that many authors require $X$ to be a geodesic metric space in order to be hyperbolic; we do not. If $X$ is a geodesic metric space, then the condition of hyperbolicity can be reformulated in several different ways, including the \emph{thin triangles condition}; see \cite[\sectionsymbol III.H.1]{BridsonHaefliger} for details.

For each $z\in X$, let $\busemann_z$ denote the \emph{Busemann function}
\begin{equation}
\label{busemanndef}
\busemann_z(x,y) := \dist(z,x) - \dist(z,y).
\end{equation}
\begin{proposition}
\label{propositionbasicidentities}
The Gromov product satisfies the following identities and inequalities:
\begin{align*}\tag{a}
\lb x|y\rb_z &= \lb y|x\rb_z \\ \tag{b}
\dist(y,z) &= \lb y|x\rb_z + \lb z|x\rb_y\\ \tag{c}
0\leq \lb x|y\rb_z &\leq \min(\dist(x,z),\dist(y,z))\\ \tag{d}
\lb x|y\rb_z &\leq \lb x|y\rb_w + \dist(z,w)\\ \tag{e}
\lb x|y\rb_w &\leq \lb x|z\rb_w + \dist(y,z)\\ \tag{f}
\lb x|y\rb_z &= \lb x|y\rb_w + \frac{1}{2}[\busemann_x(z,w) + \busemann_y(z,w)]\\ \tag{g}
\lb x|y\rb_z &= \frac{1}{2}[\dist(x,z) + \busemann_y(z,x)]\\ \tag{h}
\busemann_x(y,z) &= \lb z|x\rb_y - \lb y|x\rb_z\\ \tag{i}
\lb x|y\rb_z &= \lb x|y\rb_w + \dist(z,w) - \lb x|z\rb_w - \lb y|z\rb_w\\ \tag{j}
\lb x|y\rb_z &= \lb x|y\rb_w + \lb x|w\rb_z - \lb y|z\rb_w 
\end{align*}
\end{proposition}
The proof is a straightforward computation. We remark that (a)-(e) may be found in \cite[Lemma 2.8]{Vaisala}.

\subsection{Examples of hyperbolic metric spaces} \label{sectionexamples}
The class of hyperbolic metric spaces is rich and well-studied. For example, any subspace of a hyperbolic metric space is again a hyperbolic metric space, and any geodesic metric space quasi-isometric to a geodesic hyperbolic metric space is again a hyperbolic metric space \cite[Theorem III.H.1.9]{BridsonHaefliger}. Moreover, for every $\kappa < 0$, every CAT($\kappa$) space is hyperbolic \cite[Proposition III.H.1.2]{BridsonHaefliger}. We list below some well-known celebrated examples with no desire for completeness. Items with $*$ are CAT($\kappa$) spaces for some $\kappa < 0$, and items with \textdagger\ are proper metric spaces, i.e. those for which closed balls are compact.

\begin{itemize}
\item Standard hyperbolic space $\H^{d + 1}$ or $\B^{d + 1}$ (*,\textdagger).
\item Infinite-dimensional hyperbolic space, i.e. the open unit ball $\B^\infty$ in a separable Hilbert space with the Poincar\'e metric (*). This space and groups of isometries acting on it is investigated at length in \cite{DSU}.
\item Trees (*,\textdagger\ if locally finite)
\item $\R$-trees (*, sometimes \textdagger).
\item Word-hyperbolic groups together with their Cayley metrics (\textdagger). We remark that this includes almost every finitely presented group \cite[p.78]{Gromov4}, \cite{Olshanskii}; see also \cite{Champetier}.
\item All Riemannian manifolds with negative sectional curvature uniformly bounded away from zero (*,\textdagger).
\item Complex and quaternionic hyperbolic spaces (*,\textdagger). These will be discussed in some detail in Appendix \ref{appendixsymmetricspaces}.
\item Green metrics on word-hyperbolic groups \cite[Corollary 1.2]{BHM}
\item Quasihyperbolic metrics of uniform domains in Banach spaces \cite[Theorem 2.12]{Vaisala2}
\item Arc graphs and curve graphs \cite{HPW} and arc complexes \cite{MasurSchleimer, HilionHorbez} of finitely punctured oriented surfaces
\item Free splitting complexes \cite{HandelMosher, HilionHorbez} and free factor complexes \cite{BestvinaFeighn, KapovichRafi, HilionHorbez}
\end{itemize}

\subsection{The boundary of a hyperbolic metric space}
In this subsection, we define the Gromov boundary of a hyperbolic metric space $(X,\dist)$. If $X = \H^{d + 1}$ or $\B^{d + 1}$ for some $d\in\N$, then the boundary of $X$ is exactly what you would expect, i.e. it is isomorphic to the topological boundary of $X$ relative to $\what{\R^{d + 1}}$.

In the following, we fix a distinguished point $\zero\in X$. However, it will be clear that the definitions given are independent of which point is chosen.


\begin{definition}
A sequence $(x_n)_1^\infty$ in $X$ is called a \emph{Gromov sequence} if
\[
\lb x_n|x_m\rb_\zero \tendsto{n,m} \infty.
\]
Two Gromov sequences $(x_n)_1^\infty$ and $(y_n)_1^\infty$ are called \emph{equivalent} if
\[
\lb x_n|y_n\rb_\zero \tendsto n \infty.
\]
It is readily verified using Gromov's inequality that equivalence of Gromov sequences is indeed an equivalence relation. 
\end{definition}
\begin{definition}
\label{definitiongromovboundary}
The \emph{Gromov boundary} of $X$ is the set of Gromov sequences modulo equivalence. It is denoted $\del X$. The \emph{Gromov bordification} of $X$ is the set $\bord X := X\sqcup\del X$. (The term ``bordification'' is similar to ``compactification'', except that the resulting space is not necessarily compact.)
\end{definition}

\begin{remark}
\label{remarkisometryextension}
Let $g$ be an isometry of $X$. Then $g$ takes Gromov sequences to Gromov sequences, and equivalent Gromov sequences to equivalent Gromov sequences. Taking the quotient yields a map $g_{\del X}:\del X\to\del X$. Let $g_{\bord X}:\bord X\to\bord X$ be the piecewise combination of $g$ with $g_{\del X}$.\comdavid{Is it clear?} Abusing notation, we will in the sequel denote the maps $g$, $g_{\del X}$, and $g_{\bord X}$ by the same symbol $g$.
\end{remark}

\subsubsection{Extending the Gromov product and Busemann function to the boundary}
For all points $\xi,\eta\in\del X$ and $y,z\in X$, the appropriate Gromov products and Busemann function are defined as follows:
\begin{align} \label{gromovboundarydef1}
\lb \xi|\eta\rb_z
&:= \inf\{\liminf_{n,m\to\infty}\lb x_n|y_m\rb_z:(x_n)_1^\infty\in\xi,(y_m)_1^\infty\in\eta\}\\ \label{gromovboundarydef2}
\lb \xi|y\rb_z
&:= \lb y|\xi\rb_z := \inf\{\liminf_{n\to\infty}\lb x_n|y\rb_z:(x_n)_1^\infty\in\xi\}\\ \label{gromovboundarydef3}
\busemann_\xi(y,z) &:= \lb z|\xi\rb_y - \lb y|\xi\rb_z,
\end{align}
keeping in mind that $\xi,\eta\in\del X$ are collections of Gromov sequences.

The main properties of the Gromov product on the boundary come from the following lemma:
\begin{lemma}[{\cite[Lemma 5.11]{Vaisala}}]
\label{lemmaboundaryasymptotic}
Fix $\xi,\eta\in\del X$ and $y,z\in X$. For all $(x_n)_1^\infty\in\xi$ and $(y_m)_1^\infty\in\eta$ we have 
\begin{align} \label{boundaryasymptotic1}
\lb x_n|y_m\rb_z &\tendsto{n,m,\plus} \lb \xi|\eta\rb_z\\ \label{boundaryasymptotic2}
\lb x_n|y\rb_z &\tendsto{\;\; n,\plus \;\;} \lb \xi|y\rb_z\\ \label{boundaryasymptotic3}
\busemann_{x_n}(y,z) &\tendsto{\;\; n,\plus \;\;} \busemann_\xi(y,z).
\end{align}
\end{lemma}
(\eqref{boundaryasymptotic3} is not found in \cite{Vaisala} but it follows immediately from \eqref{boundaryasymptotic2}.)

A simple but useful consequence of Lemma \ref{lemmaboundaryasymptotic} is the following:

\begin{corollary}
\label{corollaryboundaryasymptotic}
The formulas of Proposition \ref{propositionbasicidentities} together with Gromov's inequality hold for points on the boundary as well, if the equations and inequalities there are replaced by additive asymptotics.\comdavid{Ignored proof below}
\end{corollary}


\subsubsection{A topology on $\bord X$}\label{subsubsectiontopologyclX}
One can endow $\bord X$ with a topological structure. We will call a set $S\subset\bord X$ open if $S\cap X$ is open and if for each $\xi\in S\cap\del X$ there exists $t > 0$ such that $N_t(\xi)\subset S$, where
\begin{equation} \label{nbh}
N_t(\xi) := \{y\in \bord X:\lb y|\xi\rb_\zero > t\}.
\end{equation}
We remark that with the topology defined above, $\bord X$ is metrizable; see \cite[Exercise III.H.3.18(4)]{BridsonHaefliger} if $X$ is proper and geodesic, and \cite[Corollary 3.6.14]{DSU} for the general case.

\begin{observation}
\label{observationboundaryconvergence}
A sequence $(x_n)_1^\infty$ in $\bord X$ converges to a point $\xi\in \del X$ if and only if
\begin{equation}\label{boundaryconvergence}
\lb x_n|\xi\rb_\zero \tendsto n \infty.
\end{equation}
\end{observation}

\begin{observation}
\label{observationboundaryconvergence2}
A sequence $(x_n)_1^\infty$ in $X$ converges to a point $\xi\in \del X$ if and only if $(x_n)_1^\infty$ is a Gromov sequence and $(x_n)_1^\infty \in \xi$.
\end{observation}

\begin{lemma}[Near-continuity of the Gromov product and Busemann function]
\label{lemmanearcontinuity}
Suppose that $(x_n)_1^\infty$ and $(y_n)_1^\infty$ are sequences in $\bord X$ which converge to points $x_n\tendsto n x\in\bord X$ and $y_n\tendsto n y\in\bord X$. Suppose that $(z_n)_1^\infty$ and $(w_n)_1^\infty$ are sequences in $X$ which converge to points $z_n\tendsto n z\in X$ and $w_n\tendsto n w\in X$. Then
\begin{align} \label{continuousextension1}
\lb x_n|y_n\rb_{z_n} &\tendsto{n,\plus} \lb x|y\rb_z\\ \label{continuousextension2}
\busemann_{x_n}(z_n,w_n) &\tendsto{n,\plus} \busemann_x(z,w),
\end{align}
\end{lemma}
\begin{proof}
In the proof of \eqref{continuousextension1}, there are three cases:
\begin{itemize}
\item[Case 1:] $x,y\in X$. In this case, \eqref{continuousextension1} follows directly from (d) and (e) of Proposition \ref{propositionbasicidentities}.
\item[Case 2:] $x,y\in\del X$. In this case, for each $n\in\N$, choose $\what x_n\in X$ such that either
\begin{itemize}
\item[(1)] $\what x_n = x_n$ (if $x_n\in X$), or
\item[(2)] $\lb \what x_n|x_n\rb_z\geq n$ (if $x_n\in\del X$).
\end{itemize}
Choose $\what y_n$ similarly. Clearly, $\what x_n\tendsto n x$ and $\what y_n \tendsto n y$. By Observation \ref{observationboundaryconvergence2}, $(\what x_n)_1^\infty \in x$ and $(\what y_n)_1^\infty \in y$. Thus by Lemma \ref{lemmaboundaryasymptotic},
\begin{equation}
\label{whatxwhaty}
\lb \what x_n|\what y_n\rb_z \tendsto{n,\plus} \lb x|y\rb_z.
\end{equation}
Now by Gromov's inequality and (e) of Proposition \ref{propositionbasicidentities}, either
\begin{align*} \tag{1}
\lb \what x_n|\what y_n\rb_z &\asymp_\plus \lb x_n|y_n\rb_{z_n} \text{ or }\\ \tag{2}
\lb \what x_n|\what y_n\rb_z &\gtrsim_\plus n,
\end{align*}
with which asymptotic is true depending on $n$. But for $n$ sufficiently large, \eqref{whatxwhaty} ensures that the (2) fails, so (1) holds.
\item[Case 3:] $x\in X$, $y\in\del X$, or vice-versa. In this case, a straightforward combination of the above arguments demonstrates \eqref{continuousextension1}.
\end{itemize}
Finally, note that \eqref{continuousextension2} is an immediate consequence of \eqref{continuousextension1}, \eqref{gromovboundarydef3}, and (h) of Proposition \ref{propositionbasicidentities}.
\end{proof}

\begin{remark}
If $g$ is an isometry of $X$, then the map $g_{\bord X}:\bord X\to\bord X$ defined in Remark \ref{remarkisometryextension} is a homeomorphism of $\bord X$.
\end{remark}

\subsection{Different (meta)metrics}
\label{subsectionmetametrics}

\begin{definition}
Recall that a \emph{metric} on a set $Z$ is a map $\Dist:Z\times Z\to\Rplus$ which satisfies:
\begin{itemize}
\item[(I)] Reflexivity: $\Dist(x,x) = 0$.
\item[(II)] Reverse reflexivity: $\Dist(x,y) = 0$ \implies $x = y$.
\item[(III)] Symmetry: $\Dist(x,y) = \Dist(y,x)$.
\item[(IV)] Triangle inequality: $\Dist(x,z) \leq \Dist(x,y) + \Dist(y,z)$.
\end{itemize}
Now we can define a \emph{metametric} on $Z$ to be a map $\Dist: Z\times Z\to\Rplus$ which satisfies (II), (III), and (IV), but not necessarily (I). This concept is not to be confused with the more common notion of a \emph{pseudometric}, which satisfies (I), (III), and (IV), but not necessarily (II). The term ``metametric'' was introduced by J. V\"ais\"al\"a in \cite{Vaisala}.

If $\Dist$ is a metametric, we define its \emph{domain of reflexivity} as the set $Z_\refl := \{x\in Z:\Dist(x,x) = 0\}$. Obviously, $\Dist$ restricted to its domain of reflexivity is a metric.
\end{definition}

Now let $(X,\dist)$ be a hyperbolic metric space. We will consider various metametrics on $\bord X$.

\subsubsection{The visual metametric based at a point $z\in X$}
\begin{proposition}
\label{propositioneuclideanball}
If $X = \B^{d + 1}$, then for all $\xx,\yy\in\del X$ we have
\begin{equation}
\label{euclideanball2}
\frac{1}{2}\|\yy - \xx\| = e^{-\lb \xx|\yy\rb_\0}.
\end{equation}
\end{proposition}
\begin{proof}
See \cite[\sectionsymbol III.H.3.19]{BridsonHaefliger} for $\asymp_\times$ and \cite[Lemma 3.5.1]{DSU} for equality.
\end{proof}
The following proposition generalizes \eqref{euclideanball2} to the setting of hyperbolic metric spaces:
\begin{proposition}[{\cite[Propositions 5.16 and 5.31]{Vaisala}}]
\label{propositionDist}
For each $b > 1$ sufficiently close to $1$ and for each $z\in X$, there exists a complete metametric $\Dist_z = \Dist_{b,z}$ on $\del X$ satisfying
\begin{equation}
\label{distanceasymptotic}
\Dist_{b,z}(y_1,y_2) \asymp_\times b^{-\lb y_1|y_2\rb_z}
\end{equation}
for all $y_1,y_2\in\bord X$. The implied constant may depend on $b$ but not on $z$.
\end{proposition}

We will refer to $\Dist_{b,z}$ as the ``visual (meta)metric from the point $z$ with respect to the parameter $b$''.

The metric $\Dist_{b,z}\given\del X$ has been referred to in the literature as the \emph{Bourdon metric}. We remark that this metric is compatible with the topology defined in \sectionsymbol\ref{subsubsectiontopologyclX} restricted to $\del X$.

\begin{remark}
For CAT(-1) spaces, Proposition \ref{propositionDist} holds for any $1 < b \leq e$; moreover, the asymptotic in (\ref{distanceasymptotic}) may be replaced by an equality; see \cite{Bourdon} if $X$ is proper and \cite[Proposition 3.3.4]{DSU} for the general case.
\end{remark}

\subsubsection{Standing assumptions for Section \ref{sectionhyperbolic}}\label{standingassumptions2}
For the remainder of Section \ref{sectionhyperbolic}, we will have the following standing assumptions:
\begin{itemize}
\item[(I)] $(X,\dist)$ is a hyperbolic metric space
\item[(II)] $\zero\in X$ is a distinguished point
\item[(III)] $b > 1$ is a parameter close enough to $1$ to guarantee for every $z\in X$ the existence of a visual metametric $\Dist = \Dist_{b,z}$ via Proposition \ref{propositionDist} above.
\end{itemize}

\subsubsection{The visual metametric based at a point $\xi\in\del X$}

Our next metametric is supposed to generalize the Euclidean metric on the boundary of the upper half-plane model $\H^{d + 1}$. This metric should be thought of as ``seen from the point $\infty$''.

\begin{notation}
If $X$ is a hyperbolic metric space and $\xi\in\del X$, then let $\EE_\xi := \bord X\butnot\{\xi\}$.
\end{notation}

We will motivate the visual metametric based at a point $\xi\in\del X$ by considering a sequence $(z_n)_1^\infty$ in $X$ converging to $\xi$, and taking the limits of their visual metametrics.

In fact, $\Dist_{b,z_n}(y_1,y_2)\tendsto n 0$ for every $y_1, y_2\in\EE_\xi$. Some normalization is needed.
\begin{lemma}
\label{lemmaeuclideanmetametric}
Let $(X,\dist,\zero,b)$ be as in \sectionsymbol\ref{standingassumptions2}, and suppose $z_n\tendsto n\xi\in\del X$. Then for each $y_1, y_2\in\EE_\xi$,
\[
b^{d(\zero,z_n)} \Dist_{b,z_n} (y_1,y_2) \tendsto{n,\times} b^{-[\lb y_1 | y_2 \rb_\zero - \sum_{i = 1}^2 \lb y_i | \xi \rb_\zero]} .
\]
\end{lemma}
\begin{proof}
\begin{align*}
b^{\dist(\zero, z_n)} \Dist_{b,z_n}(y_1,y_2)
&\asymp_\times b^{-[\lb y_1 | y_2 \rb_{z_n} - \dist(\zero,z_n)]} \\
&\asymp_\times b^{-[\lb y_1 | y_2 \rb_\zero - \sum_{i = 1}^2 \lb y_i | z_n \rb_\zero]} \by{(i) of Proposition \ref{propositionbasicidentities}}\\
&\tendsto{n,\times} b^{-[\lb y_1 | y_2 \rb_\zero - \sum_{i = 1}^2 \lb y_i | \xi \rb_\zero]}. \by{Lemma \ref{lemmanearcontinuity}}
\end{align*}
\end{proof}

\begin{corollary}
\label{corollaryeuclideanmetametric}
Let $(X,\dist,\zero,b)$ be as in \sectionsymbol\ref{standingassumptions2}. Then for each $\xi\in\del X$, there exists a metametric $\Dist_{\xi,\zero} = \Dist_{b,\xi,\zero}$ on $\EE_\xi$ satisfying
\begin{equation}
\label{euclideanmetametric}
\Dist_{b,\xi,\zero}(y_1,y_2) \asymp_\times b^{-[\lb y_1 | y_2 \rb_\zero - \sum_{i = 1}^2 \lb y_i | \xi \rb_\zero]} \ .
\end{equation}
\end{corollary}
The metric $\Dist_{b,\xi,\zero}\given\EE_\xi\cap\del X$ has been referred to in the literature as the \emph{Hamenst\"adt metric}.

\begin{proof}[Proof of Corollary \ref{corollaryeuclideanmetametric}]
Let
\[
\Dist_{b,\xi,\zero}(y_1,y_2) = \limsup_{z\to\xi}b^{\dist(\zero,z)}\Dist_{b,z}(y_1,y_2).
\]
Since the class of metametrics is closed under suprema and limits, it follows that $\Dist_{b,\xi,\zero}$ is a metametric.
\end{proof}

From Lemma \ref{lemmaeuclideanmetametric} and Corollary \ref{corollaryeuclideanmetametric} it immediately follows that
\begin{equation}
\label{euclideanmetrictendsasymp}
b^{d(\zero,z_n)} \Dist_{b,z_n} (y_1,y_2) \tendsto{n,\times} \Dist_{b,\xi,\zero}(y_1,y_2)
\end{equation}
whenever $(z_n)_1^\infty \in \xi$.

\begin{lemma}
\label{lemmaeuclideancomparison}
Let $(X,\dist,\zero,b)$ be as in \sectionsymbol\ref{standingassumptions2}. For all $x\in X$ and $\xi \in \del X$, we have
\begin{equation}
\label{euclideancomparison}
\Dist_{\xi,\zero}(\zero,x) \asymp_\times b^{\lb x | \xi \rb_\zero} \asymp_\times \frac{1}{\Dist_\zero(x,\xi)} .
\end{equation}
\end{lemma}

\begin{proof}
By \eqref{euclideanmetametric},
\[
\Dist_{\xi,\zero}(\zero,x) \asymp_\times b^{-[\lb \zero|x\rb_\zero - \lb \zero|\xi\rb_\zero - \lb x|\xi\rb_\zero]} = b^{\lb x|\xi\rb_\zero}.
\]
\end{proof}

\subsubsection{Comparison of different metametrics}

\begin{observation}
\label{observationGMVT}
Let $(X,\dist,\zero,b)$ be as in \sectionsymbol\ref{standingassumptions2}, and fix $y_1,y_2\in\bord X$.
\begin{itemize}
\item[(i)] For all $z_1,z_2\in X$, we have
\begin{equation}
\label{GMVT1}
\frac{\Dist_{z_1}(y_1,y_2)}{\Dist_{z_2}(y_1,y_2)} \asymp_\times b^{-\frac{1}{2}[\busemann_{y_1}(z_1,z_2) + \busemann_{y_2}(z_1,z_2)]}.
\end{equation}
\item[(ii)] For all $z\in X$ and $\xi\in\del X$, we have
\begin{equation}
\label{GMVT2}
\frac{\Dist_{\xi,z}(y_1,y_2)}{\Dist_z(y_1,y_2)} \asymp_\times b^{-[\lb y_1|\xi\rb_z + \lb y_2|\xi\rb_z]}.
\end{equation}
\item[(iii)] For all $z_1,z_2\in X$ and $\xi\in\del X$, we have
\begin{equation}
\label{GMVT3}
\frac{\Dist_{\xi,z_1}(y_1,y_2)}{\Dist_{\xi,z_2}(y_1,y_2)} \asymp_\times b^{\busemann_\xi(z_1,z_2)}.
\end{equation}
\end{itemize}
\end{observation}
\begin{proof}
(i) follows from the (f) of Proposition \ref{propositionbasicidentities}, (ii) follows from \eqref{euclideanmetametric}, and (iii) follows from \eqref{euclideanmetrictendsasymp}.
\end{proof}

\begin{corollary}
\label{corollaryGMVT}
Let $(X,\dist,\zero,b)$ be as in \sectionsymbol\ref{standingassumptions2}, and fix $y_1,y_2\in\bord X$. Then for all $g\in\Isom(X)$,
\begin{equation}
\label{GMVT4}
\frac{\Dist(g(y_1),g(y_2))}{\Dist(y_1,y_2)} \asymp_\times b^{(1/2)[\busemann_{y_1}(\zero,g^{-1}(\zero)) + \busemann_{y_2}(\zero,g^{-1}(\zero))]}.
\end{equation}
\end{corollary}
\begin{proof}
Note that $\Dist\circ g = \Dist_{g^{-1}(\zero)}$, and apply \eqref{GMVT1}.
\end{proof}

\begin{notation}
From now on we will write
\[
g'(y) = b^{\busemann_y(\zero,g^{-1}(\zero))},
\]
so that \eqref{GMVT4} may be rewritten as the \emph{geometric mean value theorem}
\begin{equation}
\label{GMVT5}
\frac{\Dist(g(y_1),g(y_2))}{\Dist(y_1,y_2)} \asymp_\times \left(g'(y_1) g'(y_2)\right)^{1/2}.
\end{equation}
\end{notation}
To justify this notation somewhat, note that if $y_1\in\bord X$ is not an isolated point then
\[
g'(y_1) \asymp_\times \lim_{y_2\to y_1}\frac{\Dist(g(y_1),g(y_2))}{\Dist(y_1,y_2)}
\]
by \eqref{GMVT5} together with Lemma \ref{lemmanearcontinuity}.

We end this subsection with a proposition which shows the relation between the derivative of an isometry $g\in\Isom(X)$ at a point $\xi\in\Fix(g)$ and the action on the metametric space $(\EE_\xi,\Dist_{\xi,\zero})$:

\begin{proposition}
\label{propositioneuclideansimilarity}
Fix $g\in\Isom(X)$ and $\xi\in\Fix(g)$. Then for all $y_1,y_2\in\EE_\xi$,
\[
\Dist_{\xi,\zero}(g(y_1), g(y_2)) \asymp_\times g'(\xi)^{-1} \Dist_{\xi,\zero}(y_1, y_2).
\]
\end{proposition}
\begin{proof}
\begin{align*}
\Dist_{\xi,\zero}(g(y_1), g(y_2)) &=_\pt \Dist_{\xi,g^{-1}(\zero)}(y_1, y_2) \\
&\asymp_\times b^{- \busemann_\xi(\zero,g^{-1}(\zero))} \Dist_{\xi,\zero}(y_1,y_2) \by{\eqref{GMVT3}} \\
&\asymp_\times g'(\xi)^{-1} \Dist_{\xi,\zero}(y_1,y_2) \ .
\end{align*}
\end{proof}

\subsection{The geometry of shadows} \label{subsectionshadows}
Recall that if $X = \H^{d + 1}$ or $\B^{d + 1}$, then for each $z\in X$ the projection map $\pi_z:X\setminus\{z\}\to\del X$ is defined to be the unique map so that for all $x\in X\setminus\{z\}$, $x$ is on the geodesic ray joining $z$ and $\pi_z(x)$. For $x\in X$ and $\sigma > 0$, it is useful to consider the set $\pi_z(B(x,\sigma))$, which is called the ``shadow'' of the ball $B(x,\sigma)$ with respect to the point $z$. This definition does not make sense in our setting, since a hyperbolic metric space does not necessarily have geodesics. Thus we replace it with the following definition which uses only the Gromov product:

\begin{definition}
Let $(X,\dist)$ be a hyperbolic metric space. For each $\sigma > 0$ and $x,z\in X$, let
\[
\Shad_z(x,\sigma) = \{\eta\in\del X:\lb z|\eta\rb_x\leq\sigma\}.
\]
We say that $\Shad_z(x,\sigma)$ is the \emph{shadow cast by $x$ from the light source $z$, with parameter $\sigma$}. For shorthand we will write $\Shad(x,\sigma) = \Shad_\zero(x,\sigma)$.
\end{definition}
Although we will not need it in our proofs, let us remark that the sets $\Shad_z(x,\sigma)$ are closed \cite[Observation 4.5.2]{DSU}.

\begin{proposition}[{\cite[Corollary 4.5.5]{DSU}}]
Let $X = \H^{d + 1}$ or $\B^{d + 1}$. For every $\sigma > 0$, there exists $\tau = \tau_\sigma > 0$ such that for any $x,z\in X$ we have
\[
\pi_z(B(x,\sigma)) \subset \Shad_z(x,\sigma) \subset \pi_z(B(x,\tau)).
\]
\end{proposition}

Let us establish up front some geometric properties of shadows.
\begin{observation}\label{observationmetricshadow}
If $\eta\in \Shad_z(x,\sigma)$, then
\[
\lb x|\eta\rb_z \asymp_{\plus,\sigma} \dist(z,x).
\]
\end{observation}
\begin{proof}
We have $\lb x|\eta\rb_z
\leq \dist(z,x)
\asymp_\plus \lb x|\eta\rb_z + \lb z|\eta\rb_x
\leq \lb x|\eta\rb_z + \sigma$.
\end{proof}

\begin{lemma}[Intersecting Shadows Lemma]
\label{lemmatau}
Let $(X,\dist)$ be a hyperbolic metric space. For each $\sigma > 0$, there exists $\tau = \tau_\sigma > 0$ such that for all $x,y,z\in X$ satisfying $\dist(z,y)\geq \dist(z,x)$ and $\Shad_z(x,\sigma)\cap \Shad_z(y,\sigma)\neq\emptyset$, we have
\begin{equation}
\label{shadcontainment}
\Shad_z(y,\sigma) \subset\Shad_z(x,\tau)
\end{equation}
and
\begin{equation}
\label{distbusemann}
\dist(x,y) \asymp_{\plus,\sigma} \busemann_z(y,x).
\end{equation}
\end{lemma}

\begin{proof}
Fix $\eta\in\Shad_z(x,\sigma)\cap \Shad_z(y,\sigma)$, so that by Observation \ref{observationmetricshadow}
\[
\lb x|\eta\rb_z \asymp_{\plus,\sigma} \dist(z,x) \text{ and } \
\lb y|\eta\rb_z \asymp_{\plus,\sigma} \dist(z,y) \geq \dist(z,x).
\]
Gromov's inequality along with (c) of Proposition \ref{propositionbasicidentities} then gives
\begin{equation}\label{distbusemann2}
\lb x|y\rb_z \asymp_{\plus,\sigma} \dist(z,x).
\end{equation}

Rearranging yields (\ref{distbusemann}). In order to show (\ref{shadcontainment}), fix $\xi\in\Shad_z(y,\sigma)$, so that $\lb y|\xi\rb_z \asymp_{\plus,\sigma} \dist(z,y)\geq \dist(z,x)$. Gromov's inequality and \eqref{distbusemann2} then give
\[
\lb x|\xi\rb_z \asymp_{\plus,\sigma} \dist(z,x),
\]
i.e. $\xi\in\Shad_z(x,\tau)$ for some $\tau > 0$ sufficiently large (depending on $\sigma$).
\end{proof}

\begin{lemma}[Bounded Distortion Lemma]
\label{lemmaboundeddistortion}
Let $(X,\dist,\zero,b)$ be as in \sectionsymbol\ref{standingassumptions2}, and fix $\sigma > 0$. Then for every $g\in\Isom(X)$ and for every $\xi\in\Shad_{g^{-1}(\zero)}(\zero,\sigma)$ we have
\begin{equation}
\label{boundeddistortion1}
g'(\xi) \asymp_{\times,\sigma} b^{-\dist(\zero,g(\zero))}.
\end{equation}
Moreover, for every $\xi_1,\xi_2\in \Shad_{g^{-1}(\zero)}(\zero,\sigma)$, we have
\begin{equation}
\label{boundeddistortion2}
\frac{\Dist (g(\xi_1),g(\xi_2) )}{\Dist(\xi_1,\xi_2)} \asymp_{\times,\sigma} b^{-\dist(\zero,g(\zero))}.
\end{equation}
\end{lemma}
\begin{proof}
By (g) of Proposition \ref{propositionbasicidentities},
\[
g'(\xi) = b^{\busemann_\xi(\zero,g^{-1}(\zero))}
\asymp_\times b^{2\lb g^{-1}(\zero)|\xi\rb_\zero - \dist(\zero,g(\zero))}
\asymp_{\times,\sigma} b^{-\dist(\zero,g(\zero))}.
\]

Now \eqref{boundeddistortion2} follows from \eqref{boundeddistortion1} and the geometric mean value theorem \eqref{GMVT5}.
\end{proof}

\begin{lemma}[Big Shadows Lemma]
\label{lemmabigshadow}
Let $(X,\dist,\zero,b)$ be as in \sectionsymbol\ref{standingassumptions2}. For every $\varepsilon > 0$, for every $\sigma > 0$ sufficiently large, and for every $z\in X$, we have
\begin{equation}
\label{bigshadow}
\Diam(\del X\butnot\Shad_z(\zero,\sigma)) \leq \varepsilon.
\end{equation}
\end{lemma}
\begin{proof}
If $\xi, \eta \in \del X \setminus \Shad_z(\zero, \sigma)$, then $\lb z | \xi \rb_\zero > \sigma$ and $\lb z | \eta \rb_\zero > \sigma$. Thus by Gromov's inequality we have
\[
\lb \xi | \eta \rb_\zero \gtrsim_\plus \sigma.
\]
Exponentiating gives $\Dist(\xi,\eta) \lesssim_\times b^{-\sigma}$. Thus
\[
\Diam (\del X \setminus \Shad_z (\zero,\sigma)) \lesssim_\times b^{-\sigma} \tendsto\sigma 0,
\]
and the convergence is uniform in $z$.
\end{proof}

\begin{lemma}[Diameter of Shadows Lemma]
\label{lemmadiameterasymptotic}
Let $(X,\dist,\zero,b)$ be as in \sectionsymbol\ref{standingassumptions2}, and let $G\leq\Isom(X)$. Then for all $\sigma > 0$ sufficiently large, we have for all $x\in G(\zero)$ and for all $z\in X$
\[
\Diam_z(\Shad_z(x,\sigma)) \lesssim_{\times,\sigma} b^{-\dist(z,x)},
\]
with $\asymp_{\times,\sigma}$ if $\#(\del X)\geq 3$.
\end{lemma}
\begin{proof}
For any $\xi,\eta\in \Shad_z(x,\sigma)$, we have
\begin{align*}
\Dist_{b,z}(\xi,\eta) \asymp_\times b^{-\lb\xi|\eta\rb_z}
&\lesssim_{\times\phantom{,\sigma}} b^{-\min\left(\lb x|\xi\rb_z,\lb x,\eta\rb_z\right)}\\
&\lesssim_{\times,\sigma} b^{-\dist(z,x)}
\end{align*}
which proves the $\lesssim$ direction.

Now let us prove the $\gtrsim$ direction, assuming $\#(\del X)\geq 3$. Fix $\xi_1,\xi_2,\xi_3\in\del X$ distinct, let $\varepsilon = \min_{i\neq j}\Dist(\xi_i,\xi_j)/2$, and fix $\sigma > 0$ large enough so that \eqref{bigshadow} holds for every $z\in X$.

Now fix $x = g(\zero)\in G(\zero)$ and $z\in X$; by \eqref{bigshadow} we have
\[
\Diam(\del X\butnot \Shad_{g^{-1}(z)}(\zero,\sigma)) \leq \varepsilon,
\]
and thus
\[
\#\big\{i = 1,2,3: \xi_i\in \Shad_{g^{-1}(z)}(\zero,\sigma)\big\}\geq 2.
\]
Without loss of generality suppose that $\xi_1,\xi_2\in \Shad_{g^{-1}(z)}(\zero,\sigma)$. By applying $g$, we have $g(\xi_1),g(\xi_2)\in \Shad_z(x,\sigma)$. Then
\begin{align*}
\Diam_z(\Shad_z(x,\sigma))
&\geq_\pt \Dist_{b,z}(g(\xi_1),g(\xi_2))\\
&\asymp_\times b^{-\lb g(\xi_1)|g(\xi_2)\rb_z}
= b^{-\lb \xi_1|\xi_2\rb_{g^{-1}(z)}}\\
&\gtrsim_\times b^{-\lb\xi_1|\xi_2\rb_\zero}b^{-\dist(\zero,g^{-1}(z))}
\asymp_{\times,\xi_1,\xi_2}b^{-\dist(z,x)}.
\end{align*}
\end{proof}

\begin{remark}
If $G$ is nonelementary then $\#(\del X)\geq\#(\Lambda_G)\geq 3$ (see Definition \ref{definitionlimitset} below). Thus in our applications, we will always have $\asymp_{\times,\sigma}$ in Lemma \ref{lemmadiameterasymptotic}.
\end{remark}

We end this subsection with the definition of a horoball:

\begin{definition}
\label{definitionhoroball}
A \emph{horoball} centered at a point $\xi\in\del X$ is a sublevelset of the Busemann function of $\xi$, i.e. a set of the form
\[
H = \{x: \busemann_\xi(x,\zero) < t\}.
\]
\end{definition}

Such sets are called horoballs because in the case $X = \H^{d + 1}$, they are balls tangent to the boundary.

\subsection{Modes of convergence to the boundary; limit sets}
Let $(X,\dist)$ be a hyperbolic metric space, and fix a distinguished point $\zero\in X$. We recall (Observation \ref{observationboundaryconvergence}) that a sequence $(x_n)_1^\infty$ in $X$ converges to a point $\xi\in\del X$ if and only if
\[
\lb x_n|\xi\rb_\zero \tendsto{n} \infty.
\]
In this subsection we define more restricted modes of convergence. To get an intuition let us consider the case $X = \H^2$.

\begin{proposition}[{\cite[Proposition 7.1.1]{DSU}}]
Let $(x_n)_1^\infty$ be a sequence in $\H^2$ converging to a point $\xi\in\del\H^2$. Then the following are equivalent:
\begin{itemize}
\item[(A)] The sequence $(x_n)_1^\infty$ lies within a bounded distance of the geodesic ray $\geo\zero\xi$.
\item[(B)] There exists $\sigma > 0$ such that for all $n\in\N$,
\[
\lb\zero|\xi\rb_{x_n}\leq\sigma,
\]
or equivalently
\begin{equation}
\label{sigmaradially}
\xi\in\Shad(x_n,\sigma).
\end{equation}
\end{itemize}
\end{proposition}

Condition (A) motivates calling this kind of convergence \emph{radial}; we shall use this terminology henceforth. However, condition (B) is best suited to a general hyperbolic metric space.

\begin{definition}
Let $X$ be a hyperbolic metric space, and let $(x_n)_1^\infty$ be a sequence in $X$ converging to a point $\xi\in\del X$. We will say that $(x_n)_1^\infty$ converges to $\xi$
\begin{itemize}
\item \emph{$\sigma$-radially} if \eqref{sigmaradially} holds for all $n\in\N$,
\item \emph{radially} if it converges $\sigma$-radially for some $\sigma > 0$,
\item \emph{$\sigma$-uniformly radially} if it converges $\sigma$-radially, if
\begin{equation}
\label{sigmauniformlyradially}
\dist(x_n,x_{n + 1}) \leq \sigma \all n\in\N,
\end{equation}
and if $x_1 = \zero$,
\item \emph{uniformly radially} if it converges $\sigma$-uniformly radially for some $\sigma > 0$, and
\item \emph{horospherically} if
\[
\busemann_\xi(\zero,x_n)\tendsto n +\infty,
\]
or equivalently, if $\{x_n:n\in\N\}$ intersects every horoball centered at $\xi$.
\end{itemize}
\end{definition}

\begin{observation}
\label{observationdependenceonbasepoint}
The concepts of convergence, radial convergence, uniformly radial convergence, and horospherical convergence are independent of the basepoint $\zero$, whereas the concept of $\sigma$-radial convergence depends on the basepoint. (However, see Lemma \ref{lemmanearinvariance} below.)
\end{observation}

Let $G\leq\Isom(X)$. We define the limit set of $G$, a subset of $\del X$ which encodes geometric information about $G$. We also define a few important subsets of the limit set.

\begin{definition}
\label{definitionlimitset}
Let
\begin{align*}
\Lambda(G) &:= \{\eta\in\del X: g_n(\zero)\to \eta \text{ for some
             sequence $(g_n)_1^\infty$ in $G$}\}\\
\Lr(G) &:= \{\eta\in\del X: g_n(\zero)\to \eta\text{ radially for some
             sequence $(g_n)_1^\infty$ in $G$}\}\\
\Lur(G) &:= \{\eta\in\del X: g_n(\zero)\to \eta\text{ uniformly
             radially for some sequence $(g_n)_1^\infty$ in $G$}\}\\
\Lrsigma(G) &:= \{\eta\in\del X: g_n(\zero)\to \eta\text{ $\sigma$-radially for some
             sequence $(g_n)_1^\infty$ in $G$}\}\\
\Lursigma(G) &:= \{\eta\in\del X: g_n(\zero)\to \eta\text{ $\sigma$-uniformly radially for some
             sequence $(g_n)_1^\infty$ in $G$}\}\\
\Lh(G) &:= \{\eta\in\del X: g_n(\zero)\to \eta\text{ horospherically for some
             sequence $(g_n)_1^\infty$ in $G$}\}.
\end{align*}
These sets are respectively called the \emph{limit set}, \emph{radial limit set},
\emph{uniformly radial limit set}, \emph{$\sigma$-radial limit set}, \emph{$\sigma$-uniformly radial limit set}, and \emph{horospherical limit set} of the group $G$.
\end{definition}
Note that
\begin{align*}
\Lr &= \bigcup_{\sigma > 0}\Lrsigma\\
\Lur &= \bigcup_{\sigma > 0}\Lursigma\\
\Lur &\subset \Lr \subset \Lh \subset \Lambda.
\end{align*}

\begin{observation}
\label{observationlimitsets}
The sets $\Lambda$, $\Lr$, $\Lur$, and $\Lh$ are invariant under the action of $G$, and are independent of the basepoint $\zero$. The set $\Lambda$ is closed.
\end{observation}
The first assertion follows from Observation \ref{observationdependenceonbasepoint} and the second follows directly from the definition of $\Lambda$ as the intersection of $\del X$ with the set of accumulation points of the set $G(\zero)$.

\begin{lemma}[Near-invariance of the sets $\Lrsigma$]
\label{lemmanearinvariance}
For every $\sigma > 0$, there exists $\tau > 0$ such that for every $g\in G$, we have
\[
g(\Lrsigma) \subset \Lrtau.
\]
\end{lemma}
\begin{proof}
Fix $\xi\in\Lrsigma$. There exists a sequence $(h_n)_1^\infty$ so that $h_n(\zero)\tendsto n \xi$ $\sigma$-radially, i.e.
\[
\lb \zero|\xi\rb_{h_n(\zero)} \leq \sigma\all n\in\N
\]
and $h_n(\zero)\tendsto n \xi$. Now
\[
\lb \zero|g^{-1}(\zero)\rb_{h_n(\zero)} \geq \dist(\zero,h_n(\zero)) - \dist(\zero,g^{-1}(\zero)) \tendsto n \infty.
\]
Thus, for $n$ sufficiently large, Gromov's inequality gives
\[
\lb g^{-1}(\zero)|\xi\rb_{h_n(\zero)} \lesssim_\plus \sigma \asymp_{\plus,\sigma} 0,
\]
i.e.
\[
\lb \zero|g(\xi)\rb_{g\circ h_n(\zero)} \asymp_{\plus,\sigma} 0.
\]
So $g\circ h_n(\zero)\tendsto n g(\xi)$ $\tau$-radially, where $\tau$ is the implied constant of this asymptotic. Thus, $g(\xi)\in\Lrtau$.

\end{proof}

\subsection{Poincar\'e exponent}
Let $(X,\dist,\zero,b)$ be as in \sectionsymbol\ref{standingassumptions2}, and let $G\leq\Isom(X)$. For each $s\geq 0$, the series
\[
\Sigma_s(G):=\sum_{g\in G}b^{-sd(\zero,g(\zero))}
\]
is called the \emph{Poincar\'e series} of the group $G$ in dimension $s$ (or ``evaluated at $s$'') relative to $b$. The number
\begin{equation}
\label{poincareexponent}
\delta_G := \inf\{s\geq 0:\Sigma_s(G) < \infty\}
\end{equation}
is called the \emph{Poincar\'e exponent} of the group $G$ relative to $b$. Here, we let $\inf\emptyset = \infty$.

\begin{definition}
A group of isometries $G$ with $\delta_G < \infty$ is said to be of \emph{convergence type} if $\Sigma_{\delta_G}(G) < \infty$. Otherwise, it is said to be of \emph{divergence type}. In the case where $\delta_G = \infty$, we say that the group is neither of convergence type nor of divergence type.
\end{definition}

\subsection{Parabolic points}
\label{subsectionparabolic}
Let $(X,\dist,\zero,b)$ be as in \sectionsymbol\ref{standingassumptions2}, and let $G\leq\Isom(X)$.

\begin{definition}
Fix $\xi\in\del X$, and let $G_\xi$ be the stabilizer of $\xi$ relative to $G$. We say that $\xi$ is a \emph{parabolic fixed point} of $G$ if $G_\xi(\zero)$ is unbounded and if
\begin{equation}
\label{parabolic}
g'(\xi) \asymp_\times 1 \all g\in G_\xi.
\end{equation}
\end{definition}

Note that this definition together with Proposition \ref{propositioneuclideansimilarity} yields the following observation:
\begin{observation}
\label{observationuniformlyLipschitz}
Let $\xi$ be a parabolic fixed point of $G$. Then the action of $G_\xi$ on $(\EE_\xi,\Dist_{\xi,\zero})$ is uniformly Lipschitz, i.e.
\begin{equation}
\label{uniformlyLipschitz}
\Dist_{\xi,\zero}(g(y_1),g(y_2)) \asymp_\times \Dist_{\xi,\zero} (y_1,y_2) \all y_1,y_2\in\EE_\xi\all g\in G_\xi.
\end{equation}
\end{observation}
\begin{notation}
The implied constant of \eqref{uniformlyLipschitz} will be denoted $C_0$.
\end{notation}

\begin{observation}
\label{observationeuclideanparabolicasymp}
Let $\xi$ be a point satisfying \eqref{parabolic}. Then for all $g\in G_\xi$,
\begin{equation}
\label{euclideanparabolicasymp}
\Dist_{\xi,\zero}(\zero,g(\zero)) \asymp_\times b^{(1/2)\dist(\zero,g(\zero))}.
\end{equation}
\end{observation}
\begin{proof}
This is a direct consequence of Lemma \ref{lemmaeuclideancomparison}, (g) of Proposition \ref{propositionbasicidentities}, and \eqref{parabolic}.
\end{proof}

As a corollary we have the following:
\begin{observation}
\label{observationparabolic}
Let $\xi$ be a point satisfying \eqref{parabolic}. Then for any sequence $(g_n)_1^\infty$ in $G_\xi$,
\begin{equation}
\label{parabolicequivalent}
\dist(\zero,g_n(\zero))\tendsto n \infty \;\;\Leftrightarrow\;\; g_n(\zero)\tendsto n\xi.
\end{equation}
In particular, if $\xi$ is parabolic then $\xi\in\Lambda_G$.
\end{observation}
\begin{proof}
\[
g_n(\zero)\tendsto n\xi \;\;\Leftrightarrow\;\; \Dist(g_n(\zero),\xi) \tendsto n 0 \;\;\Leftrightarrow\;\; \Dist_{\xi,\zero}(\zero,g_n(\zero))\tendsto n \infty \;\;\Leftrightarrow\;\; \dist(\zero,g_n(\zero))\tendsto n \infty.
\]
If $\xi$ is parabolic, then $G_\xi(\zero)$ is unbounded, so there exists a sequence $(g_n)_1^\infty$ in $G_\xi$ satisfying the left hand side of \eqref{parabolicequivalent}. It also satisfies the right hand side of \eqref{parabolicequivalent}, which implies $\xi\in\Lambda_G$.
\end{proof}

\subsubsection{$\xi$-bounded sets}
To define the notion of a \emph{bounded} parabolic point, we first introduce the following definition:

\begin{definition}
\label{definitionxibounded}
A set $S\subset\bord X\butnot\{\xi\}$ is \emph{$\xi$-bounded} if $\xi\notin\cl{S}$.
\end{definition}
The motivation for this definition is that if $X = \H^{d + 1}$ and $\xi = \infty$, then $\xi$-bounded sets are exactly those which are bounded in the Euclidean metric.

\begin{observation}
\label{observationxibounded}
Fix $S\subset\bord X\butnot\{\xi\}$. The following are equivalent:
\begin{itemize}
\item[(A)] $S$ is $\xi$-bounded.
\item[(B)] $\lb x|\xi\rb_\zero\asymp_\plus 0$ uniformly for all $x\in S$.
\item[(C)] $\Dist_{\xi,\zero}(\zero,x)\lesssim_\times 1$ uniformly for all $x\in S$.
\item[(D)] $S$ has bounded diameter in the $\Dist_{\xi,\zero}$ metametric.
\end{itemize}
\end{observation}
Condition (D) motivates the terminology ``$\xi$-bounded''.

\begin{proof}[Proof of Observation \ref{observationxibounded}]
(A) \iff (B) follows from the definition of the topology on $\bord X$, (B) \iff (C) follows from Lemma \ref{lemmaeuclideancomparison}, and (C) \iff (D) is obvious.
\end{proof}

\subsubsection{Bounded parabolic points}

\begin{definition}\label{definitionboundedparabolic}
A parabolic point $\xi\in\Lambda_G$ is a \emph{bounded} parabolic point if there exists a $\xi$-bounded set $S\subset\bord X\butnot\{\xi\}$ such that
\[
G(\zero) \subset G_\xi(S).
\]
We denote the set of bounded parabolic points by $\Lbp$.
\end{definition}

\begin{proposition}
\label{propositionboundedparabolic}
Let $\xi$ be a parabolic point of a nonelementary group $G\leq\Isom(X)$. Then the following are equivalent:
\begin{itemize}
\item[(A)] $\xi$ is a bounded parabolic point.
\item[(B)]
\begin{itemize}
\item[(i)] There exists a $\xi$-bounded set $S$ such that $\Lambda_G\setminus\{\xi\}\subset G_\xi(S)$, and
\item[(ii)] There exists a horoball $H$ centered at $\xi$ which is disjoint from $G(\zero)$, i.e. $\xi$ is not a horospherical limit point.
\end{itemize}
\end{itemize}
\end{proposition}
\begin{proof}[Proof of \textup{(A) \implies (B)}]
Let $S$ be as in Definition \ref{definitionboundedparabolic}, and let $\w S = \NN_\euc(S,1)$, where $\NN_\euc(S,r) = \{x : \Dist_\euc(x,S) \leq r\}$. Then $\Lambda_G\setminus\{\xi\}\subset G_\xi(\w S)$, demonstrating (B)(i). By (g) of Proposition \ref{propositionbasicidentities}, for $x\in S$ we have
\[
0 \asymp_\plus \lb x|\xi\rb_\zero \gtrsim_\plus \frac{1}{2}\busemann_\xi(\zero,x),
\]
and since $\xi$ is parabolic, we have $\busemann_\xi(\zero,g(x))\lesssim_\plus 0$ for all $g\in G_\xi$. Thus $G_\xi(S)$ is disjoint from some horoball $H$ centered at $\xi$, and so $G(\zero)\subset G_\xi(S)$ is also disjoint from $H$, demonstrating (B)(ii).
\end{proof}
\begin{proof}[Proof of \textup{(B) \implies (A)}]
Fix $x\in G(\zero)$, and we will show that $\Dist_\euc(x,\Lambda_G\butnot\{\xi\})\lesssim_\times 1$. Since $G$ is nonelementary, we may fix $\eta_1,\eta_2\in\Lambda_G$ distinct. Letting $x = g(\zero)$, we have
\[
\lb g(\eta_1)|g(\eta_2)\rb_x \asymp_\plus 0,
\]
and so by Gromov's inequality, there exists $i = 1,2$ such that
\[
\lb g(\eta_i)|\xi\rb_x \asymp_\plus 0;
\]
we have
\begin{align*}
\Dist_\euc(x,g(\eta_i))
&\asymp_\times b^{\busemann_\xi(\zero,x)}\Dist_{\xi,x}(x,g(\eta_i)) \by{\eqref{GMVT3}} \\
&\lesssim_\times\Dist_{\xi,x}(x,g(\eta_i)) \since{$x\notin H$} \\
&\asymp_\times b^{\lb g(\eta_i)|\xi\rb_x} \by{Lemma \ref{lemmaeuclideancomparison}}\\
&\asymp_\times 1.
\end{align*}
Letting $r$ be the implied constant of this asymptotic yields $x\in \NN_\euc(\Lambda_G\butnot \{\xi\}),r)$. Thus if $S$ is a $\xi$-bounded set such that $\Lambda_G\butnot\{\xi\}\subset G_\xi(S)$, then $G(\zero) \subset G_\xi(\NN_\euc(S,C_0 r))$.
\end{proof}

\begin{remark}
If $X = \H^{d + 1}$, then (B)(i) of Proposition \ref{propositionboundedparabolic} implies (B)(ii) of Proposition \ref{propositionboundedparabolic}. (However, it is not clear whether the implication holds in general.) Thus our definition of a bounded parabolic point agrees with the usual one in the literature, which is condition (B)(i).
\end{remark}
\begin{proof}
Suppose that $\xi$ satisfies (B)(i) of Proposition \ref{propositionboundedparabolic}. By \cite[Lemma 4.1]{Bowditch_geometrical_finiteness}, for every $\varepsilon > 0$ there exists a horoball $H_\varepsilon$ centered at $\xi$ such that
\[
H_\varepsilon\cap \CC_G \subset T_\varepsilon(G) := \{x\in\H^{d + 1}: \lb g\in G: \dist(x,g(x))\leq \varepsilon\rb\text{ is infinite}\}.
\]
Here $\CC_G$ denotes the convex hull of the limit set of $G$.

Without loss of generality suppose $\zero\in\CC_G$. Choose $\varepsilon > 0$ small enough so that $\zero\notin T_\varepsilon(G)$; such a $\varepsilon$ exists since $G$ is discrete. Since $T_\varepsilon(G)$ is invariant under the action of $G$, we have $G(\zero)\cap T_\varepsilon(G) = \emptyset$. On the other hand, $G(\zero)\subset \CC_G$, so $G(\zero)\cap H_\varepsilon = \emptyset$.
\end{proof}

\subsection{Quasiconformal measures} \label{subsectionquasiconformal}

\begin{definition}
\label{definitionquasiconformal}
Let $(X,\dist,\zero,b,G)$ be as in \sectionsymbol\ref{standingassumptions}. For each $s\geq 0$, a measure $\mu\in\MM(\del X)$ is called \emph{$s$-quasiconformal} if
\begin{equation}\label{quasiconformal}
\mu(g(A)) \asymp_\times \int_A [g'(\xi)]^s\dee\mu(\xi)
\end{equation}
for all $g\in G$ and for all $A\subset\del X$. It is \emph{ergodic} if for every $G$-invariant set $A\subset\del X$, either $\mu(A) = 0$ or $\mu(\del X\butnot A) = 0$.
\end{definition}

In this subsection we discuss sufficient conditions for the existence and uniqueness of an ergodic $\delta$-quasiconformal measure. The well-known construction of Patterson and Sullivan (\cite{Patterson2} and \cite{Sullivan_density_at_infinity}) can be generalized to show the following:
\begin{theorem}[\cite{Coornaert, GreschonigSchmidt}] \label{theorempattersonsullivangeneral}
Let $(X,\dist,\zero,b,G)$ be as in \sectionsymbol\ref{standingassumptions}, and assume that $X$ is proper and geodesic, and that $\delta = \delta_G < \infty$. Then there exists an ergodic $\delta$-quasiconformal measure supported on $\Lambda$.
\end{theorem}
\begin{proof}
Everything except ``ergodic'' was proven in \cite{Coornaert}; see also \cite[Theorem 15.4.6]{DSU}. Let $\mu$ be a $\delta$-quasiconformal measure. Let $\varrho:G\times\del X\to\R$ satisfy (1.1)-(1.3) of \cite{GreschonigSchmidt}. Then by \cite[Theorem 1.4]{GreschonigSchmidt}, $\mu$ can be written as a convex combination of ergodic measures which are ``$\varrho$-admissible'' (in the terminology of \cite{GreschonigSchmidt}). But by (1.1) of \cite{GreschonigSchmidt}, we have $e^{\varrho(g,\xi)}\asymp_\times g'(\xi)^\delta$ for $\mu$-almost every $\xi\in\del X$, so $\varrho$-admissibility implies $\delta$-quasiconformality for all measures supported on the complement of a certain $\mu$-nullset. But this means that almost every measure in the convex combination is an ergodic $\delta$-quasiconformal measure.
\end{proof}

It is much more difficult to prove the existence of quasiconformal measures in the case where $X$ is not proper. It turns out that the condition of divergence type, which in the case $X = \H^{d + 1}$ is already known to imply uniqueness of the $\delta$-conformal measure, is the right condition to guarantee existence in a non-proper setting:

\begin{theorem}[{\cite[Theorem 1.4.1]{DSU}}]\label{theoremahlforsthurstongeneral}
Let $(X,\dist,\zero,b,G)$ be as in \sectionsymbol\ref{standingassumptions}. Suppose that $G$ is of divergence type. Then there exists a $\delta$-quasiconformal measure $\mu$ for $G$, where $\delta$ is the Poincar\'e exponent of $G$. It is unique up to a multiplicative constant in the sense that if $\mu_1,\mu_2$ are two such measures then $\mu_1\asymp_\times \mu_2$, meaning that $\mu_1$ and $\mu_2$ are in the same measure class and that the Radon--Nikodym derivative $\dee \mu_1/\dee \mu_2$ is bounded from above and below. In addition, $\mu$ is ergodic and gives full measure to the radial limit set of $G$.
\end{theorem}

\draftnewpage\section{Basic facts about Diophantine approximation} \label{sectionbasic}
In this section we prove Theorem \ref{theoremdirichlet} and Proposition \ref{propositionnonlimitpoint}.

\begin{theorem}[Generalization of Dirichlet's Theorem]
\label{theoremdirichlet}
Let $(X,\dist,\zero,b,G)$ be as in \sectionsymbol\ref{standingassumptions}, and fix a distinguished point $\xi\in\del X$. Then for every $\sigma > 0$, there exists a number $C = C_\sigma > 0$ such that for every $\eta\in\Lrsigma(G)$,
\begin{equation}
\label{dirichlet}
\Dist(g(\xi),\eta) \leq C b^{-\dist(\zero,g(\zero))}\text{ for infinitely many }g\in G.
\end{equation}
Here $\Lrsigma(G)$ denotes the $\sigma$-radial limit set of $G$.

Furthermore, there exists a sequence of isometries $(g_n)_1^\infty\subset G$ such that
\[
g_n(\xi) \tendsto n \eta  \text{ and } g_n(\zero) \tendsto n \eta,
\]
and \eqref{dirichlet} holds with $g = g_n$ for all $n\in\N$. 
\end{theorem}

\begin{proof}
Since the group $G$ is of general type, by Observation \ref{observationnoglobalfixedpoint} there exists $h\in G$ such that $h(\xi)\neq\xi$. Let $\xi_1 = \xi$ and let $\xi_2 = h(\xi)$. Since $\eta$ is a $\sigma$-radial limit point, there exists a sequence $(\w g_n)_1^\infty$ in $G$ such that $\w g_n(\zero)$ converges to $\eta$ $\sigma$-radially. We have
\[
0 \asymp_{\plus,\xi} \lb \xi_1|\xi_2\rb_\zero
\gtrsim_\plus \min\left(\lb \w g_n^{-1}(\zero)|\xi_1\rb_\zero,\lb \w g_n^{-1}(\zero)|\xi_2\rb_\zero\right).
\]
Let $i_n\in\{1,2\}$ be chosen so as to minimize $\lb \w g_n^{-1}(\zero) | \xi_{i_n} \rb_\zero$. We then have 
\[
\lb \w g_n^{-1}(\zero)|\xi_{i_n}\rb_\zero \asymp_{\plus,\xi} 0,
\]
or equivalently
\[
\lb\w g_n(\zero)|\w g_n(\xi_{i_n})\rb_\zero \asymp_{\plus,\xi} \dist(\zero,\w g_n(\zero)).
\]
On the other hand, since $\w g_n(\zero)$ converges to $\eta$ $\sigma$-radially we have
\[
\lb\w g_n(\zero)|\eta\rb_\zero \asymp_{\plus,\sigma} \dist(\zero,\w g_n(\zero)).
\]
Applying Gromov's inequality yields
\[
\lb\w g_n(\xi_{i_n})|\eta\rb_\zero \gtrsim_{\plus,\xi,\sigma} \dist(\zero,\w g_n(\zero))
\]
and thus
\[
\Dist(\w g_n(\xi_{i_n}),\eta) \lesssim_{\times,\xi,\sigma} b^{-\dist(\zero,\w g_n(\zero))}.
\]
If $i_n = 1$ for infinitely many $n$, then we are done by setting $g_n := \w g_n$. Otherwise,
\[
\Dist(\w g_n\circ h(\xi),\eta)
\lesssim_{\times,\xi,\sigma} b^{-\dist(\zero,\w g_n(\zero))}
\asymp_{\times,\xi} b^{-\dist(\zero,\w g_n\circ h(\zero))} 
\]
and we are done by setting $g_n := \w g_n\circ h$.
\end{proof}

If $\xi\in\del X\butnot\Lambda$, then Theorem \ref{theoremdirichlet} together with the following proposition show that the theory of Diophantine approximation by the orbit of $\xi$ is trivial.

\begin{proposition}
\label{propositionnonlimitpoint}
Let $(X,\dist,\zero,b,G)$ be as in \sectionsymbol\ref{standingassumptions}. For every $\xi\in\del X\butnot\Lambda$, we have $\BA_\xi = \Lr$.
\end{proposition}

\begin{proof}
We have $\Dist(\xi,\Lambda) > 0$ and so
\[
\lb\xi|\eta\rb_\zero \asymp_{\plus,\xi} 0
\]
for all $\eta\in\Lambda$. Now fix $\eta\in\Lr(G)$; for all $g\in G$, we have
\[
\lb g(\xi)|\eta\rb_{g(\zero)} = \lb \xi|g^{-1}(\eta)\rb_\zero \asymp_{\plus,\xi} 0.
\]
Thus by (d) of Proposition \ref{propositionbasicidentities} we have
\begin{align*}
\lb g(\xi)|\eta\rb_\zero 
&\lesssim_{\plus\phantom{,\xi}} \lb g(\xi)|\eta\rb_{g(\zero)} + \dist(\zero,g(\zero)) \\
&\asymp_{\plus,\xi} \dist(\zero,g(\zero)).
\end{align*}
So, $\Dist(g(\xi),\eta) \gtrsim_{\times,\xi} b^{-\dist(\zero,g(\zero))}$, i.e. $\eta$ is badly approximable with respect to $\xi$. We are done.
\end{proof}

\draftnewpage\section{Schmidt's game and McMullen's absolute game}
\label{sectiongames}
Throughout this section $(Z,\Dist)$ denotes a metric space. In our applications, $Z$ will not be a hyperbolic metric space; rather, it will be the boundary of a hyperbolic metric space (or some subset thereof), and $\Dist$ will be a visual metric as defined in Proposition \ref{propositionDist}.

We recall that for $0 < c < 1$, the metric space $(Z,\Dist)$ is said to be \emph{$c$-uniformly perfect} if for every $r\in(0,1)$ and for every $z\in Z$, $B(z,r)\setminus B(z,c r)\neq\emptyset$; it is \emph{uniformly perfect} if it is $c$-uniformly perfect for some $0 < c < 1$. Thus every uniformly perfect space is perfect but not necessarily vice versa.\Footnote{Recall that a metric space is called \emph{perfect} if it is compact and has no isolated points.}
\begin{definition}
\label{definitionahlforsregular}
For $s > 0$, a measure $\mu\in\MM(Z)$ is called \emph{Ahlfors $s$-regular} if 
\[
\mu(B(z,r)) \asymp_\times r^s
\]
for all $z\in\Supp(\mu)$\Footnote{Here and from now on $\Supp(\mu)$ denotes the topological support of a measure $\mu$.} and for all $r\in(0,1)$. The set $Z$ itself is called Ahlfors $s$-regular if it is equal to the $\Supp(\mu)$ for some Ahlfors $s$-regular measure $\mu\in\MM(Z)$.
\end{definition}
We note that it is well-known that if $Z$ is an Ahlfors $s$-regular metric space then its Hausdorff, packing, and box-counting dimensions are all equal to $s$ (see e.g. \cite{Young2}). Moreover, every Ahlfors regular space is uniformly perfect, since if $Z$ is Ahlfors $s$-regular then
\[
B(z,r)\setminus B(z,c r) = \emptyset \Rightarrow B(z,r) = B(z,cr) \Rightarrow r^s \asymp_\times (cr)^s \Rightarrow c\asymp_\times 1;
\]
taking the contrapositive yields that for $0 < c < 1$ sufficiently small, $Z$ is $c$-uniformly perfect.

We now define Schmidt's game (introduced by W. M. Schmidt in \cite{Schmidt1}) for the metric space $(Z,\Dist)$. When defining the game in this level of generality, we must specify that when choosing a ball in $(Z,\Dist)$, a player is really choosing a pair $(z,r)\in Z\times (0,\infty)$ which determines to a ball $B(z,r) \subset Z$.
\begin{remark}
\label{remarkformalballs}
Unlike in the case of $Z = \R^d$, it is possible that the same ball may be written in two different ways, i.e. $B(z,r) = B(\w z,\w r\ew\ew)$ for some $(z,r)\neq(\w z,\w r\ew\ew)$. In this case, it is the player's job to disambiguate.
\end{remark}
Now fix $0 < \alpha,\beta < 1$, and suppose that two players, Bob and Alice, take turns choosing balls $B_k = B(z_k,r_k)$ and $A_k = B(\w z_k,\w r_k)$, satisfying
\[
B_1 \supset A_1 \supset B_2 \supset \dots
\]
and
\[
\w r_k = \alpha r_k, \;\; r_{k + 1} = \beta \w r_k.
\]
A set $S \subset Z$ is said to be \emph{$(\alpha,\beta)$-winning} if Alice has a strategy guaranteeing that the unique point in the set $\bigcap_{k = 1}^\infty B_k= \bigcap_{k = 1}^\infty A_k$ belongs to $S$, regardless of the way Bob chooses to play. It is said to be \emph{$\alpha$-winning} if it is $(\alpha,\beta)$-winning for all $\beta>0$, and \emph{winning} if it is $\alpha$-winning for some $\alpha$.

In \cite{McMullen_absolute_winning}, C. T. McMullen introduced the following modification:\Footnote{Strictly speaking, McMullen only defined his game for $Z = \R^d$.} Let $Z$ be a $c$-uniformly perfect metric space. Fix $0 < \beta \leq c/5$,\Footnote{The constant $c/5$ is chosen to ensure that Bob will always have a legal move; see Lemma \ref{lemmalegalmove} below.} and suppose that  Bob and Alice take turns choosing balls $B_k = B(z_k,r_k)$ and $A_k = B(\w z_k,\w r_k)$ so that
\begin{equation}
\label{absolutewinninginclusion}
B_1 \supset B_1 \setminus A_1 \supset B_2\supset B_2 \setminus A_2 \supset B_3 \supset \cdots
\end{equation}
and
\begin{equation}
\label{absolutewinningradii}
\w r_k \leq \beta r_k, \;\; r_{k + 1} \geq \beta r_k.
\end{equation}
A set $S\subset Z$ is said to be \emph{$\beta$-absolute winning} if Alice has a strategy which leads to 
\[
\bigcap_{k = 1}^\infty B_k\cap S\ne\emptyset
\]
regardless of how Bob chooses to play; $S$ is said to be \emph{absolute winning} if $S$ is $\beta$-absolute winning for all $0<\beta \leq c/5$. Note a significant difference between Schmidt's original game and the modified game: now Alice has rather limited control over the situation, since she can block fewer of Bob's possible moves at the next step. Also, in the new version the $r_k$s do not have to tend to zero; therefore $\bigcap_1^\infty B_k$ does not have to be a single point (however the outcome with $r_k\not\tendsto k 0$ is clearly winning for Alice as long as $S$ is dense).

The following lemma guarantees that Bob always has a legal move in the $\beta$-absolute winning game:
\begin{lemma}
\label{lemmalegalmove}
Let $Z$ be a $c$-uniformly perfect metric space and fix $0 < \beta \leq c/5$. Then for any two sets $B_1 = B(z_1,r_1)$ and $A_1 = B(\w z_1,\w r_1)\subset B_1$ with $\w r_1 \leq \beta r_1$, there exists a set $B_2 = B(z_2,r_2)\subset B_1$ with $r_2 = \beta r$ which is disjoint from $A_1$.
\end{lemma}
\begin{proof}
Let $r = r_1$ and $w_1 = z_1$. Since $Z$ is $c$-uniformly perfect, there exists a point $w_2\in B(z_1,(1 - \beta)r)\butnot B(z_1,c(1 - \beta)r)$. Then the balls $U_i = B(w_i,\beta r)$ are contained in $B_1$; moreover,
\[
\Dist(U_1,U_2) \geq \Dist(w_1,w_2) - 2\beta r \geq c(1 - \beta)r - 2\beta r > (c - 3\beta)r.
\]
On the other hand,
\[
\Diam(A_1)\leq 2\w r_1\leq 2\beta r \leq (c - 3\beta)r,
\]
which shows that $A_1$ can intersect at most one of the sets $U_1$ and $U_2$. Letting $B_2$ be one that it does not intersect finishes the proof.
\end{proof}

The following proposition summarizes some important properties of winning and absolute winning subsets of a $c$-uniformly perfect Ahlfors $s$-regular metric space $(Z,\Dist)$:
\begin{proposition}
\label{propositionwinningproperties}~
\begin{itemize}
\item[(i)] Winning sets are dense and have Hausdorff dimension $s$.
\item[(ii)] Absolute winning implies $\alpha$-winning for any $0 < \alpha \leq c/5$.
\item[(iii)] The countable intersection of $\alpha$-winning (resp., 
absolute winning) sets is again $\alpha$-winning (resp., 
absolute winning).
\item[(iv)] The image of an $\alpha$-winning set under a bi-Lipschitz map $f$ is $K^{-2}\alpha$-winning, where $K$ is the bi-Lipschitz constant of $f$.
\item[(v)] 
The image of an absolute winning set under a quasisymmetric map (and in particular any bi-Lipschitz map) is absolute winning.
\end{itemize}
\end{proposition}
\begin{proof}[Proofs]~
\begin{itemize}
\item[(i)] \cite[Proposition 5.1]{KleinbockWeiss2}
\item[(ii)] Fix $\beta > 0$, and let $\gamma = \min(\alpha,\beta)$. We can convert a winning strategy for Alice in the $\gamma$-absolute winning game into a winning strategy in the $(\alpha,\beta)$-game by using Lemma \ref{lemmalegalmove} with $\beta$ replaced by $\alpha$; the idea is that if Alice would have made the move $A_1$ in the $\gamma$-absolute winning game, then she should make the move $B_2$ in the $(\alpha,\beta$)-game. Details are left to the reader.
\item[(iii)] \cite[Theorem 2]{Schmidt1} (the argument can be easily modified to account for absolute winning)
\item[(iv)] \cite[Proposition 5.3]{Dani2}
\item[(v)] \cite[Theorem 2.2]{McMullen_absolute_winning}\Footnote{Full disclosure: McMullen's proof appears to the casual observer to be using crucially the fact that $Z = \R^d$. In fact, his proof has the following form:
\begin{equation}
\label{temp}
\text{Quasisymmetric}\Rightarrow\text{Conditions 1-4 on p.5-6}\Rightarrow\text{Preserves absolute winning sets}. 
\end{equation}
The first arrow depends on the assumption $Z = \R^d$; the second arrow does not. Moreover, maps satisfying McMullen's conditions 1-4 are what are usually referred to as quasisymmetric maps by those studying spaces beyond $\R^d$ (and by us below in Lemma \ref{lemmaquasisymmetric}). (Maps satisfying McMullen's definition of quasisymmetric are termed ``weakly quasisymmetric''.) Thus the second arrow of \eqref{temp}, which holds for any $Z$, is accurately represented by our statement (v) above.}
\end{itemize}
\end{proof}

Note that winning sets arise naturally in many settings in dynamics and Diophantine approximation \cite{Dani3, Dani1, Dani2, EinsiedlerTseng, Farm, FPS, KleinbockWeiss3, KleinbockWeiss2, Moshchevitin, Schmidt1, Schmidt2, Tseng2, Tseng1}. Several examples of absolute winning sets were exhibited by McMullen in \cite{McMullen_absolute_winning}, most notably the set of badly approximable numbers in $\R$. However absolute winning does not occur as frequently as winning. In particular, the set of badly approximable vectors in $\R^d$ for $d > 1$ is winning but not absolute winning (see \cite{BFKRW}).

For the purposes of this paper, we will introduce a slight modification of McMullen's absolute winning game. The only difference will be that in \eqref{absolutewinningradii} we will require equality, i.e.
\begin{equation}
\label{modifiedabsolutewinningradii}
\w r_k = \beta r_k, \;\; r_{k + 1} = \beta\w r_k.
\end{equation}
For each $0 < \beta\leq c/5$, a set for which Alice can win this modified game will be called \emph{$\beta$-modified absolute winning}.

\begin{proposition}
Fix $0 < \beta\leq c/5$. Then every $\beta$-absolute winning set is $\beta$-modified absolute winning, and every $\beta^2/4$-modified absolute winning set is $\beta$-absolute winning.
\end{proposition}
\begin{proof}
First of all, let us notice that it makes no difference whether or not Alice is required to have equality in her choice of radii; choosing the largest possible radius always puts the greatest restriction on Bob's moves, and is therefore most advantageous for Alice. On the other hand, requiring Bob to choose the minimal possible radius is an advantage for Alice, since it restricts Bob's choice of moves. This demonstrates the first statement, that every $\beta$-absolute winning set is $\beta$-modified absolute winning.

Let $S\subset Z$ be a $\beta^2/4$-modified absolute winning set. By \cite[Theorem 7]{Schmidt1}, there exists a \emph{positional} strategy for Alice to win the $\beta^2/4$-modified absolute winning game with target set $S$, i.e. a map $\FF$ which inputs Bob's ball and outputs the next move Alice should make, without any information about previous moves. Based on this map $\FF$, we will give a positional strategy $\GG$ for Alice to win the $\beta$-absolute winning game. Given any ball $B(z,r)$, let $n = n_r$ be the unique integer so that $(\beta/2)^{2n + 1}\leq r < (\beta/2)^{2n - 1}$. Write $\FF(B(z,(\beta/2)^{2n})) = B(w,(\beta/2)^{2n + 2})$ (the radius $(\beta/2)^{2n + 2}$ is forced by the definition of the modified absolute winning game). Let $\GG(B(z,r)) = B(w,2(\beta/2)^{2n + 2})$; this move is legal in the $\beta$-absolute winning game since $2(\beta/2)^{2n + 2}\leq \beta r$.

We now claim that with the positional strategy $\GG$, Alice is assured to win the $\beta$-absolute winning game. Let $\big(B_k = B(z_k,r_k)\big)_1^\infty$ and $\big(A_k = B(\w z_k,\w r_k)\big)_1^\infty$ be a sequence of moves for Bob and Alice which are legal in the $\beta$-absolute winning game, and such that $A_k = \GG(B_k)$ for all $k\in\N$. Let us assume that $r_k\tendsto k 0$, since otherwise Alice wins automatically since $S$ is dense. Let $n_1\in\N$ be such that $(\beta/2)^{2n_1} < r_1$. For each $n\geq n_1$, let $k_n$ be such that $(\beta/2)^{2n + 1}\leq r_{k_n} < \frac{1}{2}(\beta/2)^{2n}$; such a $k_n$ exists by \eqref{absolutewinningradii}, although it may not be unique. Note that $n_{r_{k_n}} = n$. Let $\w B_n = B(z_{k_n},(\beta/2)^{2n})$, and let $\w A_n = \FF(\w B_n) = B(w_n,(\beta/2)^{2n + 2})$.
\begin{claim}
\label{claimlegal}
The sequences $(\w B_n)_{n_1}^\infty$ and $(\w A_n)_{n_1}^\infty$ form a sequence of legal moves in the $\beta$-modified absolute winning game.
\end{claim}
\begin{subproof}
Alice's moves are legal because she chose them according to her positional strategy. Bob's first move $\w B_{n_1}$ is clearly legal, so fix $n\geq n_1$ and let us consider the legality of the move $\w B_{n + 1}$. It suffices to show
\begin{equation}
\label{legal}
\w B_{n + 1}\subset \w B_n \butnot \w A_n.
\end{equation}
Since the sequences $(B_k)_1^\infty$ and $(A_k)_1^\infty$ are legal for the $\beta$-absolute winning game, letting $k = k_n$ we have \eqref{absolutewinninginclusion} and \eqref{absolutewinningradii}; in particular,
\[
z_{k_{n + 1}} \in B_{k_{n + 1}} \subset B_{k_n + 1} \subset B_k\butnot A_k,
\]
i.e.
\[
\Dist(z_{k_{n + 1}},z_k)\leq r_k\text{ and }\Dist(z_{k_{n + 1}},\w z_k) > \w r_k.
\]
Now by the definition of $A_k$,
\[
\w z_k = w_n \text{ and }\w r_k = 2(\beta/2)^{2n + 2},
\]
and by the definition of $k = k_n$,
\[
r_k \leq \frac{1}{2}(\beta/2)^{2n}.
\]
Thus
\[
\Dist(z_{k_{n + 1}},z_k)\leq \frac{1}{2}(\beta/2)^{2n}\text{ and }\Dist(z_{k_{n + 1}},w_n) > 2(\beta/2)^{2n + 2},
\]
from which it follows that
\[
B(z_{k_{n + 1}},(\beta/2)^{2n + 2}) \subset B(z_k,(\beta/2)^{2n})\butnot B(w_n,(\beta/2)^{2n + 2}),
\]
which is just \eqref{legal} expanded.
\end{subproof}
Now clearly, the intersection of the sequence $(B_k)_1^\infty$ is the same as the intersection of the sequence $(\w B_n)_{n_1}^\infty$, which is in $S$ by Claim \ref{claimlegal}. Thus Alice wins the $\beta$-absolute winning game with her positional strategy $\GG$.
\end{proof}

Now since $\beta$-absolute winning implies $\w\beta$-absolute winning for every $\beta\leq\w\beta\leq c/5$, we have the following:
\begin{corollary}
\label{corollarymodifiedabsolutewinning}
If $S\subset Z$ is $\beta_m$-modified absolute winning for a sequence $\beta_m\tendsto m 0$, then $S$ is absolute winning.
\end{corollary}

\draftnewpage\section{Partition structures} \label{sectionstructures}
In this section we introduce the notion of a partition structure, an important technical tool for proving Theorems \ref{theorembishopjones}, \ref{theoremfulldimension}, and \ref{theoremextrinsic}. Moreover, we introduce stronger formulations of the former two theorems, from which they follow as corollaries. Subsection \ref{subsectionstructures} can also be found with some modifications in \cite[\69.1]{DSU}, but we include the proof here as well for completeness.

Throughout this section, $(Z,\Dist)$ denotes a metric space.
\subsection{Partition structures} \label{subsectionstructures}
\begin{notation}
Let
\[
\N^* = \bigcup_{n = 0}^\infty \N^n.
\]
If $\omega\in\N^*\cup\N^\N$, then we denote by $|\omega|$ the unique element of $\N\cup{\infty}$ such that $\omega\in \N^{|\omega|}$ and call $|\omega|$ the \emph{length} of $\omega$. We let $\emptyset$ denote the empty string, so that $|\emptyset| = 0$. For each $N\in\N$, we denote the initial segment of $\omega$ of length $N$ by
\[
\omega_1^N := (\omega_n)_1^N \in \N^N.
\]
For two words $\omega,\tau\in\N^\N$, let $\omega\wedge\tau$ denote their longest common initial segment, and let
\[
\rho_2(\omega,\tau) = 2^{-|\omega\wedge\tau|}.
\]
Then $(\N^\N,\rho_2)$ is a metric space.
\end{notation}
\begin{definition}
\label{definitiontree}
A \emph{tree} on $\N$ is a set $T\subset\N^*$ which is closed under the operation of taking initial segments.
\end{definition}
\begin{notation}
If $T$ is a tree on $\N$, then we denote its set of infinite branches by
\[
T(\infty) := \{\omega\in\N^\N:\omega_1^n\in T\all n\in\N\}.
\]
On the other hand, for $n\in\N$ we let
\[
T(n) := T\cap \N^n.
\]
For each $\omega\in T$, we denote the set of its children by
\[
T(\omega) := \{a\in\N:\omega a\in T\}.
\]
The set of infinite branches in $T(\infty)$ of which $\omega$ is an initial segment will be denoted $[\omega]_T$ or simply $[\omega]$.
\end{notation}

\begin{definition}
\label{definitionpartitionstructure}
A \emph{partition structure} on $Z$ consists of a tree $T\subset\N^*$ together with a collection of closed subsets $(\PP_\omega)_{\omega\in T}$ of $Z$, each having positive diameter and enjoying the following properties:
\begin{itemize}
\item[(I)] If $\omega\in T$ is an initial segment of $\tau\in T$ then $\PP_\tau\subset \PP_\omega$. If neither $\omega$ nor $\tau$ is an initial segment of the other then $\PP_\omega\cap \PP_\tau = \emptyset$.
\item[(II)] For each $\omega\in T$ let
\[
D_\omega = \Diam(\PP_\omega).
\]
There exist $\kappa > 0$ and $0 < \lambda < 1$ such that for all $\omega\in T$ and for all $a\in T(\omega)$, we have 
\begin{equation}\label{kappa}
\Dist(\PP_{\omega a},Z\setminus \PP_\omega) \geq \kappa D_\omega
\end{equation}
and
\begin{equation}\label{lambda}
\kappa D_\omega \leq D_{\omega a} \leq \lambda D_\omega.
\end{equation}
\end{itemize}

Fix $s > 0$. The partition structure $(\PP_\omega)_{\omega\in T}$
is called \emph{$s$-thick} if for all $\omega\in T$, 
\begin{equation}
\label{redistribution}
\sum_{a\in T(\omega)}D_{\omega a}^s \geq D_\omega^s.
\end{equation}
\end{definition}

\begin{definition}
If $(\PP_\omega)_{\omega\in T}$ is a partition structure, a \emph{substructure} of $(\PP_\omega)_{\omega\in T}$ is a partition structure of the form $(\PP_\omega)_{\omega\in \w T}$, where $\w T\subset T$ is a subtree.
\end{definition}

\begin{observation}
Let $(\PP_\omega)_{\omega\in T}$ be a partition structure on a
complete metric space $(Z,\Dist)$. For each $\omega\in T(\infty)$, the set 
\[
\bigcap_{n = 1}^\infty\PP_{\omega_1^n}
\]
is a singleton. If we define $\pi(\omega)$ to be the unique member of this set, then the map $\pi:T(\infty)\to Z$ is continuous (in fact quasisymmetric; see Lemma \ref{lemmaquasisymmetric} below). 
\end{observation}

\begin{definition}
\label{definitionpi}
The set $\pi(T(\infty))$ is called the \emph{limit set} of the partition structure.
\end{definition}

We remark that a large class of examples of partition structures comes from the theory of conformal iterated function systems \cite{MauldinUrbanski1} (or in fact even graph directed Markov systems \cite{MauldinUrbanski2}) satisfying the strong separation condition (also known as the disconnected open set condition \cite{RiediMandelbrot}; see also \cite{FalconerMarsh}, where the limit sets of iterated function systems satisfying the strong separation condition are called \emph{dust-like}). Indeed, the notion of a partition structure was intended primarily to generalize these examples. The difference is that in a partition structure, the sets $(\PP_\omega)_\omega$ do not necessarily have to be defined by dynamical means. We also note that if $Z = \R^d$ for some $d\in\N$, and if $(\PP_\omega)_{\omega\in T}$ is a partition structure on $Z$, then the tree $T$ has bounded degree, meaning that there exists $N < \infty$ such that $\#(T(\omega))\leq N$ for every $\omega\in T$.

\begin{reptheorem}{theoremahlforsgeneral}[Proven below]
Fix $s > 0$. Then any $s$-thick partition structure $(\PP_\omega)_{\omega\in T}$ on a complete metric space $(Z,\Dist)$ has a substructure $(\PP_\omega)_{\omega\in \w T}$ whose limit set is Ahlfors $s$-regular. Furthermore the tree $\w T$ can be chosen so that for each $\omega\in \w T$, we have that $\w T(\omega)$ is an initial segment of $T(\omega)$, i.e. $\w T(\omega) = T(\omega)\cap\{1,\ldots,N_\omega\}$ for some $N_\omega\in\N$. 
\end{reptheorem}

\begin{replemma}{lemmastructure}[Proven below]
Let $(X,\dist,\zero,b,G)$ be as in \sectionsymbol\ref{standingassumptions}. Then for every $0 < s < \delta_G$ and for all $\sigma > 0$ sufficiently large, there exist a tree $T$ on $\N$ and an embedding $T\ni\omega\mapsto x_\omega\in G(\zero)$ such that if 
\[
\PP_\omega := \Shad(x_\omega,\sigma),
\]
then $(\PP_\omega)_{\omega\in T}$ is an $s$-thick partition structure
on $(\del X,\Dist)$, whose limit set is a subset of $\Lur$.
\end{replemma}
Combining these results in the obvious way yields the following corollary:

\begin{corollary}
\label{corollarystructure}
Let $(X,\dist,\zero,b,G)$ be as in \sectionsymbol\ref{standingassumptions}. Then for every $0 < s < \delta_G$ and for all $\sigma > 0$ sufficiently large, there exist a tree $T$ on $\N$ and an embedding $T\ni\omega\mapsto x_\omega\in G(\zero)$ such that if 
\[
\PP_\omega := \Shad(x_\omega,\sigma),
\]
then $(\PP_\omega)_{\omega\in T}$ is a partition structure\comdavid{In general, the partition structure in this theorem will not be $s$-thick.} on $(\del X,\Dist)$, whose limit set $\limitset_s\subset\Lur(G)$ is Ahlfors $s$-regular.
\end{corollary}

Using Corollary \ref{corollarystructure}, we prove the following generalization of the Bishop--Jones theorem:

\begin{theorem}
\label{theorembishopjones}
Let $(X,\dist,\zero,b,G)$ be as in \sectionsymbol\ref{standingassumptions}. Then
\begin{repequation}{bishopjones}
\HD(\Lr) = \HD(\Lur) = \delta;
\end{repequation}
moreover, for every $0 < s < \delta$ there exists an Ahlfors $s$-regular set $\limitset_s\subset\Lur$.
\end{theorem}
\begin{proof}
The ``moreover'' clause follows directly from Corollary \ref{corollarystructure}; applying the mass distribution principle (e.g. \cite[p.55]{Falconer_book}) yields $\HD(\Lur)\geq\delta$. The proof of the remaining inequality $\HD(\Lr)\leq\delta$ is a straightforward adaptation of the argument in the standard case (e.g. \cite[p.6]{BishopJones}).
\end{proof}

\subsection{Diophantine approximation and partition structures}
Recall that for $\xi\in\del X$, $\BA_\xi$ denotes the set of points in $\Lr$ which are badly approximable with respect to $\xi$ (Definition \ref{definitionbadlyapproximable}). We have the following:

\begin{reptheorem}{theoremabsolutewinning}
Let $(X,\dist,\zero,b,G)$ be as in \sectionsymbol\ref{standingassumptions}. Fix $0 < s < \delta_G$ and let $\limitset_s$ be defined as in Corollary \ref{corollarystructure}. Then for each $\xi\in\del X$, the set $\BA_\xi\cap \limitset_s$ is absolute winning on $\limitset_s$.
\end{reptheorem}

\noindent From Theorem \ref{theoremabsolutewinning} we deduce the following corollary, which was stated in the introduction:

\begin{theorem}[Generalization of the Full Dimension Theorem] \label{theoremfulldimension}
Let $(X,\dist,\zero,b,G)$ be as in \sectionsymbol\ref{standingassumptions}, and let $(\xi_k)_1^\ell$ be a countable (finite or infinite) sequence in $\del X$. Then
\[
\HD\left(\bigcap_{k = 1}^\ell\BA_{\xi_k}\cap\Lur\right) = \delta = \HD(\Lr).
\]
In particular $\BA_\xi\neq\emptyset$ for every $\xi\in\del X$, so Theorem \ref{theoremdirichlet} cannot be improved by more than a multiplicative constant.
\end{theorem}
\begin{proof}[Proof of Theorem \ref{theoremfulldimension} assuming Theorem \ref{theoremabsolutewinning}]
For each $0 < s < \delta_G$, applying Theorem \ref{theoremabsolutewinning} (and thus implicitly Corollary \ref{corollarystructure}) and parts (i), (ii), and (iii) of Proposition \ref{propositionwinningproperties} gives that the set $\bigcap_1^\ell\BA_{\xi_k}\cap \limitset_s$ has full Hausdorff dimension in $\limitset_s$, i.e. its dimension is $s$. Letting $s$ tend to $\delta_G$ yields
\[
\HD\left(\bigcap_{k = 1}^\ell \BA_{\xi_k}\cap\Lur\right) \geq \delta.
\]
On the other hand, we clearly have $\HD\left(\bigcap_1^\ell \BA_{\xi_k}\cap\Lur\right) \leq \HD(\Lr)$, and the equation $\delta = \HD(\Lr)$ is given by Theorem \ref{theorembishopjones}. This completes the proof.
\end{proof}

\subsection{Basic facts about partition structures}
In this subsection we prove Theorem \ref{theoremahlforsgeneral}, and the following lemma which will be used in the proof of Theorem \ref{theoremabsolutewinning}:
\begin{lemma}
\label{lemmaquasisymmetric}
Let $(\PP_\omega)_{\omega\in T}$ be a partition structure on a complete metric space $(Z,\Dist)$. Then the map $\pi:T(\infty)\to\pi(T(\infty))$ is quasisymmetric.
\end{lemma}
\begin{proof}
Recall that a surjective map $\Phi:(Z_1,\Dist_1)\to (Z_2,\Dist_2)$ is said to be \emph{quasisymmetric} if there exists an increasing homeomorphism $\eta:(0,\infty)\to(0,\infty)$ such that for every $z_1,z_2,z_3\in Z_1$, we have
\begin{equation}
\label{quasisymmetric}
\frac{\Dist_2(\Phi(z_1),\Phi(z_2))}{\Dist_2(\Phi(z_1),\Phi(z_3))} \leq \eta\left(\frac{\Dist_1(z_1,z_2)}{\Dist_1(z_1,z_3)}\right).
\end{equation}
In our case we have $(Z_1,\Dist_1) = (T(\infty),\rho_2)$, $(Z_2,\Dist_2) = (\pi(T(\infty)),\Dist)$, and $\Phi = \pi$.

Fix $\omega,\tau^{(1)},\tau^{(2)}\in T(\infty)$, and let $m_i = |\omega\wedge\tau^{(i)}|$, $i = 1,2$. Suppose first that $m_1\leq m_2$. By (\ref{kappa}) and the first inequality of (\ref{lambda}), we have
\begin{align*}
\Dist(\pi(\omega),(\tau^{(2)})) \geq \kappa D_{\omega_1^{m_2}}
&\geq \kappa^{m_2 - m_1 + 1} D_{\omega_1^{m_1}}\\
&\geq \kappa^{m_2 - m_1 + 1} \Dist(\pi(\omega),\pi(\tau^{(1)}))\\
\frac{\Dist(\pi(\omega),\pi(\tau^{(1)}))}{\Dist(\pi(\omega),\pi(\tau^{(2)}))} &\leq \kappa^{-1} \left(\frac{\rho_2(\omega,\tau^{(1)})}{\rho_2(\omega,\tau^{(2)})}\right)^{-\log_2(\kappa)}.
\end{align*}
On the other hand, suppose that $m_1\geq m_2$. By (\ref{kappa}) and the second inequality of (\ref{lambda}), we have
\begin{align*}
\Dist(\pi(\omega),(\tau^{(2)})) \geq \kappa D_{\omega_1^{m_2}}
&\geq \kappa\lambda^{m_2 - m_1} D_{\omega_1^{m_1}}\\
&\geq \kappa\lambda^{m_2 - m_1} \Dist(\pi(\omega),\pi(\tau^{(1)}))\\
\frac{\Dist(\pi(\omega),\pi(\tau^{(1)}))}{\Dist(\pi(\omega),\pi(\tau^{(2)}))}
&\leq \kappa^{-1} \left(\frac{\rho_2(\omega,\tau^{(1)})}{\rho_2(\omega,\tau^{(2)})}\right)^{-\log_2(\lambda)}.
\end{align*}
Thus by letting
\[
\eta(t) = \kappa^{-1}\max\left(t^{-\log_2(\kappa)},t^{-\log_2(\lambda)}\right)
\]
we have proven (\ref{quasisymmetric}).
\end{proof}

\begin{theorem}
\label{theoremahlforsgeneral}
Fix $s > 0$. Then any $s$-thick partition structure $(\PP_\omega)_{\omega\in T}$ on a complete metric space $(Z,\Dist)$ has a substructure $(\PP_\omega)_{\omega\in \w T}$ whose limit set is Ahlfors $s$-regular. Furthermore the tree $\w T$ can be chosen so that for each $\omega\in \w T$, we have that $\w T(\omega)$ is an initial segment of $T(\omega)$, i.e. $\w T(\omega) = T(\omega)\cap\{1,\ldots,N_\omega\}$ for some $N_\omega\in\N$. 
\end{theorem}
\begin{proof}
This proof will also appear in \cite[Theorem 9.1.8]{DSU}.

We will recursively define a sequence of maps 
\[
\mu_n:T(n)\to[0,1]
\]
with the following consistency property:
\begin{equation}
\label{consistency}
\mu_n(\omega) = \sum_{a\in T(\omega)}\mu_{n + 1}(\omega a).
\end{equation}
The Kolmogorov consistency theorem will then guarantee the existence of a measure $\w\mu\in\MM(T(\infty))$ satisfying 
\begin{equation}
\label{Kolmogorov}
\w\mu([\omega]) = \mu_n(\omega)
\end{equation}
for each $\omega\in T(n)$.

Let $c = 1 - \lambda^s > 0$, where $\lambda$ is as in (\ref{lambda}). For each $n\in\N$, we will demand of our function $\mu_n$ the following property: for all $\omega\in T(n)$, if $\mu_n(\omega) > 0$, then 
\begin{equation}
\label{regularity}
cD_\omega^s \leq \mu_n(\omega) < D_\omega^s.
\end{equation}

We now begin our recursion. For the case $n = 0$, let $\mu_0(\emptyset):= c D_\smallemptyset^s$; (\ref{regularity}) is clearly satisfied.

For the inductive step, fix $n\in\N$ and suppose that $\mu_n$ has been constructed satisfying (\ref{regularity}). Fix $\omega\in T(n)$, and suppose that $\mu_n(\omega) > 0$. Formulas (\ref{redistribution}) and (\ref{regularity}) imply that
\[
\sum_{a\in T(\omega)}D_{\omega a}^s > \mu_n(\omega).
\]
Let $N_\omega\in T(\omega)$ be the smallest integer such that
\begin{equation}\label{regularitya}
\sum_{a\leq N_\omega}D_{\omega a}^s > \mu_n(\omega).\Footnote{Obviously, this and similar sums are restricted to $T(\omega)$.}
\end{equation}
Then the minimality of $N_\omega$ says precisely that
\[
\sum_{a\leq N_\omega - 1}D_{\omega a}^s \leq \mu_n(\omega).
\]
Using the above, (\ref{regularitya}), and (\ref{lambda}), we have
\begin{equation}
\label{sumaNbounds}
\mu_n(\omega) < \sum_{a\leq N_\omega}D_{\omega a}^s \leq \mu_n(\omega) +
D_{\omega N_\omega}^s \leq \mu_n(\omega) + \lambda^s D_\omega^s. 
\end{equation}
For each $a\in T(\omega)$ with $a > N_\omega$, let $\mu_{n + 1}(\omega a) =
0$, and for each $a \leq N_\omega$, let 
\[
\mu_{n + 1}(\omega a) = \frac{D_{\omega a}^s\mu_n(\omega)}{\sum_{b\leq N_\omega}D_{\omega b}^s}. 
\]
Obviously, $\mu_{n + 1}$ defined in this way satisfies (\ref{consistency}). Let us prove that (\ref{regularity}) holds (of course, with $n$ replaced by $n + 1$). The second inequality follows directly from the definition of $\mu_{n + 1}$ and from \eqref{regularitya}. Using \eqref{sumaNbounds}, \eqref{regularity} (with $n = n$), and the equation $c = 1 - \lambda^s$, we deduce the first inequality as follows:
\begin{align*}
\mu_{n + 1}(\omega a)
&\geq \frac{D_{\omega a}^s\mu_n(\omega)}{\mu_n(\omega) + \lambda^sD_\omega^s}
= D_{\omega a}^s\left[1 - \frac{\lambda^sD_\omega^s}{\mu_n(\omega) + \lambda^sD_\omega^s}\right]\\ 
&\geq D_{\omega a}^s\left[1 - \frac{\lambda^s}{c + \lambda^s}\right] \\
&= cD_{\omega a}^s.
\end{align*}
The proof of (\ref{regularity}) (with $n = n + 1$) is complete. This completes the recursive step.

Let
\[
\w T = \bigcup_{n = 1}^\infty\{\omega\in T(n):\mu_n(\omega) > 0\}.
\] 
Clearly, the limit set of the partition structure $(\PP_\omega)_{\omega\in \w T}$ is exactly the topological support of $\mu := \pi[\w\mu]$, where $\w\mu$ is defined by (\ref{Kolmogorov}). Furthermore, for each $\omega\in \w T$, we have $\w T(\omega) = T(\omega)\cap\{1,\ldots,N_\omega\}$. Thus, to complete the proof of Theorem \ref{theoremahlforsgeneral} it suffices to show that the measure $\mu$ is Ahlfors $s$-regular.

To this end, fix $z = \pi(\omega)\in\Supp(\mu)$ and $0 < r \leq \kappa D_\smallemptyset$, where $\kappa$ is as in (\ref{kappa}) and (\ref{lambda}). For convenience of notation let 
\[
\PP_n := \PP_{\omega_1^n}, \;\; D_n := \Diam(\PP_n),
\]
and let $n\in\N$ be the largest integer such that $r < \kappa D_n$. We have
\begin{equation}\label{1fsdk9}
\kappa^2 D_n \leq \kappa D_{n + 1} \leq r < \kappa D_n.
\end{equation}
(The first inequality comes from \eqref{lambda}, whereas the latter two come from the definition of $r$.)

We now claim that
\[
B(z,r)\subset \PP_n.
\]
Indeed, by contradiction suppose that $w\in B(z,r)\butnot \PP_n$. By (\ref{kappa}) we have
\[
\Dist(z,w)\geq \Dist(z,Z\butnot \PP_n)\geq \kappa D_n > r
\]
which contradicts the fact that $w\in B(z,r)$.

Let $k\in\N$ be large enough so that $\lambda^k\leq\kappa^2$. It follows from \eqref{1fsdk9} and repeated applications of the second inequality of (\ref{lambda}) that
\[
D_{n + k} \leq \lambda^k D_n \leq \kappa^2 D_n \leq r,
\]
and thus
\[
\PP_{n + k} \subset B(z,r)\subset \PP_n.
\]
Thus, invoking (\ref{regularity}), we get
\begin{equation}
\label{muomegarbounds}
(1 - \lambda^s)D_{n + k}^s \leq \mu(\PP_{n + k}) \leq \mu(B(z,r))
\leq \mu(\PP_n) \leq D_n^s. 
\end{equation}

On the other hand, it follows from \eqref{1fsdk9} and repeated applications of the first inequality of (\ref{lambda}) that
\begin{equation}
\label{above}
 D_{n + k} \geq \kappa^k D_n \geq \kappa^{k - 1}r.
\end{equation}
Combining (\ref{1fsdk9}), (\ref{muomegarbounds}), and (\ref{above}) yields
\[
(1 - \lambda^s)\kappa^{s(k - 1)}r^s \leq \mu(B(z,r)) \leq \kappa^{-2s}r^s,
\]
i.e. $\mu$ is Ahlfors $s$-regular. This completes the proof of Theorem \ref{theoremahlforsgeneral}.
\end{proof}

\subsection{Proof of Lemma \ref{lemmastructure} (A partition structure on $\del X$)}
\label{sectionlemmastructure}
In this subsection we prove Lemma \ref{lemmastructure}. The entire subsection is repeated with some modification from \cite[\69.2]{DSU}; we include the proof here for completeness.
\begin{lemma}
\label{lemmastructure}
Let $(X,\dist,\zero,b,G)$ be as in \sectionsymbol\ref{standingassumptions}. Then for every $0 < s < \delta_G$ and for all $\sigma > 0$ sufficiently large, there exist a tree $T$ on $\N$ and an embedding $T\ni\omega\mapsto x_\omega\in G(\zero)$ such that if 
\[
\PP_\omega := \Shad(x_\omega,\sigma),
\]
then $(\PP_\omega)_{\omega\in T}$ is an $s$-thick partition structure
on $(\del X,\Dist)$, whose limit set is a subset of $\Lur$.
\end{lemma}

We begin by stating our key lemma. This lemma will also be used in the proof of Theorem \ref{theoremjarnikbesicovitch}.

\begin{lemma}[Construction of children]
\label{lemmaDSU}
Let $(X,\dist,\zero,b,G)$ be as in \sectionsymbol\ref{standingassumptions}. Then for every $0 < s < \delta_G$, for every $0 < \lambda < 1$, for all $\sigma > 0$ sufficiently large, and for every $w\in G(\zero)$, there exists a finite subset $T(w)\subset G(\zero)$ (the children of $w$) such that if we let\comdavid{Although this notation gets repeated a lot, I think it is OK because the $\sigma$ is different every time.}
\begin{align*}
\PP_x &:= \Shad(x,\sigma)\\
D_x &:= \Diam(\PP_x)
\end{align*}
then the following hold:
\begin{itemize}
\item[(i)] The family $(\PP_x)_{x\in T(w)}$ consists of pairwise disjoint shadows contained in $\PP_w$.
\item[(ii)] There exists $\kappa > 0$ independent of $w$ such that for all $x\in T(w)$,
\begin{align*}
\Dist(\PP_x,\del X\butnot \PP_w) &\geq \kappa D_w\\
\kappa D_w \leq D_x &\leq \lambda D_w.
\end{align*}
\item[(iii)]
\[
\sum_{x\in T(w)}D_x^s \geq D_w^s.
\]
\end{itemize}
\end{lemma}
It is not too hard to deduce Lemma \ref{lemmastructure} from Lemma \ref{lemmaDSU}. We do it now:
\begin{proof}[Proof of Lemma \ref{lemmastructure} assuming Lemma \ref{lemmaDSU}]
Let $\lambda = 1/2$, and let $\sigma > 0$ be large enough so that Lemma \ref{lemmaDSU} holds. Let $(x_n)_1^\infty$ be an enumeration of $G(\zero)$. Let
\[
T = \bigcup_{n = 1}^\infty\{\omega\in\N^n:x_{\omega_{j + 1}}\in T(x_{\omega_j})\all j = 1,\ldots,n - 1\},
\]
and for each $\omega\in T$ let
\begin{align*}
x_\omega &= x_{\omega_{|\omega|}}\\
\PP_\omega &= \PP_{x_\omega}.
\end{align*}
Then the conclusion of Lemma \ref{lemmaDSU} precisely implies that $(\PP_\omega)_{\omega\in T}$ is an $s$-thick partition structure on $(\del X,\Dist)$.

To complete the proof, we must show that the limit set of the partition structure $(\PP_\omega)_{\omega\in T}$ is contained in $\Lur(G)$. Indeed, fix $\omega\in T(\infty)$. Then for each $n\in\N$, $\pi(\omega)\in \PP_{\omega_1^n}=\Shad(x_{\omega_1^n},\sigma)$ and $\dist(\zero,x_{\omega_1^n})\to\infty$. So, the sequence $(x_{\omega_1^n})_1^\infty$ converges radially to $\pi(\omega)$. On the other hand,
\begin{align*}
\dist(x_{\omega_1^n},x_{\omega_1^{n + 1}}) &\asymp_{\plus,\sigma} \busemann_\zero(x_{\omega_1^{n + 1}},x_{\omega_1^n}) \by{\eqref{distbusemann}}\\
&\asymp_{\plus,\sigma} -\log_b\left(\frac{D_{\omega_1^{n + 1}}}{D_{\omega_1^n}}\right) \by{the Diameter of Shadows Lemma}\\
&\leq_\pt -\log_b(\kappa) \asymp_{\plus,\kappa} 0. \by {\eqref{lambda}}
\end{align*}
Thus the sequence $(x_{\omega_1^n})_1^\infty$ converges to $\pi(\omega)$ uniformly radially.
\end{proof}

\begin{proof}[Proof of Lemma \ref{lemmaDSU}]
Choose an arbitrary $t\in (s,\delta_G)$. A point $\eta\in \Lambda$ will be called \emph{$t$-divergent} if for every neighborhood $B$ of $\eta$ the restricted Poincar\'e series 
\begin{equation}
\label{tseries}
\Sigma_t(G\given B):=\sum_{\substack{g\in G \\ g(\zero)\in B}}b^{-t\dist(\zero,g(\zero))}
\end{equation}
diverges.

The following lemma was proven in \cite{BishopJones} for the case $X = \H^{d + 1}$, using the fact that the space $\bord{\H^{d + 1}}$ is compact. A new proof is needed for the general case, since in general $\bord X$ will not be compact.
\begin{lemma}
\label{lemmatdivergent}
There exists at least one $t$-divergent point.
\end{lemma}
\begin{proof}
Since $G$ is of general type, there exists a loxodromic isometry $g\in G$. Let $g_+$ and $g_-$ be the attracting and repelling fixed points of $g$, respectively. Let $B_+$ and $B_-$ be disjoint neighborhoods of $g_+$ and $g_-$, respectively. Since the series $\Sigma_t(G)$ diverges, it follows that the series (\ref{tseries}) diverges for either $B = \del X\butnot B_-$ or $B = \del X\butnot B_+$. Without loss of generality, let us assume that it diverges for $B = \del X\butnot B_-$. Then $\Sigma_t(G\given g^n(\del X\butnot B_-))$ also diverges for every $n\in\N$. By the definition of a loxodromic isometry, $g^n$ tends to $g_+$ uniformly on $\del X\butnot B_-$, so for any neighborhood $U$ of $g_+$, we eventually have $g^n(\del X\butnot B_-)\subset U$. But then $\Sigma_t(G\given U)$ diverges, which proves that $g_+$ is a $t$-divergent point.
\QEDmod\end{proof}
\begin{remark}
\label{remarktdivergent}
It is not hard to see that the set of $t$-divergent points is invariant under the action of the group. It follows from this and from Observation \ref{observationnoglobalfixedpoint} that there are at least two $t$-divergent points.
\end{remark}

\begin{sublemma}
\label{sublemmaDSU}
Let $\eta$ be a $t$-divergent point, and let $B_\eta$ be a neighborhood of $\eta$. Then for all $\sigma > 0$ sufficiently large, there exists a set $S_\eta\subset G(\zero)\cap B_\eta$ such that for all $z\in X\butnot B_\eta$,
\begin{itemize}
\item[(i)] If
\[
\PP_{z,x} := \Shad_z(x,\sigma),
\]
then the family $(\PP_{z,x})_{x\in S_\eta}$ consists of pairwise disjoint shadows contained in $\PP_{z,\zero}\cap B_\eta$.
\item[(ii)] There exists $\kappa > 0$ independent of $z$ such that for all $x\in S_\eta$,
\begin{align} \label{epsilonnew}
\Dist_{b,z}(\PP_{z,x},\del X\butnot \PP_{z,\zero}) &\geq \kappa\Diam_z(\PP_{z,\zero})\\ \label{lambdanew}
\kappa\Diam_z(\PP_{z,\zero}) \leq \Diam_z(\PP_{z,x}) &\leq \lambda\Diam_z(\PP_{z,\zero}).
\end{align}
\item[(iii)]
\[
\sum_{x\in S_\eta}\Diam_z^s(\PP_{z,x}) \geq \Diam_z^s(\PP_{z,\zero}).
\]
\end{itemize}
\end{sublemma}
Sublemma \ref{sublemmaDSU} will be proven below; for now, let us complete the proof of Lemma \ref{lemmaDSU} assuming Sublemma \ref{sublemmaDSU}.

Let $\eta_1$ and $\eta_2$ be two distinct $t$-divergent points. Let $B_1$ and $B_2$ be disjoint neighborhoods of $\eta_1$ and $\eta_2$, respectively, and let $S_1\subset G(\zero)\cap B_1$ and $S_2\subset G(\zero)\cap B_2$ be the sets guaranteed by Sublemma \ref{sublemmaDSU}. Now suppose that $w = g_w(\zero)\in G(\zero)$. Let $z = g_w^{-1}(\zero)$. Then either $z\notin B_1$ or $z\notin B_2$; say $z\notin B_i$. Let $T(w) = g_w(S_i)$; then (i)-(iii) of Sublemma \ref{sublemmaDSU} exactly guarantee (i)-(iii) of Lemma \ref{lemmaDSU}.
\end{proof}

\begin{proof}[Proof of Sublemma \ref{sublemmaDSU}]
Choose $\rho > 0$ large enough so that
\[
\{x\in\bord X:\lb x|\eta\rb_\zero\geq \rho\} \subset B_\eta.
\]
Then fix $\sigma > 0$ large to be announced below, depending only on $\rho$.

For all $x\in X$ we have
\[
0 \asymp_{\plus,\rho} \lb z|\eta\rb_\zero \gtrsim_\plus \min(\lb x|\eta\rb_\zero,\lb x|z\rb_\zero).
\]
Fix $\w\rho\geq\rho$ large to be announced below, depending only on $\rho$ and $\sigma$. Let
\[
\w B_\eta = \{x\in X:\lb x|\eta\rb_\zero \geq \w\rho\ew\}.
\]
It follows that for all $x\in \w B_\eta$, we have 
\begin{equation}\label{xz0}
\lb x|z\rb_\zero \asymp_{\plus,\rho} 0,
\end{equation}
assuming $\w\rho$ is chosen large enough. We emphasize that the implied constants of these asymptotics are independent of $z$.

For each $n\in\N$ let 
\[
A_n := B(\zero,n + 1)\butnot B(\zero,n)
\] 
be the $n$th annulus centered at $\zero$. We shall need the following variant of the Intersecting Shadows Lemma:

\begin{claim}
\label{claimtau}
There exists $\tau > 0$ depending on $\rho$ and $\sigma$ such that for all $n\in\N$ and for all $x,y\in A_n\cap \w B_\eta$, if
\[
\PP_{z,x}\cap \PP_{z,y}\neq\emptyset,
\]
then 
\[
\dist(x,y) < \tau.
\] 
\end{claim}
\begin{proof}
Without loss of generality suppose $\dist(z,y)\geq \dist(z,x)$. Then by the Intersecting Shadows Lemma we have
\[
\dist(x,y) \asymp_{\plus,\sigma} \busemann_z(y,x)
= \busemann_\zero(y,x) + 2\lb x|z\rb_\zero - 2\lb y|z\rb_\zero.
\]
Now $\busemann_\zero(y,x) \leq 1$ since $x,y\in A_n$. On the other hand, since $x,y\in\w B_\eta$, we have
\[
\lb x|z\rb_\zero \asymp_{\plus,\rho} \lb y|z\rb_\zero \asymp_{\plus,\rho} 0.
\]
Combining gives
\[
\dist(x,y) \asymp_{\plus,\rho,\sigma} 0,
\]
i.e. there exists $\tau$ depending only on $\rho$ and $\sigma$ such that $\dist(x,y) < \tau$.
\QEDmod\end{proof}
Let 
\[
M = \#\{g\in G: g(\zero)\in B(\zero,\tau)\};
\]
$M$ is finite since $G$ is strongly discrete. Then fix $\w M > 0$ large to be announced below, depending on $\rho$ and $M$ (and thus implicitly on $\sigma$). Since $\eta$ is $t$-divergent, we have
\begin{align*}
\infty
= \Sigma_t(G\given\w B_\eta)
&=_\pt \sum_{n = 1}^\infty\Sigma_t(G\given\w B_\eta\cap A_n)\\
&=_\pt \sum_{n = 1}^\infty\sum_{\substack{g\in G\\ g(\zero)\in\w B_\eta\cap A_n}}b^{-t\dist(\zero,g(\zero))}\\
&\asymp_\times \sum_{n = 1}^\infty b^{-(t - s)n}b^{-sn}\#\{g\in G: g(\zero)\in \w B_\eta\cap A_n\}.
\end{align*}
It follows that there exist arbitrarily large numbers $n\in\N$ such that
\begin{equation}
\label{ndef}
b^{-sn}\#\{g\in G: g(\zero)\in \w B_\eta\cap A_{n}\}\geq \w M.
\end{equation}
Fix such an $n$, also to be announced below, depending on $\lambda$, $\rho$, $\w\rho$, and $\w M$ (and thus implicitly on $M$ and $\sigma$). Let $S_\eta$ be a maximal $\tau$-separated subset of $G(\zero)\cap\w B_\eta\cap A_{n}$.\Footnote{A set $S$ is \emph{$\tau$-separated} if \[x,y\in S\text{ distinct} \Rightarrow \dist(x,y)\geq\tau.\] The existence of a maximal $\tau$-separated set follows from Zorn's lemma.} We have
\[
\#(S_\eta) \geq \frac{\#\{g\in G: g(\zero)\in \w B_\eta\cap A_{n}\}}{M}
\]
and so
\begin{equation}
\label{asdist0a}
\sum_{x\in S_\eta}b^{-s\dist(\zero,x)} \asymp_\times b^{-sn}\#(S_\eta) \geq \frac{\w M}{M} \asymp_{\times,M} \w M.
\end{equation}
\begin{proof}[Proof of \textup{(i)}]
In order to see that the shadows $(\PP_{z,x})_{x\in S_\eta}$ are pairwise disjoint, suppose that $x,y\in S_\eta$ are such that $\PP_{z,x}\cap \PP_{z,y}\neq\emptyset$. By Claim \ref{claimtau} we have $\dist(x,y) < \tau$. Since $S_\eta$ is $\tau$-separated, this implies $x = y$.

Fix $x\in S_\eta$. Using \eqref{xz0} and the fact that $x\in A_{n}$, we have
\[
\lb\zero|z\rb_x \asymp_\plus \dist(\zero,x) - \lb x|z\rb_\zero\asymp_{\plus,\rho} \dist(\zero,x)\asymp_\plus n.
\]
Thus for all $\xi\in\PP_{z,x}$,
\[
0
\asymp_{\plus,\sigma} \lb z|\xi\rb_x
\gtrsim_\plus \min(\lb\zero|z\rb_x,\lb\zero|\xi\rb_x)
\asymp_\plus \min(n,\lb\zero|\xi\rb_x);
\]
taking $n$ sufficiently large (depending on $\sigma$), this gives
\[
\lb\zero|\xi\rb_x \asymp_{\plus,\sigma} 0,
\]
from which it follows that
\[
\lb x|\xi\rb_\zero
\asymp_\plus d(\zero,x)-\lb\zero|\xi\rb_x
\asymp_{\plus,\sigma} n.
\]
Therefore, since $x\in\w B_\eta$, we get
\[
\lb\xi|\eta\rb_\zero
\gtrsim_\plus \min(\lb x|\xi\rb_\zero,\lb x|\eta\rb_\zero)
\gtrsim_{\plus,\sigma} \min(n,\w\rho).
\]
Thus $\xi\in B_\eta$ as long as $\w\rho$ and $n$ are large enough (depending on $\sigma$). Thus $\PP_{z,x}\subset B_\eta$.

Finally, note that we do not need to prove that $\PP_{z,x}\subset \PP_{z,\zero}$, since it is implied by \eqref{epsilonnew} which we prove below.

\QEDmod\end{proof}
\begin{proof}[Proof of \textup{(ii)}]
Take any $x\in S_\eta$. Then by (\ref{xz0}), we have
\begin{equation}
\label{asympan1}
\dist(x,z) - \dist(\zero,z) = \dist(\zero,x) - 2\lb x|z\rb_\zero \asymp_{\plus,\rho} \dist(\zero,x) \asymp_\plus n.
\end{equation}
Combining with the Diameter of Shadows Lemma gives
\begin{equation}
\label{asympan1new}
\frac{\Diam_z(\PP_{z,x})}{\Diam_z(\PP_{z,\zero})}
\asymp_{\times,\sigma} \frac{b^{-\dist(z,x)}}{b^{-\dist(z,\zero)}}
\asymp_{\times,\rho} b^{-n}.
\end{equation}
Thus by choosing $n$ sufficiently large depending on $\sigma$, $\lambda$, and $\rho$ (and satisfying (\ref{ndef})), we guarantee that the second inequality of (\ref{lambdanew}) holds. On the other hand, once $n$ is chosen, (\ref{asympan1new}) guarantees that if we choose $\kappa$ sufficiently small, then the first inequality of (\ref{lambdanew}) holds.

In order to prove (\ref{epsilonnew}), let $\xi\in \PP_{z,x}$ and let $\gamma\in \del X\setminus \PP_{z,\zero}$. We have
\begin{align*}
\lb x|\xi\rb_z &\asymp_\plus \dist(x,z) - \lb z|\xi\rb_x \geq \dist(x,z) - \sigma\\
\lb \zero|\gamma\rb_z &\asymp_\plus \dist(\zero,z) - \lb z|\gamma \rb_\zero \leq \dist(\zero,z) - \sigma.
\end{align*}
Also, by (\ref{xz0}) we have
\[
\lb \zero|x\rb_z \asymp_\plus \dist(\zero,z) - \lb x|z\rb_\zero \asymp_{\plus,\rho} \dist(\zero,z).
\]
Applying Gromov's inequality twice and then applying (\ref{asympan1}) gives
\begin{align*}
\dist(\zero,z) - \sigma \gtrsim_\plus \lb \zero|\gamma\rb_z
&\gtrsim_{\plus\phantom{,\rho}} \min\left(\lb\zero|x\rb_z,\lb x|\xi\rb_z,\lb \xi|\gamma\rb_z\right)\\
&\gtrsim_{\plus,\rho} \min\left(\dist(\zero,z),\dist(x,z) - \sigma,\lb \xi|\gamma\rb_z\right)\\
&\asymp_{\plus\phantom{,\rho}} \min\left(\dist(\zero,z),\dist(\zero,z) + n - \sigma,\lb \xi|\gamma\rb_z\right).
\end{align*}
By choosing $n$ and $\sigma$ sufficiently large (depending on $\rho$), we can guarantee that neither of the first two expressions can represent the minimum without contradicting the inequality. Thus
\[
\dist(\zero,z) - \sigma \gtrsim_{\plus,\rho} \lb\xi|\gamma\rb_z;
\]
exponentiating and the Diameter of Shadows Lemma give
\[
\Dist_{b,z}(\xi,\gamma)
\gtrsim_{\times,\rho} b^{-(\dist(\zero,z) - \sigma)}
\asymp_{\times,\sigma} b^{-\dist(\zero,z)}
\asymp_{\times,\sigma} \Diam_z(\PP_{z,\zero}).
\]
Thus we may choose $\kappa$ small enough, depending on $\rho$ and $\sigma$, so that (\ref{epsilonnew}) holds.

\QEDmod\end{proof}
\begin{proof}[Proof of \textup{(iii)}]
\begin{align*}
\sum_{x\in S_\eta}\Diam_z^s(\PP_{z,x})
&\asymp_{\times\phantom{,M}} \sum_{x\in S_\eta}b^{-s\dist(z,x)} \by{the Diameter of Shadows Lemma}\\
&\asymp_{\times,\rho} \;\; b^{-s\dist(z,\zero)}\sum_{x\in S_\eta}b^{-\dist(\zero,x)} \by{\eqref{asympan1}}\\
&\gtrsim_{\times,M} \w Mb^{-s\dist(z,\zero)} \by{\eqref{asdist0a}}\\
&\asymp_{\times\phantom{,M}} \w M\Diam_z^s(\PP_{z,\zero}). \by{the Diameter of Shadows Lemma}
\end{align*}
Letting $\w M$ be larger than the implied constant yields the result.
\QEDmod\end{proof}
\end{proof}

\draftnewpage\section{Proof of Theorem \ref{theoremabsolutewinning} (Absolute winning of $\BA_\xi$)}
In this section we prove Theorem \ref{theoremabsolutewinning}. We repeat both Theorem \ref{theoremabsolutewinning} and Corollary \ref{corollarystructure} for convenience.

\begin{repcorollary}{corollarystructure}[Proven in Sections \ref{sectionstructures} - \ref{sectionlemmastructure}]
Let $(X,\dist,\zero,b,G)$ be as in \sectionsymbol\ref{standingassumptions}. Then for every $0 < s < \delta_G$ and for all $\sigma > 0$ sufficiently large, there exist a tree $T$ on $\N$ and an embedding $T\ni\omega\mapsto x_\omega\in G(\zero)$ such that if 
\[
\PP_\omega := \Shad(x_\omega,\sigma),
\]
then $(\PP_\omega)_{\omega\in T}$ is a partition structure on $(\del X,\Dist)$, whose limit set $\limitset_s\subset\Lur(G)$ is Ahlfors $s$-regular.
\end{repcorollary}

\begin{theorem}
\label{theoremabsolutewinning}
Let $(X,\dist,\zero,b,G)$ be as in \sectionsymbol\ref{standingassumptions}. Fix $0 < s < \delta_G$ and let $\limitset_s$ be defined as in Corollary \ref{corollarystructure}. Then for each $\xi\in\del X$, the set $\BA_\xi\cap \limitset_s$ is absolute winning on $\limitset_s$.
\end{theorem}

{\it Reductions.}
First, note that by (iii) of Proposition \ref{propositionwinningproperties}, it is enough to show that $(\BA_\xi\cup G(\xi))\cap \limitset_s$ is absolute winning on $(\limitset_s,\Dist)$. Next, by (v) of Proposition \ref{propositionwinningproperties} together with Lemma \ref{lemmaquasisymmetric}, it is enough to show that $\pi^{-1}(\BA_\xi\cup G(\xi))$ is absolute winning on $(T^\N,\rho_2)$. Finally, by Corollary \ref{corollarymodifiedabsolutewinning}, it is enough to show that $\pi^{-1}(\BA_\xi\cup G(\xi))$ is $2^{-m}$-modified absolute winning for every $m\in\N$.

Thus Theorem \ref{theoremabsolutewinning} is a direct corollary of the following theorem:

\begin{theorem}
\label{theoremmodifiedabsolutewinning}
Let $(X,\dist,\zero,b,G)$ be as in \sectionsymbol\ref{standingassumptions}. Fix $0 < s < \delta_G$, let $(\PP_\omega)_{\omega\in T}$ be defined as in Corollary \ref{corollarystructure}, and let $\pi:T^\N\to \Lur$ be as in Definition \ref{definitionpi}. Then for each $\xi\in\del X$ and for each $m\in\N$, the set $\pi^{-1}(\BA_\xi\cup G(\xi))$ is $2^{-m}$-modified absolute winning on $(T^\N,\rho_2)$.
\end{theorem}
The remainder of this section will be devoted to the proof of Theorem \ref{theoremmodifiedabsolutewinning}. Notice that a ball in $(T^\N,\rho_2)$ of radius $2^{-n}$ is simply a cylinder of length $n$, so multiplying the radius of a ball by $2^{-m}$ is the same as increasing the length of the corresponding cylinder by $m$.


Let $\sigma > 0$ be large enough so that both Corollary \ref{corollarystructure} and the Diameter of Shadows lemma hold. For each $x\in X$ let $\PP_x = \Shad(x,\sigma)$, so that $\PP_\omega = \PP_{x_\omega}$.

Fix $c > 0$ to be announced below (depending only on $m$ and on the partition structure). We define
\begin{equation}
\label{phidef}
\phi(x) := \sum_{\substack{g\in G \\ \lb g(\zero)|g(\xi)\rb_\zero\leq \dist(\zero,x) \\ g(\xi)\in \PP_x}}b^{-c\busemann_\zero(g(\zero),x)}
\end{equation}
and for each $\omega\in T$, let $\phi(\omega) := \phi(x_\omega)$. Intuitively, $\phi(x)$ measures ``how many points in the orbit of $\xi$ are close to $x$, up to a fixed height''. Since badly approximable points stay away from the orbit of $\xi$, to find a badly approximable point we should minimize $\phi$.

Thus Alice's strategy is as follows: If Bob has just chosen his $n$th ball $B_n = [\omega^{(n)}]$, then Alice will choose (i.e. remove) the ball $A_n = [\tau^{(n)}]$, where $\tau^{(n)}$ is an extension of $\omega^{(n)}$ of length $m + |\omega^{(n)}|$, and has the largest value for $\phi$ among such extensions. We will show that this strategy forces the sequence $(\phi(\omega_1^n))_1^\infty$ to remain bounded, which in turn forces $\pi(\omega)\in\BA_\xi\cup G(\xi)$.

\begin{lemma}
\label{lemmaphiba}
Fix $\omega\in T(\infty)$, and suppose that the sequence $(\phi(\omega_1^n))_1^\infty$ is bounded. Then $\omega\in\pi^{-1}(\BA_\xi\cup G(\xi))$.
\end{lemma}
\begin{proof}
Fix $\varepsilon > 0$ to be announced below, and let $\eta = \pi(\omega)$. Fix $g\in G$, and suppose for a contradiction that
\[
\Dist(g(\xi),\eta)\leq \varepsilon b^{-\dist(\zero,g(\zero))}
\]
but $\eta\neq g(\xi)$. For convenience of notation let
\[
x_n = x_{\omega_1^n}, \;\; \PP_n = \PP_{\omega_1^n}.
\]
Since $\eta\neq g(\xi)$, we have $g(\xi)\notin \PP_n$ for all $n$ sufficiently large. Let $n\in\N$ be the largest integer such that $g(\xi)\in \PP_n$. In particular, $g(\xi)\notin \PP_{n + 1}$. On the other hand, $\eta\in \PP_{n + 2}$ and so
\begin{align*}
\varepsilon b^{-\dist(\zero,g(\zero))}
\geq \Dist(g(\xi),\eta)
&\geq_\pt \Dist(\PP_{n + 2},\del X\butnot \PP_{n + 1})\\
&\asymp_\times b^{-\dist(\zero,x_{n + 1})}
\asymp_\times b^{-\dist(\zero,x_n)},\footnotemark
\end{align*}
where the last two asymptotics follow from (\ref{kappa}) and (\ref{lambda}), respectively, together with the Diameter of Shadows Lemma. Taking negative logarithms we have

\footnotetext{Here and from now on we omit the dependence on $\kappa$, $\lambda$, and $\sigma$ when writing asymptotics.}

\begin{equation}
\label{contradiction}
\dist(\zero,g(\zero))\lesssim_\plus \log_b(\varepsilon) + \dist(\zero,x_n).
\end{equation}
Combining with (c) of Proposition \ref{propositionbasicidentities} yields
\[
\lb g(\zero)|g(\xi)\rb_\zero \lesssim_\plus \log_b(\varepsilon) + \dist(\zero,x_n).
\]
If $\varepsilon$ is chosen small enough, then this bound is a regular inequality, i.e.
\[
\lb g(\zero)|g(\xi)\rb_\zero \leq \dist(\zero,x_n).
\]
Together with the fact that $g(\xi)\in \PP_n$, this guarantees that the term $b^{-c\busemann_\zero(g(\zero),x_n)}$ will be included in the summation (\ref{phidef}) (with $x = x_n$). In particular, since $(\phi(x_n))_1^\infty$ is bounded we have
\[
b^{-c\busemann_\zero(g(\zero),x_n)} \lesssim_\times 1
\]
and thus
\[
\dist(\zero,x_n) \lesssim_\plus \dist(\zero,g(\zero)).
\]
If $\varepsilon$ is small enough, then this is a contradiction to (\ref{contradiction}).
\end{proof}

To prove that Alice's strategy forces the sequence $(\phi(x_n))_1^\infty$ to remain bounded, we first show that it begins bounded. Now clearly,
\begin{align*}
\phi(x) &\leq b^{c\dist(\zero,x)}\sum_{\substack{g\in G \\ \lb \zero|\xi\rb_{g^{-1}(\zero)}\leq\dist(\zero,x)}}b^{-c \dist(\zero,g(\zero))}\\
&= b^{c\dist(\zero,x)}\sum_{\substack{g\in G \\ \xi\in \Shad(g^{-1}(\zero),\dist(\zero,x))}}b^{-c \dist(\zero,g(\zero))}.
\end{align*}
Thus, the following lemma demonstrates that $\phi(x)$ is finite for every $x\in X$:
\begin{lemma}
\label{lemmageneralfinite}
For all $\rho > 0$, we have
\[
\sum_{\substack{g\in G \\ \xi\in \Shad(g^{-1}(\zero),\rho)}}b^{-c \dist(\zero,g(\zero))} < \infty.
\]
\end{lemma}
\begin{proof}
Let
\[
S_\rho := \{z\in X:\xi\in\Shad(z,\rho)\}.
\]
As in the proof of Lemma \ref{lemmastructure} we let
\[
A_n := B(\zero,n + 1)\butnot B(\zero,n).
\]
Then for all $z,w\in S_\rho\cap A_n$, we have $\Shad(z,\rho)\cap\Shad(w,\rho)\neq\emptyset$, and so by the Intersecting Shadows Lemma
\[
\dist(z,w) \asymp_{\plus,\rho} |\busemann_\zero(z,w)| \asymp_\plus 0.
\]
Let $\tau = \tau_\rho > 0$ be the implied constant of this asymptotic, so that
\begin{equation}
\label{diamSrhoAn}
\diam(S_\rho\cap A_n)\leq\tau
\end{equation}
for all $n$. Let $M = \#\{g\in G:g(\zero)\in B(\zero,\tau)\}$; $M < \infty$ since $G$ is strongly discrete.

For each $n\in\N$, choose $z_n\in S_\rho\cap A_n\cap G(\zero)$ if possible, otherwise not. It follows from (\ref{diamSrhoAn}) that $S_\rho\cap G(\zero)\subset \bigcup_n B(z_n,\tau)$. We have
\begin{align*}
\sum_{\substack{g\in G \\ \xi\in\Shad(g^{-1}(\zero),\rho)}}b^{-c \dist(\zero,g(\zero))}
&= \sum_{\substack{g\in G \\ g(\zero)\in S_\rho}}b^{-c \dist(\zero,g(\zero))}\\
&\leq \sum_{n = 0}^\infty\sum_{\substack{g\in G \\ g(\zero)\in B(z_n,\tau)}}b^{-c \dist(\zero,g(\zero))}\\
&\leq \sum_{n = 0}^\infty M b^{-c(\dist(\zero,z_n) - \tau)}
\asymp_\times M b^{c\tau}\sum_{n = 0}^\infty b^{-cn} < \infty,
\end{align*}
which completes the proof.
\end{proof}

Let $\alpha > 0$ be large enough so that $\dist(\zero,x_{\omega a}) \leq \dist(\zero,x_\omega) + \alpha$ for all $\omega\in T$ and for all $a\in T(\omega)$; such an $\alpha$ exists by (\ref{kappa}) and the Diameter of Shadows Lemma.

\begin{lemma}
\label{lemmaphirecursion}
For each $\omega\in T$,
\begin{equation}
\label{phirecursion}
\sum_{a\in T(\omega)} \phi(x_{\omega a}) \lesssim_\plus b^{c\alpha}\phi(x_\omega).
\end{equation}
\end{lemma}
\begin{proof}
For convenience of notation let $x = x_\omega$. We have
\begin{align*}
b^{-c\alpha}\sum_{a\in T(\omega)} \phi(x_{\omega a})
&\leq \sum_{a\in T(\omega)} b^{-c\busemann_\zero(x_{\omega a},x)}\phi(x_{\omega a})
\since{$\dist(\zero,x_{\omega a}) \leq \dist(\zero,x_\omega) + \alpha$}
\\
&= \sum_{a\in T(\omega)}\sum_{\substack{g\in G \\ \lb g(\zero)|g(\xi)\rb_\zero\leq \dist(\zero,x_{\omega a}) \\ g(\xi)\in \PP_{x_{\omega a}}}}b^{-c\busemann_\zero(g(\zero),x)}
\by{the definition of $\phi$}
\\
&\leq \sum_{\substack{g\in G \\ \lb g(\zero)|g(\xi)\rb_\zero\leq \dist(\zero,x) + \alpha \\ g(\xi)\in \PP_x}}b^{-c\busemann_\zero(g(\zero),x)}
\since{$(\PP_{\omega a})_{a\in T(\omega)}$ are disjoint}
\\
&= \phi(x) + \sum_{\substack{g\in G \\ \dist(\zero,x) < \lb g(\zero)|g(\xi)\rb_\zero\leq \dist(\zero,x) + \alpha \\ g(\xi)\in \PP_x}}b^{-c\busemann_\zero(g(\zero),x)}.
\by{the definition of $\phi$}
\end{align*}
Thus, to prove Lemma \ref{lemmaphirecursion} it is enough to show that the series
\begin{equation}
\label{finitesum}
\sum_{\substack{g\in G \\ \dist(\zero,x) < \lb g(\zero)|g(\xi)\rb_\zero\leq \dist(\zero,x) + \alpha \\ g(\xi)\in \PP_x}}b^{-c\busemann_\zero(g(\zero),x)}
\end{equation}
is bounded independent of $x$. To this end, fix $g\in G$ which gives a term in (\ref{finitesum}), i.e.
\begin{align} \label{asympalpha}
\dist(\zero,x) &< \lb g(\zero)|g(\xi)\rb_\zero\leq \dist(\zero,x) + \alpha\\ \label{gxiin}
g(\xi) &\in \PP_x = \Shad(x,\sigma).
\end{align}
Now by (\ref{gxiin}) we have
\[
\lb x|g(\xi)\rb_\zero \asymp_\plus \dist(\zero,x);
\]
combining with (\ref{asympalpha}) and Gromov's inequality yields
\[
\lb x|g(\zero)\rb_\zero \asymp_\plus \dist(\zero,x),
\]
which together with (b) of Proposition \ref{propositionbasicidentities} gives
\begin{equation}
\label{0g0x0}
\lb \zero|g(\zero)\rb_x \asymp_\plus 0.
\end{equation}
By (j) of Proposition \ref{propositionbasicidentities},
\begin{align*}
\lb \zero|\xi\rb_{g^{-1}(x)} = \lb g(\zero)|g(\xi)\rb_x
&\asymp_\plus \lb g(\zero)|g(\xi)\rb_\zero + \lb \zero|g(\zero)\rb_x - \lb x|g(\xi)\rb_\zero\\
&\asymp_\plus \dist(\zero,x) + 0 - \dist(\zero,x) = 0.
\end{align*}
Thus, if we choose $\rho > 0$ large enough, we see that $\xi\in\Shad(g^{-1}(x),\rho)$ for all $g$ satisfying (\ref{asympalpha}) and (\ref{gxiin}). On the other hand, rearranging (\ref{0g0x0}) yields
\[
\busemann_\zero(g(\zero),x) \asymp_\plus \dist(g(\zero),x) = \dist(\zero,g^{-1}(x)).
\]
We return to our original series; fixing $\rho > 0$ large enough we have
\begin{align*}
b^{-c\alpha}\sum_{a\in T(\omega)}\phi(x_{\omega a}) - \phi(x)
&\leq_\pt \sum_{\substack{g\in G \\ \dist(\zero,x) < \lb g(\zero)|g(\xi)\rb_\zero\leq \dist(\zero,x) + \alpha \\ g(\xi)\in \PP_x}}b^{-c\busemann_\zero(g(\zero),x)}\\
&\lesssim_\times \sum_{\substack{g\in G \\ \xi\in\Shad(g^{-1}(x),\rho)}}b^{-c \dist(\zero,g^{-1}(x))}\\
&=_\pt \sum_{\substack{g\in G \\ \xi\in\Shad(g^{-1}(\zero),\rho)}}b^{-c \dist(\zero,g^{-1}(\zero))},
\end{align*}
since $x = x_\omega \in G(\zero)$ by construction. The sum is clearly independent of $\omega$, and by Lemma \ref{lemmageneralfinite} it is finite. This completes the proof.
\end{proof}

Iterating (\ref{phirecursion}) yields
\begin{equation}
\label{cmalpha}
\sum_\tau \phi(\tau) \lesssim_{\plus,m} b^{cm\alpha}\phi(\omega),
\end{equation}
where the sum is taken over all $\tau$ of which $\omega$ is an initial segment and for which $|\tau| = m + |\omega|$.

Recall that Alice's strategy is to remove the ball $A_n = [\tau^{(n)}]$, where $\tau^{(n)}$ is the extension of $\omega^{(n)}$ of length $m + |\omega^{(n)}|$ which has the largest value for $\phi$. Now let $B_{n + 1} = [\omega^{(n + 1)}]$ be the next ball that Bob plays. By the rules of the game, we must have $|\omega^{(n + 1)}| = m + |\omega^{(n)}|$ and $[\omega^{(n + 1)}]\cap[\tau^{(n)}] = \emptyset$, and so $\omega^{(n + 1)}$ and $\tau^{(n)}$ represent distinct terms in the sum \eqref{cmalpha}. In particular,
\[
\phi(\tau^{(n)}) + \phi(\omega^{(n + 1)}) \lesssim_{\plus,m} b^{cm\alpha}\phi(\omega^{(n)}).
\]
On the other hand, we have
\[
\phi(\omega^{(n + 1)}) \leq \phi(\tau^{(n)})
\]
and so
\begin{equation}
\label{suggestscsmall}
\phi(\omega^{(n + 1)}) \lesssim_{\plus,m} \frac{b^{cm\alpha}}{2}\phi(\omega^{(n)}).
\end{equation}
This inequality suggests that we should choose $c$ small enough so that $b^{cm\alpha} < 2$. As claimed at the beginning of this section, such a choice depends only on $m$ and on the partition structure $(\PP_\omega)_{\omega\in T}$.

It follows that for all $n\in\N$ we have
\[
\phi(\omega^{(n)}) \leq \max\left(C\sum_{i = 0}^\infty \left(\frac{b^{cm\alpha}}{2}\right)^i,\phi(\omega^{(0)})\right) < \infty,
\]
where $C$ is the implied constant in (\ref{suggestscsmall}). Thus by Lemma \ref{lemmaphiba}, $\omega\in\pi^{-1}(\BA_\xi\cup G(\xi))$.


\draftnewpage\section{Proof of Theorem \ref{theoremjarnikbesicovitch} (Generalization of the Jarn\'ik--Besicovitch Theorem)}
\begin{theorem}[Generalization of the Jarn\'ik--Besicovitch Theorem]
\label{theoremjarnikbesicovitch}
Let $(X,\dist,\zero,b,G)$ be as in \sectionsymbol\ref{standingassumptions}, and fix a distinguished point $\xi\in\del X$. Let $\delta$ be the Poincar\'e exponent of $G$, and suppose that $\delta < \infty$. Then for each $c > 0$,
\begin{equation}
\label{jarnikbesicovitch}
\begin{split}
\HD(W_{c,\xi}) &= \sup\left\{s\in(0,\delta): P_\xi(s) \leq \frac{\delta - s}{c}\right\}\\
&= \inf\left\{s\in(0,\delta): P_\xi(s) \geq \frac{\delta - s}{c}\right\}
\leq \frac{\delta}{c + 1}\cdot
\end{split}
\end{equation}
In particular
\begin{equation}
\label{HDVWAxi}
\HD(\VWA_\xi) = \sup\{s\in (0,\delta):P_\xi(s) < +\infty\} = \inf\left\{s\in(0,\delta): P_\xi(s) = +\infty\right\}
\end{equation}
and $\HD(\Liou_\xi) = 0$. In these formulas we let $\sup\emptyset = 0$ and $\inf\emptyset = \delta$.
\end{theorem}

\begin{remark}
If $\delta = \infty$, then for each $c > 0$, our proof shows that
\[
\HD(W_{c,\xi}) \geq \sup\left\{s\in(0,\infty): P_\xi(s) < +\infty\right\}.
\]
However, we cannot prove any upper bound in this case.
\end{remark}
\begin{remark}
For the remainder of this section, we assume $\delta < \infty$.
\end{remark}

\subsection{Properties of the function $P_\xi$}

Recall that the function $P_\xi:[0,\delta]\to[0,\infty]$ is defined by the equations
\begin{align} \label{Hxisigmadef}
H_{\xi,\sigma,s}(r) &:= \HH_\infty^s(B(\xi,r)\cap \Lrsigma(G))\\ \label{Pxidef}
P_\xi(s) &:= \lim_{\sigma\to\infty}\liminf_{r\to 0}\frac{\log_b(H_{\xi,\sigma,s}(r))}{\log_b(r)}\cdot
\end{align}
where $\HH_\infty^s$ is the $s$-dimensional Hausdorff content.

\begin{proposition}
\label{propositionPxifacts}
The function $s\mapsto P_\xi(s) - s$ is nondecreasing (and thus so is $P_\xi$). Moreover, $P_\xi(s)\geq s$ for all $s\in[0,\delta]$.
\end{proposition}
\begin{proof}
Fix $s,t\in[0,\delta]$ with $s > t$ and notice that for all $A\subset\del X$ and $r > 0$
\[
\Diam^s(A\cap B(\xi,r)) \leq \Diam^{s - t}(B(\xi,r))\Diam^t(A).
\]
Running this inequality through formulas \eqref{Hausdorffcontentdef}, \eqref{Hxisigmadef}, and \eqref{Pxidef} yields $P_\xi(s)\geq (s - t) + P_\xi(t)$, i.e. $s\mapsto P_\xi(s) - s$ is nondecreasing.

For all $s\in[0,\delta]$ and for all $r > 0$,
\[
H_{\xi,\sigma,s}(r) \leq \HH_\infty^s(B(\xi,r)) \leq \Diam^s(B(\xi,r)) \leq (2r)^s;
\]
running this inequality through \eqref{Pxidef} yields $P_\xi(s)\geq s$.
\end{proof}
\begin{corollary}
\label{corollaryjarnikbesicovitcheasyparts}
\begin{align*}
Q_\xi(c) := &\sup\left\{s\in(0,\delta): P_\xi(s) \leq \frac{\delta - s}{c}\right\}
= \inf\left\{s\in(0,\delta): P_\xi(s) \geq \frac{\delta - s}{c}\right\}\\
= &\sup\left\{s\in(0,\delta): P_\xi(s) < \frac{\delta - s}{c}\right\}
= \inf\left\{s\in(0,\delta): P_\xi(s) > \frac{\delta - s}{c}\right\}\\
\leq &\frac{\delta}{c + 1}.
\end{align*}
Furthermore, the function $Q_\xi:(0,\infty)\to [0,\delta]$ is continuous.
\end{corollary}
\begin{proof}
The three equalities are an easy consequence of the fact that the function $s\mapsto P_\xi(s) - \frac{\delta - s}{c}$ is strictly increasing.\Footnote{Note however that we don't know whether this function is continuous, and so in particular we don't know whether $P_\xi(s) = \frac{\delta - s}{c}$ when $s = Q_\xi(c)$.} The inequality is demonstrated by noting that
\[
P_\xi\left(\frac{\delta}{c + 1}\right) \geq \frac{\delta}{c + 1} = \frac{\delta - \frac{\delta}{c + 1}}{c},
\]
and so $s = \frac{\delta}{c + 1}$ is a member of the first infimum.

To demonstrate the continuity, fix $0 < c_1 < c_2 < \infty$, and let $s_1 = Q_\xi(c_1),s_2 = Q_\xi(c_2)$. Then $s_1\geq s_2$; let us suppose that $s_1 > s_2$. Then for every $\varepsilon > 0$ sufficiently small
\[
\frac{\delta - (s_2 + \varepsilon)}{c_2} \leq P_\xi(s_2 + \varepsilon) \leq P_\xi(s_1 - \varepsilon) \leq \frac{\delta - (s_1 - \varepsilon)}{c_1};
\]
letting $\varepsilon$ tend to zero we have
\[
\frac{\delta - s_2}{c_2} \leq \frac{\delta - s_1}{c_1};
\]
rearranging gives
\[
s_1 - s_2 \leq \frac{(c_2 - c_1)(\delta - s_2)}{c_2} \leq \delta\frac{c_2 - c_1}{c_2}.
\]
This in fact demonstrates that for every $c_0 > 0$, the function $c\mapsto s$ is Lipschitz continuous on $[c_0,\infty)$ with a corresponding constant of $\delta/a$.
\end{proof}

Corollary \ref{corollaryjarnikbesicovitcheasyparts} reduces the proof of Theorem \ref{theoremjarnikbesicovitch} to showing that $\HD(W_{c,\xi}) = Q_\xi(c)$, and then demonstrating the statements about $\HD(\VWA_\xi)$ and $\HD(\Liou_\xi)$. The former will be demonstrated in Subsections \ref{subsectionjarnikbesicovitchleqdirection} - \ref{subsectionjarnikbesicovitchgeqdirection}, and the latter will be left to the reader.

\begin{corollary}
\label{corollaryjarnikbesicovitchcontinuous}
The function $c\mapsto \HD(W_{c,\xi})$ is continuous on $(0,\infty)$.
\end{corollary}
\begin{proof}
This follows directly from Theorem \ref{theoremjarnikbesicovitch} and Corollary \ref{corollaryjarnikbesicovitcheasyparts}.
\end{proof}

\subsection{$\leq$ direction}\label{subsectionjarnikbesicovitchleqdirection}
Fix $s\in(0,\delta)$ which is a member of the second infimum in Corollary \ref{corollaryjarnikbesicovitcheasyparts}, i.e. so that $P_\xi(s) > \frac{\delta - s}{c}$. Fix $1 < \w c < c$ so that
\begin{equation}
\label{Pxigtr}
P_\xi(s) > \frac{\delta - s}{\w c}.
\end{equation}
For simplicity of exposition we will assume that $P_\xi(s) < +\infty$. Fix $\sigma > 0$, and let $\tau > 0$ be given by Lemma \ref{lemmanearinvariance}.

For each $g\in G$, let
\begin{equation}
\label{Bgs}
B_{g,\w c} := B\big(g(\xi),b^{-(1 + \w c\,)\dist(\zero,g(\zero))}\big).
\end{equation}
\begin{claim}
\label{claimboundeddistortionjarnik}
For all $g\in G$ and for all $\eta_1,\eta_2\in g^{-1}(B_{g,\w c})$,
\begin{align}\label{boundeddistortion3}
g'(\eta_1) \asymp_\times g'(\eta_2) &\asymp_\times g'(\xi)\\ \label{boundeddistortion4}
\frac{\Dist(g(\eta_1),g(\eta_2))}{\Dist(\eta_1,\eta_2)} &\asymp_\times g'(\xi)
\end{align}
\end{claim}
\begin{proof}
Fix $\eta\in g^{-1}(B_{g,\w c})$. We have
\begin{equation}
\label{asymptotictozerojarnik}
\left|\log_b\big(\frac{g'(\eta)}{g'(\xi)}\big)\right|
\asymp_\plus |\busemann_\eta(\zero,g^{-1}(\zero)) - \busemann_\xi(\zero,g^{-1}(\zero))|
\asymp_\plus |\lb g(\zero)|g(\eta)\rb_\zero - \lb g(\zero)|g(\xi)\rb_\zero|.
\end{equation}
Let
\begin{align*}
\textup{(A)} &= \lb g(\zero)|g(\eta)\rb_\zero\\
\textup{(B)} &= \lb g(\zero)|g(\xi)\rb_\zero\\
\textup{(C)} &= \lb g(\xi)|g(\eta)\rb_\zero.
\end{align*}
By Gromov's inequality, at least two of the expressions (A), (B), and (C) must be asymptotic. On the other hand, since $g(\eta)\in B_{g,\w c}$, we have
\[
\lb g(\xi)|g(\eta)\rb_\zero - \dist(\zero,g(\zero)) \asymp_\plus -\log_b(\Dist(g(\xi),g(\eta))) - \dist(\zero,g(\zero)) \gtrsim_\plus \w c\,\dist(\zero,g(\zero)) \tendsto g \infty,
\]
Since (A) and (B) are both less than $\dist(\zero,g(\zero))$, this demonstrates that for all but finitely many $g$, (C) is not asymptotic to either (A) or (B). Thus by Gromov's inequality, (A) and (B) are asymptotic (for the finitely many exceptions we just let the implied constant be $\dist(\zero,g(\zero))$), and so \eqref{asymptotictozerojarnik} is asymptotic to zero. This demonstrates \eqref{boundeddistortion3}. Now \eqref{boundeddistortion4} follows from \eqref{boundeddistortion3} and the geometric mean value theorem.
\end{proof}

Let $C > 0$ be the implied constant of \eqref{boundeddistortion4}, and let
\[
r_g = C b^{-(\w c + 1)\dist(\zero,g(\zero))}/g'(\xi).
\]
Then \eqref{boundeddistortion4} implies that
\[
g^{-1}(B_{g,\w c}) \subset B(\xi,r_g).
\]
Fix $\varepsilon > 0$ small to be determined; by the definition of $P_\xi(s)$ we have
\[
H_{\xi,\tau,s}(r_g) \lesssim_\times r_g^{P_\xi(s) - \varepsilon}.
\]
On the other hand, by the definition of $H_{\xi,\tau,s}(r_g)$, there exists a cover $\CC_g$ of $B(\xi,r_g)\cap \Lrtau$ such that
\[
\sum_{A\in\CC_g}\Diam^s(A) \leq H_{\xi,\tau,s}(r_g) + r_g^{P_\xi(s) - \varepsilon};
\]
thus
\[
\sum_{A\in\CC_g}\Diam^s(A) \lesssim_\times r_g^{P_\xi(s) - \varepsilon}.
\]
For each $A\in \CC_g$, Claim \ref{claimboundeddistortionjarnik} implies that
\[
\Diam(g(A)\cap B_{g,\w c}) \lesssim_\times g'(\xi)\Diam(A).
\]
Let $\CC = \bigcup_{g\in G}\{g(A)\cap B_{g,\w c}:A\in\CC_g\}$.
\begin{claim}
\[
\sum_{A\in\CC}\Diam^s(A) < \infty.
\]
\end{claim}
\begin{proof}
\begin{align*}
\sum_{A\in\CC}\Diam^s(A)
&= \sum_{g\in G}\sum_{A\in \CC_g}\Diam^s(g(A)\cap B_{g,\w c})\\
&\lesssim_\times \sum_{g\in G}[g'(\xi)]^s \sum_{A\in \CC_g}\Diam^s(A)\\
&\lesssim_\times \sum_{g\in G}[g'(\xi)]^s r_g^{P_\xi(s) - \varepsilon}\\
&\asymp_\times \sum_{g\in G}[g'(\xi)]^{s - (P_\xi(s) - \varepsilon)} b^{-(P_\xi(s) - \varepsilon)(\w c + 1)\dist(\zero,g(\zero))}.
\end{align*}
By Proposition \ref{propositionPxifacts}, $P_\xi(s) \geq s$. Let $B_g = \log_b(g'(\xi))$. Since $|B_g| \leq \dist(\zero,g(\zero))$, we have
\[
(P_\xi(s) - s - \varepsilon)B_g \leq |P_\xi(s) - s - \varepsilon|\cdot |B_g| \leq \max(P_\xi(s) - s - \varepsilon,\varepsilon)\dist(\zero,g(\zero))
\]
and thus
\[
\sum_{A\in\CC}\Diam^s(A) \lesssim_\times \sum_{g\in G}b^{-\min(\w c(P_\xi(s) - \varepsilon) + s,(\w c + 1) P_\xi(s) - (\w c + 2)\varepsilon)\dist(\zero,g(\zero))}.
\]
The right hand side is the Poincar\'e series of $G$ with respect to the exponent
\[
\min(\w c(P_\xi(s) - \varepsilon) + s,(\w c + 1) P_\xi(s) - (\w c + 2)\varepsilon).
\]
It can be seen by rearranging (\ref{Pxigtr}) that this exponent exceeds $\delta$ for all $\varepsilon$ sufficiently small. Therefore the series converges.
\end{proof}
\begin{claim}
For all $\eta\in W_{c,\xi}\cap\Lrsigma$, $\eta\in A$ for infinitely many $A\in\CC$.
\end{claim}
\begin{proof}
Since $\omega_\xi(\eta) \geq 1 + c > 1 + \w c$, for infinitely many $g\in G$, we have $\eta\in B_{g,\w c}$. Thus $g^{-1}(\eta)\in B(\xi,r_g)$. On the other hand, by Lemma \ref{lemmanearinvariance}, $g^{-1}(\eta)\in\Lrtau$. Thus $g^{-1}(\eta)\in\bigcup\CC_g$; fix $A_g\in\CC_g$ for which $g^{-1}(\eta)\in A_g$. Then $\eta\in g(A_g)\cap B_{g,\w c}\in\CC$.
\end{proof}
Now by the Hausdorff--Cantelli lemma (see e.g. \cite[Lemma 3.10]{BernikDodson}), we have $\HH^s(W_{c,\xi}\cap\Lrsigma) = 0$, and thus $\HD(W_{c,\xi}\cap\Lrsigma) \leq s$. But $\sigma$ was arbitrary; taking the union over all $\sigma\in\N$ gives $\HD(W_{c,\xi})\leq s$.

\draftnewpage

\subsection{Preliminaries for $\geq$ direction}
\subsubsection{Multiplying numbers and sets}
\begin{definition}
\label{definitionaS}
Let $(Z,\Dist)$ be a metric space. For each $S\subset Z$ and for each $a\geq 1$, let
\[
aS := \NN\left(S,\frac{a - 1}{2}\Diam(S)\right),
\]
where
\[
\NN(S,r) := \{z:\Dist(z,S)\leq r\}
\]
is the \emph{$r$-thickening} of $S$.
\end{definition}

\begin{observation}
\begin{align*}
\Diam(aS) &\leq a\Diam(S)\\
a B(z,r) &\subset B(z,ar)\\
a(bS) &\subset (ab)S.
\end{align*}
\end{observation}

For each inequality or inclusion, equality holds if $(Z,\Dist)$ is a Banach space.

\begin{proposition}[Variant of the Vitali covering lemma] \label{proposition5r}
Let $(Z,\Dist)$ be a metric space, and let $\CC$ be a collection of subsets of $Z$ with the property that
\[
\sup\{\Diam(S): S\in\CC\} < \infty.
\]
Then there exists a disjoint subcollection $\w\CC\subset\CC$ so that
\[
\bigcup\CC\subset \bigcup_{A\in\w\CC}5A.
\]
\end{proposition}

The proof of this proposition is a straightforward generalization of the standard proof of the Vitali covering lemma (see e.g. \cite[Theorem 1.5.1]{EvansGariepy}).

\begin{lemma}
\label{lemmaCShad}
Let $X$ be a hyperbolic metric space. For each $C,\sigma > 0$, there exists $\tau > 0$ such that for all $x\in X$
\[
C\Shad(x,\sigma) \subset \Shad(x,\tau).
\]
\end{lemma}
\begin{proof}
Suppose $\eta\in C\Shad(x,\sigma)$; then $\eta\in B(\xi,\frac{C - 1}{2}\Diam(\Shad(x,\sigma)))$ for some $\xi\in\Shad(x,\sigma)$. Now by the Diameter of Shadows Lemma
\[
\Dist(\xi,\eta) \lesssim_{\times,C}\Diam(\Shad(x,\sigma)) \lesssim_{\times,\sigma}b^{-\dist(\zero,x)};
\]
thus
\[
\lb\xi|\eta\rb_\zero \gtrsim_{\plus,C,\sigma}\dist(\zero,x);
\]
on the other hand since $\xi\in\Shad(x,\sigma)$
\[
\lb x|\xi\rb_\zero \asymp_{\plus,\sigma} \dist(\zero,x)
\]
and so by Gromov's inequality
\[
\lb x|\eta\rb_\zero \asymp_{\plus,C,\sigma}\dist(\zero,x).
\]
\end{proof}

\subsubsection{Fuzzy metric spaces}
\begin{definition}
A \emph{fuzzy metric space} consists of a metric space $(Y,\Dist)$ together with a function $f:Y\to(0,\infty)$ satisfying
\begin{equation}
\label{goodf}
f(y) \leq \Dist(y,Y\butnot\{y\}).
\end{equation}
Let $(Y,f)$ be a fuzzy metric space. A ball $B(y,r)\subset Y$ is called \emph{admissible} (with respect to $f$) if $r\geq f(y)$.
\end{definition}
The idea is that each point $y\in Y$ represents some subset of a larger metric space $Z$ which is contained in the set $B(y,f(y))$. An admissible ball is supposed to be one that contains this set.

\begin{observation}
\label{observationtwopoints}
Let $B(y,r)\subset Y$ be an admissible ball. Then $r\geq f(\w y\ew)$ for every $\w y\in B(y,r)$.
\end{observation}
\begin{proof}
Since the case $\w y = y$ holds by definition, let us suppose that $\w y \neq y$. But then by \eqref{goodf} we have
\[
r \geq \Dist(\w y,y) \geq \Dist(\w y,Y\butnot\{\w y\ew\}) \geq f(\w y\ew).
\]
\end{proof}
Now fix $s > 0$. A measure $\mu\in\MM(Y)$ is called \emph{$s$-admissible} if $\mu(B(y,r))\leq r^s$ for every admissible ball $B(y,r)\subset Y$. Consider the quantities
\begin{align*}
\alpha_{Y,f}(s) &:= \sup\{\mu(Y):\mu\in\MM(Y)\text{ is $s$-admissible}\}\\
\beta_{Y,f}(s) &:= \inf\left\{\sum_{k = 1}^\infty r_k^s: \big(B(y_k,r_k)\big)_1^\infty\text{ is a collection of admissible balls which covers $Y$}\right\}.
\end{align*}
\begin{observation}
\[
\alpha_{Y,f}(s) \leq \beta_{Y,f}(s).
\]
\end{observation}
\begin{proof}
If $\mu\in\MM(Y)$ is $s$-admissible and if $\big(B(y_k,r_k)\big)_1^\infty$ is a collection of admissible balls covering $Y$, then
\[
\mu(Y) \leq \sum_{k = 1}^\infty\mu(B(y_k,r_k)) \leq \sum_{k = 1}^\infty r_k^s.
\]
\end{proof}
\begin{claim}
\label{claimauxiliary}
If $Y$ is countable, then
\[
\beta_{Y,f}(s) \leq 32^s \alpha_{Y,f}(s).
\]
\end{claim}
\begin{proof}
Let $(y_n)_1^\infty$ be an indexing of $Y$. Fix $N\in\N$, and let $Y_N = \{y_1,\ldots,y_N\}$. The collection of $s$-admissible measures on $Y_N$ is compact, and so there exists an $s$-admissible measure $\mu_N\in\MM(Y_N)$ satisfying
\begin{equation}
\label{muNYN}
\mu_N(Y_N) = \sup_{\substack{\mu\in\MM(Y_N) \\ \text{$s$-admissible}}}\mu(Y_N) \leq \alpha_{Y,f}(s).
\end{equation}
Now fix $n\leq N$. It follows from \eqref{muNYN} that the measure $\mu_N + \frac{1}{N}\delta_{y_n}$ is not $s$-admissible. By definition, this means that there is some admissible ball $B(y_{n,N},r_{n,N})\subset Y$ such that
\begin{equation}
\label{muN1N}
\left(\mu_N + \frac{1}{N}\delta_{y_n}\right)\big(B(y_{n,N},r_{n,N})\big) > r_{n,N}^s \geq \mu_N(B(y_{n,N},r_{n,N})).
\end{equation}
Clearly, this implies
\begin{equation}
\label{ynynN}
y_n\in B(y_{n,N},r_{n,N}).
\end{equation}
Thus by Observation \ref{observationtwopoints} we have
\begin{equation}
\label{rnNgeq}
r_{n,N}\geq f(y_n).
\end{equation}
On the other hand, by \eqref{muNYN} and \eqref{muN1N} we have
\[
r_{n,N}\leq \big(\alpha_{Y,f}(s) + 1\big)^{1/s}.
\]
Thus, for each $n\in\N$ the sequence $(r_{n,N})_{N = n}^\infty$ is bounded from above and below.

Choose an increasing sequence $(N_k)_1^\infty$ such that for every $n\in\N$,
\[
r_n^{(k)} := r_{n,N_k}\tendsto k r_n \in \CO{f(y_n)}{\infty}.
\]
Now fix $n\in\N$, and notice that for all $k\in\N$ sufficiently large we have
\begin{equation}
\label{rknbounds}
\frac{1}{2}r_n \leq r_n^{(k)} \leq 2r_n
\end{equation}
and thus by \eqref{ynynN}
\[
B(y_n^{(k)},r_n^{(k)}) \subset B(y_n,2r_n^{(k)}) \subset B(y_n,4r_n).
\]
Thus by \eqref{muN1N} and \eqref{rknbounds} we have
\[
\left(\mu_{N_k} + \frac{1}{N_k}\delta_{y_n}\right)\big(B(y_n,4r_n)\big) > \big(r_n/2\big)^s
\]
i.e.
\begin{equation}
\label{muNk}
\mu_{N_k}\big(B(y_n,4r_n)\big) > \big(r_n/2\big)^s - \frac{1}{N_k}.
\end{equation}

Now by the Vitali covering lemma, there exists a set $A\subset\N$ such that the collection
\[
\big(B(y_n,4r_n)\big)_{n\in A}
\]
is disjoint and such that
\[
Y \subset \bigcup_{n = 1}^\infty B(y_n,4r_n) \subset \bigcup_{n\in A}B(y_n,16r_n).
\]
In particular, for each $k\in\N$ we have
\[
\alpha_{Y,f}(s) \geq \mu_{N_k}(Y) \geq \sum_{n\in A}\mu_{N_k}\big(B(y_n,4r_n)\big) \geq \sum_{\substack{n\in A \\ \text{\eqref{rknbounds} holds}}}\left(\big(r_n/2\big)^s - \frac{1}{N_k}\right)_+,
\]
where $x_+ = \max(0,x)$. Now let us take the limit as $k$ approaches infinity. By the monotone convergence theorem, we may consider the limit of each summand separately. But if $n$ is fixed, then \eqref{rknbounds} holds for all sufficiently large $k$. Thus
\[
\alpha_{Y,f}(s) \geq \sum_{n\in A}\left(\big(r_n/2\big)^s - 0\right)_+ = \frac{1}{2^s}\sum_{n\in A}r_n^s.
\]
On the other hand, note that for each $n\in A$, the ball $B(y_n,16r_n)$ is admissible by \eqref{rnNgeq} and \eqref{rknbounds}. Thus the collection $\big(B(y_n,16r_n)\big)_{n\in A}$ is a collection of admissible balls which covers $Y$, and so
\[
\beta_{Y,f}(s) \leq \sum_{n\in A}(16r_n)^s = 16^s \sum_{n\in A}r_n^s.
\]
This completes the proof.
\end{proof}

\subsection{$\geq$ direction}\label{subsectionjarnikbesicovitchgeqdirection}
Now let $X$ be a hyperbolic metric space, let $G$ be a strongly discrete subgroup of $\Isom(X)$ of general type, and fix distinguished points $\zero\in X$ and $\xi\in\del X$. Fix $c > 0$. Fix $s\in(0,\delta)$ which is a member of the second supremum of Corollary \ref{corollaryjarnikbesicovitcheasyparts}, i.e. so that
\begin{equation}
\label{Pxis}
P_\xi(s) < \frac{\delta - s}{c}.
\end{equation}
Fix $\tau \geq \sigma > 0$ large to be announced below. $\tau$ will depend on $\sigma$.

\subsubsection{A fuzzy metric space for each $r > 0$}
\label{subsubsectionjarnikpreliminaries2}
Let
\begin{align*}
\PP_x &:= \Shad(x,\sigma) &
\w\PP_x &:= \Shad(x,\tau)\\
D_x &:= \Diam(\PP_x) &
\w D_x &:= \Diam(\w\PP_x).
\end{align*}
Now fix $r > 0$ and let
\[
S_{\xi,\tau}(r) := \{x\in G(\zero):10\w\PP_x\subset B(\xi,r)\}.
\]
\begin{observation}
\label{observationjarnikpreliminaries}
\[
\Lrsigma\cap B(\xi,r) \subset \bigcup_{x\in S_{\xi,\tau}(r)}\PP_x.
\]
\end{observation}
\begin{proof}
Fix $\eta\in \Lrsigma\cap B(\xi,r)$. Then there exists a sequence $x_n\tendsto n \eta$ so that $\eta\in\bigcap_{n = 1}^\infty \PP_{x_n}$. Now since $\Diam(\w\PP_{x_n})\tendsto n 0$, we have
\[
10\w\PP_{x_n}\subset B(\xi,r)
\]
for all $n$ sufficiently large. Thus $x_n\in S_{\xi,\tau}(r)$ for all $n$ sufficiently large, which completes the proof.
\end{proof}
Now applying Proposition \ref{proposition5r}, we see that there exists a set $A_r\subset S_{\xi,\tau}(r)$ such that the collection
\begin{equation}
\label{2PxAr}
\big(2\w\PP_x\big)_{x\in A_r}
\end{equation}
is disjoint, and such that
\[
\bigcup_{x\in S_{\xi,\tau}(r)}\PP_x \subset \bigcup_{x\in S_{\xi,\tau}(r)}2\w\PP_x \subset \bigcup_{x\in A_r}10\w\PP_x.
\]
Combining with Observation \ref{observationjarnikpreliminaries} and \eqref{Hxisigmadef} yields
\begin{equation}
\label{Hxisigmaleq}
H_{\xi,\sigma,s}(r) \leq \HH_\infty^s\left(\bigcup_{x\in A_r}10\w\PP_x\right).
\end{equation}
Now for each $x\in A_r$, choose $\rep(x)\in\PP_x$, and let
\begin{equation}
\label{fyxdef}
f(\rep(x)) := \w D_x \leq \Dist(\rep(x),\del X\butnot 2\w\PP_x).
\end{equation}
Let
\[
Y_r = \rep(A_r).
\]
Then by the inequality of \eqref{fyxdef} and the disjointness of the collection \eqref{2PxAr}, the pair $(Y_r,f)$ is a fuzzy metric space.
\begin{claim}
\[
H_{\xi,\sigma,s}(r) \leq 7^s \beta_{Y_r,f}(s) \leq 224^s \alpha_{Y_r,f}(s).
\]
\end{claim}
\begin{proof}
Note that the second inequality is just Claim \ref{claimauxiliary}. To prove the first inequality, we prove the following:
\begin{subclaim}
\label{subclaimadmissiblecover}
Let $\big(B(y_k,r_k)\big)_{k = 1}^\infty$ be a collection of admissible balls which covers $Y_r$. Then
\[
\bigcup_{x\in A_r}10\w\PP_x \subset \bigcup_{k = 1}^\infty B(y_k,7r_k).
\]
\end{subclaim}
\begin{proof}
Fix $\eta\in 10\w\PP_x$ for some $x\in A_r$. Then $\rep(x)\in Y_r$, so $\rep(x)\in B(y_k,r_k)$ for some $k\in\N$. Since $B(y_k,r_k)$ is admissible, by Observation \ref{observationtwopoints} we have
\[
r_k \geq f(\rep(x)) = \w D_x
\]
and thus
\[
10\w\PP_x = \NN\left(\w\PP_x,\frac{9}{2}\w D_x\right) \subset \NN(\w\PP_x,5\w D_x) \subset B(\rep(x),6\w D_x) \subset B(\rep(x),6r_k) \subset B(y_k,7r_k).
\]

\QEDmod\end{proof}
Thus
\begin{align*}
7^s\beta_{Y_r,f}(s) &= \inf\left\{\sum_{k = 1}^\infty(7r_k)^s: \big(B(y_k,r_k)\big)_1^\infty\text{ is a collection of admissible balls which covers $Y_r$}\right\}\hspace{-2.2 in}\\
&\geq \inf\left\{\HH_\infty^s\left(\bigcup_{k = 1}^\infty B(y_k,7r_k)\right):\right.\\
&\hspace{1 in}\left.\big(B(y_k,r_k)\big)_1^\infty\text{ is a collection of admissible balls which covers $Y_r$}\right\}\hspace{-2.2 in}\\
&\geq \HH_\infty^s\left(\bigcup_{x\in A_r}10\w\PP_x\right) \by{Subclaim \ref{subclaimadmissiblecover}}&~\\
&\geq H_{\xi,\sigma,s}(r) \by{\eqref{Hxisigmaleq}}.&~
\qedhere\end{align*}
\end{proof}

Reading out the definition of $\alpha_{Y_r,f}(s)$, we come to the following conclusion:

\begin{corollary}
\label{corollarynur}
There exists a measure $\nu_r\in\MM(A_r)$ with
\begin{equation}
\label{nur}
\nu_r(A_r) \geq \frac{1}{225^s}H_{\xi,\sigma,s}(r),
\end{equation}
and such that the measure $\rep[\nu_r]$ is $s$-admissible.
\end{corollary}

\subsubsection{Construction of a tree}
By \eqref{Pxis}, there exist $t\in(s,\delta)$ and $\varepsilon > 0$ such that
\begin{equation}
\label{Pxits}
P_\xi(s) + \varepsilon \leq \frac{t - s - \varepsilon}{c}.
\end{equation}
Now by \eqref{Pxidef}, if $\sigma$ is large enough, then there exists a sequence $r_k\tendsto k 0$ such that
\[
\frac{\log_b(H_{\xi,\sigma,s}(r_k))}{\log_b(r_k)} \leq P_\xi(s) + \varepsilon
\]
for all $k\in\N$. Rearranging yields
\[
H_{\xi,\sigma,s}(r_k) \geq r_k^{P_\xi(s) + \varepsilon}
\]
which together with (\ref{Pxits}) yields
\begin{equation}
\label{Nrks}
H_{\xi,\sigma,s}(r_k) \geq r_k^{(t - s - \varepsilon)/c}
\end{equation}
for all $k\in\N$.

Now, since $G$ is of general type, we have $h_1(\xi) \neq \xi$ for some $h_1\in G$ (Observation \ref{observationnoglobalfixedpoint}). Choose $\varepsilon_2 > 0$ small enough so that
\[
\Dist\big(B(\xi,\varepsilon_2),h_1(B(\xi,\varepsilon_2))\big) > 0.
\]
Let $\w B = B(\xi,\varepsilon_2)$. Let $h_0 = \id$.

By the Big Shadows Lemma, we may suppose that $\sigma$ is large enough so that for all $z\in X$,
\[
\Diam(\del X\butnot\Shad_z(\zero,\sigma)) < \Dist(\w B,h_1(\w B)).
\]
Thus, there exists $i = i_z = 0,1$ such that $h_i(\w B)\subset \Shad_z(\zero,\sigma)$.

Fix $C > 0$ large and $\lambda > 0$ small to be announced below.

We will define a set $T\subset G(\zero)$,\Footnote{We will think of $T$ as being a ``tree''; however, it is not a tree in the strict sense according to Definition \ref{definitiontree}.} a binary relation $\in_T\subset T\times T$, and a map $\mu:T\to[0,1]$ recursively as follows:\Footnote{You can imagine this as a program being run by an infinite computer which has infinitely many parallel processors. Each time a command of the form ``for all X, do Y'' is performed, the computation splits up, giving each X its own processor. This parallel processing is necessary, because otherwise the program would only compute one branch of the tree. \label{footnoteparallelprocessors}}
\begin{itemize}
\item[1.] Each time a new point is put into $T$, call it $x$ and go to step 3.
\item[2.] Set $\mu(\zero) = 1$ and put $\zero$ into $T$; it is the root node.
\item[3.] Recall that $x$ is a point which has just been put into $T$, with $\mu(x) > 0$ defined. We will follow one of two different possible procedures, using the following test:
\begin{itemize}
\item[i.] If $\mu(x) \geq D_x^{t - \varepsilon}$, then go to step 4.
\item[ii.] If there is no $k\in\N$ such that
\begin{equation}
\label{RkxRk}
C^{-1}r_k \leq b^{-c\dist(\zero,x)} < r_k \leq \varepsilon_2,
\end{equation}
then go to step 4.
\item[iii.] Otherwise, go to step 5.
\end{itemize}
\item[4]
\begin{itemize}
\item[i.] Label $x$ as type 1.
\item[ii.] Apply Lemma \ref{lemmaDSU} (with $s = t$) to get a finite subset $T(x)\subset G(\zero)$.
\item[iii.] For each $y\in T(x)$, set
\[
\mu(y) = \frac{D_y^t}{\sum_{z\in T(x)}D_z^t}\mu(x),
\]
and then put $y$ into $T$, setting $y\in_T x$.
\end{itemize}
\item[5]
\begin{itemize}
\item[i.] Label $x$ as type 2.
\item[ii.] Let $j = j_x \in G$ be the unique element so that $x = j_x(\zero)$.
\item[iii.] Let $i = i_x \in \{0,1\}$ be chosen so that
\[
h_i(\w B) \subset \Shad_{j^{-1}(\zero)}(\zero,\sigma),
\]
and let $h = h_x = h_{i_x}$.
\item[iv.] Let $g = g_x = j \circ h$. Note that applying $j$ to the above equation we get
\begin{equation}
\label{gBPx}
g(\w B) \subset \PP_x.
\end{equation}
\item[v.] Choose $k = k_x \in\N$ so that \eqref{RkxRk} holds.
\item[vi.] Let $\w B_k = B(h(\xi),r_k)$; by the last inequality of \eqref{RkxRk}, we have $\w B_k \subset \w B$. In particular, by \eqref{gBPx} we have
\begin{equation}
\label{gBik}
g(\w B_k) \subset \PP_x.
\end{equation}
\item[vii.] Let
\begin{align*}
\w T(x) &= g(A_{r_k})\\
\nu_x &= g[\nu_{r_k}].
\end{align*}
where $A_{r_k}$, $\nu_{r_k}$ are as in \sectionsymbol\ref{subsubsectionjarnikpreliminaries2}.
\item[viii.] For each $y\in \w T(x)$, set
\[
\mu(y) = \frac{\nu_x(y)}{\nu_x(\w T(x))}\mu(x),
\]
and put $y$ into $T$, setting $y\in_T x$.
\end{itemize}
\end{itemize}
\begin{claim}
\label{claimparty}
Let $x\in T$ be a type 2 point. Then the collection
\[
\big(3\PP_y\big)_{y \in \w T(x)}
\]
is a disjoint collection of shadows which are contained in $g(\w B_k)$, where $k = k_x$, assuming $\tau$ is sufficiently large.
\end{claim}
Note that by \eqref{gBik}, we therefore have $3\PP_y\subset\PP_x$.
\begin{proof}[Proof of Claim \ref{claimparty}]
It is equivalent to show that the collection
\[
\big(g^{-1}(3\PP_{g(y)})\big)_{y\in A_{r_k}}
\]
is a disjoint collection of sets which are contained in $\w B_k$. But by the construction of $A_{r_k}$, the collection
\[
\big(2\w\PP_y\big)_{y\in A_{r_k}}
\]
is such a collection. Thus to complete the proof it suffices to demonstrate the following:
\begin{subclaim}
\label{subclaimparty}
\begin{equation}
\label{gPgy}
g^{-1}\big(3\PP_{g(y)}\big) \subset 2\w\PP_y
\end{equation}
for every $y\in A_{r_k}$, assuming $\tau$ is sufficiently large.
\end{subclaim}
\begin{proof}
Fix such a $y$. Let $z = g^{-1}(\zero)$, and note that
\begin{align*}
\zeta := \rep(y) &\in \PP_y \subset 10\w\PP_y \subset \w B_k \subset \w B \subset g^{-1}(\PP_x)\\
&= \Shad_z(h_x(\zero),\sigma)\\
&\subset \Shad_z(\zero,\sigma + \dist(\zero,h_x(\zero))),
\end{align*}
and thus by Gromov's inequality
\[
0 \asymp_{\plus,\sigma} \lb z|\zeta\rb_\zero \gtrsim_\plus \min(\lb y|z\rb_\zero, \lb y|z\rb_\zero).
\]
On the other hand, since $\zeta \in \PP_y$, we have
\[
\lb y|\zeta\rb_\zero \asymp_{\plus,\sigma} \dist(\zero,y) \geq \lb y|z\rb_\zero,
\]
and so
\begin{equation}
\label{yz0}
\lb y|z\rb_\zero \asymp_{\plus,\sigma} 0.
\end{equation}
Now let us demonstrate \eqref{gPgy}. Fix $\eta\in g^{-1}\big(3\PP_{g(y)}\big)$; then
\[
\lb y|\eta\rb_z \asymp_{\plus,\sigma} \dist(y,z) \asymp_{\plus,\sigma} \dist(\zero,y) + \dist(\zero,z)
\]
by \eqref{yz0}. Now by (d) of Proposition \ref{propositionbasicidentities}
\[
\lb y|\eta\rb_\zero \gtrsim_\plus \lb y|\eta\rb_z - \dist(\zero,z) \asymp_{\plus,\sigma} \dist(\zero,y).
\]
Letting $\tau$ be the implied constant of this asymptotic, we have $\eta\in\w\PP_y$.
\QEDmod\end{proof}
This completes the proof of Claim \ref{claimparty}.
\end{proof}

Let $T(\infty)$ be the set of all branches through $T$, i.e. the set of all sequences $\xx = (x_n)_0^\infty$ for which
\[
\cdots\in_T x_2\in_T x_1 \in_T x_0 = \zero.
\]
We have a map
\[
\pi:T(\infty)\to\Lrsigma(G)
\]
defined as follows: If $\xx = (x_n)_0^\infty \in T(\infty)$, then the sequence
\[
(\PP_{x_n})_0^\infty
\]
is a decreasing sequence of closed sets whose diameters tend to zero. We define $\pi(\xx)$ to be the unique intersection point and we let $\limitset = \pi(T(\infty)) \subset \Lrsigma$. By the Kolmogorov consistency theorem, there exists a unique measure $\mu\in\MM(\limitset)$ satisfying
\[
\mu(\PP_x) = \mu(x)
\]
for all $x\in T$.
\begin{claim}
\[
\limitset \subset W_{c,\xi}.
\]
\end{claim}
\begin{proof}
Fix $\xx = (x_n)_0^\infty\in T(\infty)$ and let $\eta = \pi(\xx) \in \limitset$. For each $n\in\N$ let
\begin{equation}
\label{abbreviations}
\PP_n = \PP_{x_n};\;\; D_n = D_{x_n};\;\; g_n = g_{x_n}.
\end{equation}
\begin{subclaim}
\label{subclaimtype2}
Infinitely many of the points $(x_n)_0^\infty$ are type 2.
\end{subclaim}
\begin{proof}
Suppose not; suppose that the points $x_N,x_{N + 1},\ldots$ are all type 1. Then for each $n\geq N$ we have
\[
\mu(x_{n + 1}) = \frac{D_{n + 1}^t}{\sum_{z\in T(x_n)}D_z^t} \mu(x_n);
\]
by (iii) of Lemma \ref{lemmaDSU} we have
\[
\sum_{z\in T(x_n)}D_z^t \geq D_n^t
\]
and thus
\[
\frac{\mu(x_{n + 1})}{\mu(x_n)} \leq \frac{D_{n + 1}^t}{D_n^t}.
\]
Iterating gives
\[
\frac{\mu(x_n)}{\mu(x_N)} \leq \frac{D_n^t}{D_N^t}.
\]
It follows that for all $n$ sufficiently large say $n\geq N_2\geq N$, we have
\begin{equation}
\label{muxn}
\mu(x_n) < D_n^{t - \varepsilon}.
\end{equation}
Let $k\in\N$ be large enough so that
\[
r_k < \min\big(b^{-c\dist(\zero,x_{N_2})},\varepsilon_2\big).
\]
Let $n > N_2$ be minimal so that
\[
b^{-c\dist(\zero,x_n)} \leq r_k;
\]
by the Diameter of Shadows Lemma and (ii) of Lemma \ref{lemmaDSU} we have
\[
b^{-c\dist(\zero,x_n)} \asymp_{\times,\sigma} r_k.
\]
Let $C > 0$ be the implied constant of this asymptotic; then \eqref{RkxRk} holds. But we also have \eqref{muxn}; it follows that $x_n$ is a type 2 point, contradicting our hypothesis that the points $x_N,x_{N + 1},\ldots$ are all type 1.
\QEDmod\end{proof}
The idea now is to associate each type 2 point in the sequence $(x_n)_0^\infty$ to a good approximation of $\eta$. Fix $n\in\N$ so that $x_n$ is type 2. Write $i_n = i_{x_n}$ and $k_n = k_{x_n}$.

Now by Claim \ref{claimparty} we have
\[
\eta\in \PP_{x_{n + 1}} \subset g_n(\w B_{k_n}),
\]
i.e.
\[
\Dist(\xi,g_n^{-1}(\eta)) \leq r_{k_n}.
\]
Note that $g_n(\xi)\in\PP_{x_n}$ by \eqref{gBPx}, and so
\[
\xi,g_n^{-1}(\eta) \in g_n^{-1}(\PP_{x_n}) \subset \Shad_{g_n^{-1}(\zero)}(\zero,\sigma + \dist(\zero,h_1(\zero)))
\]
Thus by (ii) of the Bounded Distortion Lemma we have
\begin{align*}
\Dist(g_n(\xi),\eta) &\asymp_{\times,\sigma} b^{-\dist(\zero,g_n(\zero))}\Dist(\xi,g_n^{-1}\eta)\\
&\leq_{\phantom{\times,\sigma}} b^{-\dist(\zero,g_n(\zero))} r_{k_n}\\
&\asymp_{\times,\sigma} b^{-\dist(\zero,g_n(\zero))}b^{-c\dist(\zero,x_n)}\\
&\asymp_{\times\phantom{,\sigma}} b^{-(c + 1)\dist(\zero,g_n(\zero))}.
\end{align*}
Thus
\[
\omega_\xi(\eta) \geq \limsup_{\substack{n\to\infty \\ \text{$x_n$ type 2}}} \frac{-\log_b\Dist(g_n(\xi),\eta)}{\dist(\zero,g_n(\zero))} \geq 1 + c,
\]
i.e. $\eta\in W_{c,\xi}$.
\end{proof}

It is left to show that $\HD(\limitset)\geq s$.
\begin{claim}
\label{claimsemiregular}
$\mu(B(\eta,r)) \lesssim_{\times,\sigma} r^s$ for all $r > 0$ and for all $\eta\in \limitset$.
\end{claim}
\begin{proof}
Fix $\xx = (x_n)_0^\infty \in T(\infty)$ so that $\eta = \pi(\xx)$. As before we use the abbreviations \eqref{abbreviations}. Let $n = n(\eta,r)\in\N$ be maximal so that $D_n\geq r$. Then
\[
B(\eta,r) \subset B(\eta,D_n)\subset 3\PP_n.
\]
Since $3\PP_n\cap\limitset \subset \PP_n$, we have $B(\eta,r)\cap\limitset \subset \PP_n$.

On the other hand, the maximality of $n$ implies that
\begin{equation}
\label{rDn1}
r > D_{n + 1}.
\end{equation}
Now we divide into cases depending on whether $x = x_n$ is a type 1 point or a type 2 point.

\begin{itemize}
\item[Case 1:] $x$ is type 2. In this case, let
\[
i = i_n = i_{x_n}, \;\; k = k_n = k_{x_n},\text{ and }g = g_n = g_{x_n}.
\]
Then
\[
\mu(B(\eta,r))
\leq \sum_{\substack{y \in \w T(x) \\ \PP_y \cap B(\eta,r)\neq\smallemptyset}} \mu(y)\\
= \sum_{\substack{y \in A_{r_k} \\ \PP_{g(y)} \cap B(\eta,r)\neq\smallemptyset}}\frac{\nu_{r_k}(y)}{\nu_{r_k}(A_{r_k})}\mu(x).
\]
On the other hand, since $x$ is type 2 we have
\[
\mu(x) \leq D_x^{t - \varepsilon} \asymp_{\times,\sigma} b^{-(t - \varepsilon)\dist(\zero,x)};
\]
moreover, by \eqref{nur} and \eqref{Nrks} we have
\[
\nu_{r_k}(A_{r_l}) \gtrsim_\times H_{\xi,\sigma,s}(r_k) \geq r_k^{(t - s - \varepsilon)/c} \asymp_{\times,\sigma} b^{-(t - s - \varepsilon)\dist(\zero,x)}.
\]
Combining gives
\begin{equation}
\label{whereweleftoff1}
\mu(B(\eta,r)) \lesssim_{\times,\sigma} b^{-s\dist(\zero,x)}\sum_{\substack{y \in A_{r_k} \\ \PP_{g(y)} \cap B(\eta,r)\neq\smallemptyset}}\nu_{r_k}(y).
\end{equation}
\begin{subclaim}
For every $y\in A_{r_k}$ for which $\PP_{g(y)}\cap B(\eta,r)\neq\emptyset$, we have
\[
g(\PP_y) \subset B(\eta,C_1 r),
\]
where $C_1 > 0$ is independent of $n$.
\end{subclaim}
\begin{proof}
We claim first that $r\geq D_{g(y)}$.
\begin{itemize}
\item[Case 1:] $g(y) = x_{n + 1}$. Then
\[
r > D_{n + 1} = D_{g(y)}.
\]
\item[Case 2:] $g(y)\neq x_{n + 1}$. Then
\[
r\geq \Dist(\PP_{g(y)},\eta)\geq \Dist(\PP_{g(y)},\del X\butnot 3\PP_{g(y)}) \geq D_{g(y)}.
\]
\end{itemize}
\begin{subclaim}
For all $\zeta\in \PP_y$ we have
\begin{equation}
\label{ogzeta}
\lb \zero|g(\zeta)\rb_{g(y)} \asymp_{\plus,\sigma} 0.
\end{equation}
\end{subclaim}
\begin{proof}
Since
\[
\zeta\in\PP_y \subset \w B_k \subset \w B \subset g^{-1}(\PP_x),
\]
we have
\[
\lb g^{-1}(\zero)|\zeta\rb_\zero = \lb \zero|g(\zeta)\rb_{g(\zero)} \asymp_\plus \lb \zero|g(\zeta)\rb_x \asymp_{\plus,\sigma} 0.
\]
On the other hand, since $\zeta\in\PP_y$
\[
\lb y|\zeta\rb_\zero \asymp_{\plus,\sigma} \dist(\zero,y) \tendsto n \infty.
\]
Gromov's inequality yields
\[
\lb g^{-1}(\zero)|y\rb_\zero \asymp_{\plus,\sigma} 0.
\]
Thus
\[
\lb \zero|g^{-1}(\zero)\rb_y \asymp_{\plus,\sigma} \dist(\zero,y) \tendsto n \infty.
\]
On the other hand
\[
\lb \zero|\zeta\rb_y \asymp_{\plus,\sigma} 0;
\]
Gromov's inequality yields
\[
\lb \zero|g(\zeta)\rb_{g(y)} = \lb g^{-1}(\zero)|\zeta\rb_y \asymp_{\plus,\sigma} 0.
\]
\QEDmod\end{proof}
Letting $\w\sigma \geq \sigma$ be the implied constant of the asymptotic \eqref{ogzeta}, we have $g(\PP_y)\subset\Shad(g(y),\w\sigma)$. Thus
\[
\Dist(\zeta,\eta) \leq r + \Diam(\Shad(g(y),\w\sigma)) \asymp_\times r + D_{g(y)} \lesssim_\times r,
\]
completing the proof.
\QEDmod\end{proof}
Thus
\begin{equation}
\label{whereweleftoff2}
\sum_{\substack{y \in A_{r_k} \\ \PP_{g(y)} \cap B(\eta,r)\neq\smallemptyset}}\nu_{r_k}(y)
\leq \sum_{\substack{y \in A_{r_k} \\ g(\PP_y) \subset B(\eta,C_1 r)}}\nu_{r_k}(y)
\leq \rep[\nu_{r_k}](g^{-1}(B(\eta,C_1 r))).
\end{equation}
Now,
\begin{align*}
\w D_{g^{-1}(x_{n + 1})}
&\asymp_{\times,\sigma} b^{-\dist(\zero,g^{-1}(x_{n + 1}))}
\asymp_\times b^{-\dist(x_n,x_{n + 1})}\noreason\\
&\asymp_{\times,\sigma} b^{\dist(\zero,x_n) - \dist(\zero,x_{n + 1})} \by{the Intersecting Shadows Lemma}\\
&\asymp_{\times,\sigma} b^{\dist(\zero,x_n)}D_{n + 1} \by{the Diameter of Shadows Lemma}\\
&\leq_{\phantom{\times,\sigma}} b^{\dist(\zero,x_n)}r. \by{\eqref{rDn1}}
\end{align*}
Let $C_2 > 0$ be the implied constant of this asymptotic; then
\begin{equation}
\label{C2def}
\w D_{g^{-1}(x_{n + 1})} \leq C_2 b^{\dist(\zero,x)}r.
\end{equation}

\begin{subclaim}
\label{subclaimC2}
\[
g^{-1}\big(B(\eta,C_1 r)\big) \subset U := B(\rep(g^{-1}(x_{n + 1})),C_3 b^{\dist(\zero,x)}r).
\]
for some $C_3 \geq 2C_2$ independent of $n$.
\end{subclaim}
\begin{proof}
Fix $\zeta\in g^{-1}(B(\eta,C_1 r))$. Then by (d) of Proposition \ref{propositionbasicidentities}, we have
\[
\Dist(g^{-1}(\eta),\zeta) \lesssim_{\times,\sigma} b^{\dist(\zero,g(\zero))}\Dist(\eta,g(\zeta)) \leq C_1 b^{\dist(\zero,g(\zero))}r \asymp_{\times,\sigma} b^{\dist(\zero,x)}r.
\]
On the other hand, since $\eta\in \PP_{x_{n + 1}}$, Subclaim \ref{subclaimparty} shows that
\[
g^{-1}(\eta)\in 2\w\PP_{g^{-1}(x_{n + 1})}
\]
and thus
\[
\Dist\big(g^{-1}(\eta),\rep(g^{-1}(x_{n + 1}))\big)
\leq_\pt 2\w D_{g^{-1}(x_{n + 1})}
\leq 2C_2 b^{\dist(\zero,x)}r.
\]
This completes the proof.
\QEDmod\end{proof}
Combining \eqref{whereweleftoff1}, \eqref{whereweleftoff2}, and Subclaim \ref{subclaimC2} gives
\begin{equation}
\label{whereweleftoff3}
\mu(B(\eta,r)) \lesssim_{\times,\sigma} b^{-s\dist(\zero,x)}\rep[\nu_{r_k}](U).
\end{equation}
Now by \eqref{C2def} we have
\[
f(\rep(g^{-1}(x_{n + 1}))) = D_{g^{-1}(x_{n + 1})} \leq 2\w D_{g^{-1}(x_{n + 1})} \leq 2C_2 b^{\dist(\zero,x)}r \leq C_3 b^{\dist(\zero,x)},
\]
and so the ball $U$ is admissible. On the other hand, the measure $\rep[\nu_{r_k}]$ is $s$-admissible; it follows that
\[
\rep[\nu_{r_k}](U) \leq \big(C_3 b^{\dist(\zero,x)}r\big)^s.
\]
Combining with \eqref{whereweleftoff3}, we have
\[
\mu(B(\eta,r)) \lesssim_\times b^{-s\dist(\zero,x)}\big(b^{\dist(\zero,x)}r\big)^s = r^s.
\]

\item[Case 2:] $x$ is type 1. In this case, by \eqref{rDn1} and (ii) of Lemma \ref{lemmaDSU} we have
\[
r \geq D_{n + 1} \asymp_\times D_n.
\]
On the other hand, since $B(\eta,r)\cap\limitset\subset \PP_n$, we have
\[
\mu(B(\eta,r)) \leq \mu(x_n).
\]
Thus to complete the proof of the claim it suffices to show
\begin{equation}
\label{muxn2}
\mu(x_n) \lesssim_\times D_n^s.
\end{equation}
Let $m < n$ be the largest number such that $x_m$ is a type 2 point. (If no such number exists, let $m = 0$.) Let
\[
\w r = \min(2D_{m + 1},D_m) > D_{m + 1};
\]
then $n(\eta,\w r\ew\ew) = m$. Thus since $x_m$ is type 2, the ball $B(\eta,\w r\ew\ew)$ falls into Case 1: So by our earlier argument, we have
\begin{equation}
\label{muetar}
\mu(B(\eta,\w r\ew\ew)) \lesssim_\times \w r\ew\ew^s.
\end{equation}
(If $m = 0$, then \eqref{muetar} holds just because the right hand side is asymptotic to $1$.) On the other hand, $\PP_{m + 1}\subset B(\eta,D_{m + 1})\subset B(\eta,\w r\ew\ew)$, and so
\[
\mu(x_{m + 1}) \leq \mu(B(\eta,\w r\ew\ew)) \lesssim_\times D_{m + 1}^s,
\]
i.e. \eqref{muxn2} is satisfied for $n = m + 1$. On the other hand, since the points $x_{m + 1},\ldots,x_{n - 1}$ are all type 1, the argument used in the proof of Subclaim \ref{subclaimtype2} shows that
\[
\frac{\mu(x_n)}{\mu(x_{m + 1})} \leq \frac{D_n^t}{D_{m + 1}^t} \leq \frac{D_n^s}{D_{m + 1}^s},
\]
since $t > s$ by \eqref{Pxits}. This demonstrates \eqref{muxn2}.
\end{itemize}
\end{proof}

Now by the mass distribution principle we have $\HD(\limitset)\geq s$, completing the proof.

\draftnewpage

\subsection{Special cases} \label{subsectionspecialcasesjarnik}
In this subsection we prove the assertions made in \sectionsymbol\ref{subsubsectionspecialcasesjarnik} concerning special cases of Theorem \ref{theoremjarnikbesicovitch}, except for Theorem \ref{theoremboundedparabolicjarnik} which will be proven in Subsection \ref{subsectionboundedparabolicjarnik} below.

\begin{example}
\label{exampleradialjarnik}
If $\xi\in\Lr$ then $P_\xi(s) = s$. Thus in this case (\ref{jarnikbesicovitch}) reduces to
\begin{equation}
\label{radialjarnik}
\HD(W_{c,\xi}) = \frac{\delta}{c + 1},
\end{equation}
i.e. the inequality of (\ref{jarnikbesicovitch}) is an equality.
\end{example}
\begin{proof}
In fact, this is a special case of Example \ref{examplehorosphericaljarnik} below, since if $\xi\in\Lr$ then $\beta(\xi) = 1$.
\end{proof}
\begin{example}
\label{examplehorosphericaljarnik}
If $\xi$ is a horospherical limit point of $G$ then $P_\xi(s) \leq 2s$. Thus in this case we have
\[
\frac{\delta}{2c + 1}\leq \HD(W_{c,\xi}) \leq \frac{\delta}{c + 1}.
\]
More generally, let $\lb\cdot|\cdot\rb$ denote the Gromov product. If we let
\begin{equation}
\label{betadef}
\beta = \beta(\xi) := \limsup_{\substack{g\in G \\ \dist(\zero,g(\zero))\to\infty}}\frac{\lb g(\zero)|\xi\rb_\zero}{\dist(\zero,g(\zero))}
\end{equation}
then $P_\xi(s)\leq \beta^{-1}s$ and so
\[
\frac{\delta}{\beta^{-1} c + 1}\leq \HD(W_{c,\xi}) \leq \frac{\delta}{c + 1}.
\]
In particular, if $\beta > 0$ then $\HD(\VWA_\xi) = \delta$.
\end{example}
\begin{proof}
Fix $s\in(0,\delta)$. Since $\HD(\Lr) = \delta > s$, there exists $\sigma > 0$ so that $\HD(\Lrsigma) > s$. Then $\HH_\infty^s(\Lrsigma) > 0$. Let $\tau > 0$ be guaranteed by Lemma \ref{lemmanearinvariance}.

Let $h_1,h_2\in G$ be as in the definition of general type. Then there exists $\varepsilon > 0$ so that for all $\xi\in\del X$,
\begin{equation}
\label{minimumtranslationdistance}
\max_{i = 1}^2 \Dist(\xi,h_i(\xi)) > \varepsilon.
\end{equation}
Let $\w\sigma > 0$ be given by the Big Shadows Lemma.

Now fix $g\in G$. By the Big Shadows Lemma, we have
\[
\Diam(\del X\butnot\Shad_{g^{-1}(\zero)}(\zero,\w\sigma)) \leq \varepsilon,
\]
and thus by \eqref{minimumtranslationdistance} we have
\[
\del X = S_g \cup h_1^{-1}(S_g)\cup h_2^{-1}(S_g), \text{ where } S_g = \Shad_{g^{-1}(\zero)}(\zero,\w\sigma).
\]
Let $h_0 = \id$. Then there exists $i = i_g = 0,1,2$ such that
\[
\HH_\infty^s(\Lrsigma\cap h_i^{-1}(S_g)) \geq \frac{1}{3}\HH_\infty^s(\Lrsigma) \asymp_{\times,\sigma} 1.
\]
Thus
\begin{align*}
b^{-s\dist(\zero,g(\zero))}
&\asymp_{\times\phantom{,\sigma}} b^{-s\dist(\zero,g\circ h_i(\zero))}\\
&\lesssim_{\times,\sigma} \HH_\infty^s\left(g\circ h_i\big(\Lrsigma\cap h_i^{-1}(S_g)\big)\right)\\
&\leq_\ptsigma \HH_\infty^s(\Lrtau\cap g(S_g)) \by{Lemma \ref{lemmanearinvariance}}\\
&=_\ptsigma \HH_\infty^s(\Lrtau\cap \Shad(g(\zero),\w\sigma)).
\end{align*}
Now fix $\eta\in\Shad(g(\zero),\w\sigma)$. We have
\begin{align*}
\lb \eta|\xi\rb_\zero &\gtrsim_{\plus\phantom{,\sigma}} \min(\lb g(\zero)|\eta\rb_\zero,\lb g(\zero)|\xi\rb_\zero)\\
&\asymp_{\plus,\w\sigma} \min(\dist(\zero,g(\zero)),\lb g(\zero)|\xi\rb_\zero) \since{$\eta\in\Shad(g(\zero),\w\sigma)$}\\
&=_{\phantom{\plus,\sigma}} \lb g(\zero)|\xi\rb_\zero,
\end{align*}
i.e.
\[
\Shad(g(\zero),\w\sigma) \subset B(\xi,C b^{-\lb g(\zero)|\xi\rb_\zero})
\]
for some $C > 0$ independent of $G$. Thus
\[
b^{-s\dist(\zero,g(\zero))}\lesssim_{\times,\sigma} \HH_\infty^s\big(\Lrtau\cap B(\xi,C b^{-\lb g(\zero),\xi\rb_\zero})\big)
= H_{\xi,\tau,s}(C b^{-\lb g(\zero),\xi\rb_\zero}).
\]
Let $(g_n)_1^\infty$ be a sequence so that
\[
\frac{\lb g_n(\zero)|\xi\rb_\zero}{\dist(\zero,g_n(\zero))} \tendsto n \beta \text{ and }\dist(\zero,g_n(\zero))\tendsto n \infty.
\]
In particular, $\lb g_n(\zero)|\xi\rb_\zero \tendsto n \infty$, and thus
\begin{align*}
P_\xi(s)
&\leq \liminf_{r\to 0} \frac{\log_b(H_{\xi,\tau,s}(r))}{\log_b(r)}\\
&\leq \liminf_{n\to\infty}\frac{\log_b(H_{\xi,\tau,s}(C b^{-\lb g_n(\zero),\xi\rb_\zero}))}{\log_b(C b^{-\lb g_n(\zero),\xi\rb_\zero})}\\
&\leq \liminf_{n\to\infty}\frac{-s\dist(\zero,g_n(\zero))}{-\lb g_n(\zero),\xi\rb_\zero}
= \beta^{-1}s.
\end{align*}
Finally, if $\xi$ is a horospherical limit point, let $(g_n(\zero))_1^\infty$ be a sequence tending horospherically to $\xi$; then
\begin{align*}
\beta(\xi) &\geq \limsup_{n\to\infty}\frac{\lb g_n(\zero)|\xi\rb_\zero}{\dist(\zero,g_n(\zero))}\\
&= \limsup_{n\to\infty}\frac{1}{2}\frac{\dist(\zero,g_n(\zero)) + \busemann_\xi(\zero,g_n(\zero))}{\dist(\zero,g_n(\zero))} \by{(g) of Proposition \ref{propositionbasicidentities}}\\
&\geq \limsup_{n\to\infty}\frac{1}{2}\frac{\dist(\zero,g_n(\zero))}{\dist(\zero,g_n(\zero))} \since{$x_n\tendsto n \xi$ horospherically}\\
&= 1/2.
\end{align*}

\end{proof}

\begin{proposition}
\label{propositionVWAHDzero}
There exists a (discrete) nonelementary Fuchsian group $G$ and a point $\xi\in\Lambda$ such that $\HD(\VWA_\xi) = 0$.
\end{proposition}
\begin{proof}
Let $\Dist_\euc$ denote the Euclidean metric on $\C$, and let $\Dist_\sph$ denote the spherical metric on $\del\H^2$, i.e. the metric $|\dee z|/(|z|^2 + 1)$. For each $n\in\N$, let $a_n = 2^{n!}$ and let $g_n\in\Isom(\H^2)$ be a M\"obius transformation such that
\[
g_n\left(\what\C\butnot B_\euc(-a_n,1/a_{n + 2})\right) = B_\euc(a_n,1/a_{n + 2}).
\]
The group $G$ generated by the sequence $(g_n)_1^\infty$ is a Schottky group and is therefore discrete. Clearly, $\infty$ is a limit point of $G$.

We claim that $\HD(\VWA_\infty) = 0$. By \eqref{HDVWAxi}, it suffices to show that
\[
P_\infty(s) = +\infty \all s\in(0,\delta).
\]
Fix $s\in(0,\delta)$ and $0 < r \leq 1$. Computation shows that
\[
B_\sph(\infty,r) \subset \what\C\butnot B_\euc(0,1/(2r)).
\]
There exists $n = n_r\in\N$ so that
\[
a_n \leq \frac{1}{2r} < a_{n + 1}.
\]
Then
\[
B_\sph(\infty,r)\cap\Lambda\subset \{\infty\}\cup\bigcup_{m\geq n} \big[B_\euc(a_m,1/a_{m + 2})\cup B_\euc(-a_m,1/a_{m + 2})\big]
\]
and thus for all $\sigma > 0$
\begin{align*}
H_{\xi,\sigma,s}(r)
&\leq \sum_{m\geq n}\big[\Diam_\sph(B_\euc(a_m,1/a_{m + 2})) + \Diam_\sph(B_\euc(-a_m,1/a_{m + 2}))\big]\\
&\asymp_\times \sum_{m\geq n}\frac{1}{a_{m + 2} a_m^2}\\
&\lesssim_\times \frac{1}{a_{n + 2}}\\
&= \frac{1}{a_{n + 1}^{n + 2}} \leq (2r)^{n_r + 2}.
\end{align*}
Thus
\[
P_\infty(s) \geq \liminf_{r\to 0} \frac{\log((2r)^{n_r + 2})}{\log(r)} = \liminf_{r\to 0}[n_r + 2] = +\infty.
\]
\end{proof}

\draftnewpage
\subsection{Proof of Theorem \ref{theoremboundedparabolicjarnik}}
\label{subsectionboundedparabolicjarnik}
In this subsection we prove Theorem \ref{theoremboundedparabolicjarnik}. The proof requires several lemmas.

\begin{theorem}
\label{theoremboundedparabolicjarnik}
Let $(X,\dist,\zero,b,G)$ be as in \sectionsymbol\ref{standingassumptions}, and let $\xi$ be a bounded parabolic point of $G$. Let $\delta$ and $\delta_\xi$ be the Poincar\'e exponents of $G$ and $G_\xi$, respectively. Then for all $s\in(0,\delta)$
\begin{equation}
\label{prevelanihill}
P_\xi(s) = \max(s,2s - 2\delta_\xi)
= \begin{cases}
s &\text{if } s \geq 2\delta_\xi\\
2s - 2\delta_\xi &\text{if } s\leq 2\delta_\xi
\end{cases}.
\end{equation}
Thus for each $c > 0$
\begin{equation}
\label{velanihillgeneral}
\HD(W_{c,\xi}) = \begin{cases}
\delta/(c + 1) &\text{if } c + 1 \geq \delta/(2\delta_\xi)\\
\displaystyle\frac{\delta + 2\delta_\xi c}{2c + 1} &\text{if } c + 1\leq \delta/(2\delta_\xi)
\end{cases}.
\end{equation}
\end{theorem}

In this subsection, $\Dist_\euc$ denotes the Euclidean metametric $\Dist_{\xi,\zero}$, and $\Dist_\sph$ denotes the spherical metametric $\Dist_\zero$.

\begin{definition}
Let $G_\xi$ be the stabilizer of $\xi$ relative to $G$. We define the \emph{orbital counting function} of $G_\xi$ with respect to the metametric $\Dist_\euc$ to be the function $f_\xi:\Rplus\to\N$ defined by
\begin{equation}
\label{orbitalcountingfunction}
f_\xi(R) := \#\{g\in G_\xi:\Dist_\euc(\zero,g(\zero))\leq R\}.
\end{equation}
Note that $f_\xi$ is different from the orbital counting function of $G_\xi$ with respect to the hyperbolic metric $\dist$, which we denote by $f_{G_\xi}$.
\end{definition}

Before proceeding further, let us remark on the function \eqref{orbitalcountingfunction}. In the case where $X$ is a real, complex, or quaternionic hyperbolic space, $f_\xi$ satisfies a power law, i.e. $f_\xi(R)\asymp_\times R^{2\delta_\xi}$, where $\delta_\xi$ is the Poincar\'e exponent of $G_\xi$ \cite[Lemma 3.5]{Newberger} (see also \cite[Proposition 2.10]{Schapira}). 
However, this is not the case for more general spaces, including the infinite-dimensional space $\H^\infty$ and proper $\R$-trees; indeed, there are essentially no restrictions on the orbital counting function of a parabolic group acting on such spaces. See Appendix \ref{appendixorbitalcounting} for details.

Let us now give a rough outline of the proof of Theorem \ref{theoremboundedparabolicjarnik}. Let
\[
H_{\xi,\sigma,s}^*(R) = \HH_{\sph,\infty}^s(\Lrsigma\butnot B_\euc(\zero,R)).
\]
Then by Lemma \ref{lemmaeuclideancomparison}, \eqref{Pxidef} can be rewritten as
\begin{equation}
\label{Pxidef2}
P_\xi(s) = \lim_{\sigma\to\infty}\liminf_{R\to\infty}\frac{\log(H_{\xi,\sigma,s}^*(R))}{\log(1/R)}.
\end{equation}
So to compute $P_\xi(s)$, we will compute $H_{\xi,\sigma,s}^*(R)$ for all $R\geq 1$ sufficiently large. Since the spherical metric is locally comparable to the Euclidean metric via \eqref{GMVT2}, we will compute $H_{\xi,\sigma,s}^*(R)$ by first computing $\HH_{\euc,\infty}^s$ of appropriately chosen subsets of $\Lrsigma\butnot B_\euc(\zero,R)$. Specifically, we will consider the sets
\[
S_{R,\alpha} := \bigcup_{\substack{g\in G_\xi \\ \alpha R < \Dist_\euc(\zero,g(\zero)) \leq R}}g(S\cap\Lrsigma),
\]
where $R\geq 1$ and $\alpha\in[0,1]$. Our first step is to estimate $\HH_{\euc,\infty}^s(S_{R,\alpha})$:

\begin{lemma}
\label{lemmaHRlocal}
Let $S$ be a $\xi$-bounded set containing $\zero$, and let $\sigma > 0$. Fix $R\geq C_0\Diam_\euc(S)$ and $\alpha\in[0,1]$. Then
\begin{itemize}
\item[(i)]
\begin{equation}
\label{HRlocal1}
\HH_{\euc,\infty}^s(S_{R,0}) \lesssim_{\times,S} \frac{f_\xi((C_0 + 1)R)}{\sup_{Q\in[1,R]}(f_\xi(Q)/Q^s)}.
\end{equation}
\item[(ii)] If
\begin{equation}
\label{HSbig}
\HH_{\euc,\infty}^s(S\cap\Lrsigma) > 0,
\end{equation}
then for all $\alpha\in[0,1]$
\begin{equation}
\label{HRlocal2}
\HH_{\euc,\infty}^s(S_{R,\alpha}) \gtrsim_{\times,S,\sigma}  \frac{f_\xi(R) - f_\xi(\alpha R)}{\sup_{Q\in[1,12C_0 R]}(f_\xi(Q)/Q^s)}.
\end{equation}
\end{itemize}
\end{lemma}
\begin{remark}
In the proof of Lemma \ref{lemmaHRlocal}, we will not indicate the dependence of implied constants on the $\xi$-bounded set $S$.
\end{remark}

\begin{proof}[Proof of \textup{(i)}]
Fix $Q\in[1,R]$, and let $A_Q$ be a maximal $2C_0 Q$-separated subset of $B_\euc(\zero,R)\cap G_\xi(\zero)$. Then the sets $\big(B_\euc(x,C_0 Q)\cap G_\xi(\zero)\big)_{x\in A_Q}$ are disjoint, so
\begin{align*}
\#(B_\euc(\zero,R + C_0 Q)\cap G_\xi(\zero))
&\geq \sum_{x\in A_Q}\#(B_\euc(x,C_0 Q)\cap G_\xi(\zero))\\
&\geq \sum_{x\in A_Q}\#(B_\euc(\zero,Q)\cap G_\xi(\zero)) \by{Observation \ref{observationuniformlyLipschitz}}\\
&= f_\xi(Q)\#(A_Q),
\end{align*}
and thus
\begin{align*}
\#(A_Q) &\leq \frac{f_\xi(R + C_0 Q)}{f_\xi(Q)} \leq \frac{f_\xi((C_0 + 1)R)}{f_\xi(Q)}\cdot
\end{align*}
On the other hand, the maximality of $A_Q$ implies
\[
B_\euc(\zero,R)\cap G_\xi(\zero) \subset \bigcup_{x\in A_Q}B_\euc(x,2C_0 Q)
\]
and thus by Observation \ref{observationuniformlyLipschitz}
\[
S_R := S_{R,0} \subset \bigcup_{x\in B_\euc(\zero,R)\cap G_\xi(\zero)}B_\euc(x,C_0\Diam_\euc(S))
\subset \bigcup_{x\in A_Q} B_\euc(x,2C_0 Q + C_0\Diam_\euc(S)),
\]
i.e. $\big(B_\euc(x,2C_0 Q + C_0\Diam_\euc(S))\big)_{x\in A_Q}$ is a cover of $S_R$. Thus
\begin{align*}
\HH_{\euc,\infty}^s(S_R)
&\leq_\pt \sum_{x\in A_Q}\big(2C_0 Q + C_0\Diam_\euc(S)\big)^s\\
&=_\pt \#(A_Q)\big(2C_0 Q + C_0\Diam_\euc(S)\big)^s\\
&\lesssim_\times \frac{f_\xi((C_0 + 1)R)}{f_\xi(Q)} Q^s,
\end{align*}
since $Q\geq 1$ by assumption, and $C_0$ and $\Diam_\euc(S)$ are constants. Thus
\[
R^{-2s}\HH_{\euc,\infty}^s(S_R) \lesssim_\times \frac{f_\xi((C_0 + 1)R)/((C_0 + 1)R)^{2s}}{f_\xi(Q)/Q^s},
\]
and taking the infimum over all $Q\in[1,R]$ finishes the proof.
\end{proof}
\begin{proof}[Proof of \textup{(ii)}]
Let $\CC$ satisfy $\bigcup(\CC) = S_{R,\alpha}$.
\begin{claim}
\label{claimgSAnonempty}
For all $A\subset\EE_\xi$,
\[
\#\{g\in G_\xi:g(S)\cap A\neq\emptyset\} \leq f_\xi(C_0\Diam_\euc(A) + 2C_0^2\Diam_\euc(S)).
\]
\end{claim}
\begin{subproof}
Suppose that the left hand side is nonzero; then there exists $g_0\in G_\xi$ such that $g_0(S)\cap A\neq\emptyset$. Then
\begin{align*}
\#\{g\in G_\xi:g(S)\cap A\neq\emptyset\}
&\leq \#\Big\{g\in G_\xi: \Dist_\euc(g_0(\zero),g(\zero)) \leq \Diam_\euc(A) + 2C_0\Diam_\euc(S)\Big\}\\
&\leq f_\xi(C_0\Diam_\euc(A) + 2C_0^2\Diam_\euc(S)).
\end{align*}
\end{subproof}
\noindent Let
\[
\CC_1 = \{A\in\CC:\Diam_\euc(A)\leq C_0\Diam_\euc(S)\},
\]
and let $\CC_2 = \CC\butnot\CC_1$. Then by Claim \ref{claimgSAnonempty}, we have
\[
A\in\CC_1 \Rightarrow \#\{g\in G_\xi:g(S)\cap A\neq\emptyset\} \leq f_\xi(3C_0^2\Diam_\euc(S)) \asymp_\times 1,
\]
and thus
\begin{align*}
\sum_{A\in\CC_1}\Diam_\euc^s(A)
&\gtrsim_{\times\phantom{,\sigma}} \sum_{g\in G_\xi}\sum_{\substack{A\in\CC_1 \\ g(S)\cap A\neq\smallemptyset}}\Diam_\euc^s(A)\\
&\geq_{\phantom{\times,\sigma}} \sum_{\substack{g\in G_\xi \\ g(S\cap\Lrsigma)\subset \bigcup(\CC_1)}}\HH_{\euc,\infty}^s(g(S\cap\Lrsigma))\\
&\asymp_{\times\phantom{,\sigma}} \sum_{\substack{g\in G_\xi \\ g(S\cap\Lrsigma)\subset \bigcup(\CC_1)}}\HH_{\euc,\infty}^s(S\cap\Lrsigma) \by{Observation \ref{observationuniformlyLipschitz}}\\
&\asymp_{\times,\sigma} \#\left\{g\in G_\xi: g(S\cap\Lrsigma)\subset\bigcup(\CC_1)\right\}. \by{\eqref{HSbig}}
\end{align*}
On the other hand, Claim \ref{claimgSAnonempty} implies that for all $A\in\CC_2$, we have
\[
\#\{g\in G_\xi:g(S)\cap A\neq\emptyset\}
\leq f_\xi(3C_0\Diam_\euc(A))
\]
Moreover, since $\bigcup(\CC) = S_R$, we have $A\subset S_R$ and so
\[
3C_0\Diam_\euc(A) \leq 3C_0\Diam_\euc(S_R) \leq 3C_0(2R + 2C_0\Diam_\euc(S)) \leq 12C_0 R,
\]
since $R\geq C_0\Diam_\euc(S)$. Thus
\begin{align*}
\#\{g\in G_\xi:g(S)\cap A\neq\emptyset\}
&\leq_\pt f_\xi(3C_0\Diam_\euc(A))\\
&\leq_\pt (3C_0 \Diam_\euc(A))^s\sup_{Q\in[1,12C_0 R]}(f_\xi(Q)/Q^s)\\
&\asymp_\times \Diam_\euc^s(A)\sup_{Q\in[1,12C_0 R]}(f_\xi(Q)/Q^s).
\end{align*}
Thus
\begin{align*}
f_\xi(R) - f_\xi(\alpha R)
&\leq_{\phantom{\times,\sigma}} \#\left\{g\in G_\xi: g(S\cap\Lrsigma)\subset\bigcup(\CC)\right\} \since{$\bigcup(\CC) = S_{R,\alpha}$}\\
&\leq_{\phantom{\times,\sigma}} \#\left\{g\in G_\xi: g(S\cap\Lrsigma)\subset\bigcup(\CC_1)\right\} + \sum_{A\in\CC_2}\#\{g\in G_\xi:g(S)\cap A\neq\emptyset\} \noreason \\
&\lesssim_{\times,\sigma} \sum_{A\in\CC_1}\Diam_\euc^s(A) + \sum_{A\in\CC_2}\Diam_\euc^s(A)\sup_{Q\in[1,12C_0 R]}(f_\xi(Q)/Q^s) \noreason \\
&\lesssim_{\times\phantom{,\sigma}} \sum_{A\in\CC}\Diam_\euc^s(A)\sup_{Q\in[1,12C_0 R]}(f_\xi(Q)/Q^s).
\end{align*}
In the last inequality, we have used the inequality
\[
1 \asymp_\times f_\xi(1)/1^2 \leq \sup_{Q\in[1,12C_0 R]}(f_\xi(Q)/Q^s)\cdot
\]
Taking the infimum over all collections $\CC$ such that $\bigcup(\CC) = S_{R,\alpha}$ gives
\[
f_\xi(R) - f_\xi(\alpha R) \lesssim_{\times,\sigma} \HH_{\euc,\infty}^s(S_{R,\alpha}) \sup_{Q\in[1,12C_0 R]}(f_\xi(Q)/Q^s),
\]
and rearranging finishes the proof.
\end{proof}

Having estimated $\HH_{\euc,\infty}^s(S_{R,\alpha})$, we proceed to estimate $H_{\xi,\sigma,s}^*(R)$:

\begin{lemma}
\label{lemmaHRasymp}
There exist $C_5,\alpha > 0$ such that for all $\sigma > 0$ sufficiently large, for all $R > 0$ sufficiently large, and for all $s\geq 0$,
\begin{itemize}
\item[(i)]
\begin{equation}
\label{HRasymp1}
H_{\xi,\sigma,s}^*(R) \lesssim_\times \sum_{\ell = 0}^\infty \frac{f_\xi\big(2^\ell (C_0 + 1) R\big)(2^\ell R)^{-2s}}{\sup_{Q\in[1,2^\ell R]}(f_\xi(Q)/Q^s)}.
\end{equation}
\item[(ii)] If $f_\xi(\alpha R) > f_\xi(\alpha R/2)$, then
\begin{equation}
\label{HRasymp2}
H_{\xi,\sigma,s}^*(R) \gtrsim_\times \frac{f_\xi\big(\alpha R\big) R^{-2s}}{\sup_{Q\in[1,C_5 R]}(f_\xi(Q)/Q^s)}.
\end{equation}
\end{itemize}
\end{lemma}
\begin{proof}[Proof of \textup{(i)}]
Since $\xi$ is a bounded parabolic point, by Proposition \ref{propositionboundedparabolic} there exists a $\xi$-bounded set $S\subset \EE_\xi$ such that $\Lr\subset\Lambda\butnot\{\xi\}\subset G_\xi(S)$. Without loss of generality, we may assume that $\zero\in S$.

Fix $R\geq 4C_0\Diam_\euc(S)$. Then
\begin{align*}
\Lr\butnot B_\euc(\zero,R)
&\subset \bigcup_{g\in G_\xi}g(S)\butnot B_\euc(\zero,R)\\
&\subset \bigcup_{\substack{g\in G_\xi \\ g(S)\butnot B_\euc(\zero,R)\neq\smallemptyset}}g(S)\\
&\subset \bigcup_{\substack{g\in G_\xi \\ \Dist_\euc(\zero,g(\zero))\geq R - C_0\Diam_\euc(S)}}g(S)\\
&\subset \bigcup_{\substack{g\in G_\xi \\ \Dist_\euc(\zero,g(\zero))\geq R/2}}g(S).
\end{align*}
Thus for all $\sigma > 0$ and for all $s\geq 0$,
\begin{equation}
\label{dividingintorings}
\begin{split}
H_{\xi,\sigma,s}^*(r) = \HH_{\sph,\infty}^s(\Lrsigma\butnot B_\euc(\zero,R))
&\leq \HH_{\sph,\infty}^s\left(\bigcup_{\substack{g\in G_\xi \\\Dist_\euc(\zero,g(\zero)) \geq R/2}}g(S)\right)\\
&\leq \sum_{\ell = 0}^\infty \HH_{\sph,\infty}^s\left(\bigcup_{\substack{g\in G_\xi \\ 2^{\ell - 1} R\leq\Dist_\euc(\zero,g(\zero)) < 2^\ell R}}g(S)\right).
\end{split}
\end{equation}
Fix $\ell\in\N$, and let
\[
A_\ell = B_\euc(\zero,2^{\ell + 1}R)\butnot B_\euc(\zero,2^{\ell - 2}R).
\]
Fix $g\in G_\xi$ such that $2^{\ell - 1} R\leq\Dist_\euc(\zero,g(\zero)) < 2^\ell R$. Then
\[
\Dist_\euc(g(\zero),\EE_\xi\butnot A_\ell) \geq \frac{R}{4}\geq C_0\Diam_\euc(S),
\]
and so
\[
g(S)\subset A_\ell.
\]
Now, for $y_1,y_2\in A_\ell$, by \eqref{GMVT2} and \eqref{euclideancomparison} we have
\[
\Dist_\sph(y_1,y_2) \asymp_\times (2^\ell R)^{-2}\Dist_\euc(y_1,y_2),
\]
and so
\begin{align*}
&\HH_{\sph,\infty}^s\left(\bigcup_{\substack{g\in G_\xi \\ 2^{\ell - 1}R\leq\Dist_\euc(\zero,g(\zero)) < 2^\ell R}}g(S)\right)\\
&\asymp_\times (2^\ell R)^{-2s} \HH_{\euc,\infty}^s\left(\bigcup_{\substack{g\in G_\xi \\ 2^{\ell - 1}R\leq\Dist_\euc(\zero,g(\zero)) < 2^\ell R}}g(S)\right) \noreason \\
&\leq_\pt  (2^\ell R)^{-2s} \HH_{\euc,\infty}^s\left(\bigcup_{\substack{g\in G_\xi \\ \Dist_\euc(\zero,g(\zero)) < 2^\ell R}}g(S)\right) \noreason \\
&\lesssim_\times (2^\ell R)^{-2s}\frac{f_\xi(2^\ell(C_0 + 1)R)}{\sup_{Q\in[1,2^\ell R]}(f_\xi(Q)/Q^s)}\cdot \by{Lemma \ref{lemmaHRlocal}}
\end{align*}
Combining with \eqref{dividingintorings} finishes the proof.
\end{proof}

\begin{proof}[Proof of \textup{(ii)}]
Since $G$ is of general type, there exists $h\in G$ so that $h(\xi)\neq\xi$. Let $S$ be a $\xi$-bounded set so that $\bord X = S\cup h(S)$. Without loss of generality, we may suppose that $\zero\in S$.
\begin{claim}
\label{claimHSbig}
There exists $\sigma > 0$ satisfying \eqref{HSbig}.
\end{claim}
\begin{subproof}
Since $\HD(\Lr) = \delta > s$, there exists $\sigma > 0$ so that $\HD(\Lrsigma) > s$. Letting $\w h = \id$ or $h$, we have $\HD(\Lrsigma\cap \w h(S)) > s$. By Lemma \ref{lemmanearinvariance}, we have $\HD(\Lrtau\cap S) > s$ for some $\tau > 0$,  finishing the proof. (Hausdorff dimension is the same in the Euclidean and spherical metrics.)
\end{subproof}
Let $\alpha = 4(C_0 + 1)$. Fix $R\geq C_0\Diam_\euc(S)$, and assume that $f_\xi(\alpha R) > f_\xi(\alpha R/2)$. Then there exists $g_0\in G_\xi$ such that
\begin{equation}
\label{gzerodef}
\alpha R/2 = 2(C_0 + 1) R < \Dist(\zero,g_0(\zero)) \leq \alpha R = 4(C_0 + 1)R.
\end{equation}
Now for every $x\in B_\euc(\zero,2R)\cap G_\xi(\zero)$,
\[
\Dist_\euc(g_0(\zero),g_\zero(x)) \leq 2C_0 R
\]
and thus
\[
2 R < \Dist_\euc(\zero,g_0(x)) \leq (6C_0 + 4) R.
\]
So
\begin{equation}
\label{fxiincreasingfast}
f_\xi((6C_0 + 4) R) \geq 2 f_\xi(2 R),
\end{equation}
and thus by Lemma \ref{lemmaHRlocal},
\begin{align*}
\HH_{\euc,\infty}^s(S_{(6C_0 + 4) R,1/(3C_0 + 2)})
&\gtrsim_\times  \frac{f_\xi((6C_0 + 4) R) - f_\xi(\beta R)}{\sup_{Q\in[1,12C_0 (6C_0 + 4) R]}(f_\xi(Q)/Q^s)}\\
&\asymp_\times \frac{f_\xi((6C_0 + 4) R)}{\sup_{Q\in[1,12C_0 (6C_0 + 4) R]}(f_\xi(Q)/Q^s)}\\
&\geq_\pt \frac{f_\xi(\alpha R)}{\sup_{Q\in[1,12C_0 (6C_0 + 4) R]}(f_\xi(Q)/Q^s)}\\
\end{align*}
On the other hand,
\begin{equation}
\label{S6Cetc}
\begin{split}
S_{(6C_0 + 4) R,1/(3C_0 + 2)}
&\subset B_\euc(\zero,(6C_0 + 4) R + C_0 \Diam_\euc(S)) \butnot B_\euc(\zero,2 R - C_0 \Diam_\euc(S))\\
&\subset B_\euc(\zero,(6C_0 + 4) R + C_0 \Diam_\euc(S)) \butnot B_\euc(\zero,R).
\end{split}
\end{equation}
Let $\tau$ be as in Lemma \ref{lemmanearinvariance}. Then
\begin{align*}
H_{\xi,\tau,s}^*(R) = \HH_{\sph,\infty}^s(\Lrtau\butnot B_\euc(\zero,R))
&\geq_\pt \HH_{\sph,\infty}^s(S_{(6C_0 + 4)R,1/(3C_0 + 2)}) \by{\eqref{S6Cetc}}\\
&\asymp_\times R^{-2s}\HH_{\euc,\infty}^s(S_{(6C_0 + 4)R,1/(3C_0 + 2)}) \by{\eqref{S6Cetc}}\\
&\gtrsim_\times R^{-2s}\frac{f_\xi(\alpha R)}{\sup_{Q\in[1,12C_0 (6C_0 + 4) R]}(f_\xi(Q)/Q^s)}\cdot
\end{align*}
\end{proof}

We now begin the proof of \eqref{prevelanihill}. It suffices to consider the two cases $s\in (0,2\delta_\xi)$ and $s\in(2\delta_\xi,\delta)$; the case $s = 2\delta_\xi$ follows by continuity.

\begin{remark}
By Observation \ref{observationeuclideanparabolicasymp},
\begin{equation}
\label{deltaxilimsup}
2\delta_\xi = \limsup_{R\to\infty} \frac{\log f_\xi(R)}{\log(R)}\cdot
\end{equation}
\end{remark}

\begin{itemize}
\item[Case 1:] $s > 2\delta_\xi$. In this case, fix $\varepsilon > 0$ so that $2\delta_\xi + \varepsilon < s$; then by \eqref{deltaxilimsup}, for all $R\geq 1$ we have
\begin{equation}
\label{fxiRbound}
f_\xi(R) \lesssim_{\times,\varepsilon} R^{2\delta_\xi + \varepsilon},
\end{equation}
so by \eqref{HRasymp1},
\begin{align*}
H_{\xi,\sigma,s}^*(R) &\lesssim_{\times\phantom{,\varepsilon}} \sum_{\ell = 0}^\infty \frac{f_\xi\big(2^\ell C_5 R\big)(2^\ell R)^{-2s}}{\sup_{Q\in[1,2^\ell R]}(f_\xi(Q)/Q^s)}\\
&\lesssim_{\times,\varepsilon} \sum_{\ell = 0}^\infty \big(2^\ell C_5 R\big)^{2\delta_\xi + \varepsilon}(2^\ell R)^{-2s}\\
&\asymp_{\times\phantom{,\varepsilon}}  R^{-(2s - 2\delta_\xi - \varepsilon)}\sum_{\ell = 0}^\infty (2^\ell)^{2\delta_\xi + \varepsilon - 2s}\\
&\asymp_{\times,\varepsilon} R^{-(2s - 2\delta_\xi - \varepsilon)}.
\end{align*}
Applying \eqref{Pxidef2} gives
\[
P_\xi(s) \geq 2s - 2\delta_\xi - \varepsilon,
\]
and taking the limit as $\varepsilon\to 0$ gives $P_\xi(s)\geq 2s - 2\delta_\xi$.

To demonstrate the other direction, note that by \eqref{fxiRbound} we have
\[
\sup_{Q\in[1,R]}(f_\xi(Q)/Q^s) \asymp_{\times,\varepsilon} 1
\]
and so by \eqref{HRasymp2} we have
\[
H_{\xi,\sigma,s}^*(R) \gtrsim_\times \frac{f_\xi\big(\alpha R\big) R^{-2s}}{\sup_{Q\in[1,C_5 R]}(f_\xi(Q)/Q^s)}
\asymp_{\times,\varepsilon} f_\xi\big(\alpha R\big) R^{-2s}
\]
whenever $f_\xi(\alpha R) > f_\xi(\alpha R/2)$. Thus
\begin{align*}
P_\xi(s) &= \lim_{\sigma\to\infty}\liminf_{R\to\infty}\frac{\log(H_{\xi,\sigma,s}^*(R))}{\log(1/R)}\\
&\leq \lim_{\sigma\to\infty}\liminf_{\substack{R\to\infty \\ f_\xi(\alpha R) > f_\xi(\alpha R/2)}}\frac{\log(R^{-2s} f_\xi(\alpha R))}{\log(1/R)}\\
&= 2s + \liminf_{\substack{R\to\infty \\ f_\xi(R) > f_\xi(R/2)}}\frac{\log f_\xi(R)}{\log(1/R)}\\
&= 2s - \limsup_{\substack{R\to\infty \\ f_\xi(R) > f_\xi(R/2)}}\frac{\log f_\xi(R)}{\log(R)}\\
&= 2s - \limsup_{R\to\infty}\frac{\log f_\xi(R)}{\log(R)} & \textup{(see below)}\\
&= 2s - 2\delta_\xi. \by{\eqref{deltaxilimsup}}
\end{align*}
To justify removing the restriction on the limsup when moving to the penultimate line, simply notice that for each $R$, if we let $\w R$ be the smallest value so that $f_\xi(\w R) = f_\xi(R)$, then $\w R$ satisfies the restriction, and moreover
\[
\frac{\log f_\xi(\w R)}{\log(\w R)} = \frac{\log f_\xi(R)}{\log(\w R)} \geq \frac{\log f_\xi(R)}{\log(R)}.
\]
So letting $R_n\tendsto n \infty$ be a sequence such that $\frac{\log f_\xi(R_n)}{\log(R_n)}$ tends to the lim sup, the sequence $(\w R_n)_1^\infty$ also tends to the lim sup; moreover, $\w R_n\tendsto n \infty$ because otherwise $f_\xi$ would be eventually constant, $G_\xi$ would be finite, and $\xi$ would fail to be a parabolic point.

\item[Case 2:] $s < 2\delta_\xi$. In this case, by \eqref{deltaxilimsup}, we have
\begin{equation}
\label{toinfinity}
\limsup_{R\to\infty}f_\xi(R)/R^s = \infty.
\end{equation}
\begin{claim}
\label{claimC7}
There exists a constant $C_7 > 0$ and a sequence $R_n\tendsto n \infty$ such that for all $n\in\N$,
\[
\sup_{Q\in[1,C_5 R_n]}(f_\xi(Q)/Q^s) \leq C_7 f_\xi(\alpha R_n)/R_n^s
\]
and
\[
f_\xi(\alpha R_n) > f_\xi(\alpha R_n/2).
\]
\end{claim}
\begin{subproof}
For convenience let $h_{\xi,s}(R) = f_\xi(R)/R^s$ for all $R\geq 1$ and let $C_8 = C_5/\alpha$. Fix $C_7 > 0$ large to be announced below, and by contradiction suppose that there exists $R_0$ such that for all $R\geq R_0$, either
\begin{equation}
\label{C71}
\sup_{[1,C_8 R]} h_{\xi,s} > C_7 h_{\xi,s}(R)
\end{equation}
or
\begin{equation}
\label{C72}
f_\xi(R) = f_\xi(R/2).
\end{equation}
(In each equation, we have changed variables replacing $R$ by $R/\alpha$.) By \eqref{toinfinity}, there exists $R_1\geq R_0$ so that
\begin{equation}
\label{R1def}
h_{\xi,s}(R_1) = \sup_{[1,R_1]}h_{\xi,s},
\end{equation}
and in particular $h_{\xi,s}(R_1)\geq h_{\xi,s}(R_2)$, which implies $f_\xi(R_1) > f_\xi(R_1/2)$. Thus $R_1$ does not satisfy \eqref{C72}, so it does satisfy \eqref{C71}; there exists $R_2\in[1,C_8 R_1]$ such that
\begin{equation}
\label{R2def}
h_{\xi,s}(R_2) \geq C_7 h_{\xi,s}(R_1).
\end{equation}
By \eqref{R1def}, such an $R_2$ must be larger than $R_1$. Letting $R_2$ be minimal satisfying \eqref{R2def}, we see that
\[
h_{\xi,s}(R_2) = \sup_{[1,R_2]}h_{\xi,s}.
\]
We can continue this process indefinitely to get a sequence $R_1 < R_2 < \cdots$ satisfying
\[
h_{\xi,s}(R_{n + 1}) \geq C_7 h_{\xi,s}(R_n) \text{ and } R_{n + 1} \leq C_8 R_1.
\]
By iterating, we see that
\[
\frac{h_{\xi,s}(R_n)}{h_{\xi,s}(R_1)} \geq C_7^{n - 1} = C_8^{\log_{C_8}(C_7)(n - 1)} \geq \left(\frac{R_n}{R_1}\right)^{\log_{C_8}(C_7)}
\]
and thus
\[
f_\xi(R_n) \gtrsim_{\times,C_7,R_1} R_n^{\log_{C_8}(C_7) + s}.
\]
Applying \eqref{deltaxilimsup} gives
\[
\delta_\xi \geq \log_{C_8}(C_7) + s,
\]
but if $C_7$ is large enough, then this is a contradiction, since $\delta_\xi\leq\delta < \infty$ by assumption.
\end{subproof}

\noindent Applying \eqref{HRasymp2} gives
\[
H_{\xi,\sigma,s}^*(R_n) \gtrsim_\times R_n^{-s}
\]
and applying \eqref{Pxidef2} shows that $P_\xi(s)\leq s$. On the other hand, the direction $P_\xi(s)\geq s$ holds in general  (Proposition \ref{propositionPxifacts}).
\end{itemize}
This completes the proof of \eqref{prevelanihill}. Deducing \eqref{velanihillgeneral} from \eqref{jarnikbesicovitch} and \eqref{prevelanihill} is trivial, so this completes the proof of Theorem \ref{theoremboundedparabolicjarnik}.

\draftnewpage\section{Proof of Theorem \ref{theoremkhinchin} (Generalization of Khinchin's Theorem)}

\begin{theorem}[Generalization of Khinchin's Theorem]
\label{theoremkhinchin}
Let $(X,\dist,\zero,b,G)$ be as in \sectionsymbol\ref{standingassumptions}, and fix a distinguished point $\xi\in\del X$. Let $\delta$ be the Poincar\'e exponent of $G$, and suppose that $\mu\in\MM(\Lambda)$ is an ergodic $\delta$-quasiconformal probability measure satisfying $\mu(\xi) = 0$. Let
\[
\Delta_{\mu,\xi}(r) = \mu(B(\xi,r)).
\]
Let $\Phi:\Rplus\to(0,\infty)$ be a function such that the function $t\mapsto t\Phi(t)$ is nonincreasing. Then:
\begin{itemize}
\item[(i)] If the series
\begin{equation}
\label{khinchindiverges}
\sum_{g\in G}b^{-\delta\dist(\zero,g(\zero))} \Delta_{\mu,\xi} \left(b^{\dist(\zero,g(\zero))} \Phi(K b^{\dist(\zero,g(\zero))})\right)
\end{equation}
diverges for all $K > 0$, then $\mu(\WA_{\Phi,\xi}) = 1$.
\item[(iiA)] If the series
\begin{equation}
\label{khinchinconverges}
\sum_{g\in G}(g'(\xi))^\delta \Delta_{\mu,\xi} \left(\Phi(K b^{\dist(\zero,g(\zero))})/g'(\xi)\right)
\end{equation}
converges for some $K > 0$, then $\mu(\WA_{\Phi,\xi}) = 0$.
\item[(iiB)] If $\xi$ is a bounded parabolic point of $G$ (see Definition \ref{definitionboundedparabolic}), and if the series \eqref{khinchindiverges} converges for some $K > 0$, then $\mu(\WA_{\Phi,\xi}) = 0$.
\end{itemize}
\end{theorem}

Let us begin by introducing some notation. First of all, we will assume that the map $G\ni g\mapsto g(\zero)$ is injective; avoiding this assumption would only change the notation and not the arguments. For each $x\in G(\zero)$, let $g_x\in G$ be the unique element so that $g_x(\zero) = x$, and for each $K > 0$ let
\begin{equation}
\label{BxKdef}
B_{x,K} := B(g_x(\xi),\Phi(K b^{\dist(\zero,x)}));
\end{equation}
then
\begin{equation}
\label{WAphiequals}
\WA_{\Phi,\xi} = \bigcap_{K > 0}\bigcup_{x\in G(\zero)}B_{x,K}.
\end{equation}

We will also need the following:

\begin{lemma}[Sullivan's Shadow Lemma]\label{Sullivan's Shadow Lemma}
Let $X$ be a hyperbolic metric space, and let $G$ be a subgroup of $\Isom(X)$. Fix $s\geq 0$, and let $\mu\in\MM(\del X)$ be an $s$-quasiconformal measure which is not a pointmass. Then for all $\sigma > 0$ sufficiently large and for all $x\in G(\zero)$,
\begin{equation}
\label{Sullivan's shadow}
\mu(\Shad(x,\sigma)) \asymp_{\times,\sigma} b^{-s \dist(\zero,x)}.
\end{equation}
\end{lemma}
The proof of this theorem is similar to the proof in the case of standard hyperbolic space \cite[Proposition 3]{Sullivan_density_at_infinity}, and can be found in \cite{Coornaert} or \cite[Lemma 15.4.1]{DSU}.

\subsection{The convergence case}
Suppose that either
\begin{itemize}
\item[(1)] the series \eqref{khinchinconverges} converges for some $K > 0$, or that
\item[(2)] $\xi$ is a bounded parabolic point, and the series \eqref{khinchindiverges} converges for some $K > 0$.
\end{itemize}
By \eqref{WAphiequals} and by the easy direction of the Borel--Cantelli lemma,\Footnote{The condition $\mu(\xi) = 0$ guarantees that for $\mu$-almost every $\eta$, we have $\eta\in\WA_{\Phi,\xi}$ if and only if for all $K > 0$ there exist infinitely many $x\in G(\zero)$ such that $\eta\in B_{x,K}$.} to show that $\mu(\WA_{\Phi,\xi}) = 0$ it suffices to show that there exists $K > 0$ so that the series
\begin{equation}
\label{muBxK}
\sum_{x\in G(\zero)} \mu(B_{x,K})
\end{equation}
converges. Fix $K > 0$ large to be announced below. Fix $x\in G(\zero)$ and let $g = g_x$. Note that
\[
\mu(B_{x,K}) \asymp_\times \int_{g^{-1}(B_{x,K})} (g')^\delta \dee\mu.
\]
\begin{claim}
\label{claimboundeddistortionkhinchin}
By choosing $K$ sufficiently large, we may ensure that for all $x\in G(\zero)$ and for all $\eta\in g^{-1}(B_{x,K})$,
\begin{align} \label{boundeddistortion5}
g'(\eta) &\asymp_\times g'(\xi)\\ \label{distetaxi}
\Dist(\xi,\eta) &\lesssim_\times \frac{\Phi(K b^{\dist(\zero,x)})}{g'(\xi)}\cdot
\end{align}
In both equations $g = g_x$.
\end{claim}
\begin{subproof} (cf. proof of Claim \ref{claimboundeddistortionjarnik}\comdavid{I figured it was still better to repeat the proofs since they are slightly different... Should I add a similar note to the proof of Claim \ref{claimboundeddistortionjarnik}?})
Fix $\eta\in g^{-1}(B_{x,K})$. We have
\begin{equation}
\label{asymptotictozerokhinchin}
\left|\log_b\big(\frac{g'(\eta)}{g'(\xi)}\big)\right|
= |\busemann_\eta(\zero,g^{-1}(\zero)) - \busemann_\xi(\zero,g^{-1}(\zero))|
\asymp_\plus |\lb g(\zero)|g(\eta)\rb_\zero - \lb g(\zero)|g(\xi)\rb_\zero|.
\end{equation}
Let
\begin{align*}
\textup{(A)} &= \lb g(\zero)|g(\eta)\rb_\zero\\
\textup{(B)} &= \lb g(\zero)|g(\xi)\rb_\zero\\
\textup{(C)} &= \lb g(\xi)|g(\eta)\rb_\zero.
\end{align*}
By Gromov's inequality, at least two of the expressions (A), (B), and (C) must be asymptotic. On the other hand, since $g(\eta)\in B_{x,K}$, we have
\begin{align*}
\lb g(\xi)|g(\eta)\rb_\zero - \dist(\zero,g(\zero)) &\asymp_\plus -\log_b(\Dist(g(\xi),g(\eta))) - \dist(\zero,g(\zero))\\
&\geq_\pt -\log_b(\Phi(K b^{\dist(\zero,x)})) - \dist(\zero,x) \by{\eqref{BxKdef}}\\
&\geq_\pt -\log_b(\Phi(1)/(K b^{\dist(\zero,x)})) - \dist(\zero,x) \since{$t\mapsto t\Phi(t)$ is nonincreasing}\\
&=_\pt \log_b(K/\Phi(1))
\end{align*}
Since (A) and (B) are both less than $\dist(\zero,x)$, this demonstrates that (C) is not asymptotic to either (A) or (B) if $K$ is sufficiently large. Thus by Gromov's inequality, (A) and (B) are asymptotic, and so \eqref{asymptotictozerokhinchin} is asymptotic to zero. This demonstrates \eqref{boundeddistortion5}. Finally, \eqref{distetaxi} follows from \eqref{boundeddistortion5}, \eqref{BxKdef}, and the geometric mean value theorem.
\end{subproof}
Let $C > 0$ be the implied constant in \eqref{distetaxi}. Then
\begin{equation}
\label{muBxKbound}
\begin{split}
\mu(B_{x,K})
&\asymp_\times \int_{g^{-1}(B_{x,K})} (g')^\delta \dee\mu\\
&\asymp_\times (g'(\xi))^\delta \mu\big(g^{-1}(B_{x,K})\big)\\
&\leq_\pt (g'(\xi))^\delta \mu\left(B\big(\xi, C\frac{\Phi(K b^{\dist(\zero,x)})}{g'(\xi)}\big)\right)\\
&=_\pt (g'(\xi))^\delta \Delta_{\mu,\xi}\left(C \Phi(K b^{\dist(\zero,x)}) / g'(\xi) \right)\\
&\leq_\pt (g'(\xi))^\delta \Delta_{\mu,\xi}\left(\Phi(C^{-1}K b^{\dist(\zero,x)}) / g'(\xi) \right),
\end{split}
\end{equation}
the last inequality due to the fact that the function $t\mapsto t\Phi(t)$ is nonincreasing. Now the argument splits according to the cases (1) and (2):
\begin{itemize}
\item[(1)] In this case, let $\w K > 0$ be a value of $K$ for which \eqref{khinchinconverges} converges, and choose $K > 0$ large enough so that $K \geq C\w K$ and so that Claim \ref{claimboundeddistortionkhinchin} is satisfied. Then we have \eqref{muBxK} $\lesssim_\times$ \eqref{khinchinconverges} $< \infty$, completing the proof.
\item[(2)] In this case, let us call a point $x\in G(\zero)$ \emph{minimal} if
\[
\dist(\zero,x) = \min\{\dist(\zero,\w g(\zero)):\w g(\xi) = g_x(\xi)\}.
\]
Then
\[
\WA_{\Phi,\xi} = \bigcap_{K > 0}\bigcup_{\substack{x\in G(\zero) \\ \text{minimal}}}B_{x,K}.
\]
Again, by the easy direction of the Borel--Cantelli lemma, to show that $\mu(\WA_{\Phi,\xi}) = 0$ it suffices to show that there exists $K > 0$ so that the series
\begin{equation}
\label{muBxKminimal}
\sum_{\substack{x\in G(\zero) \\ \text{minimal}}}\mu(B_{x,K})
\end{equation}
converges.

Since $\xi$ is a bounded parabolic point, there exists a $\xi$-bounded set $S$ so that $G(\zero)\subset G_\xi(S)$, where $G_\xi$ is the stabilizer of $\xi$ relative to $G$. Fix $x\in G(\zero)$ minimal. Then there exists $h\in G_\xi$ such that $h\circ g^{-1}(\zero)\in S$; by Observation \ref{observationxibounded} we have
\begin{equation}
\label{boundedparabolicnew}
\lb h\circ g^{-1}(\zero)|\xi\rb_\zero \asymp_\plus 0.
\end{equation}
Then
\begin{align*}
\dist(\zero,x) &\leq_\pt \dist(\zero,g\circ h^{-1}(\zero)) \since{$x$ is minimal}\\
&=_\pt \dist(h\circ g^{-1}(\zero),\zero)\\
&\asymp_\plus \busemann_\xi(h\circ g^{-1}(\zero),\zero) \by{\eqref{boundedparabolicnew}}\\
&\asymp_\plus \busemann_\xi(g^{-1}(\zero),\zero); \since{$\xi$ is parabolic}
\end{align*}
exponentiating gives
\[
g'(\xi) \asymp_\times b^{-\dist(\zero,x)}.
\]
Let $C_2 > 0$ be the implied constant of this asymptotic. Then by \eqref{muBxKbound}, we have
\[
\mu(B_{x,K}) \lesssim_\times b^{-\delta\dist(\zero,x)} \Delta_{\mu,\xi}\left(b^{\dist(\zero,x)}\Phi((CC_2)^{-1}K b^{\dist(\zero,x)}) \right).
\]
Let $\w K > 0$ be a value of $K$ for which \eqref{khinchindiverges} converges, and choose $K > 0$ large enough so that $K \geq CC_2\w K$ and so that Claim \ref{claimboundeddistortionkhinchin} is satisfied. Then we have \eqref{muBxKminimal} $\lesssim_\times$ \eqref{khinchindiverges} $< \infty$, completing the proof.
\end{itemize}

\subsection{The divergence case}\Footnote{This proof was heavily influenced by the proof of \cite[Theorem 4 of Section VII]{Ahlfors}.}
By contradiction, suppose that $\mu(\WA_{\Phi,\xi}) = 0$. Fixing $\lambda > 0$ small to be announced below, we have
\[
\mu\left(\bigcup_{x\in G(\zero)}B_{x,K}\right) \leq \lambda
\]
for some $K > 0$. For convenience of notation let
\[
U := \bigcup_{x\in G(\zero)}B_{x,K}.
\]
Let $\sigma > 0$ be large enough so that \eqref{Sullivan's shadow} holds for all $x\in G(\zero)$, and fix $\varepsilon > 0$ small to be announced below. Let $\phi(t) = t\Phi(t)$, so that the function $\phi:\Rplus\to(0,\infty)$ is nonincreasing. For each $x\in G(\zero)$ let
\begin{align*}
\PP_x &= \Shad(x,\sigma)\\
B_x &= g_x(B(\xi,\varepsilon \phi(Kb^{\dist(\zero,x)}))).
\end{align*}
Let
\begin{equation}
\label{Gsigmadef}
A := \{x\in G(\zero):B_x\subset \PP_x\}.
\end{equation}
In a specific sense $A$ comprises ``at least half'' of $G(\zero)$; see Lemma \ref{lemmahalf} below.

Let $d:A\to\N$ be a bijection with the following property: $d_x \leq d_y$ \implies $\dist(\zero,x) \leq \dist(\zero,y)$. Such a bijection is possible since $G$ is strongly discrete. Note that $d_\zero = 0$.

We will define three sets $T_1,T_2,T_3\subset A$, two binary relations $\in_1\subset T_1\times T_1$ and $\in_2\subset T_2\times T_1$, and one ternary relation $\in_3 \subset T_3\times T_1\times T_2$ via the following procedure. Points in $T_i$ will be called \emph{type i}.\Footnote{See footnote \ref{footnoteparallelprocessors}.} \comdavid{Do not allow the numbers in this enumeration to change.}
\begin{enumerate}[1.]
\item Each time a new point is put into $T_1$, call it $x$ and go to step 3.
\item Put $\zero$ into $T_1$; it is the root node.
\item Recall that $x$ is a point which has just been put into $T_1$. Let
\begin{equation}
\label{Dxdef}
S_x = \{y\in A:\PP_x\cap \PP_y\neq\emptyset\text{ and }d_y > d_x\}
\end{equation}
and
\begin{equation}
\label{wDxdef}
\w S_x = \{y\in S_x: \PP_y\subset g_x(U)\}.
\end{equation}
\item Construct a sequence of points $z_n\in\w S_x$ recursively as follows: If $z_1,\ldots,z_{n - 1}$ have been chosen, then let $z_n\in\w S_x$ be such that
\begin{equation}
\label{SznSzj}
\PP_{z_n}\cap \PP_{z_j} = \emptyset \all j = 1,\ldots,n - 1.
\end{equation}
If more than one such $z_n$ exists, then choose $z_n$ so as to minimize $d_{z_n}$. Note that the sequence may terminate in finitely many steps if \eqref{SznSzj} is not satisfied for any $z_n\in\w S_x$.
\item For each $n$, put $z_n$ into $T_1$, setting $z_n\in_1 x$.
\item Construct a sequence of points $y_n\in S_x\butnot\w S_x$ recursively as follows: If $y_1,\ldots,y_{n - 1}$ have been chosen, then let $y_n\in\w S_x$ be such that
\begin{equation}
\label{BynByj}
B_{y_n}\cap B_{y_j} = \emptyset \all j = 1,\ldots,n - 1.
\end{equation}
If more than one such $y_n$ exists, then choose $y_n$ so as to minimize $d_{y_n}$.
\item For each $n$,
\begin{enumerate}[i.]
\item Put $y = y_n$ into $T_2$, setting $y\in_2 x$.
\item Let
\[
S_{x,y} = \{z\in S_x\butnot\w S_x:B_y\cap B_z\neq\emptyset\text{ and }d_z > d_y\}.
\]
\item For each $z\in S_{x,y}$, put $z$ into $T_3$, setting $z\in_3 (x,y)$.
\end{enumerate}
\end{enumerate}

\begin{notation}
Write $y > x$ if $y\in S_x$. Note that by construction $y > x$ \implies $d_y > d_x$.

Let us warn that the relation $>$ is not necessarily transitive, due to the fact that the relation of having nonempty intersection is not an equivalence relation.
\end{notation}
\begin{lemma}
\label{lemmaeverypoint}
\[
T_1\cup T_2\cup T_3 = A.
\]
\end{lemma}
\begin{proof}
By contradiction, suppose that $w\in A\butnot(T_1\cup T_2\cup T_3)$. We claim that there is an infinite sequence of type 1 points $(x_n)_0^\infty$ so that $\cdots \in_1 x_2\in_1 x_1\in_1 x_0 = \zero$ and $x_n < w$ for all $n$. Then $d_{x_0} < d_{x_1} < \cdots < d_w$, which is absurd since these are natural numbers.

Let $x_0 = \zero$. By induction, suppose that $x = x_n < w$ is a type 1 point.
\begin{itemize}
\item[Case 1:] $w\in \w S_x$. In this case, note that $w$ could have been chosen in step 4, but was not (otherwise it would have been put into $T_1$). The only explanation for this is that $\PP_w\cap \PP_z \neq \emptyset$ for some $z\in_1 x$ for which $d_z < d_w$. But then $x_n < z < w$; letting $x_{n + 1} = z$ completes the inductive step.
\item[Case 2:] $w\notin \w S_x$. In this case, note that $w$ could have been chosen in step 6, but was not (otherwise it would have been put into $T_2$). The only explanation for this is that $B_w\cap B_z \neq \emptyset$ for some $y\in_2 x$ for which $d_y < d_w$. But then $w\in_3 (x,y)$, which demonstrates that $w\in T_3$, contradicting our hypothesis.
\end{itemize}
\end{proof}

\begin{lemma}
\label{lemmaSyStau}
There exists $\tau > 0$ (depending on $\sigma$) so that for all $x,y\in A$ with $y > x$ we have
\[
\PP_y\subset \Shad(x,\tau).
\]
\end{lemma}
\begin{proof}
This follows directly from the Intersecting Shadows Lemma, \eqref{Dxdef}, and the fact that $d_y > d_x$ \implies $\dist(\zero,y)\geq\dist(\zero,x)$.
\end{proof}

\begin{lemma}
\label{lemma3By}
For all $x,y\in A$ such that $y > x$, we have
\[
3B_y\subset g_x(U),
\]
if $\varepsilon$ is chosen small enough.
\end{lemma}
\begin{proof}
By Lemma \ref{lemmaSyStau} and \eqref{Gsigmadef} we have
\[
B_y \subset \PP_y \subset \Shad(x,\tau),
\]
where $\tau$ is as in Lemma \ref{lemmaSyStau}. Thus by Lemma \ref{lemmaCShad}, there exists $\w\tau > 0$ such that
\[
3B_y \subset \Shad(x,\w\tau).
\]
Suppose $\eta\in B_y$; then
\[
\Dist(\xi,g_y^{-1}(\eta)) \leq \varepsilon \phi(Kb^{\dist(\zero,y)});
\]
since $B_y\subset\PP_y$, the Bounded Distortion Lemma gives
\[
\Dist(g_y(\xi),\eta) \lesssim_{\times,\sigma} b^{-\dist(\zero,y)} \varepsilon \phi(Kb^{\dist(\zero,y)}).
\]
Thus
\[
\Diam(3B_y) \leq 3\Diam(B_y) \lesssim_{\times,\sigma} b^{-\dist(\zero,y)} \varepsilon \phi(Kb^{\dist(\zero,y)}).
\]
Now fix $\eta\in 3B_y$; then
\begin{align*}
\Dist(g_y(\xi),\eta) &\lesssim_{\times,\sigma} b^{-\dist(\zero,y)} \varepsilon \phi(Kb^{\dist(\zero,y)})\\
\Dist(g_x^{-1}\circ g_y(\xi),g_x^{-1}(\eta))
&\lesssim_{\times,\w\tau} b^{\dist(\zero,x) - \dist(\zero,y)} \varepsilon \phi(Kb^{\dist(\zero,y)}) \by{the Bounded Distortion Lemma}\\
&\asymp_{\times,\sigma} b^{-\dist(x,y)}\varepsilon \phi(Kb^{\dist(\zero,y)}) \by{the Intersecting Shadows Lemma}\\
&\leq_{\phantom{\times,\sigma}} b^{-\dist(x,y)}\varepsilon \phi(Kb^{\dist(x,y)}) \since{$\phi$ is nonincreasing}\\
&=_{\phantom{\times,\sigma}} \varepsilon K \Phi(Kb^{\dist(x,y)})
\asymp_{\times,K} \varepsilon\Phi(K b^{\dist(x,y)}).\noreason
\end{align*}
If we choose $\varepsilon$ to be less than the reciprocal of the implied constant, then we have
\[
\Dist(g_x^{-1}\circ g_y(\xi),g_x^{-1}(\eta)) \leq \Phi(K b^{\dist(x,y)})
\]
and thus
\[
g_x^{-1}(\eta) \in B\big(g_x^{-1}\circ g_y(\xi),\Phi(Kb^{\dist(\zero,g_x^{-1}\circ g_y(\zero))})\big) = B_{g_x^{-1}(y),K} \subset U,
\]
i.e. $\eta\in g_x(U)$. Since $\eta\in 3B_y$ was arbitrary, this demonstrates that $3B_y\subset g_x(U)$.
\end{proof}

\begin{lemma}
\label{lemmatype2type3}
If $x\in T_1$, $y\in_2 x$, and $z\in_3 (x,y)$, then $B_y \subset \Shad(z,\rho)$ for some $\rho > 0$ large independent of $z$. In particular $g_y(\xi)\in\Shad(z,\rho)$.
\end{lemma}
\begin{proof}
Since $z\in_3 (x,y)$, we have $B_y\cap B_z\neq \emptyset$, but by \eqref{Gsigmadef}, $B_z\subset \PP_z$, so
\begin{equation}
\label{BySz}
B_y\cap \PP_z\neq\emptyset.
\end{equation}
Since $z\in S_{x,y}\subset S_x\butnot\w S_x$, we have
\[
\PP_z\nsubseteq g_x(U),
\]
but on the other hand, $3B_y\subset g_x(U)$ by Lemma \ref{lemma3By}, so
\[
\PP_z\nsubseteq 3B_y.
\]
Combining with \eqref{BySz}, we have
\[
\Diam(\PP_z) \geq \Dist(B_y,\del X\butnot 3B_y) \geq \Diam(B_y)
\]
and thus by Lemma \ref{lemmaCShad}, we have
\[
B_y \subset 3\PP_z \subset \Shad(z,\rho)
\]
if $\rho > 0$ is large enough.
\end{proof}

\begin{lemma}
\label{lemmatypebounds}
For each $x\in T_1$,
\begin{align} \label{type1bound}
\sum_{z\in_1 x}\mu(\PP_z) &\lesssim_{\times,\sigma} \lambda\mu(\PP_x)\\ \label{type2bound}
\sum_{y\in_2 x}\mu(B_y) &\lesssim_{\times,\sigma} \mu(\PP_x).
\end{align}
For each $x\in T_1$ and for each $y\in_2 x$,
\begin{equation} \label{type3bound}
\sum_{z\in_3 (x,y)}\mu(B_z) \lesssim_{\times,\sigma} \mu(B_y).
\end{equation}
\end{lemma}
\begin{proof}~
\begin{itemize}
\item[(\ref{type1bound}):]
By Lemma \ref{lemmaSyStau}, \eqref{wDxdef}, and \eqref{SznSzj}, the collection $(\PP_z)_{z\in_1 x}$ is a disjoint collection of shadows contained in $g_x(U)\cap\Shad(x,\tau)$. By the Bounded Distortion Lemma,
\begin{align*}
\sum_{z\in_1 x}\mu(\PP_z)
&\leq_{\phantom{\times,\tau}} \mu(g_x(U)\cap \Shad(x,\tau))\\
&\asymp_{\times,\tau} b^{-\delta\dist(\zero,x)}\mu(U\cap g^{-1}(\Shad(x,\tau)))\\
&\leq_{\phantom{\times,\tau}} \lambda b^{-\delta\dist(\zero,x)}
\asymp_{\times,\sigma} \lambda \mu(\PP_x).
\end{align*}
\item[(\ref{type2bound}):]
By Lemma \ref{lemmaSyStau}, \eqref{Gsigmadef}, and \eqref{BynByj}, the collection $(B_y)_{y\in_2 x}$ is a disjoint collection of balls contained in $\Shad(x,\tau)$. Thus
\[
\sum_{y\in_2 x}\mu(B_y)
\leq \mu(\Shad(x,\tau))
\asymp_{\times,\tau} b^{-\delta\dist(\zero,x)} \asymp_{\times,\sigma} \mu(\PP_x).
\]
\item[(\ref{type3bound}):]
By the Bounded Distortion Lemma, we have for each $z\in A$
\[
\mu(B_z) \asymp_{\times,\sigma} b^{-\delta\dist(\zero,z)}\mu(B(\xi,\varepsilon \phi(Kb^{\dist(\zero,z)}))).
\]
Since $\phi$ is nonincreasing, we have for each $z\in_3 (x,y)$
\[
\mu(B(\xi,\varepsilon \phi(Kb^{\dist(\zero,z)}))) \leq \mu(B(\xi,\varepsilon \phi(Kb^{\dist(\zero,y)})))
\]
and thus
\[
\mu(B_z) \lesssim_{\times,\sigma} b^{-\delta\busemann_\zero(z,y)}\mu(B_y),
\]
and so
\begin{equation}
\label{zxymuBz}
\sum_{z\in_3 (x,y)}\mu(B_z) \lesssim_{\times,\sigma} \mu(B_y)b^{\delta\dist(\zero,y)}\sum_{z\in_3 (x,y)}b^{-\delta\dist(\zero,z)}.
\end{equation}
Recall that by Lemma \ref{lemmatype2type3}, there exists $\rho > 0$ so that $g_y(\xi)\in\Shad(z,\rho)$ for every $z\in_3 (x,y)$; moreover, $\dist(\zero,z)\geq\dist(\zero,y)$ for every such $z$. Thus, modifying the proof of Lemma \ref{lemmageneralfinite}, we see that
\begin{align*}
\sum_{z\in_3 (x,y)}b^{-\delta\dist(\zero,z)} \leq M b^{\delta\tau}\sum_{n = \lfloor \dist(\zero,y)\rfloor}^\infty b^{-\delta n},
\end{align*}
where $\tau > 0$ is large enough so that \eqref{diamSrhoAn} is satisfied for every $n\in\N$, and where $M = \#\{g\in G:g(\zero)\in B(\zero,\tau)\}$. Summing the geometric series, we see
\[
\sum_{z\in_3 (x,y)}b^{-\delta\dist(\zero,z)} \lesssim_{\times,\rho} b^{-\delta\dist(\zero,y)}.
\]
Applying this to \eqref{zxymuBz} yields \eqref{type3bound}.
\end{itemize}
\end{proof}
Next, choose $\lambda$ small enough so that (\ref{type1bound}) becomes
\[
\sum_{z\in_1 x}\mu(\PP_z) \leq \frac{1}{2}\mu(\PP_x).
\]
Iterating yields
\[
\sum_{z\in T_1}\mu(\PP_z) \leq 2\mu(\PP_\zero) = 2.
\]
Combining with \eqref{type2bound}, \eqref{type3bound}, \eqref{Gsigmadef}, and Lemma \ref{lemmaeverypoint} yields
\begin{equation}
\label{muBxconverges}
\sum_{x\in A}\mu(B_x) < \infty.
\end{equation}

Recall that we are trying to derive a contradiction by finding a value of $K$ for which the series \eqref{khinchindiverges} converges. (It won't necessarily be the same $K$ that we chose earlier.) For each $\w K > 0$ and $x\in X$ let
\[
f_{\w K}(x) := b^{-\delta\dist(\zero,x)}\Delta_{\mu,\xi}(\varepsilon \phi(Kb^{\dist(\zero,x)})).
\]
Then if $x\in A$, then by the Bounded Distortion Lemma we have
\[
\mu(B_x) \asymp_{\times,\sigma} f_K(x).
\]
Thus by \eqref{muBxconverges}, $\sum_A f_K < \infty$. To complete the proof we need to show that $\sum_{G(\zero)} f_{\w K} < \infty$ for some $\w K > 0$.

By Observation \ref{observationnoglobalfixedpoint}, there exists $h\in G$ so that $h(\xi)\neq \xi$.
\begin{lemma}
\label{lemmahalf}
For every $g\in G$, either $g(\zero)\in A$ or $g\circ h(\zero)\in A$, assuming $\sigma > 0$ is chosen large enough and $\varepsilon > 0$ is chosen small enough.
\end{lemma}
\begin{proof}
By contradiction, suppose $g(\zero),g\circ h(\zero)\notin A$. Then
\begin{align*}
B(\xi,\varepsilon \phi(Kb^{\dist(\zero,g(\zero))})) &\nsubseteq \Shad_{g^{-1}(\zero)}(\zero,\sigma)\\
h(B(\xi,\varepsilon \phi(Kb^{\dist(\zero,g\circ h(\zero))}))) &\nsubseteq \Shad_{g^{-1}(\zero)}(h(\zero),\sigma).
\end{align*}
On the other hand,
\[
\Shad_{g^{-1}(\zero)}(h(\zero),\sigma) \supset \Shad_{g^{-1}(\zero)}(\zero,\sigma - \dist(\zero,h(\zero)))
\]
and thus
\[
h(B(\xi,\varepsilon \phi(Kb^{\dist(\zero,g\circ h(\zero))}))) \nsubseteq \Shad_{g^{-1}(\zero)}(\zero,\sigma - \dist(\zero,h(\zero))).
\]
But by the Big Shadows Lemma, we may choose $\sigma$ large enough so that
\[
\Diam\big(\del X\butnot \Shad_{g^{-1}(\zero)}(\zero,\sigma - \dist(\zero,h(\zero)))\big) \leq \frac{1}{4}\Dist(\xi,h(\xi)),
\]
and so
\[
\Dist(B(\xi,\varepsilon \phi(Kb^{\dist(\zero,g(\zero))})),h(B(\xi,\varepsilon \phi(Kb^{\dist(\zero,g\circ h(\zero))})))) \leq \frac{1}{4}\Dist(\xi,h(\xi)).
\]
On the other hand, we may choose $\varepsilon$ small enough so that
\begin{align*}
\varepsilon\phi(K) &\leq \frac{1}{4}\Dist(\xi,h(\xi))\\
\big(\sup h'\big)\varepsilon \phi(K) &\leq \frac{1}{4}\Dist(\xi,h(\xi)),
\end{align*}
which implies
\[
\Dist(\xi,h(\xi)) \leq \frac{3}{4}\Dist(\xi,h(\xi)),
\]
a contradiction.
\end{proof}

Fix $\w K \geq K$ to be announced below, and note that $f_{\w K}\leq f_K$ since $\phi$ and $\Delta_{\mu,\xi}$ are both monotone (nonincreasing and nondecreasing, respectively). Now
\begin{align*}
\sum_{G(\zero)}f_{\w K} &\leq \sum_{\substack{g\in G \\ g(\zero)\in A}}f_{\w K}(g(\zero)) + \sum_{\substack{g\in G \\ g\circ h(\zero)\in A}}f_{\w K}(g(\zero))\\
&= \sum_A f_{\w K} + \sum_{\substack{g\in G \\ g\circ h(\zero)\in A}}b^{-\delta\dist(\zero,g(\zero))}\Delta_{\mu,\xi}(\varepsilon \phi(\w K b^{\dist(\zero,g(\zero))})).
\end{align*}
Let $\w K = K b^{\dist(\zero,h(\zero))}$; then
\[
\w K b^{\dist(\zero,g(\zero))} \geq K b^{\dist(\zero,g\circ h(\zero))}
\]
and so
\begin{align*}
\sum_{G(\zero)}f_{\w K}
&\leq \sum_A f_K + b^{\delta\dist(\zero,h(\zero))}\sum_{\substack{g\in G \\ g\circ h(\zero)\in A}} b^{-\delta\dist(\zero,g\circ h(\zero))}\Delta_{\mu,\xi}(\varepsilon \phi(K b^{\dist(\zero,g\circ h(\zero))}))\\
&\leq \sum_A f_K + b^{\delta\dist(\zero,h(\zero))}\sum_A f_K < \infty.
\end{align*}

\subsection{Special cases} \label{subsectionspecialcaseskhinchin}
In this subsection we prove the assertions made in \sectionsymbol\ref{subsubsectionspecialcaseskhinchin}-\sectionsymbol\ref{subsubsectionverifying} concerning special cases of Theorem \ref{theoremkhinchin}. We will be interested in the following asymptotic formulas:

\begin{align} \label{Deltamuxiasymp}
\Delta_{\mu,\xi}(r) &\asymp_\times r^\delta \all r > 0 \\ \label{orbitalcountingfunctionasymp}
f_G(t) := \#\{g\in G:\dist(\zero,g(\zero))\leq t\} &\asymp_\times b^{\delta t} \all t > 0.
\end{align}

\begin{proposition}[Reduction of \eqref{khinchindiverges} and \eqref{khinchinconverges} assuming power law asymptotics]
\label{propositionkhinchinpowerlaw}
If \eqref{orbitalcountingfunctionasymp} holds then for each $K > 0$
\begin{equation}
\label{orbitalcountingfunctionasympimplies}
\eqref{khinchindiverges} \asymp_{\plus,\times} \int_{t = 0}^\infty \Delta_{\mu,\xi}(K^{-1} e^t \Phi(e^t))\dee t.
\end{equation}
If \eqref{Deltamuxiasymp} holds then for each $K > 0$
\begin{equation}
\label{Deltamuxiasympimplies}
\eqref{khinchindiverges} \asymp_\times \eqref{khinchinconverges} \asymp_\times \sum_{g\in G}\Phi^\delta(K b^{\dist(\zero,g(\zero))}).
\end{equation}
If both \eqref{orbitalcountingfunctionasymp} and \eqref{Deltamuxiasymp} hold then for each $K > 0$
\begin{equation}
\label{perfectkhinchin}
\eqref{khinchindiverges} \asymp_\times \eqref{khinchinconverges} \asymp_{\plus,\times} \int_{t = 0}^\infty e^{\delta t} \Phi^\delta(e^t) \dee t,
\end{equation}
and so by Theorem \ref{theoremkhinchin}
\begin{equation}
\label{perfectkhinchin2}
\mu(\WA_{\Phi,\xi}) = 0 \;\;\Leftrightarrow\;\; \int_{t = 0}^\infty e^{\delta t} \Phi^\delta(e^t) \dee t < \infty.
\end{equation}
\end{proposition}
\begin{proof}
We prove only \eqref{orbitalcountingfunctionasympimplies}, as \eqref{Deltamuxiasympimplies} is obvious, and \eqref{perfectkhinchin} is obvious given \eqref{orbitalcountingfunctionasympimplies}.

If $f:\Rplus\to(0,\infty)$ is any $\CC^1$ decreasing function for which $f(t)\tendsto t 0$, then
\begin{align*}
\sum_{g\in G}f(\dist(\zero,g(\zero)))
&=_\pt \sum_{g\in G}\int_{t = \dist(\zero,g(\zero))}^\infty (-f'(t))\dee t \since{$f(t)\tendsto t 0$}\\
&=_\pt \sum_{g\in G} \int_{t = 0}^\infty (-f'(t))\chi\big(t\geq \dist(\zero,g(\zero))\big)\dee t\\
&=_\pt \int_{t = 0}^\infty \sum_{g\in G} (-f'(t))\chi\big(\dist(\zero,g(\zero))\leq t\big)\dee t \by{Tonelli's theorem}\\
&=_\pt \int_{t = 0}^\infty (-f'(t))\#\{g\in G:\dist(\zero,g(\zero))\leq t\}\dee t\\
&\asymp_\times \int_{t = 0}^\infty (-f'(t))b^{\delta t}\dee t \by{\eqref{orbitalcountingfunctionasymp}}\\
&=_\pt f(0) + \log(b^\delta)\int_{t = 0}^\infty f(t) b^{\delta t}\dee t. & \textup{(integration by parts)}
\end{align*}
Thus
\begin{equation}
\label{fdogo}
\sum_{g\in G}f(\dist(\zero,g(\zero))) \asymp_\times f(0) + \int_{t = 0}^\infty f(t) b^{\delta t}\dee t.
\end{equation}
The implied constant is independent of $f$. Now \eqref{fdogo} makes sense even if $f$ is not $\CC^1$. An approximation argument shows that \eqref{fdogo} holds for any nonincreasing function $f:\Rplus\to(0,\infty)$ for which $f(t)\tendsto t 0$. Now let
\[
f(t) = b^{-\delta t}\Delta_{\mu,\xi}(b^t\Phi(K b^t)).
\]
Since by assumption $t\mapsto t\Phi(t)$ is nonincreasing, it follows that $f$ is also nonincreasing. On the other hand, clearly $f(t)\leq b^{-\delta t}$, so $f(t)\tendsto t 0$. Thus \eqref{fdogo} holds; substituting yields
\[
\eqref{khinchindiverges} \asymp_{\plus,\times} \int_{t = 0}^\infty \Delta_{\mu,\xi}(b^t \Phi(K b^t))\dee t.
\]
Substituting $s = t\log(b) + \log(K)$, we have
\begin{align*}
\eqref{khinchindiverges} \asymp_{\plus,\times}
\int_{t = 0}^\infty \Delta_{\mu,\xi}(b^t \Phi(K b^t))\dee t &=_\pt  \int_{s = \log(K)}^\infty \Delta_{\mu,\xi}(K^{-1}e^s\Phi(e^s))\log(b)\dee s\\
&\asymp_{\plus} \int_{s = 0}^\infty \Delta_{\mu,\xi}(K^{-1} e^s \Phi(e^s))\dee s.
\qedhere\end{align*}
\end{proof}

\begin{corollary}[Measures of $\BA_\xi$ and $\VWA_\xi$ assuming power law asymptotics]
\label{corollarykhinchinpowerlaw}
Assume $\xi\in\Lambda_G$. If \eqref{Deltamuxiasymp} holds, then $\mu(\VWA_\xi) = 0$, and $\mu(\BA_\xi) = 0$ if and only if $G$ is of divergence type. If \eqref{orbitalcountingfunctionasymp} holds, then $\mu(\BA_\xi) = 0$. If both hold, then $\mu(\BA_\xi) = \mu(\VWA_\xi) = 0$.
\end{corollary}
\begin{proof}
Fix $c > 0$, and let $\Phi_c(t) = t^{-(1 + c)}$. If \eqref{Deltamuxiasymp} holds, then
\[
\eqref{khinchinconverges} \asymp_\times \Sigma_{\delta(1 + c)}(G) < \infty,
\]
so by Theorem \ref{theoremkhinchin} we have $\mu(\WA_{\Phi_c,\xi}) = 0$. Since $c$ was arbitrary we have $\mu(\VWA_\xi) = 0$.

Now let $\Psi(t) = t^{-1}$. If \eqref{Deltamuxiasymp} holds, then
\[
\eqref{khinchindiverges} \asymp_\times \eqref{khinchinconverges} \asymp_\times \Sigma_\delta(G),
\]
so by Theorem \ref{theoremkhinchin}, $\mu(\BA_\xi) = \mu(\BA_{\Phi,\xi}) = 0$ if and only if $G$ is of divergence type. If \eqref{orbitalcountingfunctionasymp} holds rather than \eqref{Deltamuxiasymp}, then
\[
\eqref{khinchindiverges} \asymp_{\plus,\times} \int_{t = 0}^\infty \Delta_{\mu,\xi}(K^{-1})\dee t = \infty
\]
since $\xi\in\Lambda \;\;\Rightarrow\;\;\Delta_{\mu,\xi}(1/K) > 0$. Thus $\mu(\BA_\xi) = \mu(\BA_{\Phi,\xi}) = 0$.
\end{proof}

\begin{example}
\label{exampleuniformlyradialkhinchin}
If $\xi$ is a uniformly radial limit point, then \eqref{Deltamuxiasymp} holds.
\end{example}
The proof will show that if $\xi\in\Lursigma$, then the implied constant of \eqref{Deltamuxiasymp} depends only on $\sigma$.
\begin{proof}[Proof of Example \ref{exampleuniformlyradialkhinchin}]
Let $(g_n)_1^\infty$ be a sequence such that $g_n(\zero)\tendsto n \xi$ $\sigma$-uniformly radially. Fix $0 < r < 1$, and let $n = n_r\in\N$ be maximal such that
\[
b^{-\dist(\zero,g_n(\zero))}\geq r.
\]
Now for each $\eta\in B(\xi,r)$, we have
\[
\lb g_n(\zero)|\eta\rb_\zero \gtrsim_\plus \min(\lb g_n(\zero)|\xi\rb_\zero, \lb\eta|\xi\rb_\zero)\gtrsim_{\plus,\sigma} \min(\dist(\zero,g_n(\zero)),-\log_b(r)) = \dist(\zero,g_n(\zero)).
\]
Letting $\tau > 0$ be the implied constant of this asymptotic, we have
\[
B(\xi,r) \subset \Shad(g_n(\zero),\tau).
\]
Now by Sullivan's Shadow Lemma,
\[
\mu(B(\xi,r)) \leq \mu(\Shad(g_n(\zero),\tau)) \asymp_{\times,\tau} b^{-\delta\dist(\zero,g_n(\zero))}.
\]
On the other hand, by \eqref{sigmauniformlyradially}, we have
\[
b^{-\dist(\zero,g_n(\zero))} \asymp_{\times,\sigma} b^{-\dist(\zero,g_{n + 1}(\zero))} < r,
\]
and so
\[
\mu(B(\xi,r)) \lesssim_{\times,\sigma} r^\delta.
\]
The opposite direction is similar.
\end{proof}

\begin{example}[Quasiconvex-cocompact group]
\label{examplequasiconvexcocompact}
Let $X$ be a proper and geodesic hyperbolic metric space, and let $G\leq\Isom(X)$. The group $G$ is called \emph{quasiconvex-cocompact} if the set $G(\zero)$ is \emph{quasiconvex}, i.e. there exists $\rho > 0$ such that for all $x,y\in G(\zero)$,
\begin{equation}
\label{quasiconvexcocompact}
\geo xy \subset B(G(\zero),\rho).
\end{equation}
Any quasiconvex-cocompact group is perfect.
\end{example}
\begin{proof}
By Lemma \ref{lemmanonhorospherical} below, we have $\Lambda = \Lursigma$ for some $\sigma > 0$; without loss of generality, assume that $\sigma$ is large enough so that Sullivan's Shadow Lemma holds. In particular, \eqref{Deltamuxiasymp} follows from  Example \ref{exampleuniformlyradialkhinchin}. To obtain \eqref{orbitalcountingfunctionasymp}, we generalize an argument of Sullivan \cite[p. 180]{Sullivan_density_at_infinity}. Fix $t\geq 0$, and let $A_t = \{x\in G(\zero): (t - \sigma)\leq \dist(\zero,x) < t\}$. Then for all $\xi\in\Lambda = \Lursigma$, there exists $x\in A_t$ such that $\xi\in\Shad(x,\sigma)$. Let $A_{t,\xi} := \{x\in A_t: \xi\in\Shad(x,\sigma)\}$, so that $A_{t,\xi}\neq\emptyset$. On the other hand, by the Intersecting Shadows Lemma, we have for $x_1,x_2\in A_{t,\xi}$
\[
\dist(x_1,x_2) \asymp_{\plus,\sigma} |\busemann_\zero(x_1,x_2)| \asymp_{\plus,\sigma} 0,
\]
i.e. the set $A_{t,\xi} := \{x\in A_t: \xi\in\Shad(x,\sigma)\}$ has diameter bounded depending only on $\sigma$. Since $G$ is strongly discrete, this means that the cardinality of $A_{t,\xi}$ is bounded depending only on $\sigma$; since $A_{t,\xi}\neq\emptyset$, we have
\[
\#(A_{t,\xi}) \asymp_{\times,\sigma} 1.
\]
Thus
\begin{align*}
1 = \mu(X) &\asymp_{\times,\sigma} \int \#(A_{t,\xi}) \dee \mu(\xi)\\
&=_\ptsigma \sum_{x\in A_t}\mu(\Shad(x,\sigma))\\
&\asymp_{\times,\sigma} \sum_{x\in A_t}b^{-\delta \dist(\zero,x)}\by{Sullivan's Shadow Lemma}\\
&\asymp_{\times,\sigma} b^{-\delta t}\#(A_t).
\end{align*}
So $\#(A_t) \asymp_{\times,\sigma} b^{\delta t}$; thus for all $n\in\N$
\begin{align*}
f_G(n\sigma) = \sum_{i = 1}^n \#(A_{i\sigma})
&\asymp_{\times,\sigma} \sum_{i = 1}^n b^{\delta i\sigma}\\
&\asymp_{\times,\sigma} b^{\delta n\sigma},
\end{align*}
the last asymptotic following from the fact that $\sigma\delta > 0$. This proves \eqref{orbitalcountingfunctionasymp} in the case $t \in\N\sigma$; the general case follows from an easy approximation argument.
\end{proof}

\begin{proposition}
\label{propositionparabolic43}
Let $X = \H^{d + 1}$ for some $d\in\N$, and let $G\leq\Isom(X)$ be geometrically finite. If $\xi$ is a bounded parabolic point whose rank is strictly greater than $4\delta/3$, then there exists a function $\Phi$ such that $t\mapsto t\Phi(t)$ is nonincreasing and such that for every $K > 0$, \eqref{khinchinconverges} diverges but \eqref{khinchindiverges} converges.
\end{proposition}
\begin{proof}
Let $k = k_\xi$ be the rank of $\xi$. Fix a function $\Phi:(0,\infty)\to(0,\infty)$ to be announced below and fix any $K > 0$. Then by \cite[Lemma 3.2]{StratmannVelani} (a special case of the global measure formula),
\[
\eqref{khinchinconverges}
\asymp_\times \sum_{g\in G}(g'(\xi))^{\delta - k}\Phi^{2\delta - k}(K e^{\dist(\zero,g(\zero))}).
\]
Let $G_\xi$ be the stabilizer of $\xi$ in $G$, and let $S$ be a $\xi$-bounded set satisfying $G(\zero)\subset G_\xi(S)$ (see Definition \ref{definitionboundedparabolic}). Let $A$ be a transversal\Footnote{Recall that a \emph{transversal} of a partition is a set which intersects each partition element exactly once.} of $G_\xi\backslash G$ satisfying $a(\zero)\in S$ for all $a\in A$.
\begin{claim}
\label{claimSGxi}
For $a\in A$ and $h\in G_\xi$,
\[
\lb a(\zero)|h(\zero)\rb_\zero \asymp_\plus 0.
\]
\end{claim}
\begin{subproof}
By contradiction, suppose $\lb a_n(\zero)|h_n(\zero)\rb_\zero \tendsto n \infty$ with $a_n\in A$, $h_n\in G_\xi$. Then $\dist(\zero,h_n(\zero))\tendsto n \infty$, so by Observation \ref{observationparabolic} we have $h_n(\zero)\tendsto n \xi$. But then $a_n(\zero)\tendsto n \xi$, contradicting that $S$ is $\xi$-bounded.
\end{subproof}
Clearly, the same holds with $h$ replaced by $h^{-1}$, i.e.
\[
0 \asymp_\plus 2\lb a(\zero)|h^{-1}(\zero)\rb_\zero = \dist(\zero,a(\zero)) + \dist(\zero,h(\zero)) - \dist(\zero,h\circ a(\zero)).
\]
Writing $g = h\circ a$, we have
\begin{equation}
\label{gah}
\dist(\zero,g(\zero))\asymp_\plus \dist(\zero,a(\zero)) + \dist(\zero,h(\zero)).
\end{equation}
Note that every $g\in G$ can be written uniquely as $h\circ a$ with $h\in G_\xi$, $a\in A$ by the transversality of $A$. On the other hand,
\[
\busemann_\xi(g(\zero),\zero) = \busemann_\xi(h(\zero),\zero) + \busemann_{h^{-1}(\xi)}(a(\zero),\zero) = \busemann_\xi(a(\zero),\zero) = \dist(\zero,a(\zero)) - 2\lb a(\zero)|\xi\rb_\zero \asymp_\plus \dist(\zero,a(\zero)),
\]
the last asymptotic being a restatement of the fact that $\xi\notin\cl{A(\zero)}$. Thus
\begin{align*}
\eqref{khinchinconverges}
&\asymp_\times \sum_{g\in G}e^{(\delta - k)\busemann_\xi(g(\zero),\zero)}\Phi^{2\delta - k}(K e^{\dist(\zero,g(\zero))})\\
&\asymp_\times \sum_{a\in A}\sum_{h\in G_\xi}e^{(\delta - k)\dist(\zero,a(\zero))}\Phi^{2\delta - k}(K e^{\dist(\zero,h\circ a(\zero))})\\
&\geq_\pt \sum_{a\in A}\sum_{h\in G_\xi}e^{(\delta - k)\dist(\zero,a(\zero))}\Phi^{2\delta - k}(K e^{\dist(\zero,a(\zero)) + \dist(\zero,h(\zero))}). \since{$\Phi$ is nonincreasing}
\end{align*}
\begin{claim}
\[
\#(A_t) := \#\{a\in A:\dist(\zero,a(\zero))\leq t\} \asymp_\times e^{\delta t}.
\]
\end{claim}
\begin{proof}
By Example \ref{examplepinched}, we have $f_G(t)\asymp_\times e^{\delta t}$; in particular, since $\#(A_t)\leq f_G(t)$ we have
\begin{equation}
\label{Atupperbound}
\#(A_t) \lesssim_\times e^{\delta t}.
\end{equation}
To establish the lower bound, let $C$ be the implied constant of \eqref{gah}. We have
\begin{align*}
e^{\delta t} \asymp_\times e^{\delta(t - C)} \asymp_\times f_G(t - C)
&=_\pt \#\{(h,a)\in G_\xi\times A: \dist(\zero,h\circ a(\zero))\leq t - C\}\\
&\leq_\pt \#\{(h,a)\in G_\xi\times A: \dist(\zero,h(\zero)) + \dist(\zero,a(\zero))\leq t\} \by{\eqref{gah}}\\
&\leq_\pt \sum_{n = 0}^{\lfloor t\rfloor}\#\{(h,a)\in G_\xi\times A: n\leq \dist(\zero,h(\zero)) < n + 1, \dist(\zero,a(\zero)) \leq t - n\} \noreason\\
&=_\pt \sum_{n = 0}^{\lfloor t\rfloor}\#\{h\in G_\xi: n\leq \dist(\zero,h(\zero)) < n + 1\}\#(A_{t - n}) \noreason\\
&\leq_\pt \sum_{n = 0}^{\lfloor t\rfloor}\#\{h\in G_\xi: \dist(\zero,h(\zero)) < n + 1\}\#(A_{t - n})\\
&\asymp_\times \sum_{n = 0}^{\lfloor t\rfloor} e^{\delta_\xi n}\#(A_{t - n}).
\end{align*}
The last asymptotic follows from \eqref{powerlaw} and \eqref{euclideanparabolicasymp}. Now fix $N\in\N$ large to be determined. We have
\begin{align*}
e^{\delta t} &\lesssim_\times \sum_{n = 0}^{N - 1} e^{\delta_\xi n}\#(A_{t - n}) + \sum_{n = N}^{\lfloor t\rfloor} e^{\delta_\xi n}\#(A_{t - n})\\
&\lesssim_\times \sum_{n = 0}^{N - 1} e^{\delta_\xi n}\#(A_{t - n}) + \sum_{n = N}^{\lfloor t\rfloor} e^{\delta_\xi n}e^{\delta(t - n)} \by{\eqref{Atupperbound}}\\
&\leq_\pt \sum_{n = 0}^{N - 1} e^{\delta_\xi n}\#(A_{t - n}) + e^{\delta t}\sum_{n = N}^\infty e^{(\delta_\xi - \delta) n}\\
&\asymp_\times \sum_{n = 0}^{N - 1} e^{\delta_\xi n}\#(A_{t - n}) + e^{\delta t} e^{(\delta_\xi - \delta) N}. \since{$\delta_\xi < \delta$}
\end{align*}
Let $C_2$ be the implied constant, and let $N$ be large enough so that
\[
e^{(\delta - \delta_\xi)N} \geq 2C_2.
\]
Then
\[
e^{\delta t} \leq C_2 \sum_{n = 0}^{N - 1} e^{\delta_\xi (n + 1)}\#(A_{t - n}) + \frac{1}{2}e^{\delta t};
\]
rearranging gives
\[
e^{\delta t}
\lesssim_\times \sum_{n = 0}^{N - 1} e^{\delta_\xi (n + 1)}\#(A_{t - n})
\lesssim_{\times,N} \#(A_t).
\]
\end{proof}

Thus by an argument similar to the argument used in the proof of Proposition \ref{propositionkhinchinpowerlaw}, we have
\[
\sum_{a\in A}f(\dist(\zero,a(\zero))) \asymp_\times f(0) + \int_{t = 0}^\infty f(t) e^{\delta t}\dee t
\]
for every nonincreasing function $f\to 0$. Thus
\begin{align*}
\eqref{khinchinconverges}
&\gtrsim_\times \sum_{a\in A}\sum_{h\in G_\xi}e^{(\delta - k)\dist(\zero,a(\zero))}\Phi^{2\delta - k}(K e^{\dist(\zero,a(\zero)) + \dist(\zero,h(\zero))})\\
&\gtrsim_\times \int_{t = 0}^\infty e^{\delta t}\left[\sum_{h\in G_\xi}
e^{(\delta - k)t}\Phi^{2\delta - k}(K e^{t + \dist(\zero,h(\zero))})\right]\dee t.
\end{align*}
Substituting $s = t + \dist(\zero,h(\zero))$ (and letting $\|h\| = \dist(\zero,h(\zero))$),
\[
\eqref{khinchinconverges} \gtrsim_{\plus,\times} \int_{s = 0}^\infty\sum_{\substack{h\in G_\xi \\ \|h\|\leq s}}
e^{(2\delta - k)[t - \|h\|]}\Phi^{2\delta - k}(K e^s)\dee s.
\]
Now
\begin{align*}
\sum_{\substack{h\in G_\xi \\ \|h\|\leq s}}e^{(2\delta - k)[-\|h\|]}
&\asymp_\times \int_{\substack{\xx\in\R^k \\ \|\xx\| \leq s}}(1 + \|\xx\|^2)^{2\delta - k}\dee\xx\\
&\asymp_\times \max(s^{3k - 4\delta},1)
\end{align*}
and so
\[
\eqref{khinchinconverges}
\gtrsim_{\plus,\times} \int_{s = 0}^\infty\max(s^{3k - 4\delta},1)
e^{(2\delta - k)s}\Phi^{2\delta - k}(K e^s)\dee s.
\]
In particular, since $3k - 4\delta > 0$, $\Phi$ may be chosen in such a way so that this integral diverges whereas the integral
\[
\eqref{khinchindiverges}
\asymp_\times \int_{s = 0}^\infty
e^{(2\delta - k)s}\Phi^{2\delta - k}(K e^s)\dee s
\]
converges. For example, let
\[
\Phi(e^s) = e^{-s}s^{[-1 - (3k - 4\delta)]/(2\delta - k)};
\]
then the series reduce to
\begin{align*}
\eqref{khinchinconverges} &\gtrsim_{\plus,\times} \int_{s = 1}^\infty s^{-1}\dee s = \infty,\\
\eqref{khinchindiverges} &\asymp_{\times\phantom{,\plus}} \int_{s = 1}^\infty s^{-1 - (3k - 4\delta)}\dee s < \infty.
\end{align*}
\end{proof}

\subsection{The measure of $\Lur$}
We end this section by proving the following proposition, which was stated in the introduction:
\begin{proposition}
\label{propositionmuLurzero}
Let $(X,\dist,\zero,b,G)$ be as in \sectionsymbol\ref{standingassumptions}, with $X$ proper and geodesic. Let $\mu$ be a $\delta_G$-quasiconformal measure on $\Lambda$. Then $\mu(\Lur) = 0$ if and only if $G$ is not quasiconvex-cocompact.
\end{proposition}

To prove Proposition \ref{propositionmuLurzero} we will need the following lemmas:

\begin{lemma}
\label{lemmaLurBA}
Let $(X,\dist,\zero,b,G)$ be as in \sectionsymbol\ref{standingassumptions}. If $\xi\in\del X\butnot\Lh$, then
\[
\Lur \subset \BA_\xi.
\]
\end{lemma}
\begin{proof}
By contradiction fix $\eta\in \Lur\cap\WA_\xi$. Let $g_m(\zero)\tendsto m \eta$ uniformly radially, and let $(h_n)_1^\infty$ be such that
\[
\Dist(h_n(\xi),\eta) \leq b^{-\dist(\zero,h_n(\zero)) - n}.
\]
For each $n$, there exists $m_n$ such that $\dist(\zero,g_{m_n}(\zero)) \asymp_\plus \dist(\zero,h_n(\zero)) + n$. Let $h = h_n$ and $g = g_{m_n}$; we have
\begin{align*}
\busemann_\xi(\zero,h^{-1}\circ g(\zero))
&\asymp_\plus \busemann_{h(\xi)}(h(\zero),g(\zero))\\
&\asymp_\plus \busemann_{h(\xi)}(h(\zero),\zero) + \busemann_{h(\xi)}(\zero,g(\zero))\\
&\gtrsim_\plus 2\lb h(\xi),g(\zero)\rb_\zero - \dist(\zero,h(\zero)) - \dist(\zero,g(\zero))\\
&\asymp_\plus 2\lb h(\xi),g(\zero)\rb_\zero - 2\dist(\zero,h(\zero)) - n.
\end{align*}
Now,
\begin{align*}
\lb h(\xi),g(\zero)\rb_\zero
&\gtrsim_\plus \min(\lb h(\xi),\eta\rb_\zero,\lb g(\zero),\eta\rb_\zero) \by{Gromov's inequality}\\
&\gtrsim_\plus \min(\dist(\zero,h(\zero)) + n,\dist(\zero,g(\zero))) \since{$\eta\in\Shad(g(\zero),\sigma)$}\\
&\asymp_\plus \dist(\zero,h(\zero)) + n,
\end{align*}
and so
\[
\busemann_\xi(\zero,h^{-1}\circ g(\zero)) \gtrsim_\plus n.
\]
Thus $h_n^{-1}\circ g_{m_n}(\zero)\tendsto n \xi$ horospherically, contradicting that $\xi\notin\Lh$.
\end{proof}

\begin{lemma}
\label{lemmanonhorospherical}
Let $(X,\dist,\zero,b,G)$ be as in \sectionsymbol\ref{standingassumptions}, with $X$ proper and geodesic. The following are equivalent:
\begin{itemize}
\item[(A)] $G$ is quasiconvex-cocompact.
\item[(B)] $\Lambda = \Lursigma$ for some $\sigma > 0$.
\item[(C)] $\Lambda = \Lh$.
\end{itemize}
\end{lemma}
\begin{proof}
(B) \implies (C) is obvious as $\Lursigma \subset \Lur \subset \Lh$. To demonstrate (A) \implies (B), suppose that $G$ is quasiconvex-cocompact, let $\rho > 1$ be as in the definition of quasiconvex-cocompact, and fix $\xi\in\Lambda$. Fix $n\in\N$, and let $z_n\in G(\zero)$ be such that
\[
\lb z_n|\xi\rb_\zero \geq n;
\]
such a $z_n$ exists since $\xi\in\Lambda$. Let $y_n\in\geo\zero{z_n}$ be the unique point so that $\dist(\zero,y_n) = n$. Fix $x_n\in G(\zero)\cap B(y_n,\rho)$; such a point exists by \eqref{quasiconvexcocompact}. Without loss of generality, we let $x_1 = \zero$; this is possible since $\zero\in G(\zero)\cap B(y_1,\rho)$.
\begin{claim}
$x_n \tendsto n \xi$ $\sigma$-uniformly radially, with $\sigma > 0$ independent of $\xi$.
\end{claim}
\begin{subproof}
\begin{align*}
\lb x_n|\xi\rb_\zero &\gtrsim_\plus \min(\lb x_n|y_n\rb_\zero,\lb y_n|z_n\rb_\zero,\lb z_n|\xi\rb_\zero)\\
&\geq_\pt \min(\dist(\zero,y_n) - \dist(\zero,x_n), \dist(\zero,y_n),n)\\
&\geq_\pt \min(n - \rho,n,n) = n - \rho \asymp_\rho n \asymp_\rho \dist(\zero,x_n).
\end{align*}
This demonstrates that $x_n \tendsto n \xi$ radially. Now the Intersecting Shadows Lemma together with the asymptotic $\dist(\zero,x_n) \asymp_{\plus,\rho} n$ implies that $x_n \tendsto n \xi$ uniformly radially. The implied constants depend only on $\rho$ and not on $\xi$.
\end{subproof}

Finally, we demonstrate (C) \implies (A). Let
\[
\DD_{\zero,1} = \{x\in X: \busemann_x(\zero,g(\zero)) \leq 1 \all g\in G\}
\]
be the \emph{$1$-approximate Dirichlet domain centered at $\zero$}, and let
\[
\CC_G = \bigcup_{x,y\in G(\zero)} \geo xy
\]
be the \emph{approximate convex core} of $G$.
\begin{claim}
$\CC_G\cap \DD_{\zero,1}$ is bounded.
\end{claim}
\begin{subproof}
Since $X$ is proper and geodesic, $\bord X$ is compact \cite[Exercise III.H.3.18(4)]{BridsonHaefliger}, and so if $\CC_G\cap \DD_{\zero,1}$ is unbounded, then there is a point $\xi\in \cl{\CC_G\cap \DD_{\zero,1}} \cap \del X$. We claim that $\xi\in\Lambda\butnot\Lh$, contradicting our hypothesis. Indeed, by Lemma \ref{lemmanearcontinuity}, since $\xi\in\cl{\DD_{\zero,1}}$ we have $\busemann_\xi(\zero,g(\zero))\lesssim_\plus 1 \all g\in G$, which demonstrates that $\xi\notin\Lh$. On the other hand, since $\xi\in\cl{\CC_G}$, there exists a sequence $x_n\tendsto n \xi$ in $\CC_G$; for each $n\in\N$, there exist $y_n,z_n\in G(\zero)$ such that $x_n\in \geo{y_n}{z_n}$. Without loss of generality suppose that $\lb x_n|y_n\rb_\zero \leq \lb x_n|z_n\rb_\zero$ Now
\begin{align*}
2\lb x_n|z_n\rb_\zero &\geq \lb x_n|y_n\rb_\zero + \lb x_n|z_n\rb_\zero\\
&= \dist(\zero,x_n) + \lb y_n|z_n\rb_\zero \since{$\lb y_n|z_n\rb_{x_n} = 0$}\\
&\geq \dist(\zero,x_n) \tendsto n \infty.
\end{align*}
Thus $z_n \tendsto n \xi$, and so $\xi\in\Lambda$.
\end{subproof}
Write $\CC_G\cap \DD_{\zero,1}\subset B(\zero,\rho)$ for some $\rho > 0$. We claim that \eqref{quasiconvexcocompact} holds for all $x,y\in G(\zero)$. Indeed, fix $z\in\geo xy$, and let $g\in G$ be such that
\[
\dist(z,g(\zero)) < \dist(z,G(\zero)) + 1.
\]
Then $g^{-1}(z)\in \geo{g^{-1}(x)}{g^{-1}(y)}\subset\CC_G$; moreover, for all $h\in G$
\begin{align*}
\busemann_{g^{-1}(z)}(\zero,h(\zero)) &\leq \dist(\zero,g^{-1}(z)) - \dist(g^{-1}(z),G(\zero))\\
&= \dist(z,g(\zero)) - \dist(z,G(\zero)) \leq 1,
\end{align*}
i.e. $g^{-1}(z)\in\DD_{\zero,1}$. Thus $g^{-1}(z)\in B(\zero,)$, and in particular
\[
z\in B(g(\zero),\rho) \subset B(G(\zero),\rho).
\]
\end{proof}

\begin{proof}[Proof of Proposition \ref{propositionmuLurzero}]
Suppose first that $G$ is quasiconvex-cocompact. Then by Lemma \ref{lemmanonhorospherical}, $\Lambda = \Lur$, so $\mu(\Lur) = 1$.

Next, suppose that $G$ is not quasiconvex-cocompact. By Lemma \ref{lemmanonhorospherical}, there exists a point $\xi\in\Lambda\butnot\Lh$, and by Lemma \ref{lemmaLurBA} we have $\Lur \subset \BA_\xi$. Now if $G$ is of divergence type, we are done, since by (i) of Theorem \ref{theoremkhinchin} we have $\mu(\BA_\xi) = 0$. On the other hand, if $G$ is of convergence type, then for each $\sigma > 0$ we have
\[
\sum_{g\in G}\mu(\Shad(g(\zero),\sigma) \lesssim_\times \sum_{g\in G} b^{-\delta\dist(\zero,g(\zero))} < \infty,
\]
and so by the Borel--Cantelli lemma, we have $\mu(\Lrsigma) = 0$. Since $\sigma$ was arbitrary we have $\mu(\Lr) = 0$, and in particular $\mu(\Lur) = 0$.
\end{proof}

\draftnewpage\section{Proof of Theorem \ref{theoremextrinsic} ($\BA_d$ has full dimension in $\Lr(G)$)} \label{sectionextrinsic}
In this section, fix $d\in\N$, let $X = \H^{d + 1}$, let $Z = \R^d$, and let $G\leq\Isom(X)$ be a discrete group. We are interested in approximating the points of $\Lr(G)$ by rational vectors. We recall the following definitions: 

\begin{definition}
A vector $\xx\in\R^d$ is called \emph{badly approximable} if for all $\pp/q\in\Q^d$, we have
\[
\left\|\xx - \frac{\pp}{q}\right\| \gtrsim_\times \frac{1}{q^{1 + 1/d}}.
\]
The set of badly approximable vectors in $\R^d$ will be denoted $\BA_d$.
\end{definition}

\begin{definition}
We define the \emph{geodesic span} of a set $S\subset\bord{\H^{d + 1}}$ to be the smallest totally geodesic subspace of $\H^{d + 1}$ whose closure contains $S$.

We say that $G$ \emph{acts irreducibly on $\H^{d + 1}$} if it is nonelementary and if there is no nonempty proper\Footnote{In this section ``proper'' means ``not equal to the entire space'', not ``bounded closed sets are compact''.} totally geodesic subspace $V\subset \H^{d + 1}$ such that $G(V) = V$. Equivalently, $G$ acts irreducibly if there is no proper totally geodesic subspace of $\H^{d + 1}$ whose closure contains $\Lambda_G$.
\end{definition}

The remainder of this section will be devoted to proving the following theorem:

\begin{theorem}
\label{theoremextrinsic}
Fix $d\in\N$, let $X = \H^{d + 1}$, and let $G\leq\Isom(X)$ be a discrete group acting irreducibly on $X$. Then
\[
\HD(\BA_d\cap \Lur) = \delta = \HD(\Lr),
\]
where $\BA_d$ denotes the set of badly approximable vectors in $\R^d$.
\end{theorem}

\begin{notation}
There are two natural metrics to put on $\R^d$: the usual Euclidean metric
\[
\Dist_\euc(\xx,\yy) := \|\xx - \yy\|,
\]
and the \emph{spherical metric}
\[
\Dist_\sph := \Dist_{e,\zero},
\]
where $\zero = (1,0,\ldots,0)\in \H^{d + 1}$. From now on, we will use the subscripts $\euc$ and $\sph$ to distinguish between these metrics.
\end{notation}

We recall the following definition and theorem:

\begin{definition}
\label{definitionhyperplanediffuse}
A closed set $K\subset\R^d$ is said to be \emph{hyperplane diffuse} if there exists $\gamma > 0$ such that for every $\xx\in K$, for every $0 < r\leq 1$, and for every affine hyperplane $\LL\subset\R^d$, we have
\[
K\cap B_\euc(\xx,r)\setminus \NN_\euc(\LL,\gamma r) \neq \emptyset.
\]
Here $\NN_\euc(\LL,\gamma r)$ denotes the $r$-thickening of $\LL$ with respect to the Euclidean metric $\Dist_\euc$.
\end{definition}



\begin{theorem} \label{theoremBFKRW}
Let $K$ be a subset of $\R^d$ which is Ahlfors regular and hyperplane diffuse. Then
\[
\HD(\BA_d\cap K) = \HD(K).
\]
\end{theorem}
\begin{proof}
See Theorems 2.5, 4.7, and 5.3 from \cite{BFKRW}.
\end{proof}

\begin{remark}
The conclusion of Theorem \ref{theoremBFKRW}, and thus of Theorem \ref{theoremextrinsic}, holds when $\BA_d$ is replaced by any subset of $\R^d$ which is hyperplane absolute winning (see \cite{BFKRW} for the definition).
\end{remark}

Fix $0 < s < \delta_G$. According to Corollary \ref{corollarystructure}, there exists an Ahlfors $s$-regular set $\limitset_s\subset\Lur(G)$.
\begin{claim}
\label{claimmodification}
The proof of Corollary \ref{corollarystructure} can be modified so that the resulting set $\limitset_s$ is hyperplane diffuse.
\end{claim}
\begin{proof}[Proof of Theorem \ref{theoremextrinsic} assuming Claim \ref{claimmodification}]
By Theorem \ref{theoremBFKRW}, we have $\HD(\BA_d\cap\Lur)\geq \HD(\limitset_s) = s$ for every $0 < s < \delta_G$; letting $s\to\delta$ finishes the proof.
\end{proof}
\begin{proof}[Proof of Claim \ref{claimmodification}]
The first modification to be made is in the proof of Lemma \ref{lemmaDSU} assuming Sublemma \ref{sublemmaDSU}. Instead of picking just two $t$-divergent points $\eta_1$ and $\eta_2$, let us notice that the set of $t$-divergent points is closed and invariant under the action of the group. Thus by a well-known theorem, the set of $t$-divergent points contains the entire limit set of $G$.

Recall that a collection of $(d + 2)$ points $\eta_1,\ldots,\eta_{d + 2}\in\del X$ is in \emph{general position} if there is no totally geodesic subspace $V\subset \H^{d + 1}$ such that $\eta_1,\ldots,\eta_{d + 2}\in \cl V$.

\begin{claim}
There exist $(2d + 4)$ points $\eta_1,\ldots,\eta_{2d + 4}\in\Lambda_G$, such that for each $i = 1,\ldots,d + 2$, the collection
\begin{equation}
\label{etad2i}
\eta_1,\ldots,\what{\eta_i},\ldots,\eta_{d + 2},\eta_{d + 2 + i}
\end{equation}
is in general position. (Here the hat indicates that $\eta_i$ is omitted.)
\end{claim}
\begin{proof}
First, pick $(d + 2)$ points $\eta_1,\ldots,\eta_{d + 2}\in \Lambda_G$ in general position; this is possible since $G$ acts irreducibly on $\H^{d + 1}$. Then, for each $i = 1,\ldots,d + 2$, we may choose $\eta_{d + 2 + i}\in\Lambda_G$ so that the collection \eqref{etad2i} is in general position by choosing $\eta_{d + 2 + i}\in\Lambda_G$ sufficiently close to $\eta_i$; this is possible since by a well-known theorem, the limit set $\Lambda_G$ is perfect.
\end{proof}

The following observation can be proven by an elementary compactness argument:
\begin{observation}
\label{observationcompactness}
For each $i = 1,\ldots,d + 2$, there exists $\varepsilon_i > 0$ such that if $V$ is a proper geodesic subspace of $X$, if $\del V := \cl V\cap \del X$, and if $\NN_\sph(\del V,\varepsilon_i)$ is the $\varepsilon_i$-thickening of $\del V$ with respect to the spherical metric, then $\NN_\sph(\del V,\varepsilon_i)$ does not contain all $(d + 2)$ of the points \eqref{etad2i}.
\end{observation}
Let
\[
\varepsilon = \frac{1}{2}\min_{i = 1}^{d + 2}\varepsilon_i.
\]
We observe that $\Dist_\sph(\eta_i,\eta_j)\geq 2\varepsilon$ for all $i\neq j$. For each $i = 1,\ldots,d + 2$ let
\[
B_i = B_\sph(\eta_i,\varepsilon)\cup\{x\in X:\Shad(x,\sigma)\subset B_\sph(\eta_i,\varepsilon)\};
\]
the sets $(B_i)_1^{2d + 4}$ are open and disjoint.

Now apply Sublemma \ref{sublemmaDSU} to each pair $(\eta_i,B_i)$ for $i = 1,\ldots,2d + 4$ to get sets $S_i\subset G(\zero)\cap B_i$. For each $i = 1,\ldots,2d + 4$, fix a point $p_i\in S_i$. Suppose that $w = g_w(\zero)\in G(\zero)$, and let $z = g_w^{-1}(\zero)$. Let
\[
T(w) = \bigcup_{\substack{i = 1 \\ z\notin B_i}}^{2d + 4}g_w(S_i).
\]
Note that at most one index $i$ is omitted from the union. Now (i)-(iii) of Sublemma \ref{sublemmaDSU} directly imply (i)-(iii) of Lemma \ref{lemmaDSU}. Note that here we need the fact that $\PP_{z,x}\subset B_i$ for $x\in S_i$ to prove disjointness in (i) of Lemma \ref{lemmaDSU} (whereas we did not need it in the earlier proof).

The next change to be made in the proof of Corollary \ref{corollarystructure} is in the application of Theorem \ref{theoremahlforsgeneral} to Lemma \ref{lemmastructure}. This time, we will want to use the second part of Theorem \ref{theoremahlforsgeneral} - the tree $\w T$ can be chosen so that for each $\omega\in \w T$, we have that $\w T(\omega)$ is an initial segment of $T(\omega)$. Obviously, the application of this additional information depends on the ordering which is given to the set $\w T(\omega)$. In the original proof of Lemma \ref{lemmastructure}, we used an ordering which had no meaning - any enumeration of $G(\zero)$. Now, we will choose some ordering so that the points $g_w(p_1),\ldots,g_w(p_{2d + 4})$ (possibly minus one) mentioned above are first in the ordering.

By increasing $n_1,\ldots,n_{2d + 4}$ if necessary, we can ensure that the points $g_w(p_1),\ldots,g_w(p_{2d + 4})$ are all in $\w T(w)$ for every $w$, since there is no way that \eqref{regularitya} could hold if $N_w < 2d + 4$.

Now we have constructed the set $\limitset_s$, and we come to the crucial point: with these modifications, the new set $\limitset_s$ is hyperplane diffuse.
\begin{definition}
A set $K\subset \what{\R^d}$ is \emph{hyperbolically hyperplane diffuse} if there exists $\gamma > 0$ so that every $\eta\in K$, for every $0 < r \leq 1$, and for every totally geodesic subspace $V\subset X$, we have
\[
K\cap B_\sph(\eta,r)\setminus \NN_\sph(\del V,\gamma r) \neq \emptyset.
\]
\end{definition}
Here (and from now on), balls and thickenings are all assumed to be with respect to the spherical metric $\Dist_\sph$.
\begin{proposition}
Hyperbolically hyperplane diffuse sets are hyperplane diffuse.
\end{proposition}
(Proof left to the reader.)

\noindent Thus we are left with showing that $\limitset_s$ is hyperbolically hyperplane diffuse.

Fix $\gamma > 0$ to be announced below, $\eta\in \limitset_s$, $0 < r \leq \varepsilon S_\smallemptyset$, and a totally geodesic subspace $V\subset X$.
\begin{notation}
For each $x\in G(\zero)$ let
\[
\PP_x := \Shad(x,\sigma).
\]
Let $(x_n)_1^\infty$ be an enumeration of $G(\zero)$. Let
\[
T = \bigcup_{n = 1}^\infty\{\omega\in\N^n:x_{\omega_{j + 1}}\in T(x_{\omega_j})\all j = 1,\ldots,n - 1\},
\]
and for each $\omega\in T$ let
\[
x_\omega = x_{\omega_{|\omega|}} \text{ and } \PP_\omega = \PP_{x_\omega}.
\]
Let $\eta = \pi(\omega)$; for each $n\in\N$ let
\[
x_n = x_{\omega_1^n}, \;\;
\PP_n = \PP_{\omega_1^n}, \text{ and }
D_n = D_{\omega_1^n}.
\]
\end{notation}

Let $n\in\N$ be the largest integer such that $r < \kappa D_n$. Then by the argument used in the proof of Theorem \ref{theoremahlforsgeneral}, we have
\[
\kappa^2 D_n \leq r \leq \kappa D_n
\]
and
\[
\PP_{n + k} \subset B(\eta,r) \subset \PP_n,
\]
where $k$ is large enough so that $\lambda^k \leq \kappa^2$. In particular
\[
\limitset_s\cap B(\eta,r)\butnot \NN(\del V,\gamma r) \supset \limitset_s\cap \PP_{n + k}\butnot \NN(\del V,\gamma r),
\]
so to complete the proof it suffices to show
\begin{equation}
\label{JsBnk}
\limitset_s\cap \PP_{n + k}\nsubset \NN(\del V,\gamma r).
\end{equation}
By contradiction suppose not. Let $g\in G$ be such that $g(\zero) = x_{n + k}$, and let $z = g^{-1}(\zero)$. Then for some $j = 1,\ldots,d + 2$ we have
\[
g(p_1),\ldots,\what{g(p_j)},\ldots,g(p_{d + 2}),g(p_{d + 2 + j})\in \w T(x_{n + k}).
\]
For each $i = 1,\ldots,\what{j},\ldots,d + 2,d + 2 + j$, we have
\[
\emptyset \neq \limitset_s\cap\PP_{g(p_i)} \subset \limitset_s\cap\PP_{n + k} \subset \NN(\del V,\gamma r),
\]
which implies
\[
\PP_{g(p_i)}\cap \NN(\del V,\gamma r)\neq \emptyset
\]
and thus
\[
\Shad_z(p_i,\sigma) \cap g^{-1}\left(\NN(\del V,\gamma r)\right) \neq \emptyset.
\]
On the other hand, by Sublemma \ref{sublemmaDSU} we have
\[
\Shad_z(p_i,\sigma)\subset B_i,
\]
so
\[
B_i\cap g^{-1}\left(\NN(\del V,\gamma r)\right) \neq \emptyset.
\]
Now,
\[
g^{-1}\left(\NN(\del V,\gamma r)\right) \subset \NN\left(\del g^{-1}(V),e^{\dist(\zero,g(\zero))}\gamma r)\right);
\]
on the other hand
\[
r \asymp_\times D_n \asymp_\times D_{n + k} \asymp_{\times,\sigma} e^{-\dist(\zero,x_{n + k})} = e^{-\dist(\zero,g(\zero))};
\]
letting $C_4 > 0$ be the implied asymptotic, we have
\[
g^{-1}\left(\NN(\del V,\gamma r)\right) \subset \NN(\del W,C_4\gamma),
\]
where $W = g^{-1}(V)$. Thus
\[
B_i\cap \NN(\del W,C_4\gamma) \neq\emptyset.
\]
Letting $\gamma = \varepsilon/C_4$, this contradicts Observation \ref{observationcompactness}.
\end{proof}

\draftnewpage
\appendix
\section{Any function is an orbital counting function for some parabolic group} \label{appendixorbitalcounting}
Let $(X,\dist,\zero,b,G)$ be as in \sectionsymbol\ref{standingassumptions}. In Subsection \ref{subsectionboundedparabolicjarnik}, we defined the \emph{orbital counting function} of a parabolic fixed point $\xi$ of $G$ to be the function
\begin{repequation}{orbitalcountingfunction}
f_\xi(R) := \#\{g\in G_\xi:\Dist_\euc(\zero,g(\zero))\leq R\},
\end{repequation}
where $G_\xi$ is the stabilizer of $\xi$ in $G$ and $\Dist_\euc = \Dist_{\xi,\zero}$. We remarked that when $X$ is a real, complex, or quaternionic hyperbolic space, the orbital counting function $f_\xi$ satisfies a power law, but that this is not true in general. Since the proof of Theorem \ref{theoremboundedparabolicjarnik} would be simplified somewhat if $f_\xi$ were assumed to satisfy a power law, we feel that it is important to give some examples where it is not satisfied, to justify the extra work. Moreover, these examples corroborate the statement made in Remark \ref{remarksurprising} that the fact that \eqref{velanihillgeneral} depends only on the Poincar\'e exponents $\delta$ and $\delta_\xi$ is surprising.

To begin with, let us give some background on the infinite-dimensional hyperbolic space $\H^\infty$. Its boundary $\HH := \del\H^\infty$ is itself a Hilbert space. As in finite dimensions, any isometry of $\HH$ with respect to the Euclidean metric extends uniquely to an isometry of $\H^\infty$ which $\infty$. So immediately, there is a correspondence between parabolic subgroups of $\Stab(\Isom(\H^\infty);\infty)$ and subgroups of $\Isom(\HH)$ whose orbits are unbounded. However, unlike in finite dimensions, strongly discrete subgroups of $\Isom(\HH)$ are not necessarily virtually nilpotent. Groups which embed as strongly discrete subgroups of $\Isom(\HH)$ are said to have the \emph{Haagerup property}, and they include both the amenable groups and the free groups. Moreover, even cyclic subgroups of $\Isom(\HH)$ are quite different from cyclic subgroups of $\Isom(\R^d)$ for $d < \infty$; for example, a well-known example of M. Edelstein \cite{Edelstein} is a cyclic subgroup of $\Isom(\HH)$ whose orbits are unbounded but which is not strongly discrete.

To relate the orbital counting function $f_\xi$ with the intrinsic structure of the Hilbert space $\HH$, we need the following lemma:
\begin{lemma}
\label{lemmaDistH}
For $g\in\Stab(\Isom(\H^\infty);\infty)$,
\[
\Dist_\euc(\ee_0,g(\ee_0)) \asymp_\times \max(1,\|g(\ee_0) - \ee_0\|) = \max(1,\|g(\0)\|).
\]
\end{lemma}
\begin{proof}
By \eqref{euclideanparabolicasymp}, we have
\[
\Dist_\euc(\ee_0,g(\ee_0)) \asymp_\times e^{(1/2)\dist(\ee_0,g(\ee_0))}.
\]
The remaining asymptotic is a geometric calculation which is left to the reader, and the equality follows from the fact that $g$ is an affine map whose linear part preserves $\ee_0$.
\end{proof}

For the remainder of the appendix, rather than working with $f_\infty$, we will work with the modification
\[
\w f_\infty(R) := \{g\in G_\infty:\|g(\0)\|\leq R\}.
\]
The relation between $f_\infty$ and $\w f_\infty$ can be deduced easily from Lemma \ref{lemmaDistH}.

\begin{proposition}
\label{propositionorbitalcounting2}
For any nondecreasing function $f: 2^\N\to 2^\N$ with $f\to\infty$, there exists a parabolic group $G_\infty\leq\Stab(\Isom(\H^\infty);\infty)$ such that
\[
\w f_\infty(2^n) = f(2^n)
\]
for all $n\in\N$.
\end{proposition}
\begin{remark}
A similar statement holds for proper $\R$-trees.
\end{remark}
\begin{proof}
For each $n\in\N$, let
\[
m_n = \frac{f(2^{n + 1})}{f(2^n)} \in 2^\N
\]
and let $T_n:\R^{m_n}\to\R^{m_n}$ be an isometry of order $m_n$ such that any two points in the orbit of $\0$ have a distance of exactly $2^{n - 1/2}$. For example, we can let
\[
T_n(\xx) = \big(x_{k - 1}\big)_{k = 1}^{m_n} + 2^{n - 1}(\ee_2 - \ee_1).
\]
(Here by convention $x_0 = x_{m_n}$.) We will also denote by $T_n$ the extension of $T_n$ to the product
\[
\HH := \bigoplus_{n = 1}^\infty \R^{m_n}
\]
obtained by letting $T_n$ act as the identity in the other dimensions. By construction, the $T_n$s commute. Let $G = \lb T_n\rb_{n\in\N}$. Then each element of $G$ can be written in the form $T = \prod_{n = 1}^\infty T_n^{q_n}$, where $(q_n)_1^\infty$ is a sequence of integers satisfying $0\leq q_n < m_n$, only finitely many of which are nonzero. Such an isometry satisfies
\[
\|T(\0)\|^2 = \sum_{n = 1}^\infty \|T_n^{q_n}\|^2
= \sum_{\substack{n\in\N \\ q_n\neq 0}} 4^{n - 1/2}
\in \left[\frac{1}{2} 4^{n_{\max}}, 4^{n_{\max}}\right],
\]
where
\[
n_{\max} = \max\{n\in\N:q_n\neq 0\}.
\]
In particular, for all $N\in\N$
\[
\|T(\0)\| \leq 2^N \;\;\Leftrightarrow\;\; N \geq n_{\max},
\]
and thus
\[
\w f_\infty(2^N) = \#\{(q_n)_1^\infty: n_{\max} \leq N\} = \#\{(q_n)_1^N\} = \prod_{n = 1}^N m_n = f(2^N).
\]
\end{proof}

\begin{corollary}
For any nondecreasing function $f: (0,\infty) \to (0,\infty)$ with $f\to\infty$, there exists a parabolic group $G\leq\Stab(\Isom(\H^\infty);\infty)$ such that
\[
\w f_\infty(2^n) \asymp_\times f(2^n)
\]
for all $n\in\N$.

If the function $f$ satisfies $f(2^{n + 1}) \asymp_\times f(2^n)$, then
\[
\w f_\infty(R) \asymp_\times f(R) \all R\geq 1.
\]
\end{corollary}
\begin{proof}
The first assertion is immediate upon applying the above example to the function $\w f(2^n) = 2^{\lfloor \log_2 f(2^n)\rfloor}$, and the second follows from the fact that the functions $f$ and $f_\xi$ are both nondecreasing.
\end{proof}

\draftnewpage
\section{Real, complex, and quaternionic hyperbolic spaces}\label{appendixsymmetricspaces}
In this appendix, we give some background on the ``non-exceptional'' rank one symmetric spaces of noncompact type, i.e. real, complex, and quaternionic hyperbolic spaces, and prove the assertions made about them in Remarks \ref{remarkreductionjarnik} and \ref{remarkreductionkhinchin} and Example \ref{examplekhinchinlattice1}. Our presentation loosely follows \cite[\sectionsymbol II.10]{BridsonHaefliger}.

Let $\R$, $\C$, and $\H$ denote the sets of real, complex, and quaternionic numbers, respectively. Fix $\F \in \{\R,\C,\H\}$ and $d\in\N$. Note that since the noncommutative $\F = \H$ is a possibility, we will have to be careful about the order products are written in. We define a sesquilinear inner product on $\F^{d + 1}$ by
\[
\lb\xx,\yy\rb := \sum_{i = 1}^{d + 1}\wbar{x_i}y_i.
\]
Let $\|\xx\| = \sqrt{\lb\xx,\xx\rb}$.

We define (cf. \cite[II.10.24]{BridsonHaefliger}) \emph{$(d + 1)$-dimensional $\F$-hyperbolic space} to be the set
\[
\H_\F^{d + 1} = \{\xx \in \F^{d + 1}: \|\xx\| < 1\}
\]
equipped with the metric
\[
\cosh\dist(\xx,\yy) = \frac{|1 - \lb\xx,\yy\rb|}{\sqrt{1 - \|\xx\|^2}\sqrt{1 - \|\yy\|^2}}\cdot
\]
\begin{proposition}[{\cite[Theorem II.10.10]{BridsonHaefliger}}]
$\H_\F^{d + 1}$ is a CAT(-1) space, and in particular (by \cite[Proposition III.H.1.2]{BridsonHaefliger}) is a hyperbolic metric space.
\end{proposition}

Since $X := \H_\F^{d + 1}$ is an open subset of $\F^{d + 1}$, it has a natural topological boundary $\del_T X = \{\xx\in \F^{d + 1}:\|\xx\| = 1\}$. On the other hand, since $X$ is a hyperbolic metric space, it has a Gromov boundary $\del_G X$. The following relation between them is more or less a direct consequence of the fact that every Gromov sequence in $X$ converges to a point in $\del_T X$ (for details see \cite[Proposition 3.5.3]{DSU}):

\begin{proposition}
There is a unique homeomorphism $j:\del_G X\to \del_T X$ so that the extension $\id\cup j:X\cup\del_G X\to X\cup\del_T X$ is a homeomorphism.
\end{proposition}

From now on we will not distinguish between the two boundaries and will simply write $\del X$ for either one.

Let $\lambda_X$ denote normalized Lebesgue measure on $\del X$. Let $K = \dim_\R(\F)$, and let
\begin{equation}
\label{deltaXdef}
\delta_X := dK + 2(K - 1).
\end{equation}

\begin{proposition}
\label{propositionlebesgue}
~
\begin{itemize}
\item[(i)] $\lambda_X$ is an Ahlfors $\delta_X$-regular measure with respect to the visual metric on $\del X$.
\item[(ii)] For any group $G\leq\Isom(X)$, $\lambda_X$ is $\delta_X$-conformal with respect to $G$.
\end{itemize}
\end{proposition}
\begin{remark}
If $\F\neq\R$, then (i) is very different from saying that $\lambda_X$ is Ahlfors regular with respect to the metric inherited from $\F^{d + 1}$, since the visual metric on $\del X$ is not bi-Lipschitz equivalent to the usual metric as we will see below (Remark \ref{remarkballbox}).

If $\F = \R$, then the visual metric on $\del X$ is bi-Lipschitz equivalent to the usual metric, so the proofs below can be simplified somewhat in this case.
\end{remark}

\begin{proof}[Proof of Proposition \ref{propositionlebesgue}]
~
\begin{itemize}
\item[(i)] The first step is to write the visual metric on $\del X$ in terms of the sesquilinear form $B$. Indeed, direct calculation (for details see \cite[p. 32]{MackayTyson}) yields that for $\xx,\yy\in\del X$,
\[
\Dist_{\0}(\xx,\yy) = \Dist_{\0,e}(\xx,\yy) = \sqrt{\frac{1}{2}|1 - \lb\xx,\yy\rb|}.
\]
We may define a coordinate chart on $\del X$ as follows: Let $W = \{\xx\in\F^{d + 1}:\mathrm{Re}[x_1] = 0\}$, and let $g:W\to \del X$ be stereographic projection through $-\ee_1$, i.e.
\[
g(\xx) = -\ee_1 + \frac{2}{\|\ee_1 + \xx\|^2}(\ee_1 + \xx).
\]
Then for $\xx\in W$
\begin{align*}
\Dist_{\0}(\ee_1,g(\xx))
&= \sqrt{\frac{1}{2}|1 - g(\xx)_1|}\\
&= \sqrt{\frac{1}{2}\left|1 - \left(-1 + \frac{2}{\|\ee_1 + \xx\|^2}(1 + x_1)\right)\right|}\\
&= \sqrt{\left|1 - \frac{1}{1 + \|\xx\|^2}(1 + x_1)\right|}\\
&= \sqrt{\left|\frac{1}{1 + \|\xx\|^2}(\|\xx\|^2 - x_1)\right|}\\
&= \frac{\sqrt[4]{\|\xx\|^4 + |x_1|^2}}{\sqrt{1 + \|\xx\|^2}}\cdot
\end{align*}
Now fix $r > 0$ small enough so that $g^{-1}(B(\ee_1,r))$ is bounded. Then for $\xx\in g^{-1}(B(\ee_1,r))$,
\[
r \geq \Dist_{\0}(\ee_1,g(\xx)) \asymp_\times \max(\|\xx\|,\sqrt{|x_1|}).
\]
Letting $C$ be the implied constant, we have
\begin{equation}
\label{ballbox1}
g^{-1}(B(\ee_1,r)) \subset \{\xx\in W:|x_1|\leq (Cr)^2,\; |x_i|\leq Cr\text{ for }i\neq 1\}.
\end{equation}
Now, since the Jacobian of $g$ is bounded from above and below on $g^{-1}(B(\ee_1,r))$, we have
\[
\lambda_X(g^{-1}(B(\ee_1,r))) \asymp_\times \lambda_W(B(\ee_1,r)),
\]
where $\lambda_W$ is Lebesgue measure on $W$. Thus
\[
\lambda_X(B(\ee_1,r)) \lesssim_\times \lambda_W(\{\xx\in W:|x_1|\leq (Cr)^2,\; |x_i|\leq Cr\text{ for }i\neq 1\}) \asymp_\times r^{\delta_X}.
\]
On the other hand, we have
\begin{equation}
\label{ballbox2}
\{\xx\in W:|x_1|\leq (r/(d + 2))^2,\; |x_i|\leq r/(d + 2)\text{ for }i\neq 1\} \subset g^{-1}(B(\ee_1,r))
\end{equation}
since if $\xx$ is a member of the left hand side, then
\[
\Dist_{\0}(\ee_1,g(\xx)) \leq \|\xx\| + \sqrt{|x_1|} \leq (d + 2) \frac{r}{d + 2} = r.
\]
A similar argument yields
\[
\lambda_X(B(\ee_1,r)) \gtrsim_\times r^{\delta_X}.
\]
Now since the group $\scrO(\F^{d + 1})$ (i.e. the group of $\F$-linear isometries of $\F^{d + 1}$) preserves both $\lambda_X$ and $\Dist_{\0}$ and acts transitively on $\del X$, the expression $\lambda_X(B([\xx],r))$ depends only on $r$ and not on $[\xx]$. Thus $\lambda_X$ is Ahlfors $\delta_X$-regular.
\item[(ii)] By (i), $\lambda_X$ is Ahlfors $\delta_X$-regular and so $\lambda_X\asymp_\times\HH^{\delta_X}$, where $\HH^{\delta_X}$ is Hausdorff $\delta_X$-dimensional measure on $\del X$. Let
\[
f = \frac{\dee \lambda_X}{\dee \HH^{\delta_X}}
\]
be the Radon--Nikodym derivative of $\lambda_X$ with respect to $\HH^{\delta_X}$. Since both $\lambda_X$ and $\HH^{\delta_X}$ are invariant under $\scrO(\F^{d + 1})$, for each $g\in\scrO(\F^{d + 1})$, the equation $f = f\circ g$ holds $\lambda_X$-almost everywhere. 
\begin{claim}
$f$ is constant $\lambda_X$-almost everywhere.
\end{claim}
\begin{subproof}
By contradiction, suppose that there exists $\alpha$ such that $\lambda_X(A_1),\lambda_X(A_2) > 0$ where $A_1 = \{\xi\in\del X: f(\xi) > \alpha\}$ and $A_2 = \{\xi\in\del X:f(\xi) < \alpha\}$. Let $\xi_1,\xi_2\in \del X$ be density points of $A_1$ and $A_2$, respectively. Since $\scrO(\F^{d + 1})$ acts transitively on $\del X$, there exists $g\in \scrO(\F^{d + 1})$ such that $g(\xi_1) = \xi_2$. Since $f =_{\lambda_X} f\circ g$, we have $g(A_1) =_{\lambda_X} A_1$ and thus $\xi_2$ is a density point of $A_1$, which contradicts that it is a density point of $A_2$.
\end{subproof}
Thus $\lambda_X$ and $\HH^{\delta_X}$ are proportional; since $\HH^{\delta_X}$ is $\delta_X$-conformal with respect to any group $G\leq\Isom(X)$ (by abstract metric space theory), it follows that $\lambda_X$ is $\delta_X$-conformal for any such group as well.
\end{itemize}
\end{proof}
\begin{remark}
\label{remarkballbox}
The inclusions \eqref{ballbox1} and \eqref{ballbox2} are reminiscent of the Ball-Box theorem of sub-Riemannian geometry (see e.g. \cite[Theorem 0.5.A]{Gromov5}). The difference is that while the Ball-Box theorem applies to a Carnot--Carath\'eodory metric, the above argument is about visual metrics. We remark that it is possible to prove using \eqref{ballbox1}-\eqref{ballbox2} together with the Ball-Box theorem that the visual metric is bi-Lipschitz equivalent with the Carnot--Carath\'eodory metric of an appropriate sub-Riemannian metric on $\del X$. This is irrelevant for our purposes since \eqref{ballbox1}-\eqref{ballbox2} already contain the information we need to analyze the visual metric.

Note further that \eqref{ballbox1}-\eqref{ballbox2} show that $\Dist_{\0}$ is bi-Lipschitz equivalent to the Euclidean metric on $\del X$ if and only if $\F = \R$.
\end{remark}

Now we are able to prove the assertions about real, complex, and quaternionic hyperbolic spaces made in Remarks \ref{remarkreductionjarnik} and \ref{remarkreductionkhinchin} and Example \ref{examplekhinchinlattice1}.

\begin{proposition}
\label{propositionreduction}
Let $X$ be a real, complex, or quaternionic hyperbolic space, and let $G\leq\Isom(X)$ be a lattice. Then
\begin{itemize}
\item[(i)] For all $\xi\in\Lbp(G)$, we have
\begin{equation}
\label{deltaxidelta}
\delta_\xi = \delta/2.
\end{equation}
\item[(ii)] The Patterson--Sullivan measure of $G$ is Ahlfors regular.
\end{itemize}
\end{proposition}
\begin{proof}
~
\begin{itemize}
\item[(i)] Let $\delta_X$ be defined by \eqref{deltaXdef}. Then
\begin{align*}
\delta &= \HD(\Lr(G)) \by{Theorem \ref{theorembishopjones}}\\
&= \HD(\del X) \since{$G$ is a lattice}\\
&= \delta_X. \by{Proposition \ref{propositionlebesgue}}
\end{align*}
On the other hand, by \cite[Lemma 3.5]{Newberger} together with \eqref{deltaxilimsup}, we have
\[
2\delta_\xi = k + 2\ell,
\]
where $(\ell,k)$ is the \emph{rank} of $\xi$ (see \cite[p. 226]{Newberger} for the definition). It is readily checked that if $G$ is a lattice, then $(\ell,k) = (dK,K - 1)$ (in the notation of \cite{Newberger}, we have $H = N$), and thus
\[
2\delta_\xi = dK + 2(K - 1) = \delta_X = \delta,
\]
demonstrating \eqref{deltaxidelta}.
\item[(ii)] As demonstrated above, $\delta = \delta_X$, so by (ii) of Proposition \ref{propositionlebesgue}, $\lambda_X$ is $\delta$-conformal, and so it is the Patterson--Sullivan measure for $G$. On the other hand, $\lambda_X$ is Ahlfors regular by (i) of Proposition \ref{propositionlebesgue}.
\end{itemize}
\end{proof}

\draftnewpage
\section{The potential function game}
In this appendix, we give a short and self-contained exposition of the technique used to prove Theorem \ref{theoremabsolutewinning} (absolute winning of $\BA_\xi$), and prove a general theorem about when it can be applied, which can be used to simplify the proof of Theorem \ref{theoremabsolutewinning} (see Subsection \ref{subsectionsimplification}). Our hope is that this theorem can be used as a basis for future results.

Let $(Z,\Dist)$ be a complete metric space, and let $\HH$ be a collection of closed subsets of $Z$. We define the \emph{$\HH$-absolute game} as follows:

\begin{definition}
Given $\beta > 0$, Alice and Bob play the \emph{$(\beta,\HH)$-absolute game} on $Z$ as follows:
\begin{itemize}
\item[1.] Bob begins by choosing a ball $B_0 = B(z_0,r_0)\subset Z$.\Footnote{Cf. Remark \ref{remarkformalballs}.}
\item[2.] On Alice's $n$th turn, she chooses a set of the form $\NN(A_n,\w r_n)$ with $A_n\in\HH$, $0 < \w r_n \leq \beta r_n$, where $r_n$ is the radius of Bob's $n$th move $B_n = B(z_n,r_n)$. Here $\NN(A_n,\w r_n)$ denotes the $\w r_n$-thickening of $A_n$. We say that Alice \emph{deletes} her choice $\NN(A_n,\w r_n)$.
\item[3.] On Bob's $(n + 1)$st turn, he chooses a ball $B_{n + 1} = B(z_{n + 1},r_{n + 1})$ satisfying
\begin{equation}
\label{bobrulesabsolute}
r_{n + 1}\geq\beta r_n \text{ and } B_{n + 1} \subset B_n \butnot \NN(A_n,\w r_n),
\end{equation}
where $B_n = B(z_n,r_n)$ was his $n$th move, and $\NN(A_n,\w r_n)$ was Alice's $n$th move.
\item[4.] If at any time Bob has no legal move, Alice wins by default. If $r_n\not\to 0$, then Alice also wins by default. Otherwise, the balls $(B_n)_1^\infty$ intersect at a unique point which we call the \emph{outcome} of the game.
\end{itemize}
If Alice has a strategy guaranteeing that the outcome lies in a set $S$ (or that she wins by default), then the set $S$ is called \emph{$(\beta,\HH)$-absolute winning}. If a set $S$ is $(\beta,\HH)$-absolute winning for all $\beta > 0$ then it is called \emph{$\HH$-absolute winning}.
\end{definition}

\begin{example}
If $\HH = \{\{z\}:z\in Z\}$, then this game corresponds to the absolute winning game described in Section \ref{sectiongames}. If $Z = \R^d$ and $\HH = \{\text{affine hyperplanes in $\R^d$}\}$, then this game corresponds to the hyperplane game introduced in \cite{BFKRW}. We will abbreviate these special cases as $\HH = \text{points}$ and $\HH = \text{hyperplanes}$.
\end{example}
\begin{remark}
If $\HH = \text{points}$, then $\HH$-absolute winning on $Z$ implies winning for Schmidt's game on all uniformly perfect subsets of $Z$, and is invariant under quasisymmetric maps ((ii) and (v) of Proposition \ref{propositionwinningproperties}). If $\HH = \text{hyperplanes}$, then $\HH$-absolute winning on $\R^d$ implies winning for Schmidt's game on all hyperplane diffuse subsets of $\R^d$ (see Definition \ref{definitionhyperplanediffuse}), and is invariant under $\CC^1$ diffeomorphisms \cite[Propositions 2.3(c) and 4.7]{BFKRW}. Moreover, the $\HH$-absolute winning game is readily seen to have the countable intersection property.
\end{remark}

We now proceed to describe another game which is equivalent to the $\HH$-absolute game, but for which in some cases it is easier to describe a winning strategy. It is not immediately obvious, for example, that sets which are winning for the new game are nonempty, but this (and in fact full dimension of winning sets) follows from the equivalence of the two games, which we prove below (Theorem \ref{theoremgamesequivalent}).

\begin{definition}
Given $\beta,c > 0$, Alice and Bob play the \emph{$(\beta,c,\HH)$-potential game} as follows:
\begin{itemize}
\item[1.] Bob begins by choosing a ball $B(z_0,r_0)\subset Z$.
\item[2.] On Alice's $n$th turn, she chooses a countable collection of sets of the form $\NN(A_{i,n},r_{i,n})$, with $A_{i,n}\in\HH$ and $r_{i,n} > 0$, satisfying
\begin{equation}
\label{alicerules}
\sum_i r_{i,n}^c \leq (\beta r_n)^c,
\end{equation}
where $r_n$ is the radius of Bob's $n$th move.
\item[3.] On Bob's $(n + 1)$st turn, he chooses a ball $B_{n + 1} = B(z_{n + 1},r_{n + 1})$ satisfying
\begin{equation}
\label{bobrules}
r_{n + 1}\geq\beta r_n \text{ and } B_{n + 1} \subset B_n,
\end{equation}
where $B_n = B(z_n,r_n)$ was his $n$th move.
\item[4.] If $r_n\not\to 0$, then Alice wins by default. Otherwise, the balls $(B_n)_1^\infty$ intersect at a unique point which we call the \emph{outcome} of the game. If the outcome is an element of any of the sets $\NN(A_{i,n},r_{i,n})$ which Alice chose during the course of the game, she wins by default.
\end{itemize}
If Alice has a strategy guaranteeing that the outcome lies in a set $S$ (or that she wins by default), then the set $S$ is called \emph{$(\beta,c,\HH)$-potential winning}. If a set is $(\beta,c,\HH)$-potential winning for all $\beta,c > 0$, then it is \emph{$\HH$-potential winning}.
\end{definition}

We remark that in this version of the game Alice is not deleting balls or hyperplane-neighborhoods per se, in that Bob can make moves which intersect Alice's choices, but he must eventually avoid them in order to avoid losing.

\begin{remark}
We call this game the \emph{potential} game because of the use of a potential function $\phi$ in the proof of equivalence (Theorem \ref{theoremgamesequivalent}). We call $\phi$ a potential function because Alice's strategy is to act so as to minimize $\phi$ of Bob's balls, in loose analogy with physics where nature acts to minimize potential energy.
\end{remark}

In order to guarantee that the two games are equivalent, we make the following assumption:

\begin{assumption}
\label{assumptionNF}
For all $\beta > 0$, there exists $N = N_\beta\in\N$ such that for any ball $B(z,r)\subset Z$, there exists $\FF = \FF_{B(z,r)} \subset\HH$ with $\#(\FF)\leq N$ such that for all $A_1\in\HH$,
\begin{equation}
\label{A1A2}
\NN(A_1,\beta r/3)\cap B(z,r) \subset \NN(A_2,\beta r) \text{ for some $A_2\in\FF$.}
\end{equation}
\end{assumption}
\begin{observation}
Assumption \ref{assumptionNF} is satisfied if either
\begin{itemize}
\item[(1)] $\HH = \text{points}$ and $Z$ is a doubling metric space, or
\item[(2)] $\HH = \text{hyperplanes}$.
\end{itemize}
\end{observation}
\begin{proof}[Proof if $\HH$ = points and $Z$ is doubling]
Let $F\subset B(z,r)$ be a maximal $\beta r/3$-separated subset, and let $\FF = \{\{z\}:z\in F\}$. Since $Z$ is doubling, $\#(\FF) = \#(F)\leq N$ for some $N$ depending only on $\beta$. Given any $\{z_1\}\in\HH$, if $B(z_1,\beta r/3)\cap B(z,r)\neq\emptyset$, then by the maximality of $F$, there exists $z_2\in F$ such that $B(z_1,\beta r/3)\cap B(z_2,\beta r/3)\neq\emptyset$. This implies \eqref{A1A2}.
\end{proof}
\begin{proof}[Proof if $\HH$ = hyperplanes]
By applying a similarity, we may without loss of generality assume that $B(z,r) = B(\0,1)$. Consider the surjective map $\pi:S^{d - 1}\times\R\to \HH$ defined by $\pi(\vv,t) = \{\xx\in\R^d:\vv\cdot\xx = t\}$. Let $F$ be a maximal $\beta/3$-separated subset of $S^{d - 1}\times [-(1 + \beta/3),1 + \beta/3]$, and let $\FF = \pi(F)$. We claim that $\FF$ satisfies the conclusion of Assumption \ref{assumptionNF}. Clearly, $N := \#(\FF) < \infty$. Indeed, fix $\LL_1 = \pi(\vv_1,t_1)\in\HH$. If $|t_1| > 1 + \beta/3$, then $\NN(\LL_1,\beta/3)\cap B(\0,1) = \emptyset$, so \eqref{A1A2} holds trivially. Otherwise, fix $\LL_2 = \pi(\vv_2,t_2)\in\FF$ with $\|\vv_1 - \vv_2\|,|t_1 - t_2| \leq \beta/3$. For $\xx\in \NN(\LL_2,\beta/3)\cap B(\0,1)$, we have
\[
|\vv_1\cdot\xx - t_1| \leq |\vv_2\cdot\xx - t_2| + \|\vv_1 - \vv_2\|\cdot\|\xx\| + |t_1 - t_2| \leq \frac{\beta}{3} + \frac{\beta}{3} + \frac{\beta}{3} = \beta,
\]
and so $\xx\in\NN(\LL_1,\beta)$, demonstrating \eqref{A1A2}.
\end{proof}

\begin{theorem}
\label{theoremgamesequivalent}
Suppose that Assumption \ref{assumptionNF} is satisfied. Then $S\subset Z$ is $\HH$-absolute winning if and only if $S$ is $\HH$-potential winning.
\end{theorem}
\begin{proof}
Suppose that $S$ is $\HH$-absolute winning. Fix $\beta,c > 0$, and consider a strategy of Alice which is winning for the $(\beta^2/3,\HH)$-absolute game. Each time Bob makes a move $B(z,r)$, Alice deletes a set of the form $\NN(A,\beta^2 r/3)$, $A\in \HH$. Alice's corresponding strategy in the $(\beta,c,\HH)$-potential game will be to choose the set $\NN(A,\beta r)$, and then to make dummy moves until the radius of Bob's ball $B(w,r')$ is less than $\beta r/3$. By \eqref{bobrules}, we have $r'\geq\beta^2 r/3$, allowing us to interpret the move $B(w,r')$ as Bob's next move in the $(\beta^2/3,\HH)$-absolute game. If this move is disjoint from $\NN(A,\beta^2 r/3)$, then it is legal in the $(\beta^2/3,\HH)$-absolute game and the process can continue. Otherwise, we have
\[
B(w,r') \subset \NN(A,\beta^2 r/3 + 2r') \subset \NN(A,\beta r).
\]
Thus Alice can make dummy moves until the game ends, and the outcome will lie in $\NN(A,\beta r)$, and so Alice will win by default.

On the other hand, suppose that $S$ is $\HH$-potential winning. Let $\beta > 0$. Fix $\w\beta,c > 0$ small to be determined, and consider a strategy of Alice which is winning for the $(\w\beta,c,\HH)$-potential game. Each time Bob makes a move $B_n = B(z_n,r_n)$, Alice chooses a collection of sets $\{\NN(A_{i,n},r_{i,n})\}_{i = 1}^{N_n}$ (with $N_n\in\N\cup\{\infty\}$) satisfying \eqref{alicerules}. Alice's corresponding strategy in the $(\beta,\HH)$-absolute game will be to choose her set $\NN(A_n,\beta r_n)\subset Z$ so as to maximize
\[
\phi(B_n;\NN(A_n,\beta r_n)) := \sum_{m = 0}^n \sum_{\substack{i \\ \NN(A_{i,m},\beta r_{i,m})\subset \NN(A_n,\beta r_n) \\ \NN(A_{i,m},\beta r_{i,m})\cap B_n\neq\smallemptyset}} r_{i,m}^c. 
\]
After Alice deletes the set $\NN(A_n,\beta r_n)$, Bob will choose his ball $B_{n + 1}\subset B_n\butnot \NN(A_n,\beta r_n)$, and we will consider this also to be the next move in the $(\w\beta,c,\HH)$-potential game. This gives us an infinite sequence of moves for both games. We will show that if Alice wins the $(\w\beta,c,\HH)$-potential game, then she also wins the $(\beta,\HH)$-absolute game.

\begin{claim}
\label{claiminductive}
For each $n$ let
\[
\phi(B_n) = \sum_{m = 0}^n \sum_{\substack{i \\ \NN(A_{i,m},\beta r_{i,m})\cap B_n\neq\smallemptyset}} r_{i,m}^c.
\]
Then there exists $\epsilon > 0$ independent of $n,\w\beta,c$ (but depending on $\beta$) such that if
\begin{equation}
\label{inductivehypothesis}
\phi(B_n) \leq (\epsilon r_n)^c
\end{equation}
then
\begin{equation}
\label{AnBn}
\phi(B_n;\NN(A_n,\beta r_n)) \geq \epsilon \phi(B_n).
\end{equation}
\end{claim}
\begin{subproof}[Proof of Claim \ref{claiminductive}]
Let $N = N_\beta$ and $\FF = \FF_{B_n}$ be as in Assumption \ref{assumptionNF}. For each pair $(i,m)$ for which $\NN(A_{i,m},\beta r_{i,m})\cap B_n\neq\emptyset$, we have
\[
r_{i,m}^c \leq \phi(B_n) \leq (\epsilon r_n)^c;
\]
letting $\epsilon \leq \beta/3$, we have $\NN(A_{i,m},\beta r_{i,m})\subset \NN(A_{i,m},\beta r_n/3)$. Thus by Assumption \ref{assumptionNF}, there exists $A\in\FF$ for which $\NN(A_{i,m},\beta r_{i,m})\subset \NN(A,\beta r_n)$. It follows that
\[
\phi(B_n) \leq \sum_{A\in\FF}\phi(B_n;\NN(A,\beta r_n)) \leq N \phi(B_n;\NN(A_n,\beta r_n)),
\]
and letting $\epsilon \leq 1/N$ finishes the proof.
\end{subproof}

\begin{claim}
If $\w\beta$ and $c$ are chosen sufficiently small, then \eqref{inductivehypothesis} holds for all $n\in\N$.
\end{claim}
\begin{subproof}
We proceed by strong induction. Indeed, fix $N\in\N$ suppose that \eqref{inductivehypothesis} holds for all $n < N$. Fix such an $n$. We observe that since $B_{n + 1}\subset B_n\butnot \NN(A_n,\beta r_n)$, we have
\[
\phi(B_{n + 1})\leq \phi(B_n) - \phi(B_n;\NN(A_n,\beta r_n)) + \sum_{\substack{i \\ \NN(A_{i,n + 1},r_{i,n + 1})\cap B_{n + 1}\neq\smallemptyset}} r_{i,n + 1}^c.
\]
Combining with \eqref{AnBn} and \eqref{alicerules} gives
\[
\phi(B_{n + 1}) \leq (1 - \epsilon)\phi(B_n) + (\w\beta r_n)^c.
\]
On the other hand, $r_{n + 1}\geq \beta r_n$ since Bob originally played $B_{n + 1}$ as a move in the $(\beta,\HH)$-absolute game. Thus, dividing by $(\beta r_n)^c$ gives
\begin{equation}
\label{iterated}
\frac{\phi_{n + 1}(B_{n + 1})}{r_{n + 1}^c} \leq \beta^{-c} \left[(1 - \epsilon) \frac{\phi_n(B_n)}{r_n^c} + \w\beta^c \right].
\end{equation}
For $c$ sufficiently small,
\[
\frac{1 - \epsilon}{\beta^c} < 1,
\]
and iterating \eqref{iterated} yields
\[
\frac{\phi_N(B_N)}{r_N^c} \lesssim_{\times,\beta} \w\beta^c .
\]
Let $C$ be the implied constant. Setting $\w\beta = \epsilon/C^{1/c}$, we see that \eqref{inductivehypothesis} holds for $N$.
\end{subproof}

Now suppose that $r_n\to 0$; otherwise Alice wins the $(\beta,\HH)$-absolute game by default. Then \eqref{inductivehypothesis} implies that $\phi(B_n)\to 0$. In particular, for each $(i,m)$, $r_{i,m}^c > \phi(B_n)$ for all $n$ sufficiently large which implies $\NN(A_{i,m},\beta r_{i,m})\cap B_n = \emptyset$. Thus the outcome of the game does not lie in $\NN(A_{i,m},\beta r_{i,m})$ for any $(i,m)$, so Alice does not win the $(\w\beta,c,\HH)$-potential game by default. Thus if she wins, then she must win by having the outcome lie in $S$. Since the outcome is the same for the $(\w\beta,c,\HH)$-potential game and the $(\beta,\HH)$-absolute game, this implies that she also wins the $(\beta,\HH)$-absolute game.
\end{proof}

\draftnewpage

\subsection{Proof of Theorem \ref{theoremabsolutewinning} using the $\HH$-potential game, where $\HH$ = points}
\label{subsectionsimplification}
In this subsection, we give a simpler proof of Theorem \ref{theoremabsolutewinning} using Theorem \ref{theoremgamesequivalent}. This proof uses Corollary \ref{corollarystructure} and Lemma \ref{lemmageneralfinite}, but not Lemma \ref{lemmaquasisymmetric} or Corollary \ref{corollarymodifiedabsolutewinning}.

Fix $\beta,c > 0$, and let us play the $(\beta,c,\HH)$-potential game, where $\HH$ = points and $Z = \JJ_s$, where $\JJ_s$ is as in Corollary \ref{corollarystructure}. Fix $\epsilon > 0$ small to be determined. Alice's strategy is as follows: When Bob makes a move $B = B(z,r)$, then for each $g\in G$ satisfying
\begin{align} \label{gcondition1}
-\log_b(r/r_0) \leq \lb g(\zero)|g(\xi)\rb_\zero &< -\log_b(\beta r/r_0)\\ \label{gcondition2}
B(g(\xi),\epsilon b^{-\dist(\zero,g(\zero))})\cap B(z,r) &\neq\emptyset,
\end{align}
Alice chooses the set $B(g(\xi),\epsilon b^{-\dist(\zero,g(\zero))})$. Here $r_0$ represents the radius of Bob's first move.
\begin{claim}
For $\epsilon > 0$ sufficiently small, the move described above is legal.
\end{claim}
\begin{proof}
By the construction of $\JJ_s$, there exists $x\in G(\zero)$ such that $z\in \PP_x\subset B(z,r)$ but $\Diam(\PP_x) \asymp_\times r$. In particular, $\dist(\zero,x) \asymp_\plus -\log_b(r)$. Moreover,
\begin{equation}
\label{xzasymp}
\lb x|z\rb_\zero \asymp_\plus \dist(\zero,x).
\end{equation}
Suppose that $g\in G$ satisfies \eqref{gcondition1} and \eqref{gcondition2}. Then
\[
\Dist(g(\xi),z) \leq r + \epsilon b^{-\dist(\zero,g(\zero))} \lesssim_\times \max(r, b^{-\lb g(\zero)|g(\xi)\rb_\zero}) \asymp_\times r;
\]
combining with \eqref{gcondition1}, \eqref{xzasymp}, and Gromov's inequality yields
\[
\lb x|g(\zero)\rb_\zero \asymp_\plus \dist(\zero,x),
\]
which together with (b) of Proposition \ref{propositionbasicidentities} gives
\begin{equation}
\label{0g0x0v2}
\lb \zero|g(\zero)\rb_x \asymp_\plus 0.
\end{equation}
By (j) of Proposition \ref{propositionbasicidentities},
\begin{align*}
\lb \zero|\xi\rb_{g^{-1}(x)} = \lb g(\zero)|g(\xi)\rb_x
&\asymp_\plus \lb g(\zero)|g(\xi)\rb_\zero + \lb \zero|g(\zero)\rb_x - \lb x|g(\xi)\rb_\zero\\
&\asymp_\plus \dist(\zero,x) + 0 - \dist(\zero,x) = 0.
\end{align*}
Thus, if we choose $\rho > 0$ large enough, we see that $\xi\in\Shad(g^{-1}(x),\rho)$ for all $g$ satisfying (\ref{gcondition1}) and (\ref{gcondition2}). On the other hand, rearranging (\ref{0g0x0v2}) yields
\[
\dist(\zero,g(\zero)) \asymp_\plus \dist(g(\zero),x) + \dist(\zero,x) \asymp_\plus \dist(\zero,g^{-1}(x)) - \log_b(r),
\]
and thus for $\rho > 0$ sufficiently large,
\begin{align*}
\sum_{\substack{g\in G \\ \text{\eqref{gcondition1} and \eqref{gcondition2} hold}}}\left(\epsilon b^{-\dist(\zero,g(\zero))}\right)^c
&\lesssim_\times \sum_{\substack{g\in G \\ \xi\in\Shad(g^{-1}(x),\rho)}}\left(\epsilon r b^{-\dist(\zero,g^{-1}(x))}\right)^c\\
&=_\pt (\epsilon r)^c \sum_{\substack{g\in G \\ \xi\in\Shad(g^{-1}(\zero),\rho)}}b^{-c \dist(\zero,g^{-1}(\zero))},
\end{align*}
since $x \in G(\zero)$ by construction. The sum is clearly independent of $B(z,r)$, and by Lemma \ref{lemmageneralfinite} it is finite. Thus the right hand side is asymptotic to $(\epsilon r)^c$, and so for $\epsilon$ sufficiently small, \eqref{alicerules} holds.
\end{proof}

\begin{claim}
The strategy above guarantees that $\BA_\psi$ is absolute winning.
\end{claim}
\begin{proof}
Suppose that Alice does not win by default, and let $\eta$ be the outcome of the game. Fix $g\in G$. Since $\lb g(\zero)|g(\xi)\rb_\zero\geq 0$, \eqref{gcondition1} holds for one of Bob's moves $B(z,r)$. If \eqref{gcondition2} also holds, then Alice chose the ball $B(g(\xi),\epsilon b^{-\dist(\zero,g(\zero))})$ on her $n$th turn; otherwise $B(g(\xi),\epsilon b^{-\dist(\zero,g(\zero))})\cap B(z,r) = \emptyset$. Either way, since we assume that Alice did not win by default, we have $\eta\notin B(g(\xi),\epsilon b^{-\dist(\zero,g(\zero))})$. Since $g$ was arbitrary, this demonstrates that $\eta\in\BA_\xi$, so Alice wins.
\end{proof}

\section{Winning sets and partition structures}
In this paper, when we proved Theorem \ref{theoremabsolutewinning}, we first constructed a partition structure, and then we played Schmidt's game on it. It is also possible to do this in the reverse order, using Schmidt's game to prove the existence of certain partition structures. This fact can be used to show that any winning subset of an Ahlfors regular metric space contains subsets which are Ahlfors regular. Specifically, we have the following:

\begin{proposition}
Let $(Z,\Dist)$ be an Ahlfors $s$-regular metric space, and let $S$ be winning on $Z$. Then for every $t < s$, there exists a $t$-thick partition structure on $Z$ whose limit set is contained in $S$. In particular, for every $t < s$, there exists an Ahlfors $t$-regular set contained in $S$.
\end{proposition}
\begin{proof}
We will need the following:
\begin{lemma}[{\cite[Theorem 5.1]{Fishman}}]
\label{lemmafishman}
There exists a constant $\epsilon > 0$ such that for every $0 < r,\beta < 1$ and for every $z\in Z$, there exists a disjoint collection of balls $\big(B(z_i,\beta r)\big)_{i = 1}^{N_\beta}$ contained in $B(z,r)$, where
\[
N_\beta = \lfloor \epsilon\beta^{-s}\rfloor.
\]
\end{lemma}
Now since $S$ is winning, there exists $\alpha > 0$ such that for every $0 < \beta < 1$, $S$ is $(\alpha,\beta)$-winning. Fix $0 < t < s$, fix $0 < \beta < 1/2$ small to be determined, and let $N = N_{2\beta}$.

Now $S$ is an $(\alpha,\beta)$-winning set. To simplify notation, we will use the fact \cite[Theorem 7]{Schmidt1} that Alice can win using a \emph{positional strategy}, i.e. a strategy where her move depends only on Bob's last move. Let $f$ be such a positional strategy, so that $f(B)$ is Alice's response when Bob plays the move $B$.

Let $T^* = \bigcup_{n\in\N}\{1,\ldots,N\}^n$. For each $\omega\in T^*$, we define the balls $B_\omega$ and $A_\omega$ recursively as follows:
\begin{itemize}
\item $B_\smallemptyset$ is any ball.
\item If $B_\omega = B(z,r)$ has been defined, let $A_\omega = f(B_\omega) = B(\w z,\w r)$. Note that $\w r = \alpha r$. Let $\big(B(z_i,2\beta\w r)\big)_{i = 1}^N$ be the disjoint collection of balls guaranteed by Lemma \ref{lemmafishman}. For each $i = 1,\ldots,N$ let
\[
B_{\omega i} = B(z_i,\beta\w r) = B(z_i,\alpha\beta r).
\]
\end{itemize}
It is readily verified that the collection of sets $(\PP_\omega = B_\omega)_{\omega\in T^*}$ is a partition structure on $Z$. Note that this verification requires the use of the asymptotic $\Diam(B(z,r))\asymp_\times r$, which follows from the fact that $Z$ is Ahlfors regular and therefore uniformly perfect.

For each $\omega\in T^\N$, $\pi(\omega)$ is the unique intersection point of the sequence $(B_{\omega_1^n})_{n = 1}^\infty$, which is a sequence of possible moves that Bob could make while Alice is using her positional strategy $f$. Thus $\pi(\omega)\in S$. Since $\omega$ was arbitrary, $\pi(T^\N)\subset S$.

Finally, we will show that the partition structure $(\PP_\omega)_{\omega\in T^*}$ is $t$-thick if $\beta$ is sufficiently small. Indeed, for any $\omega\in T^*$ we have
\[
\frac{\sum_{a\in E_\omega}D_{\omega a}^t}{D_\omega^t} \asymp_\times \frac{N (\alpha\beta r)^t}{r^t} = N_{2\beta} (\alpha\beta)^t \asymp_{\times,\alpha} \beta^{t - s}.
\]
Since $t < s$, we can choose $\beta$ small enough so that $\beta^{t - s}$ is greater than the implied constant of this asymptotic. This completes the proof of the existence of a $t$-thick partition structure; the second part of the theorem follows from Theorem \ref{theoremahlforsgeneral}.
\end{proof}

\draftnewpage

\bibliographystyle{amsplain}

\bibliography{bibliography}

\providecommand{\bysame}{\leavevmode\hbox to3em{\hrulefill}\thinspace}
\providecommand{\MR}{\relax\ifhmode\unskip\space\fi MR }
\providecommand{\MRhref}[2]{%
  \href{http://www.ams.org/mathscinet-getitem?mr=#1}{#2}
}
\providecommand{\href}[2]{#2}
\begin{thebibliography}{10}

\bibitem{Ahlfors}
L.~V. Ahlfors, \emph{{M}\"obius transformations in several dimensions}, Ordway
  Professorship Lectures in Mathematics, University of Minnesota, School of
  Mathematics, Minneapolis, Minn., 1981.

\bibitem{Aravinda}
C.~S. Aravinda, \emph{Bounded geodesics and {H}ausdorff dimension}, Math. Proc.
  Cambridge Philos. Soc. \textbf{116} (1994), no. 3, 505--511.

\bibitem{AravindaLeuzinger}
C.~S. Aravinda and E.~Leuzinger, \emph{Bounded geodesics in rank-1 locally
  symmetric spaces}, Ergodic Theory Dynam. Systems \textbf{15} (1995), no. 5,
  813--820.

\bibitem{AGP}
J.~S. Athreya, A.~Ghosh, and A.~Prasad, \emph{Ultrametric logarithm laws,
  {II}}, Monatsh. Math. \textbf{167} (2012), no. 3--4, 333--356.

\bibitem{Beardon1}
A.~F. Beardon, \emph{The exponent of convergence of {P}oincar\'e series}, Proc.
  London Math. Soc. (3) \textbf{18} (1968), 461--483.

\bibitem{BDV}
V.~V. Beresnevich, D.~Dickinson, and S.~L. Velani, \emph{Measure theoretic laws
  for lim sup sets}, Mem. Amer. Math. Soc. \textbf{179} (2006), no. 846, x+91
  pp.

\bibitem{BeresnevichVelani}
V.~V. Beresnevich and S.~L. Velani, \emph{A mass transference principle and the
  {D}uffin-{S}chaeffer conjecture for {H}ausdorff measures}, Ann. of Math. (2)
  \textbf{164} (2006), no. 3, 971--992.

\bibitem{BernikDodson}
V.~I. Bernik and M.~M. Dodson, \emph{Metric {D}iophantine approximation on
  manifolds}, Cambridge Tracts in Mathematics, vol. 137, Cambridge University
  Press, Cambridge, 1999.

\bibitem{BestvinaFeighn}
M.~Bestvina and M.~Feighn, \emph{Hyperbolicity of the complex of free factors},
  Adv. Math. \textbf{256} (2014), 104--155.

\bibitem{BishopJones}
C.~J. Bishop and P.~W. Jones, \emph{{H}ausdorff dimension and {K}leinian
  groups}, Acta Math. \textbf{179} (1997), no. 1, 1--39.

\bibitem{BHM}
S.~Blach\`ere, P.~Ha\"issinsky, and P.~Mathieu, \emph{Harmonic measures versus
  quasiconformal measures for hyperbolic groups}, Ann. Sci. \'Ec. Norm.
  Sup\'er. (4) \textbf{44} (2011), no. 4, 683--721.

\bibitem{BonkSchramm}
M.~Bonk and O.~Schramm, \emph{Embeddings of {G}romov hyperbolic spaces}, Geom.
  Funct. Anal. \textbf{10} (2000), no. 2, 266--306.

\bibitem{Bourdon}
M.~Bourdon, \emph{Structure conforme au bord et flot g\'eod\'esique d'un
  {CAT}(-1)-espace ({C}onformal structure at the boundary and geodesic flow of
  a {CAT}(-1)-space)}, Enseign. Math. (2) \textbf{41} (1995), no. 1--2, 63--102
  (French).

\bibitem{Bowditch_geometrical_finiteness}
B.~H. Bowditch, \emph{Geometrical finiteness for hyperbolic groups}, J. Funct.
  Anal. \textbf{113} (1993), no. 2, 245--317.

\bibitem{BridsonHaefliger}
M.~R. Bridson and A.~Haefliger, \emph{Metric spaces of non-positive curvature},
  Grundlehren der Mathematischen Wissenschaften, vol. 319, Springer-Verlag,
  Berlin, 1999.

\bibitem{BFKRW}
R.~Broderick, L.~Fishman, D.~Y. Kleinbock, A.~Reich, and B.~Weiss, \emph{The
  set of badly approximable vectors is strongly {$C^1$} incompressible}, Math.
  Proc. Cambridge Philos. Soc. \textbf{153} (2012), no.~02, 319--339.

\bibitem{Broise-AlamichelPaulin}
A.~Broise-Alamichel and F.~Paulin, \emph{Sur le codage du flot g\'eod\'esique
  dans un arbre (coding geodesic flow in a tree)}, Ann. Fac. Sci. Toulouse
  Math. (6) \textbf{16} (2007), no. 3, 477--527 (French).

\bibitem{Bugeaud}
Y.~Bugeaud, \emph{Approximation by algebraic numbers}, Cambridge Tracts in
  Mathematics, vol. 160, Cambridge University Press, Cambridge, 2004.

\bibitem{CCMT}
P.-E. Caprace, Y.~de~Cornulier, N.~Monod, and R.~Tessera, \emph{Amenable
  hyperbolic groups}, \url{http://arxiv.org/abs/1202.3585v1}, preprint 2012.

\bibitem{Champetier}
C.~Champetier, \emph{Propri\'et\'es statistiques des groupes de pr\'esentation
  finie (statistical properties of finitely presented groups)}, Adv. Math.
  \textbf{116} (1995), no. 2, 197--262 (French).

\bibitem{Coornaert}
M.~Coornaert, \emph{Mesures de {P}atterson-{S}ullivan sur le bord d'un espace
  hyperbolique au sens de {G}romov ({P}atterson-{S}ullivan measures on the
  boundary of a hyperbolic space in the sense of {G}romov)}, Pacific J. Math.
  \textbf{159} (1993), no. 2, 241--270 (French).

\bibitem{DOP}
F.~Dal'bo, J.-P. Otal, and M.~Peign\'e, \emph{S\'eries de {P}oincar\'e des
  groupes g\'eom\'etriquement finis. ({P}oincar\'e series of geometrically
  finite groups)}, Israel J. Math. \textbf{118} (2000), 109--124 (French).

\bibitem{Dani3}
S.~G. Dani, \emph{Bounded orbits of flows on homogeneous spaces}, Comment.
  Math. Helv. \textbf{61} (1986), no. 4, 636--660.

\bibitem{Dani1}
\bysame, \emph{On orbits of endomorphisms of tori and the {S}chmidt game},
  Ergodic Theory Dynam. Systems \textbf{8} (1988), 523--529.

\bibitem{Dani2}
\bysame, \emph{On badly approximable numbers, {S}chmidt games and bounded
  orbits of flows}, Number theory and dynamical systems (York, 1987), London
  Math. Soc. Lecture Note Ser., vol. 134, Cambridge Univ. Press, Cambridge,
  1989, pp.~69--86.

\bibitem{DSU}
T.~Das, D.~S. Simmons, and M.~Urba\'nski, \emph{Geometry and dynamics in
  {G}romov hyperbolic metric spaces: with an emphasis on non-proper settings},
  \url{http://arxiv.org/abs/1409.2155}, preprint 2014.

\bibitem{DMPV}
M.~M. Dodson, M.~V. Meli\'an, D.~Pestana, and S.~L. Velani, \emph{{P}atterson
  measure and ubiquity}, Ann. Acad. Sci. Fenn. Ser. A I Math. \textbf{20}
  (1995), no. 1, 37--60.

\bibitem{Edelstein}
M.~Edelstein, \emph{On non-expansive mappings of {B}anach spaces}, Math. Proc.
  Cambridge Philos. Soc. \textbf{60} (1964), 439--447.

\bibitem{EinsiedlerTseng}
M.~Einsiedler and J.~Tseng, \emph{Badly approximable systems of affine forms,
  fractals, and {S}chmidt games}, J. Reine Angew. Math. \textbf{660} (2011),
  83--97.

\bibitem{EvansGariepy}
L.~C. Evans and R.~F. Gariepy, \emph{Measure theory and fine properties of
  functions}, Studies in Advanced Mathematics, CRC Press, Boca Raton, FL, 1992.

\bibitem{Falconer_book}
K.~J. Falconer, \emph{Fractal geometry: {M}athematical foundations and
  applications}, John Wiley \& Sons, Ltd., Chichester, 1990.

\bibitem{FalconerMarsh}
K.~J. Falconer and D.~T. Marsh, \emph{On the {L}ipschitz equivalence of
  {C}antor sets}, Mathematika \textbf{39} (1992), no. 2, 223--233.

\bibitem{Farm}
D.~F\"arm, \emph{Simultaneously non-dense orbits under different expanding
  maps}, Dyn. Syst. \textbf{25} (2010), no. 4, 531--545.

\bibitem{FPS}
D.~F\"arm, T.~Persson, and J.~Schmeling, \emph{Dimension of countable
  intersections of some sets arising in expansions in non-integer bases}, Fund.
  Math. \textbf{209} (2010), 157--176.

\bibitem{FernandezMelian}
J.~L. Fern\'andez and M.~V. Meli\'an, \emph{Bounded geodesics of {R}iemann
  surfaces and hyperbolic manifolds}, Trans. Amer. Math. Soc. \textbf{347}
  (1995), no. 9, 3533--3549.

\bibitem{Fishman}
L.~Fishman, \emph{{S}chmidt's game on fractals}, Israel J. Math. \textbf{171}
  (2009), no. 1, 77--92.

\bibitem{FishmanSimmons5}
L.~Fishman, D.~S. Simmons, and M.~Urba\'nski, \emph{{D}iophantine approximation
  in {B}anach spaces}, \url{http://arxiv.org/abs/1302.2275}, preprint 2013, to
  appear in J. Th\'eor. Nombres Bordeaux.

\bibitem{GreschonigSchmidt}
G.~Greschonig and K.~Schmidt, \emph{Ergodic decomposition of quasi-invariant
  probability measures}, Colloq. Math. \textbf{84/85} (2000), part 2, 493--514.

\bibitem{Gromov3}
M.~Gromov, \emph{Hyperbolic groups}, Essays in group theory, Math. Sci. Res.
  Inst. Publ., vol.~8, Springer, New York, 1987, pp.~75--263.

\bibitem{Gromov4}
\bysame, \emph{Asymptotic invariants of infinite groups}, Geometric group
  theory, Vol. 2 (Sussex, 1991), London Math. Soc. Lecture Note Ser., vol. 182,
  Cambridge Univ. Press, Cambridge, 1993, pp.~1--295.

\bibitem{Gromov5}
\bysame, \emph{{C}arnot-{C}arath\'eodory spaces seen from within},
  Sub-Riemannian geometry, Progr. Math., 144, Birkh\"auser, Basel, 1996,
  pp.~79--323.

\bibitem{HandelMosher}
M.~Handel and L.~Mosher, \emph{The free splitting complex of a free group, {I}:
  hyperbolicity}, Geom. Topol. \textbf{17} (2013), no. 3, 1581--1672.

\bibitem{HPW}
S.~W. Hensel, P.~Przytycki, and R.~C.~H. Webb, \emph{Slim unicorns and uniform
  hyperbolicity for arc graphs and curve graphs},
  \url{http://arxiv.org/abs/1301.5577}, preprint 2013.

\bibitem{HersonskyHubbard}
S.~D. Hersonsky and J.~H. Hubbard, \emph{Groups of automorphisms of trees and
  their limit sets}, Ergodic Theory Dynam. Systems \textbf{17} (1997), no. 4,
  869--884.

\bibitem{HP_Jarnik-Besicovitch_internal}
S.~D. Hersonsky and F.~Paulin, \emph{{H}ausdorff dimension of {D}iophantine
  geodesics in negatively curved manifolds}, J. Reine Angew. Math. \textbf{539}
  (2001), 29--43.

\bibitem{HP_Dirichlet}
\bysame, \emph{{D}iophantine approximation for negatively curved manifolds},
  Math. Z. \textbf{241} (2002), no. 1, 181--226.

\bibitem{HP_Khinchin}
\bysame, \emph{Counting orbit points in coverings of negatively curved
  manifolds and {H}ausdorff dimension of cusp excursions}, Ergodic Theory
  Dynam. Systems \textbf{24} (2004), no. 3, 803--824.

\bibitem{HP_Khinchin_trees}
\bysame, \emph{A logarithm law for automorphism groups of trees}, Arch. Math.
  (Basel) \textbf{88} (2007), no. 2, 97--108.

\bibitem{HilionHorbez}
A.~Hilion and C.~Horbez, \emph{The hyperbolicity of the sphere complex via
  surgery paths}, \url{http://arxiv.org/abs/1210.6183}, preprint 2012.

\bibitem{HillVelani}
R.~M. Hill and S.~L. Velani, \emph{The {J}arn\'ik-{B}esicovitch theorem for
  geometrically finite {K}leinian groups}, Proc. London Math. Soc. (3)
  \textbf{77} (1998), no. 3, 524--550.

\bibitem{KapovichBenakli}
I.~Kapovich and N.~Benakli, \emph{Boundaries of hyperbolic groups},
  Combinatorial and geometric group theory (New York, 2000/Hoboken, NJ, 2001,
  Contemporary Mathematics, vol. 296, American Mathematical Society,
  Providence, RI, 2002, pp.~39--93.

\bibitem{KapovichRafi}
I.~Kapovich and K.~Rafi, \emph{On hyperbolicity of free splitting and free
  factor complexes}, Groups Geom. Dyn. \textbf{8} (2014), no. 2, 391--414.

\bibitem{KleinbockMargulis}
D.~Y. Kleinbock and G.~A. Margulis, \emph{Logarithm laws for flows on
  homogeneous spaces}, Invent. Math. \textbf{138} (1999), no. 3, 451--494.

\bibitem{KleinbockWeiss1}
D.~Y. Kleinbock and B.~Weiss, \emph{Badly approximable vectors on fractals},
  Israel J. Math. \textbf{149} (2005), 137--170.

\bibitem{KleinbockWeiss2}
\bysame, \emph{Modified {S}chmidt games and {D}iophantine approximation with
  weights}, Advances in Math. \textbf{223} (2010), 1276--1298.

\bibitem{KleinbockWeiss3}
\bysame, \emph{Modified {S}chmidt games and a conjecture of {M}argulis},
  \url{http://arxiv.org/abs/1001.5017}, preprint 2010.

\bibitem{Lundh}
T.~Lundh, \emph{Geodesics on quotient manifolds and their corresponding limit
  points}, Michigan Math. J. \textbf{51} (2003), no. 2, 279--304.

\bibitem{MackayTyson}
J.~M. Mackay and J.~T. Tyson, \emph{Conformal dimension: Theory and
  application}, University Lecture Series, 54, American Mathematical Society,
  Providence, RI, 2010.

\bibitem{MasurSchleimer}
H.~A. Masur and S.~Schleimer, \emph{The geometry of the disk complex}, J. Amer.
  Math. Soc. \textbf{26} (2013), no. 1, 1--62.

\bibitem{MauldinUrbanski1}
R.~D. Mauldin and M.~Urba\'nski, \emph{Dimensions and measures in infinite
  iterated function systems}, Proc. London Math. Soc. (3) \textbf{73} (1996),
  no. 1, 105--154.

\bibitem{MauldinUrbanski2}
\bysame, \emph{Graph directed {M}arkov systems: Geometry and dynamics of limit
  sets}, Cambridge Tracts in Mathematics, vol. 148, Cambridge University Press,
  Cambridge, 2003.

\bibitem{MayedaMerrill}
D.~Mayeda and K.~Merrill, \emph{Limit points badly approximable by horoballs},
  Geom. Dedicata \textbf{163} (2013), 127--140.

\bibitem{McMullen_absolute_winning}
C.~T. McMullen, \emph{Winning sets, quasiconformal maps and {D}iophantine
  approximation}, Geom. Funct. Anal. \textbf{20} (2010), no. 3, 726--740.

\bibitem{MelianPestana}
M.~V. Meli\'an and D.~Pestana, \emph{Geodesic excursions into cusps in
  finite-volume hyperbolic manifolds}, Michigan Math. J. 40 (1993), no. 1,
  77--93.

\bibitem{Moshchevitin}
N.~G. Moshchevitin, \emph{A note on badly approximable affine forms and winning
  sets}, Moscow Math. J \textbf{11} (2011), no. 1, 129--137.

\bibitem{Newberger}
F.~Newberger, \emph{On the {P}atterson-{S}ullivan measure for geometrically
  finite groups acting on complex or quaternionic hyperbolic space ({S}pecial
  volume dedicated to the memory of {H}anna {M}iriam {S}andler (1960-1999))},
  Geom. Dedicata \textbf{97} (2003), 215--249.

\bibitem{Oshika}
K.~Ohshika, \emph{Discrete groups}, Translations of Mathematical Monographs,
  vol. 207, American Mathematical Society, Providence, RI, 2002.

\bibitem{Olshanskii}
A.~Y. Ol'shanski\u\i, \emph{Almost every group is hyperbolic}, Internat. J.
  Algebra Comput. \textbf{2} (1992), no. 1, 1--17.

\bibitem{Patterson1}
S.~J. Patterson, \emph{{D}iophantine approximation in {F}uchsian groups},
  Philos. Trans. Roy. Soc. London Ser. A \textbf{282} (1976), no. 1309,
  527--563.

\bibitem{Patterson2}
\bysame, \emph{The limit set of a {F}uchsian group}, Acta Math. \textbf{136}
  (1976), no. 3--4, 241--273.

\bibitem{Patterson3}
\bysame, \emph{Further remarks on the exponent of convergence of {P}oincar\'e
  series}, Tohoku Math. J. (2) \textbf{35} (1983), no. 3, 357--373.

\bibitem{Paulin}
F.~Paulin, \emph{On the critical exponent of a discrete group of hyperbolic
  isometries}, Differential Geom. Appl. \textbf{7} (1997), no. 3, 231--236.

\bibitem{Paulin2}
\bysame, \emph{Groupes g\'eom\'etriquement finis d'automorphismes d'arbres et
  approximation diophantienne dans les arbres ({G}eometrically finite groups of
  tree automorphisms and {D}iophantine approximation in trees)}, Manuscripta
  Math. \textbf{113} (2004), no. 1, 1--23 (French).

\bibitem{RiediMandelbrot}
R.~H. Riedi and B.~B. Mandelbrot, \emph{Multifractal formalism for infinite
  multinomial measures}, Adv. in Appl. Math. \textbf{16} (1995), no. 2,
  132--150.

\bibitem{Roblin1}
T.~Roblin, \emph{Ergodicit\'e et \'equidistribution en courbure n\'egative
  ({E}rgodicity and uniform distribution in negative curvature)}, M\'em. Soc.
  Math. Fr. (N. S.) \textbf{95} (2003), vi+96 pp. (French).

\bibitem{Schapira}
B.~Schapira, \emph{Lemme de l'ombre et non divergence des horosph\'eres d'une
  vari\'et\'e g\'eom\'etriquement finie. (the shadow lemma and nondivergence of
  the horospheres of a geometrically finite manifold)}, Ann. Inst. Fourier
  (Grenoble) \textbf{54} (2004), no. 4, 939--987 (French).

\bibitem{Schmidt1}
W.~M. Schmidt, \emph{On badly approximable numbers and certain games}, Trans.
  Amer. Math. Soc. \textbf{123} (1966), 27--50.

\bibitem{Schmidt2}
\bysame, \emph{Badly approximable systems of linear forms}, J. Number Theory
  \textbf{1} (1969), 139--154.

\bibitem{Stratmann3}
B.~O. Stratmann, \emph{{D}iophantine approximation in {K}leinian groups}, Math.
  Proc. Cambridge Philos. Soc. \textbf{116} (1994), no. 1, 57--78.

\bibitem{Stratmann4}
\bysame, \emph{Fractal dimensions for {J}arn\'ik limit sets of geometrically
  finite {K}leinian groups; the semi-classical approach}, Ark. Mat. \textbf{33}
  (1995), no. 2, 385--403.

\bibitem{Stratmann1}
\bysame, \emph{The {H}ausdorff dimension of bounded geodesics on geometrically
  finite manifolds}, Ergodic Theory Dynam. Systems \textbf{17} (1997), no. 1,
  227--246.

\bibitem{StratmannVelani}
B.~O. Stratmann and S.~L. Velani, \emph{The {P}atterson measure for
  geometrically finite groups with parabolic elements, new and old}, Proc.
  London Math. Soc. (3) \textbf{71} (1995), no. 1, 197--220.

\bibitem{Sullivan_density_at_infinity}
D.~P. Sullivan, \emph{The density at infinity of a discrete group of hyperbolic
  motions}, Inst. Hautes \'Etudes Sci. Publ. Math. \textbf{50} (1979),
  171--202.

\bibitem{Sullivan_disjoint_spheres}
\bysame, \emph{Disjoint spheres, approximation by imaginary quadratic numbers,
  and the logarithm law for geodesics}, Acta Math. \textbf{149} (1982), no.
  3--4, 215--237.

\bibitem{Sullivan_entropy}
\bysame, \emph{Entropy, {H}ausdorff measures old and new, and limit sets of
  geometrically finite {K}leinian groups}, Acta Math. \textbf{153} (1984), no.
  3--4, 259--277.

\bibitem{Tseng2}
J.~Tseng, \emph{Badly approximable affine forms and {S}chmidt games}, J. Number
  Theory \textbf{129} (2009), 3020--3025.

\bibitem{Tseng1}
\bysame, \emph{{S}chmidt games and {M}arkov partitions}, Nonlinearity
  \textbf{22} (2009), no. 3, 525--543.

\bibitem{Vaisala}
J.~V\"ais\"al\"a, \emph{{G}romov hyperbolic spaces}, Expo. Math. \textbf{23}
  (2005), no. 3, 187--231.

\bibitem{Vaisala2}
\bysame, \emph{Hyperbolic and uniform domains in {B}anach spaces}, Ann. Acad.
  Sci. Fenn. Math. \textbf{30} (2005), no. 2, 261--302.

\bibitem{Velani1}
S.~L. Velani, \emph{{D}iophantine approximation and {H}ausdorff dimension in
  {F}uchsian groups}, Math. Proc. Cambridge Philos. Soc. \textbf{113} (1993),
  no. 2, 343--354.

\bibitem{Velani2}
\bysame, \emph{An application of metric {D}iophantine approximation in
  hyperbolic space to quadratic forms}, Publ. Mat. \textbf{38} (1994), no. 1,
  175--185.

\bibitem{Velani3}
\bysame, \emph{Geometrically finite groups, {K}hintchine-type theorems and
  {H}ausdorff dimension}, Math. Proc. Cambridge Philos. Soc. \textbf{120}
  (1996), no. 4, 647--662.

\bibitem{Young2}
L.-S. Young, \emph{Dimension, entropy and {L}yapunov exponents}, Ergodic Theory
  Dynam. Systems \textbf{2} (1982), no. 1, 109--124.

\end{thebibliography}

\end{document}